%% file: 1main.tex
\documentclass[12pt]{amsart}

\usepackage[T2A]{fontenc}
\usepackage[utf8]{inputenc}

\usepackage{amssymb}
\usepackage{fullpage}
\usepackage[mathcal]{euscript}

\usepackage[x11names]{xcolor}
\usepackage{graphicx}

\usepackage{mdframed}

\mdfsetup{
  linewidth=0pt,
  innerleftmargin=0pt,
  innerrightmargin=0pt,
  innertopmargin=3.5pt,
  innerbottommargin=0pt,
  startinnercode={\hskip\parindent}
}

\newmdenv[backgroundcolor=blue!2!white]{raw1}
\newmdenv[backgroundcolor=yellow!3!white]{raw2}
\newmdenv[backgroundcolor=orange!3!white]{raw3}
\newmdenv[backgroundcolor=red!20!white]{raw4}

\newmdenv[backgroundcolor=blue!6!white,innertopmargin=0pt]{bgassert1}

\usepackage{soul} 
\setstcolor{red} 
\sethlcolor{Yellow1} 

\makeatletter

\def\subsection{\@startsection{subsection}{2}%
  \z@{.5\linespacing\@plus.7\linespacing}{-.4em}%
  {\normalfont\itshape}}

\def\subsubsection{\@startsection{subsubsection}{3}%
  \z@{.5\linespacing\@plus.7\linespacing}{-.4em}%
  {\normalfont\itshape\color{blue}}}

\makeatother

\usepackage{hyperref}

\newtheorem{theorem}{Theorem}

\newtheorem*{maincorollary*}{Corollary}

\swapnumbers
\newtheorem{lemma}[subsection]{Lemma}
\newtheorem{proposition}[subsection]{Proposition}
\newtheorem{corollary}[subsection]{Corollary}
\newtheorem*{corollary*}{Corollary}

\newtheorem*{proposition-wn}{Proposition \theorem-ref}
\newenvironment{proposition-cited}[1]{\def\theorem-ref{#1}\begin{proposition-wn}}{\end{proposition-wn}}

\newtheorem*{corollary-wn}{Corollary \theorem-ref}
\newenvironment{corollary-cited}[1]{\def\theorem-ref{#1}\begin{corollary-wn}}{\end{corollary-wn}}

\theoremstyle{definition}
\newtheorem{definition}[subsection]{Definition}

\theoremstyle{remark}
\newtheorem{remark}[subsection]{Remark}

\newtheorem*{remark*}{Remark}

\def\theenumi{\roman{enumi}}  

\makeatletter
\@addtoreset{equation}{section}

\newcounter{claim}[subsection]

\makeatother

\newenvironment{claim}%
{\refstepcounter{claim}%
   \trivlist
   \item[\hskip\labelsep
   {\it Claim \theclaim:}]\ignorespaces\it}
{\endtrivlist}

\newenvironment{claim*}%
{%
   \trivlist
   \item[\hskip\labelsep
   {\it Claim:}]\ignorespaces\it}
{\endtrivlist}

\makeatletter
\@addtoreset{enumi}{subsection}
\makeatother

\newenvironment{enum-par}%
{\refstepcounter{enumi}%
   \trivlist
   \item[\hskip\labelsep
   {\theenumi}]\ignorespaces}
{\endtrivlist}

\newenvironment{labeled-par}[1]%
{%
   \trivlist
   \item[\hskip\labelsep
   {#1}]\ignorespaces}
{\endtrivlist}

\let\al\alpha
\let\be\beta
\let\ga\gamma
\let\de\delta
\let\ep\varepsilon
\let\ze\zeta
\let\th\theta 
\let\io\iota
\let\ka\kappa
\let\la\lambda

\let\om\omega

\let\Ga\Gamma
\let\De\Delta
\let\Th\Theta
\let\La\Lambda
\let\Si\Sigma
\let\Om\Omega

\let\seq\subseteq
\let\sm\setminus
\let\ti\tilde
\let\bd\partial
\let\lequiv\Leftrightarrow
\let\greq\eqcirc
\let\isom\simeq

\mathcode`\*="0203 	

\mathcode`\!="8000
{\catcode`\!=\active \gdef!#1{\mathsf{#1}}}

\def\cA{\mathcal A}

\def\cE{\mathcal E}

\def\cF{\mathcal F}

\def\cH{\mathcal H}

\def\cM{\mathcal M}

\def\cP{\mathcal P}
\def\cR{\mathcal R}
\def\cS{\mathcal S}
\def\cT{\mathcal T}

\def\cX{\mathcal X}

\def\Z{\mathbb Z}
\def\R{\mathbb R}

\newcommand\setof[2]{\{ #1 \;|\; #2 \}}

\newcommand\set[1]{ \{ #1 \} }
\newcommand\bbset[1]{ \left\{ #1 \right\} }
\newcommand\spresn[2]{\langle #1 \, | \, #2 \rangle}

\newcommand\bpresn[2]{\bigl\langle #1 \;\bigm|\; #2 \bigr\rangle}

\newcommand\sgp[1]{\langle #1 \rangle}

\newcommand\ceil[1]{\lceil #1 \rceil}
\newcommand\bbceil[1]{\left\lceil #1 \right\rceil}
\newcommand\floor[1]{\lfloor #1 \rfloor}
\newcommand\bbfloor[1]{\left\lfloor #1 \right\rfloor}

\DeclareMathOperator{\per}{Per}

\DeclareMathOperator{\lab}{\mathit{label}}
\DeclareMathOperator{\Area}{Area}

\newcommand\muf{\mu_{\mathrm f}}

\begin{document}

%
\def\theenumi{(\roman{enumi})}
\makeatletter
\def\p@enumi{\thelemma}
\def\labelenumi{\theenumi}
\let\@savedlabel\label
\def\label#1{\@savedlabel{#1}\ifnum\@listdepth=1%
\protected@edef\@currentlabel{\theenumi}\@savedlabel{#1-it@m}\fi}
\def\itemref#1{\ref{#1-it@m}}
\makeatother

\title{A sample iterated small cancellation theory for groups of Burnside type}
\author{Igor Lysenok}
\date{}
\thanks{This research was supported by the Russian Science Foundation (project No.~21-11-00318)}

\begin{abstract}
We develop yet another technique to 
present the free Burnside group $B(m,n)$ of odd exponent $n$ with $m\ge2$
generators as a group satisfying a certain iterated small cancellation condition.
Using the approach, 
we provide a reasonably accessible proof that $B(m,n)$ is infinite with a moderate bound $n > 2000$
on the odd exponent $n$.
\end{abstract}

\maketitle

\section{Introduction}

The free $m$-generated Burnside group $B(m,n)$ of exponent $n$ is, by definition, the relatively free 
group in the variety of groups satisfying the identity $x^n =1$, i.e.\
$
  B(m,n) \isom F_m / F_m^n
$
where $F_m$ is the free group of rank $m$ and $F_m^n$ is the subgroup of $F_m$ generated by all $n$-th powers.
Obtaining a structural information about groups $B(m,n)$ is known to be a difficult problem.
The primary question of this sort is whether $B(m,n)$ is finite for given $m, n \ge 2$.
The question is known as the {\em Burnside problem} \cite{Bur02} and it is still not completely answered. 
The group is shown to be finite for exponents $n=2$,~$3$ \cite{Bur02}, $n=4$ \cite{San40} and $n=6$ \cite{Hal58}.
A negative solution to the Burnside problem is given
by the Novikov--Adian theorem \cite{NovAdi68,Adi75} stating that the Burnside group $B(m,n)$ 
of odd exponent $n \ge  665$ with $m \ge 2$ generators  is infinite. 
As for now, infiniteness 
of $B(m,n)$ is established for exponents 
of the form $n=665r$ or $n \ge 8000$ and any number $m\ge 2$ of generators.
Note that $B(m,r)$ is a homomorphic image of $B(m,n)$ if $n$ is a multiple of $r$,
so in this case infiniteness of $B(m,r)$ implies infiniteness of $B(m,n)$.
The case when the exponent $n$ does not have a large odd divisor was treated in \cite{Iva94,Lys96}. 
Although it is believable that free Burnside groups $B(m,n)$ are infinite for considerably lower values of $n$
(and there are several announcements of results of this sort) the lowest published and carefully checked 
bound is still 665, obtained by Adian \cite{Adi75} for the case of odd exponent ~$n$.

A principal step in understanding the structure of the group $B(m,n)$ in the infinite case
was made in the fundamental work by Novikov and Adian \cite{NovAdi68} and its
improved version ~\cite{Adi75}. 
One of the ingredients of the proof was a tightly interweaved 
version of the small cancellation theory similar to one developed by Tartakovski\u\i\ \cite{Tar49}.
It was also shown in \cite{Adi75} that for $m \ge 2$ and odd $n \ge 665$ the group $B(m,n)$ has several 
properties similar to key properties of small cancellation groups. A basic one is 
{\em layered Dehn's property:} a freely reduced nonempty word representing the identity 
in the group contains 
a large part of a defining relator modulo relations of the previous layer.
This easily implies that any such word should contain a subword of the form $X^t$ 
for sufficiently large $t$ which in turn implies that $B(m,n)$ is infinite.

Unfortunately, the approach due to Novikov--Adian, even in its polished and improved form in ~\cite{Adi75},
is extremely technical and has a complicated logical structure.
Several later works \cite{Ols82,Ols91,DelGro08,Cou14} pursued 
the goal to find a more conceptually explicit and 
technically simpler approach to infinite Burnside groups, and more generally, to 
``infinite quotient of bounded exponent'' phenomena in wider classes of groups
as in \cite{IvaOls96,DelGro08,Cou14}.
As an underlying basic idea, all these approaches utilize a small cancellation
theory in a more or less explicit form though based on different implementation techniques.
It was eventually realized that iterated small cancellation theory
is indeed a relevant framework to present Burnside groups of large exponents as well as many other 
examples of infinitely presented groups of a ``monster'' nature.
In an explicit form, a relevant version of the theory 
was formulated by Gromov and Delzant \cite{DelGro08} and Coulon \cite{Cou14}. 
However, both approaches need extremely large exponents to be applied to Burnside groups.
(In fact, the both incorporate ``non-constructive'' tools 
so that the proof does not provide any explicit lower bound on the exponent $n$.)

Two questions naturally arise. 
What is the lower bound on the exponent $n$ for which the iterated small cancellation 
approach can be applied to Burnside groups $B(m,n)$?
Do we need a sophisticated technical framework to use the approach for reasonably small values of the exponent;
for example, for values which are about several hundreds or less? 

The main goal of the present paper is to develop a sample version of the iterated small cancellation theory
specially designed for free Burnside groups $B(m,n)$ with a ``moderate'' lower bound on the exponent $n$.
More precisely, our technique works for odd exponents $n > 2000$.

We consider our approach as a first approximation and an introduction to 
a considerably more technical result on infiniteness of Burnside groups with substantially smaller
bounds on the exponent.

\section{The iterated small cancellation condition} \label{s:condition}

\subsection{} \label{ss:relators}
We fix a group $G$ given by a graded presentation 
\begin{equation} \label{eq:G-presn}
  \bpresn{\cA}{R= 1 \ (R \in \bigcup_{\al\ge 1} \cX_\al)}.
\end{equation}
Here we assume that the set of defining relators is partitioned into the union of subsets $\cX_\al$
indexed by a positive integer ~$\al$. 
We call cyclic shifts of words $R \in \cX_\al^{\pm1}$ {\em relators of rank ~$\al$}.
Thus, the set of all relators of rank ~$\al$ is symmetrized, i.e.\ closed under cyclic shifts and taking inverses.

With the presentation of $G$, there are naturally associated 
{\em level groups} $G_\al$ defined by all relations of rank up to $\al$, i.e.\
\begin{equation} \label{eq:G-al-presn}
  G_\al = \bpresn{\cA}{R= 1 \ (R \in \bigcup_{\be\le \al} \cX_\be)}
\end{equation}

\subsection{} \label{ss:parameter-roles}
Our small cancellation condition depends on two positive real-valued parameters $\la$ and ~$\Om$
satisfying 
\begin{equation} \label{eq:ISC-main}
  \la \le \frac{1}{24}, 
  \quad \la \Om \ge 20.
\end{equation}
We introduce also two other parameters with fixed value:
$$
  \rho = 1 - 9\la, \quad \ze = \frac{1}{20}.
$$ 
The role of $\la$, $\Om$, $\rho$ and $\ze$ can be described as follows:
\begin{itemize}
\item
$\la$ is an analog of the small cancellation parameter in the classical condition $C'(\la)$;
\item
$\Om$ is the lower bound on the size of a relator $R$ of rank $\al$ in terms of the 
length function $|\cdot|_{\al-1}$ associated with $G_{\al-1}$ (defined below in \ref{ss:alpha-length});
see condition (S1) in ~\ref{ss:condition}.
\item 
$\rho$ is the reduction threshold
used in the definition of a reduced in $G_\al$ word.
Informally, a reduced in $G_\al$ word cannot contain more that $\rho$-th part of a relator of rank ~$\al$
up to closeness in $G_{\al-1}$.
\item
$\ze$ is the rank scaling factor; it determines how the function $|\cdot|_\al$ rescales when incrementing the rank.
\end{itemize}

\subsection{} \label{ss:bridge-word}
For any $\al \ge 0$, 
we introduce the set $\cH_\al$ of {\em bridge words of rank $\al$} recursively by setting
$$
  \cH_0 = \set{\text{the empty word}},
$$ 
$$
  \cH_{\al} = \setof{u S v}{u,v \in \cH_{\al-1}, \ S \text{  is a subword of a relator of rank } \al}.
$$
The definition immediately implies that $\cH_{\al-1} \seq \cH_{\al}$. 
Note also that all sets $\cH_\al$ are closed under taking inverses.

\subsection{} \label{ss:close-words}

We call two elements $x,y \in G_\al$ {\em close} if $x = uyv$ for some $u,v \in \cH_\al$. 
This relation will be often used in the case when $x$ and $y$ are represented by words in the generators $\cA$.
In that case we say that words $X$ and $Y$ are {\em close in rank $\al$} if 
they represent close elements of ~$G_\al$, or, equivalently,
$X = uYv$ in $G_\al$ for some $u,v \in \cH_\al$.

\subsection{} \label{ss:reduced-word} 
For $\al\ge0$, the set $\cR_\al$ of words {\em reduced in $G_\al$}, 
the set of {\em fragments of rank $\al$} 
and the length function $|\cdot|_\al$ are defined by joint recursion.

A word $X$ in the generators $\cA$ is {\em reduced in $G_0$} if $X$ is freely reduced.
A word $X$ is {\em reduced in $G_\al$} for $\al\ge 1$ if it is reduced in $G_{\al-1}$ and the following is true:
if a subword $S$ of a relator ~$R$ of rank $\al$ is close in rank $\al-1$
to a subword of $X$ then 
$$
  |S|_{\al-1}  \le \rho |R|_{\al-1}.
$$
A word $X$ is {\em cyclically reduced in $G_\al$}
if any cyclic shift of $X$ is reduced in $G_\al$.

\subsection{} \label{ss:fragment}
A nonempty word $F$ is a {\em fragment of rank $\al\ge 1$} if $F$ 
is reduced in $G_{\al-1}$
and is close in rank $\al-1$ to a subword $P$ of a word of the form $R^k$ where $R$ is a relator of rank $\al$. 
(In almost all situations $P$ will be a subword of a cyclic shift of $R$.)
A {\em fragment of rank 0} is a word of length 1, i.e.\ a single letter of the alphabet $\cA^{\pm1}$.

It is convenient to assume that each fragment $F$ of rank $\al\ge 1$ is considered with
fixed associated words $P$, $u$, $v$ and a relator $R$ of rank $\al$
such that $F = uPv$ in $G_{\al-1}$, $u,v \in \cH_{\al-1}$ and $P$ is a subword of $R^k$ for some $k > 0$, i.e.\ a fragment
is formally a quintuple $(F,P,u,v,R)$.

\subsection{} \label{ss:alpha-length}
A {\em fragmentation of rank $\al$} of a (linear or cyclic) word $X$ is a partition of $X$ into 
nonempty subwords of fragments of ranks $\be \le \al$.
If $\cF$ is a fragmentation of rank $\al$ of $X$ then by definition, 
the {\em weight of $\cF$ in rank $\al$}
is defined by
$$
  \text{weight}_\al(\cF) = m_\al + \ze m_{\al-1} + \ze^2 m_{\al-2} + \dots + \ze^\al m_0 
$$
where $m_\be$ is the number of subwords of fragments of rank $\be$ in $\cF$.
Here we assume that each subword in $\cF$ is assigned a unique rank $\be$.

We now define a semi-additive length function $|\cdot|_\al$ on words in the generators $\cA$:
$$
  |X|_\al = \min\setof{\text{weight}_\al(\cF)}{\text{$\cF$ is a fragmentation of rank $\al$ of $X$}}.
$$
Note that $|X|_0$ is the usual length $|X|$ of $X$.

\subsection{} \label{ss:condition}
The iterated small cancellation condition consists of the following three conditions (S0)--(S3)
where the quantifier `for all $\al\ge1$' is assumed.

\begin{itemize}
\item[(S0)] (``Relators are reduced'')
Any relator of rank $\al$ is cyclically reduced in $G_{\al-1}$.
\item[(S1)] (``Relators are large'')
Any relator $R$ of rank $\al$ satisfies
$$
  |R|_{\al-1} \ge \Om.
$$
\item[(S2)] (``Small overlapping'')
For $i=1,2$, 
let $S_i$ be a starting segment of a relator $R_i$ of rank ~$\al$.
Assume that $S_1 = u S_2 v$ in $G_{\al-1}$ 
for some $u,v \in \cH_{\al-1}$ and $|S_1|_{\al-1} \ge \la |R_1|_{\al-1}$.
Then $R_1 = u R_2 u^{-1}$  in $G_{\al-1}$.
\end{itemize}

\subsection{} \label{ss:no-inverse-conjugation}
It can be proved that a group $G$ satisfying conditions (S0)--(S2)
possesses core properties of small cancellation groups, in particular, a version of Dehn's property.
We will impose, however, an extra condition on the graded presentation of $G$ which 
implies cyclicity of all finite subgroups of groups $G_\al$ and avoids difficulties
caused by existence of non-cyclic finite subgroups
in the case of Burnside groups $B(m,n)$ of even exponent ~$n$.

\begin{itemize}
\item[(S3)] (``No inverse conjugate relators'')
No relator of rank $\al$ is conjugate in $G_{\al-1}$ to its inverse.
\end{itemize}

As we see below, this condition is satisfied if each relator $R$ of rank $\al$ has the form $R_0^n$
where the exponent $n$ (which can vary for different $R$) is odd and $R_0$ is a non-power in $G_{\al-1}$.
See Corollary \ref{co:no-involutions-no-inverse-conjugates}.

Starting from Section \ref{s:fragments}, we will use a mild extra assumption on 
the graded presentation \eqref{eq:G-presn} by requiring it to be normalized in the following sense.
The assumption is not essential and just makes 
arguments simpler (mainly due to Lemma ~\ref{lm:compatible-lines}) slightly improving
bounds on the constants.

\begin{definition} \label{df:normalized-presentation}
We call a graded presentation \eqref{eq:G-presn}
{\em normalized} if the following assertions hold:
\begin{enumerate}
\item
Every relator $R \in \cX_\al$ has the form $R \eqcirc R_0^t$ where $R_0$ represents
a non-power element of $G_{\al-1}$ (i.e.\ $R_0$ does not represent in $G_{\al-1}$ an element
of the form $g^k$ for $k \ge 2$); we call $R_0$ the {\em root} of a relator $R$.
\item
If $R, S \in \cX_\al$ and $R \ne S$ then $R$ and $S$ are not conjugate in $G_{\al-1}$.
\end{enumerate}
\end{definition}
Note that the condition to be normalized is not restrictive: 
every graded presentation can be replaced with a normalized one
(although formally speaking, this replacement could affect the iterated small cancellation condition; however,
in real applications this would hardly be the case).

\begin{remark*}
Checking conditions (S0)--(S3) requires knowledge about groups $G_{\al-1}$. 
Thus presenting a group by relations satisfying the iterated small cancellation condition already
requires a proof of properties of groups $G_\al$ by induction on the rank.
\end{remark*}

\section{Main results} \label{s:main-results}

As in the case of classical small cancellation, the iterated 
small cancellation condition has strong consequences on the presented group $G$. 
A basic one is an analog of the Dehn property: every non-empty freely reduced
word representing the trivial element of the group ``contains a large part'' of a relator. 

In what follows,  
we assume that a group $G$ is given by a normalized 
graded presentation satisfying conditions (S0)--(S3) above 
and for any $\al\ge0$, $G_\al$ denotes the group defined by all 
relations of ranks up to $\al$.
We say that a word $X$ is {\em reduced in $G$} if it is reduced in $G_\al$ for all
$\al \ge 0$. The following theorem is an immediate consequence of Proposition \ref{pr:reduced-nontrivial}.

\begin{theorem} \label{th:reduced-nontrivial}
Let $X$ be a non-empty word in the generators $\cA$. 
If $X$ reduced in $G_\al$ then $X \ne 1$ in $G_\al$.
If $X$ is reduced in $G$ then $X \ne 1$ in $G$.
\end{theorem}

By expanding the definition of a reduced word in $G$ we get an equivalent
formulation which is more in the spirit of the small cancellation theory.

\begin{maincorollary*}
Let $X$ be a freely reduced non-empty word. If $X = 1$ in $G$ then for some 
$\al\ge1$, $X$ has a subword close in $G_{\al-1}$ to a subword $P$ of a relator $R$
of rank $\al$ with $ |P|_{\al-1} \ge \rho |R|_{\al-1}$.
\end{maincorollary*}

In the classical small cancellation theory, existence of a Dehn reduced 
representatives for group elements is a simple consequence of the fact that
a word containing more than a half of a relator can be shortened by applying 
the corresponding relation. This approach does not work in our version of 
the iterated small cancellation and existence of reduced representatives
is a nontrivial fact proved below and formulated in 
Proposition \ref{pr:reduction} and Corollary \ref{co:absolute-reduction}.

\begin{theorem} \label{th:absolute-reduction}
Every element of $G_\al$ can be represented by a word reduced in $G_\al$. 
Every element of $G$ can be represented by a word reduced in $G$. 
\end{theorem}

Many other properties of groups $G_\al$ and $G$ are established in 
Sections \ref{s:diagrams}--\ref{s:compare-close}.
Our principal result shows that
our version of the iterated small cancellation theory can be applied to 
free Burnside groups of odd exponent $n$ with a moderate lower bound on $n$.
The following theorem is a consequence of 
Propositions \ref{pr:G-small-cancellation} and Corollary \ref{co:G-Burnside} 
(see also Remark \ref{rm:G-normalized}).

\begin{theorem} \label{th:main}
For odd $n > 2000$ and $m \ge 2$, 
the free Burnside group $B(n,m)$ has a normalized graded presentation 
$$
    \bpresn{\cA}{C^n= 1 \ (C \in \bigcup_{\al\ge 1} \cE_\al)}
$$
satisfying conditions (S0)--(S3) with $\la = \frac{80}{n}$, $\Om = 0.25 n$. 
\end{theorem}

The following theorem is a 
well known property of Burnside groups of sufficiently large odd exponent.
It is direct consequence of Propositions \ref{pr:nonidentity-active-piece} and \ref{pr:length-in-periods} 
(the definition of $\om$ is given in \ref{ss:extra-parameters}).

\begin{theorem} \label{th:trivial-has-periodic}
Let $n > 2000$ be odd. 
Let $X$ be a non-empty freely reduced word that is 
equal~1 in $B(m,n)$.
Then $X$ has a subword of the form $C^{480}$ where 
$C \in \bigcup_{\al\ge 1} \cE_\al$.
\end{theorem}

Note that, with  existence of infinite aperiodic words in the 2-letter alphabet
(see for example \cite[\S I.3]{Adi75})
this implies infiniteness of $B(n,m)$ for odd $n > 2000$ and $m\ge 2$.

\subsection*{Some remarks}
The present approach has much in common with paper \cite{Lys96}. However, the approach in \cite{Lys96} 
was based on  the assumption that defining relations of the group under consideration are of the  form $x^n=1$ for sufficiently large $n$. 
Although the general scheme of a large portion of our proofs is the same as in \cite{Lys96}, 
our arguments are in different technical environment.

We tried to make the iterated small cancellation condition as simple possible. 
In particular, we use a simple version of closeness in groups $G_\al$ 
(see \ref{ss:bridge-word} and \ref{ss:close-words}).
However, when presenting the free Burnside
group as an iterated small cancellation group, this version is not optimal for the bound on the exponent. 
A more refined version would significantly lower the bound. 
Nevertheless, we consider the bound $n > 2000$ on the exponent as a 
reasonable balance between its optimality and the complexity of definitions and proofs. 

The whole approach relies essentially on the simultaneous induction on the rank $\al$. 
Since the proof of required statements about groups $G_\al$ needs a comprehensive
analysis of certain types of relations in groups of previous ranks, the number of 
inductive hypotheses in quite large (several tens). We think that a large number of inductive hypotheses
is an unavoidable feature of any ``small cancellation'' approach to infinite Burnside groups 
with a reasonably small lower bound on the exponent. 
Note that in the ``basic'' small cancellation theory in Sections 
\ref{s:diagrams}--\ref{s:diagram-bounds} we use 
Proposition \ref{pr:principal-bound-sides}
(with its immediate consequence Proposition \ref{pr:principal-bound})
as the only inductive hypothesis.

We briefly mention essential ingredients of our approach.

Sections \ref{s:diagrams}--\ref{s:diagram-bounds} are devoted to analysis of van Kampen diagrams over the presentation \eqref{eq:G-al-presn} of the group $G_\al$.
In \ref{df:boundary-marking} we introduce diagrams with a special marking of the boundary 
so that the boundary loops
of a diagram are divided into sides and bridges. 
The label of a side is a word reduced in $G_\al$ and bridges
are ``small'' sections between sides labeled by bridge words of rank $\al$.
According to the marking, there are diagrams of bigon, trigon, etc.\ type.
We then analyze a global structure of a diagram with marked boundary using the notion of
contiguity subdiagram (see \ref{df:contiguity-subdiagram}). 
For the quantitative analysis, we use a version of discrete connection in the spirit of 
\cite{Lys18} and the corresponding discrete analog 
of the Gauss-Bonnet formula (Proposition \ref{pr:Gauss-Bonnet}).
The main outcomes are the bound on the total size of sides of a diagram with no 
bonds (Propositions \ref{pr:principal-bound} and ~\ref{pr:small-trigons-tetragons})
and the ``single layered'' structure of diagrams of small complexity 
(Propositions ~\ref{pr:single-layer} and ~\ref{pr:small-complexity-cell}).

The results of Sections \ref{s:diagrams}--\ref{s:diagram-bounds}
serve as a background for further analysis of relations in ~$G_\al$.
The most important type of relations under consideration are ``closeness'' relations in $G_\al$ of 
the form $X = uYv$ where $X,Y \in \cR_\al$ and $u,v \in \cH_\al$.
The structural description of diagrams over the presentation of $G_\al$ 
transfers naturally to the language of the Cayley graph $\Ga_\al$ of ~$G_\al$, 
see \ref{ss:active-fragments}.
In $\Ga_\al$, words in the generators of the group are represented by paths and 
relations in ~$G_\al$ are represented by loops.
The relation $X = uYv$ becomes a loop $!X^{-1} !u !Y !v$ in ~$\Ga_\al$ which 
can be viewed as a coarse bigon; we say also that paths $!X$ and $!Y$ are close.
The single layered structure of the filling diagram implies 
one-to-one correspondence between 
fragments of rank ~$\al$ in $!X$ and in $!Y$ that
come from the 2-cells of the diagram, called {\em active} fragments
of rank ~$\al$ with respect to the coarse bigon $!X^{-1} !u !Y !v$.
To express the correspondence,
we use the {\em compatibility} relation, defined in \ref{df:fragment-compatibility},
on the set of fragments of rank $\al$ in $\Ga_\al$
(i.e.\ paths in $\Ga_\al$ labeled by fragments of rank ~$\al$): 
if $!K$ and $!M$
are the corresponding active fragments of rank ~$\al$ in $!X$ and $!Y$, respectively, then
$!K$ and  $!M^{-1}$ are compatible (Proposition \ref{pr:active-large}).

In Section \ref{s:relations} we perform this passage from diagrams over the presentation 
of $G_\al$ to the Cayley graph $\Ga_\al$. We establish several properties of coarse bigons, trigons
and more generally, coarse polygons in $\Ga_\al$. 
We consider also conjugacy relations in $G_\al$ which are represented by
parallel infinite lines in $\Ga_\al$ (see \ref{ss:relations}).

A fundamental property of close paths $!X$ and $!Y$ in $\Ga_\al$ 
with $\lab(!X), \lab(!Y) \in \cR_\al$ is that 
the correspondence between fragments of rank $\al$ in $!X$ and $!Y$
extends to non-active ones. If $!K$ is a fragment in $!X$
of sufficiently large size then there exists a fragment of $!M$ of rank $\al$ in $!Y$
such that $!K$ is compatible with either $!M$ or $!M^{-1}$, with possible exceptions 
of extreme positions of $!K$ in $!X$
(Proposition ~\ref{pr:fragment-stability-bigon}).
Speaking informally, fragments of rank $\al$ play the role of letters when coincidence of words
is replaced by closeness in $G_\al$. This property of close paths $!X$ and $!Y$ in $\Ga_\al$
and its analogs for coarse trigons in $G_\al$ (Proposition \ref{pr:fragment-stability-trigon})
and for conjugacy relations in $G_\al$ (Propositions \ref{pr:fragment-stability-cyclic} 
and \ref{pr:fragment-stability-cyclic1}) 
provide a technical base to analyze further properties of groups $G_\al$ and $G$. 
In particular, the correspondence between fragments of rank $\al$ in coarse bigons,
under an appropriate adaptation, is crucial 
when we consider in Section \ref{s:overlapping-periodicity} close in $G_\al$ periodic words.

In Section \ref{s:reduction} we prove that any element of $G_\al$ can be represented by a reduced
word (Proposition \ref{pr:reduction}) and is conjugate to an element represented by a cyclically reduced
word and, moreover, by a strongly cyclically reduced word if it has infinite order 
(definition \ref{ss:reduced-words}, Proposition ~\ref{pr:cyclic-reduction}).

Sections \ref{s:coarsely-periodic} and \ref{s:overlapping-periodicity} are preparatory for 
analysis of periodic relations over $G_\al$. In Section  ~\ref{s:coarsely-periodic} we introduce
the set of {\em coarsely periodic words} over $G_\al$ which are close (in a stronger sense
then defined in \ref{ss:close-words}) to periodic words with a strongly reduced in $G_\al$ period
(Definition ~\ref{df:coarsely-periodic-segment}).
The main result of Section \ref{s:overlapping-periodicity}, 
Proposition \ref{pr:coarsely-periodic-overlapping-general}, is an analog
of a well known property of periodic words stating that if two periodic words
have a sufficiently large overlapping (for example, 
if they overlap for at least two of each of the periods) then they have 
a common period. 

In the last two Sections \ref{s:elementary-periods} and \ref{s:hierarchical-containment}
we define a set of defining relations of the form $C^n = 1$ $(C \in \bigcup_{\al\ge 1} \cE_\al)$
for the Burnside group $B(m,n)$
and prove that this set satisfies the iterated small cancellation condition (S0)--(S3).
More precisely, in Definitions \ref{df:start-suspended}--\ref{df:elementary-periods} we describe the recursive step to define $\cE_{\al+1}$ given $\cE_\be$ for $\be\le\al$, i.e.\ given the presentation 
of $G_\al$. The principal idea to build sets $\cE_\al$ can be roughly described as ``classification of 
periodic words by depth of periodicity'' and is similar to one used in \cite{NovAdi68,Adi75}.
Note that other approaches \cite{Ols82,Ols91,Iva94,IvaOls96,DelGro08,Cou14} 
to groups of ``Burnside type'' use 
construction of periodic relations $C^n = 1$ where for the next rank, $C$ are chosen to be ``short in size''
with respect to the current group. 
We believe that the ``depth of periodicity'' approach, allthough more technical in several aspects, gives
a more optimal lower bound on the exponent $n$.

\section{Preliminaries}

Starting from Section ~\ref{s:diagrams} we assume fixed a value of rank $\al \ge 0$
and  a presentation 
\eqref{eq:G-al-presn} of a group $G_\al$ with relators $R \in \cX_\be$ 
defined for all ranks $\be \le \al$.
We assume that the presentation of $G_\al$ is normalized and satisfies conditions (S0)--(S3)
and inequalities \eqref{eq:ISC-main} for all ranks up to the fixed value $\al$.
In the proofs we will use forward references to statements
for smaller values of rank, as already established. 
We will use references like ``Proposition 2.3$_{\al-1}$'' or ``Lemma 3.4$_{<\al}$'' etc.\ which mean
``statement of Proposition 2.3 for rank $\al-1$'' or 
``statement of Lemma 3.4 for all ranks $\be < \al$'' respectively.
With a few exceptions, statements whose formulation includes the case $\al=0$, are trivial or follow directly
from definitions in that case. 

\subsection{Words} \label{ss:words}
We fix a set $\cA$ of generators for a group $G$.
By a word we always mean a group word over the alphabet $\cA^{\pm1} = \cA \cup \setof{a^{-1}}{a \in \cA}$.
We use notation $X \eqcirc Y$ for identical equality of words $X$ and $Y$. By
$X^\circ$ we denote the cyclic word represented by a plain word ~$X$.

A {\em subword} $Y$ of a word $X$ is always considered with an associated occurrence of ~$Y$ in $X$
that is clear from the context. To make it formal,
we associate with a subword $Y$ of $X$ a pair of words $(U, V)$ such that $UYV \eqcirc X$.
If $Y$ is a subword of $X$ with an associated pair $(U, V)$ then
writing $Y \eqcirc WZ$ we mean that $W$ and $Z$ are viewed as subwords of $X$ with associated 
pairs $(U, ZV)$ and $(UW, V)$ respectively.
Note that `subword $Y$ of $X_1$' and `subword $Y$ of $X_2$' are formally two distinct objects if 
$X_1\ne X_2$. It will be always clear from the context which ambient word is assumed for $Y$.

A {\em periodic word with period $A$}, or an {\em $A$-periodic word} for short,
is any subword of  $A^t$ for $t > 0$.
According to the convention about subwords,  an $A$-periodic word $P$ is always considered with an associated occurrence of $P$ in a word $A^t$.

A {\em partition} of a word $X$ is a representation of $X$ as concatenation $X = X_1 \cdot X_2 \cdot \ldots \cdot X_k$ of some subwords $X_i$.
A word $X$ is {\em covered} by a collection of words $(Y_i)_i$ if $X$ admits a partition $X = X_1 \cdot X_2 \cdot \ldots \cdot X_k$
such that $X_i$ is a subword of some $Y_{t_i}$ and $t_i \ne t_j$ for $i \ne j$.

\subsection{Graphs} \label{ss:Cayley}

We use the term `graph' as a synonym for `combinatorial 1-complex'. Edges of a graph are considered
as having one of the two possible directions, so formally all our graphs are directed.
By $\io(!e)$ and $\tau(!e)$ we denote the starting and the ending vertices
of an edge $!e$, respectively, and $!e^{-1}$ denotes the inverse edge.
An {\em $\cA$-labeling} on a graph $\Ga$ is a function from the set of edges of $\Ga$ with values in 
$\cA^{\pm1} \cup \set{1}$ such that $\lab(!e^{-1}) = \lab(!e)^{-1}$ for any $!e$; here $1$ denotes the empty word.
An $\cA$-labeling naturally transfers to paths in $\Ga$, so the label of a path $!P$ is a word in $\cA^{\pm1}$.
If $!P$ is a path in ~$\Ga$ then $\io(!P)$ and $\tau(!P)$ denote the starting and the ending vertices
of $!P$, respectively. For any vertex ~$!a$ of $\Ga$, there is the unique {\em empty path at $!a$}.
We identify this empty path with vertex ~$!a$ itself, so 
$\io(!a) = \tau(!a) = !a$ and $\lab(!a) = 1$.
A path is {\em simple} if it visits no vertex twice.
Two paths are {\em disjoint} if they have no common and no mutually inverse edges. 
A {\em line} in $\Ga$ is a bi-infinite path (we do not assume that lines have no loops).

If $!X$ and $!Y$ are subpaths of a simple path $!Z$ then we write $!X \ll !Y$ if $!Z = !Z_1 !X !Z_2 !Y !Z_3$ 
for some $!Z_i$
and $!X < !Y$ if $!Z = !Z_1 !X !u !Z_2 = !Z_1 !v !Y !Z_2$ for some $!Z_i$ and non-empty $!u$ and $!v$. 
Although both relations depend on $!Z$, it will be always clear from the context which $!Z$ is assumed.
Clearly, if neither $!X$ and $!Y$ is contained in the other then  either $!X < !Y$ or $!Y < !X$.
The {\em union} $!X \cup !Y$ of subpaths $!X$ and $!Y$ of $!Z$ is the shortest subpath of $!Z$ containing 
both $!X$ and ~$!Y$.

The Cayley graph $\Ga(G,\cA)$ of a group $G$ with a generating set $\cA$ 
is naturally viewed as an $\cA$-labeled graph.
We identify vertices of $\Ga(G,\cA)$ with elements of $G$, so if $\io(!P) = !a$ and $\tau(!P) = !b$
then $\lab(!P)$ is a word representing $!a^{-1} !b$.

The group $G$ acts on $\Ga(G,\cA)$ by left multiplication.

A path $!P$ in $\Ga(G,\cA)$ labeled by an $A$-periodic word is an {\em $A$-periodic segment}.
An {\em $A$-periodic line} is a bi-infinite path labeled by $A^\infty$.
Since an $A$-periodic word is assumed to have an associated occurrence in some $A^t$, 
an $A$-periodic segment ~$!P$ can be uniquely extended to an $A$-periodic line called the 
{\em infinite periodic extension} of $!P$. If $!P$ and $!Q$ are $A$-periodic segments, $!P$ is
a subpath of $!Q$ and the both have the same infinite periodic extension then $!Q$ is a 
{\em periodic extension} of $!P$.

We define also the {\em translation element} $s_{A,!P} \in G$ that shifts 
the infinite periodic extension ~$!L$ of $!P$ forward by a period $A$. 
By definition, $s_{A,!P}$ can be computed as follows. 
Take any vertex ~$!a$ on $!L$ such that the label of $!L$ at $!a$ starts with $A$.
Then $s_{A,!P} = !a A !a^{-1}$.

If $!L_1$ and $!L_2$ are two periodic lines with periods $A_1$ and $A_2$ respectively
then $!L_1$ and $!L_2$ are {\em parallel} if 
$s_{A_1,!L_1} = s_{A_2,!L_2}$. 

\subsection{Mapping relations in the Cayley graph} \label{ss:relations}
It follows from the definition of the Cayley graph that a word $X$ 
in the generators $\cA$ represents the identity of $G$ if and only if some (and therefore, any)
path ~$!X$ in $\Ga(G,\cA)$ 
with $\lab(!X) \greq X$ is a loop. Thus relations in $G$ are represented by loops in $\Ga(G,\cA)$.
This representation will be our basic tool to analyze relations in a group using geometric
properties of its Cayley graph.

We will often use the following notational convention. 
If $X_1 X_2 \dots X_n = 1$ is a relation in a group $G$ then we represent it by a loop $!X_1 !X_2 \dots !X_n$ 
in the Cayley graph of $G$ typed with the same letters in sans serif where, by default, $\lab(!X_i) \greq X_i$ for all $i$.

We represent also conjugacy relations in $G$ by parallel periodic lines in $\Ga(G,\cA)$ as follows.
Let $X = Z^{-1} Y Z$ in ~$G$. Consider a loop $!X^{-1} !Z^{-1} !Y !Z'$ in $\Ga(G,\cA)$ with $\lab(!X) \greq X$, $\lab(!Y) \greq Y$ 
and $\lab(!Z) \greq \lab(!Z') \greq Z$. We extend $!X$ to an $X$-periodic line $!L_1 = \dots !X_{-1} !X_0 !X_1 \dots$
with $\lab(!X_i) \greq X$ and $!X_0 = !X$ and, in a similar way, extend $Y$ to a $Y$-periodic line 
$!L_2 = \dots !Y_{-1} !Y_0 !Y_1 \dots$ with $\lab(!Y_i) \greq Y$ and $!Y_0 = !Y$.
Then we get a pair of parallel lines $!L_1$ and $!L_2$
that represents conjugacy of $X$ and $Y$ in $G$.

We will be freely switch between the language of paths in Cayley graphs and 
word relations. 

\subsection{Van Kampen diagrams} \label{ss:diagrams}
Let $G$ be a group with a presentation $\cP = \spresn{\cA}{\cR}$.
A {\em diagram ~$\De$ over ~$\cP$} is a finite 2-complex $\De$ embedded in $\R^2$ with a given $\cA$-labeling of
the 1-skeleton $\De^{(1)}$ such that the label of the boundary loop of every 2-cell of $\De$ is either empty, has the
form $a^{\pm1} a^{\mp1}$ for $a \in \cA$ or is a relator in $\cR^{\pm1}$.
Note that here we use an extended version of the widely used definition by allowing boundary loops of 2-cells labeled with empty word or freely cancellable pair of letters. 
This allows us to avoid technical issues related to singularities 
(see \cite[\S 11.5]{Ols91} or \cite[\S4]{Lys96}).

By default, all diagrams are assumed to be connected.

We refer to 2-cells of a diagram $\De$ simply as to {\em cells};
1-cells and 0-cells are {\em edges} and {\em vertices} as usual.
By $\de !D$ we denote the boundary loop of a cell ~$!D$ and by
$\de \De$ we denote the unique boundary loop of $\De$ in case when $\De$ is simply connected.
We fix an orientation of ~$\R^2$ and assume that boundary loops of cells of $\De$
and boundary loops of $\De$ are positively oriented with respect to the cell or to the diagram, respectively.
This means, for example, that $(\de !D)^{-1}$ is a boundary loop of the diagram $\De - !D$ obtained 
by removal of a cell $!D$ from ~$\De$. 
Note that boundary loops of $\De$ and of its cells are defined up to cyclic shift.

According to van Kampen lemma (\cite[Theorem V.1.1]{LynSch01} and \cite[Theorem 11.1]{Ols91}) 
a word $X$ in the generators $\cA$
represents the identity in $G$ if and only if there exists a simply connected diagram $\De$
over $\cP$ with $\lab(\de\De) \greq X$. Words $X$ and $Y$ represent conjugate elements of $G$
if and only if there exists an annular (i.e.\ homotopy equivalent to an annulus) diagram over ~$\cP$
with boundary loops $!X$ and $!Z$ such that $\lab(!X) \greq X$ and $\lab(!Z) \greq Y^{-1}$
 (\cite[Lemma V.5.2]{LynSch01} and \cite[Theorem 11.2]{Ols91}).
 
If $\Si$ is a subdiagram of $\De$ then $\De-\Si$ denotes the subdiagram of $\De$ obtained as the topological closure of the complement $\De \sm \Si$.

Let $\De$ and $\De'$ be diagrams over $\cP$ such that $\De'$ 
is obtained from $\De$
by either 
\begin{itemize}
\item 
contracting an edge $!e$ with $\lab(!e) \greq 1$ to a vertex,
\item
contracting a cell $!D$ with $\lab(\de !D) \greq 1$ to a vertex, or
\item
contracting a cell $!D$ with $\lab(\de !D) \greq a^{\pm1} a^{\mp1}$, $a \in \cA$, to an edge labeled $a^{\pm1}$.
\end{itemize}
We call the inverse transition from $\De'$ to $\De$ an {\em elementary refinement}. 
A sequence of elementary refinements is a {\em refinement}.

There are several common use cases for refinement:
\begin{itemize}
\item 
Any diagram can be made by refinement {\em non-singular}, i.e.\ homeomorphic to a punctured disk.
In particular, any simply connected diagram can be refined to a non-singular disk.
\item
If $!C$ is a boundary loop of $\De$ represented as a product $!C = !X_1 \dots !X_k$ of paths $!X_i$ then, after
refinement, the corresponding boundary loop of a new diagram $\De'$ becomes $!X_1' \dots !X_k'$ 
where each $!X_i$ refines to a nonempty path $!X_i'$ 
(see the definition in \ref{ss:combinatorially-continuous}).
\end{itemize}

\subsection{Combinatorially continuous maps of graphs}   \label{ss:combinatorially-continuous}
We consider the class of maps between $\cA$-labeled graphs
which are label preserving and can be realized as continuous maps 
of topological spaces. More precisely, a map $\phi: \La \to \La'$ between 
$\cA$-labeled graphs $\La$ and $\La'$ is {\em combinatorially continuous} if
\begin{itemize}
\item 
$\phi$ sends vertices to vertices and edges to edges or vertices;
for any edge $!e$ of $\La$, $\phi(!e)$ is a vertex only if $e$ has the empty label;
if $\phi(!e)$ is an edge then $\lab(\phi(!e)) = \lab(!e)$.
\item
if $\phi(!e)$ is an edge then $\phi$ preserves the starting and the ending vertices of $!e$; if $\phi(!e)$ is a vertex 
then $\phi(!e) = \phi(\io(!e)) = \phi(\tau(!e))$.
\end{itemize}

A combinatorially continuous map $\phi: \La \to \La'$ extends in a natural way to the map denoted also by
$\phi$, from the set of paths in $\La$ to the set of paths in $\La'$. Clearly, $\phi$ preserves 
path labels.

If a diagram $\De'$ is obtained from a diagram $\De$ by refinement then we have 
a combinatorially continuous map $\phi: \De'^{(1)} \to \De^{(1)}$ induced by 
the sequence of contractions $\De' \to \De$. If $!P$ is a path in $\De$ and $!P' = \phi(!P)$ then $!P$ {\em refines} to $!P'$.

\subsection{Mapping diagrams in Cayley graphs}
It is well known that simply connected diagrams can be viewed as combinatorial surfaces in the Cayley complex of a group.
Since we do not make use of two-dimensional structure, we adapt this view to the 
case of Cayley graphs.

If $\De$ is a simply connected diagram over $\cP$ then there exists a combinatorially continuous 
map $\phi: \De^{(1)} \to \Ga(G,\cA)$.
Any two such maps $\phi_1, \phi_2: \De^{(1)} \to \Ga(G,\cA)$ differ by translation by some element
$g \in G$, i.e.\ $\phi_1 = t_g \phi_2$ where $t_g : x \mapsto gx$ is the translation.

In particular, 
if $!X$ is a loop in $\Ga(G,\cA)$ and for the boundary 
loop $\bar{!X}$ of $\De$ we have $\lab(\bar{!X}) = \lab(!X)$ then there is a map 
$\phi: \De^{(1)} \to \Ga(G,\cA)$ such that $\phi(\bar{!X}) \greq !X$.
In this case we say that $\De$ {\em fills} $!X$ via $\phi$.

If $\De$ is not simply connected then we can consider a combinatorially continuous 
map $\phi: \ti\De^{(1)} \to \Ga(G,\cA)$ where $\ti\De$ is the universal cover of $\De$.
Again, any two such maps $\phi_1, \phi_2: \ti\De^{(1)} \to \Ga(G,\cA)$ differ by translation by an element of $G$.
The set $\set{!L_i}_i$ of boundary loops of $\De$ lifts to a (possibly infinite) 
set of bi-infinite boundary 
lines $\set{\ti{!L}_i^j}_{i,j}$ of $\ti\De$ and thus produces a set of lines $\set{\phi(\ti{!L}_i^j)}_{i,j}$
in $\Ga(G,\cA)$. Each $\phi(\ti{!L}_i^j)$ can be viewed as an $P_i$-periodic line with period $P_i =\lab(!L_i)$.
We will be interested mainly in the case when $\De$ is an {\em annular} diagram, i.e.
homotopy equivalent to a circle. In this case, boundary loops $!L_1$ and $!L_2$ of $\De$ produce
two $P_i$-periodic lines $\phi(\ti{!L}_i)$ $(i=1,2)$ in $\Ga(G,\cA)$ such that $\phi(\ti{!L}_1)$
and $\phi(\ti{!L}_2)^{-1}$ are parallel.

\begin{definition} \label{df:frame-type}
Let $\De$ and $\De'$ be diagrams 
of the same homotopy type over a presentation of a group $G$.
We assume that a label preserving 
bijection $!L_i \mapsto !L_i'$ is given between boundary loops of $\De$ and $\De'$
(which is usually clear from the context). 
We say that $\De$ and $\De'$ have the same {\em frame type} if there exist 
combinatorially continuous maps $\phi: \ti\De^{(1)} \to \Ga(G,\cA)$
and $\psi: \ti{\De}'^{(1)} \to \Ga(G,\cA)$ such that for each $i$ we have the same sets of lines
(or loops if $\De$ and $\De'$ are simply connected)
$\set{\phi(\ti{!L}_i^j)}_{j} = \set{\psi(\ti{!L}'^j_i)}_{j}$.
\end{definition}

The following two observations follow easily from the definition.

\begin{lemma} \label{lm:disk-annular-frame-type}
Two simply connected diagrams $\De$ and $\De'$ have the same frame type if and only if the labels of
their boundary loops are equal words.
 
Let $\De$ and $\De'$ be annular diagrams with boundary loops $\set{!L_1,!L_2}$ and $\set{!L_1',!L_2'}$.
Then $\De$ and ~$\De'$ have the same frame type if and only if the following is true.
Take any vertices $!a_i$ on $!L_i$ $(i=1,2)$ and let $!p$ be a path from $!a_1$ to $!a_2$ in $\De$.
Then there exist vertices $!a_i'$ on $!L_i'$ $(i=1,2)$ and a path $!p'$ from $!a_1'$ to $!a_2'$ in $\De'$
such that the label of $!L_i$ read at ~$!a_i$ and the label of $!L_i'$ read at $!a_i'$ are equal words
and $\lab(!p) = \lab(!p')$ in $G$.
\end{lemma}

\begin{lemma} \label{lm:preserving-frame-type}
Diagrams $\De$ and $\De'$ have the same frame type in the following two cases:
\begin{itemize}
\item 
$\De'$ is obtained from $\De$ by refinement;
\item
$\De'$ is obtained from $\De$ by cutting off a simply connected subdiagram 
and replacing it with another simply connected subdiagram.
\end{itemize}
\end{lemma}

\subsection{Groups $G_\al$}
Throughout the paper we will study a fixed family of groups $G_\al$ given by a presentation \eqref{eq:G-al-presn}.
Consequently, most of the related terminology will involve rank $\al$ as a parameter
(though in some cases, it is not mentioned explicitly; for example, the already 
introduced measure $\muf(F)$ of fragments of rank $\al$ formally depends 
on $\al$).

Diagrams over the presentation of $G_\al$ are referred simply as diagrams over $G_\al$.
For $1 \le \be \le \al$, 
a cell of a diagram $!D$ over $G_\al$ with $\lab(\de !D) \in \cX_\be$ is a {\em cell of rank $\be$}. 
Cells with trivial boundary labels (i.e.\ empty or of the form $aa^{-1}$) are {\em cells of rank $0$}.

The Cayley graph of $G_\al$ is denoted $\Ga_\al$.
Note that if $\be > \al$ then we have a natural covering map $\Ga_\be \to \Ga_\al$ of labeled graphs.
A loop $!L$ in $\Ga_\al$ lifts to $\Ga_\be$ as a loop if and only if $\lab(!L) = 1$ in $G_\be$.

\subsection{Pieces} \label{ss:piece} \label{ss:mu-def}

By a {\em piece of rank $\al$} we call any (including empty) subword of a relator of rank ~$\al$.
If $S$ is a subword of a cyclic shift of a relator $R$ then we say also that $S$ is a {\em piece of $R$}.
We admit that a piece of rank $\al$ be the empty word.
Note that our definition differs from the traditional view on a piece in the
small cancellation theory as a common
starting segment of two distinct relators.

We assume that a piece $S$ of rank $\al$ always has an 
associated relator $R$ of rank ~$\al$ such that $S$ is a start of $R$;
so formally a piece of rank $\al$ should be 
viewed as a pair of the form $(S,R)$.
Associated relators are naturally inherited under taking subwords and inversion:
if $S$ is a piece of rank $\al$ with associated relator $R = ST$ and 
$S = S_1 S_2$ then $S_1$ and $S_2$ are viewed as pieces of rank $\al$ with associated relators $R$ and $S_2 T S_1$
respectively and $S^{-1}$ is viewed as a piece of rank $\al$ with associated relator $S^{-1} T^{-1}$.

For pieces of rank $\al$ we use a ``measure'' $\mu(S) \in [0,1]$
defined by $\mu(S) = \frac{|S|_{\al-1}}{|R^\circ|_{\al-1}}$ as 
in \eqref{eq:mu-def} where $R$ is the associated relator.
(Recall that $R^\circ$ denotes the cyclic word represented by ~$R$.)
If for some $\be$, $!S$ is a path in $\Ga_\be$ or in a diagram over the 
presentation of $G_\be$ 
and $!S$ is labeled by a piece of a relator of rank $\al$ 
(or by an $R$-periodic word where $R$ is a relator 
of rank $\al$) then we abbreviate $\mu(\lab(!S))$ simply as $\mu(!S)$.

\subsection{Reformulation of conditions (S2) and (S3) in terms of Cayley graph}
\label{ss:S2-S3-Cayley}

The following conditions on the presentation \eqref{eq:G-presn} are equivalent to (S2) and (S3), respectively.
\begin{labeled-par}{(S2-Cayley)}
Let $!L_i$ $(i=1,2)$ be an $R_i$-periodic line in $\Ga_{\al-1}$ where $R_i$ is a relator of rank $\al$.
If $!L_1$ and $!L_2$ have close subpaths $!P_1$ and $!P_2$ with $|!P_i| \le |R_i|$
and $\mu(!P) \ge \ga$ then $!L_1$ and $!L_2$ are parallel. 
\end{labeled-par}

\begin{labeled-par}{(S3-Cayley)}
There are no parallel $R$-periodic and $R^{-1}$-periodic lines in 
$\Ga_{\al-1}$
where $R$ is a relator of rank $\al$.
\end{labeled-par}

\subsection{Bridge partition}  \label{ss:bridge-partition}

We define also a 
{\em bridge partition of rank $\al$} of a word $w \in \cH_\al$ as follows.
A bridge partition of rank $0$ is empty.
A bridge partition of rank $\al \ge 1$ either 
\begin{itemize}
\item
 has the form
$w_1 \cdot S \cdot w_2$ where $w_i \in \cH_{\al-1}$ 
and $S$ is a piece of rank $\al$ called the {\em central piece} of $w$; or
\item
is a single factor $w$ itself in the case $w \in \cH_{\al-1}$.
\end{itemize}

If $w$ is a bridge word of rank $\al$ endowed with a bridge partition 
$u \cdot S \cdot v$
and $ST$ is the relator of rank $\al$ associated with $S$ then $w' = u T^{-1} v$
is a bridge word of rank $\al$ equal to $w$ in ~$G_\al$.
We say that $w'$ is obtained from ~$w$ by {\em switching}. 
In this case we assume also that $w'$ is endowed with the 
bridge partition $u \cdot T^{-1} \cdot v$.
Thus, applying the switching operation twice results in the initial word ~$w$.

We will be considering paths in Cayley graphs $\Ga_\be$ labeled by bridge words of rank $\al$.
We call them {\em bridges of rank $\al$} (with a slight abuse of terminology,
we will also use this term in Section \ref{s:diagrams}
for boundary paths with appropriate label in diagrams over the presentation of $G_\al$). 
If $!w$ is bridge of rank ~$\al$ in $\Ga_\be$ then 
a {\em bridge partition of rank $\al$ of $!w$} is either a factorization 
$!w = !u \cdot !S \cdot !v$ where $!u$ and $!v$ are bridges of rank $\al-1$
and $\lab(!S)$ is a piece of rank $\al$ or a trivial factorization with
the single factor $!w$ if $!w$ is bridge of rank $\al-1$.
In the former case, if also $\be \ge \al$,
we define the {\em switching operation} on $!w$ in a similar way as in the case 
of words; namely, we take the word $w'$ obtained from $w \greq \lab(!w)$ by switching and consider the path 
$!w'$ with $\lab(!w') \greq w'$ starting at the same vertex as $!w$. Since $w = w'$ in ~$\Ga_\be$, 
bridges $!w$ and $!w'$ have the same endpoints.

\subsection{}
\label{ss:alpha-length-properties}
The following properties of the function $|\cdot|_\al$ follow from the definition:
\begin{enumerate}
\item 
$|X|_\al + |Y|_\al -1 \le |XY|_\al \le |X|_\al + |Y|_\al$;
in particular, if $Y$ is a subword of ~$X$ then $|Y|_\al \le |X|_\al$.
\item
More generally, if a collection of words $(X_i)_i$ covers a (plain or cyclic) word $X$ then 
$$
  |X|_\al \le \sum_i |X_i|_\al.
$$
If $(X_i)_{1\le i \le k}$ is a collection $k$ of disjoint subwords of $X$ then 
$$
  \sum_i |X_i|_\al \le |X|_\al + k.
$$
\item \label{cl:om-rank-increment}
$|X|_\al \le \ze |X|_{\al-1}$.
\item
$
  |X^\circ|_\al = \min\setof{|Y|_\al}{\text{$Y$ is a cyclic shift of $X$}}.
$
\end{enumerate}

If $!X$ is a path in $\Ga_\be$ or in a diagram over the presentation of $G_\be$ then we use abbreviation 
$|!X|_\al = |\!\lab(!X)|_\al$.

\subsection{Reduced words} 
\label{ss:reduced-words}
The set of words reduced in $G_\al$ is denoted $\cR_\al$. 
The definition immediately implies that $\cR_\al$ is closed under taking subwords.

A word $X$ is {\em strongly cyclically reduced in $G_\al$} if any power $X^t$ is reduced in $G_\al$.

\subsection{Coarse polygon relations} \label{ss:polygon-relations}
%
A relation in $G_\al$ of the form $X_1 u_1 \dots X_m u_m = 1$ where words ~$X_i$ 
are reduced in $G_\al$ and $u_i$ are bridge words of rank $\al$, is called a {\em coarse $m$-gon relation}
in $G_\al$. We can write coarse polygon relations in different forms. 
For example, a coarse bigon relation can be written as $X = u Y v$
where $X$ and $Y$ are reduced in $G_\al$ and $u,v \in \cH_\al$.
In this form, the relation represents closeness of words $X$ and $Y$ in $G_\al$.

\subsection{}
We transfer some terminology
from words to paths in $\Ga_\al$. 

We call paths in $\Ga_\al$ with label reduced in $G_\al$ simply {\em reduced}.
Note that, according to Proposition \ref{pr:reduced-nontrivial}, a reduced path $!X$ in $\Ga_\al$ 
is simple. This implies that we can correctly treat the ordering of subpaths of $!X$, intersections of subpaths, unions etc.

Two vertices of $\Ga_\al$ are {\em close} if they can be joined by a bridge of 
rank $\al$ (see \ref{ss:bridge-partition}). 
Two paths $!X$ and $!Y$ in $\Ga_\al$ are {\em close} if their starting vertices
and their ending vertices are close.

We say that a loop $!P = !X_1 !u_1 !X_2 !u_2 ,\dots,!X_r !u_r$ in $\Ga_\al$ is a {\em coarse $r$-gon} if 
each $!X_i$ is reduced and each $!u_i$ is a bridge of rank $\al$. Paths $!X_i$ are {\em sides} of $!P$.

Note that paths $!X$ and $!Y$ in $\Ga_\al$ are close if and only if $!X^{-1} !u !Y !v$
is a coarse bigon for some ~$!u$ and $!v$.

\subsection{Symmetry}
All concepts (i.e.\ relations, functions etc.) and statements 
involving paths in the Cayley graphs $\Ga_\al$ 
are invariant under the action of $G_\al$ in a natural way. For example, if paths $!X$ and $!Y$ in $\Ga_\al$
are close then paths $g!X$ and $g!Y$ are also close for any $g \in G_\al$. 
We adopt a convention (which is essential for the invariance) that
the action of $G_\al$ is extended onto extra data associated with paths in $\Ga_\al$:
for example, if $!F$ is a fragment of rank $\be$ with base ~$!P$ then then $g !F$ is considered
as a fragment of rank $\be$ with base $g !P$ and so on. 
This implies, for example, that $\muf(!F) = \muf(g!F)$ for any  $g \in G_\al$. 

We will implicitly use symmetry with respect to inversion. For example, 
if $!F$ is a fragment of rank $\be$ with base ~$!P$ then $!F^{-1}$ is a fragment 
of rank $\be$ with base ~$!P^{-1}$ and $\muf(!F^{-1}) = \muf(!F)$. If a statement admits
two symmetric forms then only one of them is formulated 
(as in case of Lemma \ref{lm:closeness-bigon-1side}, for instance).

\subsection{Numerical parameters} \label{ss:extra-parameters}
In many cases, it will be notationally more convenient to use instead of $\Om$ its inverse: 
$$
  \om = \frac{1}{\Om}.
$$
Note that by \eqref{eq:ISC-main}, 
\begin{equation} \label{eq:om-bounds}
  \om \le \frac{1}{480} \quad\text{and}\quad \la \ge 20\om.
\end{equation}
We will extensively use $\om$ as a unit to measure pieces and fragments of rank $\al$. 

Condition (S1) in \ref{ss:condition} will be often used in the following form: {\em if $P$ is a piece of a relator ~$R$
of rank $\al$ then}
\begin{equation} \label{eq:S2-piece-form}
  \mu(P) \le \om |P|_{\al-1}.
\end{equation}

For reader's convenience, we list our other global numerical parameters indicating the places
where they first appeared.

$$
  \nu = \frac{\ze}{1 - 2\ze} = \frac{1}{18}, \quad
  \th = \frac16 (5 - 22\nu) = \frac{17}{27} \quad \text{(Proposition \ref{pr:principal-bound-cells})},
$$
$$
  \eta = \frac{1+2\nu}{\th} = \frac{30}{17} \quad \text{(Proposition \ref{pr:principal-bound})},
$$
$$
  \xi_0 = 7\la - 1.5\om \quad \text{(Proposition \ref{pr:active-large})},
$$
$$
  \xi_1 = \xi_0 - 2.6\om \quad \text{(Definition \ref{df:activity-rank})},
$$
$$
  \xi_2 = \xi_1 - 2\la - 3.4\om \quad \text{(Definition \ref{df:coarsely-periodic-segment})}.
$$

\section{Diagrams with marked boundary} \label{s:diagrams}


\subsection{Boundary marking of rank $\al$} \label{df:boundary-marking}
We start with introducing a class of diagrams over the presentation  \eqref{eq:G-al-presn}
of $G_\al$ with extra data which,
in particular, represent coarse polygon relations in $G_\al$.

Let $\De$ be a non-singular diagram over the presentation \eqref{eq:G-al-presn}. 
We say that $\De$ has a 
{\em boundary marking of rank ~$\al$} 
if for each boundary loop $!L$ of $\De$, there is fixed a representation as a product $!L = !X_1 !u_1 \dots !X_m !u_m$
of nonempty paths ~$!X_i$ and $!u_i$ 
where labels of $!X_i$ are reduced in ~$G_\al$ and the label of each $!u_i$ belongs to ~$\cH_\al$.
Paths ~$!X_i$ are called {\em sides} and paths $!u_i$ are called {\em bridges} of $\De$.
We allow also that the whole boundary loop $!L$ of $\De$ is viewed a side 
called a {\em cyclic side}. 
In this case we require that the label of $!L$ is cyclically reduced in $G_\al$.

If $X_1 u_1 \dots X_m u_m = 1$ is a coarse polygon relation in $G_\al$ then there exists
a disk diagram with boundary label $!X_1 !u_1 \dots !X_m !u_m$ such that 
$\lab(!X_i) \greq X_i$ and $\lab(!u_i) \greq u_i$ for all $i$.
Refining $\De$ if necessary (see \ref{ss:diagrams}) we can assume that $\De$ is non-singular 
and all paths $!X_i$
and ~$!u_i$ are nonempty, i.e.\ $\De$ satisfies the definition above.
In a similar way, we can associate with a conjugacy relation in $G_\al$ an annular diagram over the presentation
of $G_\al$ with an appropriate boundary marking.

Unless otherwise stated, ``a diagram of rank $\al$'' will always mean ``a non-singular diagram over 
the presentation \eqref{eq:G-al-presn} with
a fixed boundary marking of rank $\al$''. 
We use terms ``diagrams of monogon, bigon, trigon type etc.'' to name disk diagrams of rank $\al$ 
with the appropriate number of sides.
%

\subsection{Complexity} \label{ss:diagram-complexity}
If $\De$ is a diagram of rank $\al$ then
by $b(\De)$ we denote the number of bridges of $\De$. 
We define the {\em complexity $c(\De)$} of $\De$ by
$$
  c(\De) = b(\De) - 2 \chi(\De).
$$

\subsection{Decrementing the rank} \label{ss:diagram-previous}
Let $\De$ be a diagram of rank $\al\ge 1$. By $\De_{\al-1}$ we denote the diagram over the presentation
of $G_{\al-1}$
obtained by removal from $\De$ of all cells of rank $\al$.
Up to refinement of $\De$, we assume that $\De_{\al-1}$ is non-singular.


We assume that every bridge $!w$ of $\De$ is given a bridge partition of rank $\al$
as defined in ~\ref{ss:bridge-partition}, i.e.\ for some bridges $!w$ 
a factorization $!w = !u \cdot !S \cdot !v$ is fixed where $\lab(!u), \lab(!v) \in \cH_{\al-1}$
and $\lab(!S)$ is a piece of rank $\al$, and for all other $!w$ we have 
$\lab(!w) \in \cH_{\al-1}$. 
In the case when $!w$ has a nontrivial bridge partition $!u \cdot !S \cdot !v$
we say that $!w$ has {\em native rank $\al$} and call $!S$ the {\em central arc} 
of ~$!u$.

We will be always assuming that all factors $!u$, $!v$ and $!S$
are nonempty paths (this can be achieved by refinement).

We then define a naturally induced
boundary marking of rank $\al-1$ of $\De_{\al-1}$
(see Figure ~\ref{fig:diagram-previous}):
\begin{itemize}
\item 
Sides of $\De$ become sides of $\De_{\al-1}$; we have also extra
sides of $\De_{\al-1}$ defined as follows.
\item
If $!D$ is a cell of rank $\al$ of $\De$ then boundary loop $(\de !D)^{-1}$ of $\De_{\al-1}$
becomes a cyclic side of $\De_{\al-1}$.
\item
For each bridge $!w$ of rank $\al$ of $\De$ we do the following.
If the bridge partition of $!w$ is of the form $!u = !v \cdot !S \cdot !w$
then we take $!v$ and $!w$ as bridges of $\De_{\al-1}$ and the central arc ~$!S$ as a side of 
$\De_{\al-1}$. Otherwise we have $\lab(!w) \in \cH_{\al-1}$ and we take $!w$ as a bridge of $\De_{\al-1}$.
\end{itemize}
\begin{figure}[h]
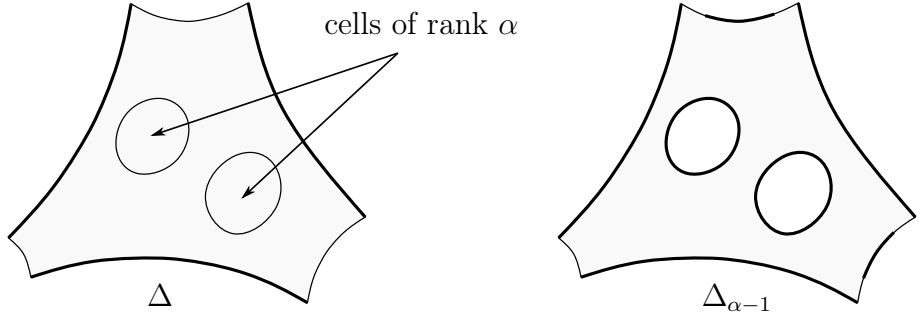
\caption{Producing $\De_{\al-1}$ from $\De$.
Sides of ~$\De$ and $\De_{\al-1}$ are drawn by thicker lines}  
\label{fig:diagram-previous}
\end{figure}

\subsection{Cell cancellation} \label{ss:diagram-reduction}
We introduce two types of elementary reductions of a diagram $\De$ of rank $\al \ge 1$.
In both cases, we reduce the number of cells of rank $\al$.
As in \ref{ss:diagram-previous}, we assume that a bridge partition
is fixed for each bridge $\De$.

Let $!C$ and $!D$ be two cells of rank $\al$ of $\De$. We say that $!C$ and $!D$
form a {\em cell-cell cancellable pair} if there exists a simple path $!p$ joining two vertices $!a$ and $!b$
in the boundaries of $!C$ and ~$!D$ respectively, 
so that the label of the path $!Q !p !R !p^{-1}$ is equal 1 in $G_{\al-1}$
where $!Q$ and $!R$ are boundary loops of $!C$ and $!D$ starting at $!a$ and $!b$ respectively 
see Figure~\ref{fig:diagram-cell-cell}a).  
\begin{figure}[h]
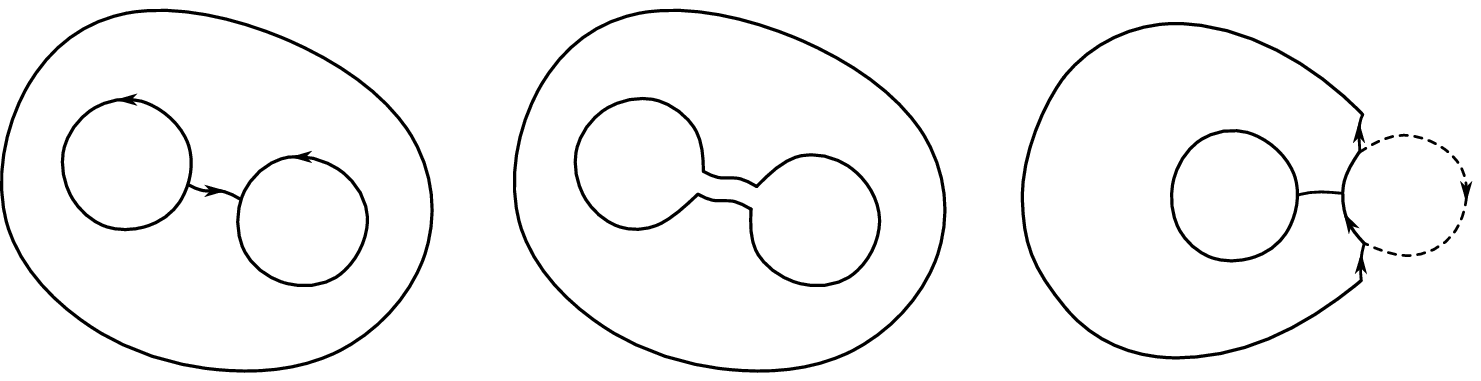
\caption{}  \label{fig:diagram-cell-cell}
\end{figure}
In this case, we can perform the procedure of {\em cell-cell cancellation} as follows.
We remove cells $!C$ and $!D$ from ~$\De$, 
cut the remaining diagram along $!p$ and fill in the resulting region
by a diagram $\Th$ over the presentation of $G_{\al-1}$ (see Figure~\ref{fig:diagram-cell-cell}b). 
The boundary marking of the new diagram naturally inherits the boundary marking of $\De$ 
and the labels of sides and bridges are not changed.

Now let $!u$ be a bridge of native rank $\al$ of $\De$ with
bridge partition $!u = !v \cdot !S \cdot !w$.
The label ~$S$ of $!S$ has an associated relator $R$ of rank $\al$ such that $R \greq ST$ for some $T$
(according to the convention in \ref{ss:piece}).
We attach
a cell $!C$ of rank $\al$ to ~$\De$ along ~$!S$ so that $(ST)^{-1}$ becomes the label of the boundary loop $(!S !T)^{-1}$ 
of $!C$ (see Figure~\ref{fig:diagram-cell-cell}c).
For the new diagram $\De \cup !C$ we define the boundary marking of rank $\al$ 
with a new bridge $!v !T^{-1} !w$ instead of $!u$.
We call this operation {\em switching of $!u$}.

If $!C$ and another cell ~$!D$ of rank ~$\al$ of $\De$ form a cell-cell cancellation pair in $\De \cup !C$ 
then we say that $!u$ and $!D$
form a {\em bridge-cell cancellable pair}. In this case, after performing a cell-cell cancellation in 
$\De \cup !C$ we obtain a diagram $\De'$ having one cell of rank $\al$ less than ~$\De$.
We will refer to this reduction step as {\em bridge-cell cancellation}.

\begin{definition}[Reduced diagram] \label{df:reduced-diagram}
Let $\De$ be a diagram of rank $\al\ge1$ with fixed bridge partitions for all 
bridges of $\De$. We say that $\De$ is {\em reduced} if it has 
no cancellable pairs after any refinement. 
\end{definition}

\begin{remark}
In what follows, we will be assuming that a diagram $\De$ of rank $\al\ge 1$ has fixed bridge partitions of 
all bridges of $\De$ if it is required by context.  In particular, this applies when we consider the 
subdiagram $\De_{\al-1}$ and the property of $\De$ to be reduced.
\end{remark}

\subsection{Reduction process} \label{ss:reduction-transformations-good}
If a diagram $\De$ of rank $\al$ is not reduced then, after possible refinement, we obtain a cancellable pair which
can be removed by performing the reduction procedure described above. Thus, any diagram 
of rank $\al\ge 1$ can be transformed
to a reduced one. Note that we use a sequence of transformations of the following two types in the reduction process:
\begin{itemize}
\item 
transformations preserving the frame type (see Lemma \ref{lm:preserving-frame-type});
\item
bridge switching.
\end{itemize}
Thus, after reduction the new diagram $\bar\De$ has the same frame type as $\De$ up to bridge switching.

The following observation follows from 
definitions \ref{ss:diagram-reduction} and \ref{df:reduced-diagram}
and will be used without explicit reference.

\begin{proposition} 
Let $\Si$ be a subdiagram of a reduced diagram $\De$ of rank $\al\ge 1$ such that the central arc of any bridge of $\Si$
is either a subpath of the central arc of a bridge of $\De$ 
or a subpath of $(\de !D)^{-1}$ where $!D$ is a cell of rank $\al$
of $\De$. Then $\Si$ is reduced as well.
\end{proposition}

\section{Reduction to the previous rank}

\begin{definition} \label{df:bond}
Let $\De$ be a diagram of rank $\al$. 
A {\em bond\/} in $\De$ is a simple path $!u$ 
satisfying the following conditions:
\begin{enumerate}
\item
$!u$ joins two vertices on sides of $\De$
and intersects the boundary of $\De$ only at the endpoints of $!u$;
\item
$\lab(!u)$ is equal in $G_\al$ to a word in $\cH_{\al}$.
\item 
$!u$ is not homotopic in $\De$ (rel endpoints) to a subpath of a side of $\De$;
\item
$!u$ does not cut off from $\De$ a simply connected subdiagram 
with boundary loop $!u^{\pm1} !p !v !q$ 
where $!p$ is an end of a side of $\De$, $!v$ is a bridge of $\De$, $!q$ is a start of a side of $\De$ 
and labels of $!p$ and $!q$ are empty words. See Figure~\ref{fig:bond-excluded-cases}.
\end{enumerate}
\end{definition}
\begin{figure}[h]
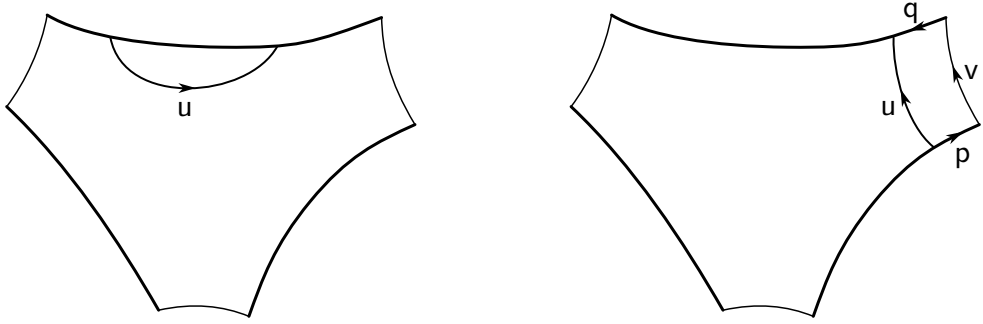
\caption{Excluded cases in (iii) and (iv)}  \label{fig:bond-excluded-cases}
\end{figure}

\subsection{} \label{ss:bond-enhacement}
In most cases, we will assume that the label of a bond $!u$ already belongs to $\cH_\al$.
Note that this condition can always be achieved by cutting $\De$ along $!u$ and attaching 
a subdiagram with boundary loop $!u^{\pm1} !v$ where $\lab(!v) \in \cH_\al$ and its mirror copy,
see Figure \ref{fig:bond-to-bridge}.
\begin{figure}[h]
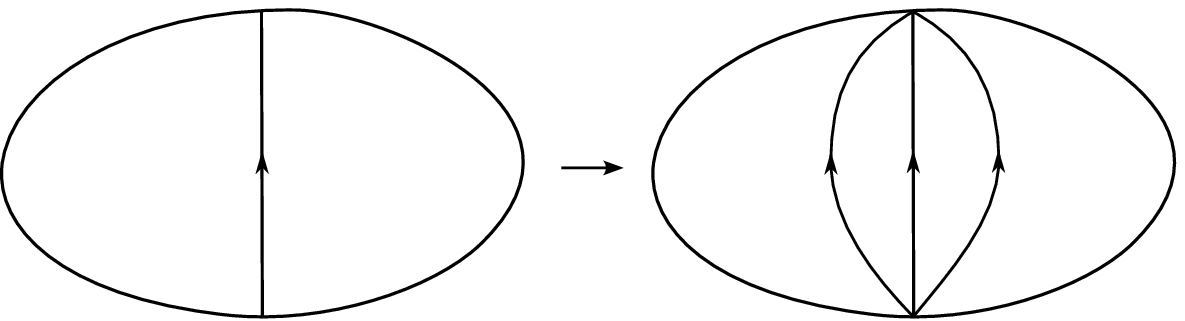
\caption{}  \label{fig:bond-to-bridge}
\end{figure}


\begin{definition} \label{df:small-diagram}
A diagram of rank $\al$ is {\em small} if it has no bonds after any refinement.
\end{definition}

The following observation is straightforward.

\begin{proposition} \strut\par
\begin{enumerate}
\item 
The property of a diagram $\De$ of rank $\al$ to be small depends only on the frame type of $\De$.
\item
The property of a diagram of rank $\al$ to be small is preserved under switching of bridges.
\item \label{pri:small-rank0-diagram}
If $\De$ is a small diagram of rank 0 with $c(\De) > 0$ then labels of all sides of $\De$ are empty words.
\end{enumerate}
\end{proposition}

\begin{definition} \label{df:contiguity-subdiagram}
Let $\De$ be a diagram of rank ~$\al\ge 1$. 
A disk subdiagram ~$\Pi$ of $\De_{\al-1}$ is a {\em contiguity subdiagram} 
of $\De$ if the boundary loop of $\Pi$ 
has the form $!P!u_1 !Q !u_2$ where $!P^{-1}$ and ~$!Q^{-1}$ are nonempty subpaths of sides 
of $\De_{\al-1}$ and each of the two paths ~$!u_i$ is either a bond in ~$\De_{\al-1}$ 
with $\lab(!u_i) \in \cH_{\al-1}$
or a bridge of $\De_{\al-1}$. Note that here we use Definition ~\ref{df:bond} 
with rank $\al-1$ instead of ~$\al$.

The paths $!P^{\pm1}$ and $!Q^{\pm1}$ are {\em contiguity arcs} of $\Pi$.
If $!P^{-1}$ and $!Q^{-1}$ occur, respectively, in sides $!S$ and ~$!T$ of $\De_{\al-1}$ then we say 
that $\Pi$ is a contiguity subdiagram {\em of $!S$ to $!T$} (or {\em between ~$!S$ and ~$!T$}).
\end{definition}

According to definition \ref{ss:close-words}, if $!P$ and $!Q$ are contiguity arcs of a contiguity 
subdiagram with boundary loop $!P!u_1 !Q !u_2$ then labels of $!P^{-1}$ and $!Q$ are close in $G_{\al-1}$.

\begin{lemma}[small cancellation in reduced diagrams] \label{lm:small-cancellation-diagram}
Let $\De$ be a reduced diagram of rank ~$\al$. Let $\Pi$ be a contiguity subdiagram of $\De$ 
with boundary loop $\de \Pi = !P !u !Q !v$ where $!P$ and $!Q$ are the contiguity arcs of $\Pi$.
Assume that $!P^{-1}$ occurs in the boundary loop of a cell $!D$ of rank $\al$ 
and $!Q^{-1}$ occurs in a side $!S$ of ~$\De_{\al-1}$.
Then:
\begin{enumerate}
\item 
If $!S$ is a side of $\De$ then $\mu(!P) < \rho$;
\item
If $!S$ is the boundary loop of a cell $!D'$ distinct from $!D$ then $\mu(!P) < \la$;
\item
If $!S$ is the central arc of a bridge of $\De$ then $\mu(!P) < \la$;
\end{enumerate}
\end{lemma}

\begin{proof}
If $!S$ is a side of $\De$ then the label of $!S$ is reduced in $G_\al$ (or cyclically reduced in $G_\al$
if $!S$ is a cyclic side), as defined in \ref{df:boundary-marking}.
Then $\mu(!P) < \rho$ by the definition of a reduced word in ~\ref{ss:reduced-word}.

Assume that $\mu(!P) \ge \ga$ and 
$!S = \de !D'$ where $!D'$ is a cell distinct from $!D$.
Let $!R$ and $!R'$ be boundary loops of $!D$ and $!D'$ starting at the initial and terminal vertices 
of $!u$, respectively. By the small cancellation condition (S2) we have 
$\lab(!R) = \lab(!u !R' !u^{-1})$ in $G_{\al-1}$, hence $!D$ and $!D'$ form a cell-cell cancellable pair
contrary to the hypothesis that $\De$ is reduced.

If $\mu(\lab(!P)) \ge \la$ and $!S$ is the central arc of a bridge of $\De$ then in a similar way
we see that $!D$ and $!S$ form a cell-bridge cancellable pair.
\end{proof}

Note that the lemma leaves uncovered a possibility when $!S = \de !D$, i.e.\ when $\Pi$ is a contiguity
subdiagram of $!D$ to itself. This case needs a special consideration.

\begin{definition} \label{df:folded-cell}
A cell $!D$ of rank $\al$ in a diagram $\De$ of rank $\al\ge 1$ is {\em folded} if there exists a simple 
path $!u$ joining two vertices $!a$ and $!b$ in the boundary of $!D$ so that $\lab(!P!Q !u !Q!P!u^{-1}) = 1$
in $G_{\al-1}$ where $!P$ and $!Q$ are subpaths of $\de !D$ from $!a$ to $!b$ and from $!b$ to $!a$ respectively
(Figure ~\ref{fig:folded-cell}).
\begin{figure}[h]
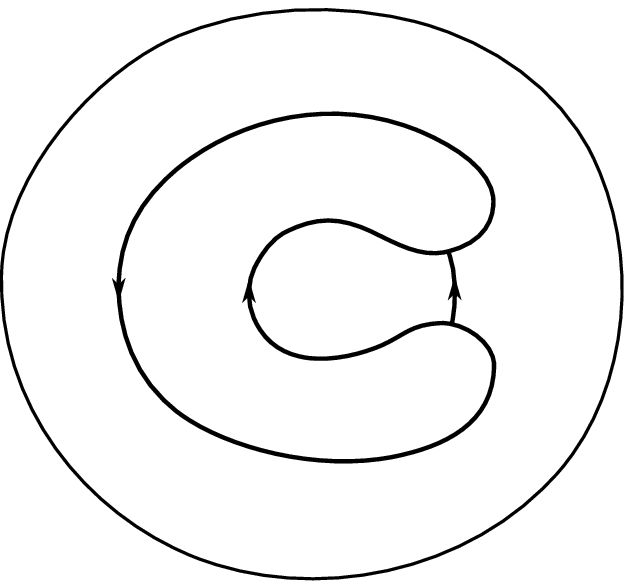
\caption{}  \label{fig:folded-cell}
\end{figure}
\end{definition}


\begin{lemma}[no folded cells] \label{lm:no-folded-cells}
Assume that no relator of rank $\al$ is conjugate in $G_{\al-1}$ to its inverse.
Then folded cells do not exist. Consequently, if $\Pi$ is a contiguity subdiagram of a cell of rank ~$\al$
to itself then for a contiguity arc $!P$ of $\Pi$ we have $\mu(\lab(!P)) < \la$.
\end{lemma}

\begin{proof}
The first statement is an immediate consequence of Definition \ref{df:folded-cell}.
If $\Pi$ is a contiguity subdiagram of a cell $!D$ of rank ~$\al$ to itself
and $!P$ is a contiguity arc of $\Pi$ with  $\mu(\lab(!P)) \ge \la$ then, as in the proof of 
Lemma \ref{lm:small-cancellation-diagram}, we conclude that $!D$ is a folded cell.
\end{proof}

%

\subsection{}
We will be considering finite sets of disjoint contiguity subdiagrams of a diagram $\De$ of rank ~$\al\ge1$.
Our goal is to produce a maximal, in an appropriate sense, such a set.

Let $\set{\Pi_i}$ be a finite set of pairwise disjoint contiguity subdiagrams of $\De$.
Each connected component $\Th$ of the complement $\De_{\al-1} - \bigcup \Pi_i$
is a diagram of rank ${\al-1}$ with
a naturally induced boundary marking of rank $\al-1$ defined as follows:
\begin{itemize}
\item 
Bridges of $\De_{\al-1}$ occurring in the boundary of $\Th$ become bridges of $\Th$;
\item
If $!u$ is a bond of $\De_{\al-1}$ occurring in the boundary of some contiguity subdiagram $\Pi_i$ and
$!u^{-1}$ occurs in the boundary of $\Th$ then $!u^{-1}$ becomes a bridge of $\Th$;
\item 
The rest of the boundary of $\Th$ consists of subpaths of sides of $\De_{\al-1}$,
or possibly cyclic sides of $\De_{\al-1}$, which are viewed as sides of $\Th$.
\end{itemize}

The following observation follows easily by induction on the number of contiguity subdiagrams 
in a set $\set{\Pi_i}$.
\begin{lemma} \label{lm:filling-geometry}
Let $\set{\Pi_i}$ be a set of
$r$ pairwise disjoint contiguity subdiagrams of a diagram ~$\De$ of rank $\al\ge1$.
Let $\set{\Th_j}$ be the set of all connected components of the complement $\De_{\al-1} - \bigcup_i \Pi_i$.
Then 
$$
  \sum_j c(\Th_j) = c(\De_{\al-1}),
$$
$$
  \sum_j \chi(\Th_j) = \chi(\De_{\al-1}) + r.
$$
\end{lemma}

\begin{proposition} \label{pr:contiguity-map-exists}
Let $\De$ be a diagram of rank $\al\ge1$. Then there exists another diagram ~$\De'$ of rank $\al$ 
and a finite set $\set{\Pi_i}$ of pairwise disjoint contiguity subdiagrams of ~$\De'$ such that:
\begin{enumerate}
\item 
$\De'$ is obtained from $\De$ by replacing its subdiagram $\De_{\al-1}$ with another subdiagram over the presentation of $G_{\al-1}$
of the same frame type; in particular, $\De$ and $\De'$ have the same boundary marking and the same frame type.
\item
any connected component $\Th$ of $\De'_{\al-1} - \bigcup_i \Pi_i$ is a small diagram of rank ${\al-1}$.
\item
if $c(\De_{\al-1}) > 0$ then $c(\Th) >0$ for each connected component $\Th$ of $\De'_{\al-1} - \bigcup_i \Pi_i$.
\end{enumerate}
\end{proposition}

\begin{proof}
Let $\De$ be a diagram of rank $\al$ and let
$\set{\Pi_i}$ be a finite set of pairwise disjoint contiguity subdiagrams of ~$\De$.
Assume that a connected component $\Th$ of $\De_{\al-1} - \bigcup_i \Pi_i$
has a bond, possibly after refinement. 
We describe how to obtain from $\set{\Pi_i}$ a 
new set of disjoint contiguity subdiagrams by either increasing the set or increasing the part of $\De$ covered by ~$\set{\Pi_i}$.
We track on two inductive parameters: the number $N$ of 
connected components of $\De_{\al-1} - \bigcup_i \Pi_i$ and the total length $L$ of sides of these
components. 

Refining $\Th$ inside $\De$ we may assume that $\Th$ has a bond ~$!u$.
An easy analysis shows that any bond in $\Th$ is also a bond in $\De_{\al-1}$. 
Performing surgery as described in \ref{ss:bond-enhacement} 
we may assume that the label of $!u$ belongs to $\cH_{\al-1}$.

Observe that $!u$ cuts $\Th$ into a subdiagram $\Th_1$ or two subdiagrams $\Th_1$ and $\Th_2$ which 
inherit the boundary marking of rank $\al-1$. From the definition of complexity $c(*)$ we immediately see
that $c(\Th) = \sum_i c(\Th_i)$ in either of the two cases.
Since $!u$ is not homotopic to a subpath of a side of $\Th$
we have $c(\Th_i) \ge 0$ for each $\Th_i$. We change the set $\set{\Pi_i}$ 
depending on the following two cases:

{\em Case\/} 1: $!u$ cuts $\Th$ into two subdiagrams $\Th_1$ and $\Th_2$ and at least one of them, 
say $\Th_1$,
satisfies $c(\Th_1) = 0$. Then $\Th_1$ is a simply connected subdiagram with two bridges, and hence
a contiguity subdiagram of $\De$. 
Note that if for both $\Th_1$ and $\Th_2$ we have $c(\Th_1) =  c(\Th_2) = 0$ then $\De$ has no cells of 
rank $\al$ and is itself a contiguity subdiagram. We then can take $\set{\Pi_i} = \set{\De}$.
We assume that this is not the case.

Let $!v$ be the other bridge
of ~$\Th_1$. If $!u$ is a bridge of $\De_{\al-1}$ then we simply add $\Th_1$ to the set ~$\set{\Pi_i}$.
Otherwise $!v^{-1}$ is a bond of $\De_{\al-1}$ occurring in the boundary loop of some $\Pi_i$;
then we attach $\Th_1$ to $\Pi_i$ (see Figure \ref{fig:contiguity-map-exists-1}.
Note that the label of at least one side of $\Th_1$ is nonempty 
(by condition (iv) of Definition \ref{df:bond} applied to $\Th$ and ~$!u$).
Hence after performing this operation, $L$ is strictly decreased and $N$ is not changed.
\begin{figure}[h]
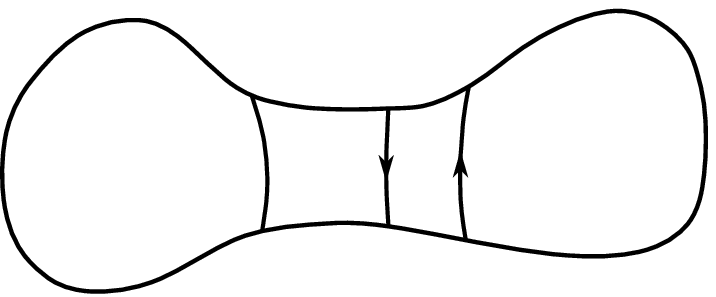
\caption{}  \label{fig:contiguity-map-exists-1}
\end{figure}

{\em Case\/} 2: Case 1 does not hold. 
We refine $\De$ so that $!u$ ``bifurcates'' into two paths ~$!u'$ and ~$!u''$ 
(Figure \ref{fig:contiguity-map-exists-2})
and obtain a ``degenerate'' contiguity subdiagram $\Pi$ of $\De$ between ~$!u'$ and ~$!u''$. 
We then add $\Pi$ to the set $\set{\Pi_i}$.
The operation strictly increases $N$ not changing $L$.
\begin{figure}[h]
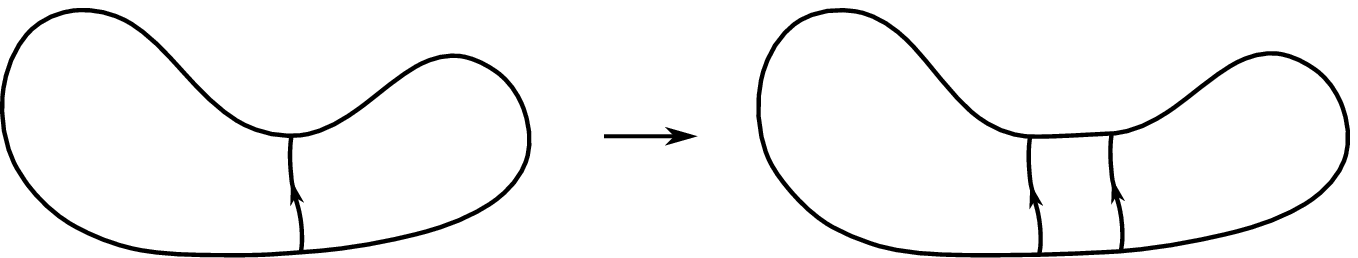
\caption{}  \label{fig:contiguity-map-exists-2}
\end{figure}

Starting from the empty set of contiguity subdiagrams $\Pi_i$,
we perform recursively the procedure described above. Each step we either decrease $L$ not changing $N$
or increase $N$ not changing $L$.
Furthermore, each time there is at most one connected component ~$\Th$
of $\De_{\al-1} - \bigcup_i \Pi_i$ with $c(\Th) \le 0$ and it exists only if $c(\De_{\al-1}) \le 0$ for 
the initial diagram ~$\De$. 
By Lemma ~\ref{lm:filling-geometry}, $N$ is bounded from above,
so the procedure terminates after finitely many steps.
Upon termination, all connected components of $\De_{\al-1} - \bigcup_i \Pi_i$ become small by construction.
\end{proof}

\begin{definition} \label{df:tight-set}
We say that a set $\set{\Pi_i}$ 
satisfying the conclusion of Proposition \ref{pr:contiguity-map-exists} 
is a {\em tight} set of contiguity subdiagrams of ~$\De'$.
\end{definition}

\section{Global bounds on diagrams} \label{s:diagram-bounds}
%
%

\subsection{}
Let $\De$ be a diagram of rank $\al\ge1$ and $\set{\Pi_j}$ a set of disjoint
contiguity subdiagrams of ~$\De$. We have a tiling of $\De$ by
subdiagrams of three types:
cells of rank $\al$, contiguity subdiagrams ~$\Pi_i$ and connected components of 
the complement $\De_{\al-1} - \bigcup \Pi_i$. 
We name these subdiagrams {\em tiles
of index 2, 1 and 0} respectively and refer to them also as 
{\em internal} tiles. 
We consider also external 2-cells of $\De$ as tiles of index ~2, so 
with these extra tiles we obtain a tiling of the 2-sphere. 
Boundary loops of all tiles carry naturally
induced partitions into subpaths (allowed to be whole loops) called {\em tiling sides}, defined precisely as follows  
(see Figure~\ref{fig:tiling-sides}):
\begin{figure}[h]
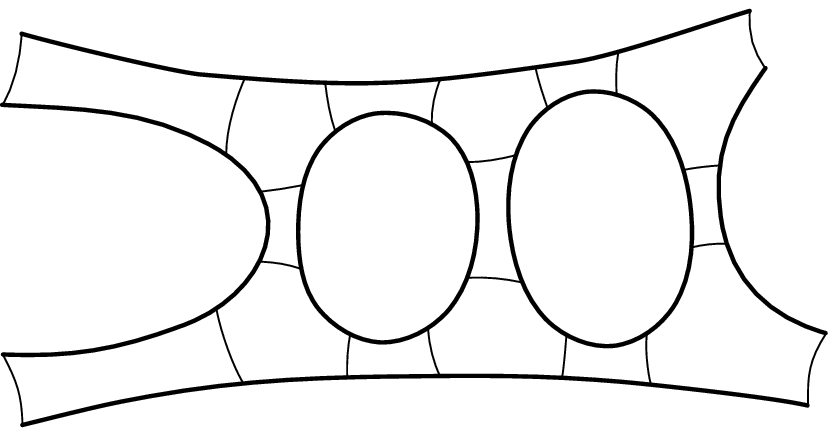
\caption{}  \label{fig:tiling-sides}
\end{figure}
\begin{itemize}
\item 
The boundary loop $\de \Pi_i$ of each contiguity subdiagram ~$\Pi_i$ is partitioned as 
$!P \cdot !u \cdot !Q \cdot !v$ where $!P$ and $!Q$ are the contiguity arcs; 
thus $\de \Pi_i$ consists of four tiling sides.
\item
A component $\Th$ of $\De_{\al-1} - \bigcup_i \Pi_i$ has the induced boundary marking of rank $\al-1$
(in this case, a tiling side can be a cyclic side of $\Th$).
\item 
The boundary loop of a cell of rank $\al$ either has no nontrivial partition  
(in this case it is considered as a cyclic tiling side)
or is partitioned as an alternating product of 
contiguity arcs of subdiagrams $\Pi_i$ and paths $!S$ where 
$!S^{-1}$ is a side of a component of $\De_{\al-1} - \bigcup_i \Pi_i$.
\item
The partition of the boundary loop $!L$ of an external cell is defined as follows:
we take the partition of $!L$ induced by the boundary marking of rank $\al-1$ of $\De_{\al-1}$
and additionally subdivide sides of rank $\al-1$ into
alternating products of contiguity arcs of subdiagrams $\Pi_i$ and paths $!S$ where 
$!S^{-1}$ is a side of a component of $\De_{\al-1} - \bigcup_i \Pi_i$.
\end{itemize}

Note that we view on tiling sides as paths, i.e.\ they are considered with direction. 
By construction, the set of all tiling sides is closed under inversion, 
and each tiling side occurs in a unique way in a boundary loop of a tile. 

\begin{definition} \label{df:curvature}
Let $\cS$ be the set of tiling sides associated with $\set{\Pi_i}$.
For every tile $T$, we denote $\cS(T)$ the set of tiling sides occurring in the boundary loops of $T$.

A {\em discrete connection} on a pair $(\De,\set{\Pi_i})$ is a function $w: \cS \to \R$ such that 
 $w(!s^{-1}) = -w(!s)$ for any $!s$.
Given $w$, we define the {\em curvature} $\ka(T)$ of each internal tile $T$:
$$
  \ka(T) = (-1)^{\text{index}(T)} \chi(T) + \sum_{!s \in \cS(T)} w(!s).
$$
(Note that inequality $\chi(T) \ne 1$ is possible only if $T$ has index 0.)
For an external tile ~$T$, by definition, 
$$
  \ka(T) = \sum_{!s \in \cS(T)} w(!s).
$$
By definition, the total curvature $\ka(\De)$ of $\De$ is the sum of curvatures of all internal tiles of $\De$.
The total curvature of external tiles of $\De$ is the {\em curvature along the boundary of $\De$},
denoted $\ka(\bd\De)$.
\end{definition}

\begin{proposition}[A discrete version of the Gauss--Bonnet theorem] 
\label{pr:Gauss-Bonnet}
For any diagram $\De$ of rank $\al\ge1$ and any set $\set{\Pi_i}$ of disjoint contiguity subdiagrams of $\De$,
$$
  \ka(\De) + \ka(\bd\De)  = \chi(\De).
$$
In particular, if $\ka(T)$ is non-positive for any internal tile $T$ then $\ka(\bd\De)  \ge \chi(\De)$.
\end{proposition}

\begin{proof}
Let $t$ be the number of cells of rank $\al$ of $\De$. It follows from the second equality of 
Lemma \ref{lm:filling-geometry} that
$$
  \sum_T (-1)^{\text{index}(T)} \chi(T) = \chi(\De_{\al-1}) + t = \chi(\De)
$$
where the sum is taken over all internal tiles $T$ of $\De$. In the expansion of $\ka(\De) + \ka(\bd\De)$
all summands $w(!s)$ are canceled because of the assumption $w(!s^{-1}) = -w(!s)$.
\end{proof}

\begin{proposition}[bounding the number of cells] \label{pr:principal-bound-cells}
Let $\De$ be a reduced diagram of rank ~$\al\ge1$ with $c(\De_{\al-1}) > 0$.
Denote
\begin{equation} \label{eq:de-th} 
  \nu = \frac{\ze}{1 - 2\ze} = \frac{1}{18}, \quad
  \th = \frac16 (5 - 22\nu) = \frac{17}{27} .
\end{equation}

Let $\cT$ be a tight set of contiguity subdiagrams of ~$\De$.
We assume that the following extra condition is satisfied: 
\begin{itemize}
\item[(*)]
Each cell of rank $\al$ of $\De$ has at most one contiguity subdiagram $\Pi \in \cT$ to sides of ~$\De$.
\end{itemize}
Let $M$ be the number of cells of rank $\al$ of $\De$. Then
\begin{equation} \label{eq:principal-bound-cells}
  \th M \le \frac23 (1+\nu) b(\De) - \chi(\De).
\end{equation}
\end{proposition}

For the proof, we 
define a discrete connection $w$ on the pair $(\De, \set{\Pi_i})$.
Note that $w(!S^{-1}) = -w(!S)$ by Definition \ref{df:curvature} and thus defining $w(!S)$
automatically defines $w(!S^{-1})$.

Recall that sides of $\De_{\al-1}$ are divided into three types: sides of ~$\De$, 
central arcs of bridges of native rank $\al$ and the boundary loops of cells of rank $\al$.
If $!S$ is a side of $\De_{\al-1}$ or a subpath of a side of $\De_{\al-1}$ then
we assign to $!S$ type I, II or III respectively. 

Before defining $w$, we perform on $\De$ the following ``cleaning'' procedure:
if a bridge of $\De_{\al-1}$ occurs in the boundary of some contiguity subdiagram $\Pi_i$ 
then we cut off ~$\Pi_i$ from $\De$
taking the bond in the boundary of $\Pi_i$ as a new bridge of the resulting $\De_{\al-1}$.
Thus we may assume that 
\begin{itemize}
\item[(**)]
every bridge of $\De_{\al-1}$ occurs in the boundary of a tile of index 0
(i.e.\ a connected component of $\De_{\al-1} - \bigcup_{\Pi \in \cT} \Pi$).
\end{itemize}

\medskip 
We define $w$ as follows:

\begin{enum-par} \label{it:principal-bound-filling-weights}
Let $\Th$ be a connected component of $\De_{\al-1} - \bigcup_{\Pi \in \cT} \Pi$.
For each bond or bridge $!u$ of rank $\al-1$ occurring in the boundary of $\Th$, 
define
$$
  w(!u) = - \frac13 (1 + \nu).
$$
For each side $!S$ of $\Th$, 
$$
  w(!S) = \ze\th |!S|_{\al-1}.
$$
\end{enum-par}

\begin{enum-par} \label{it:conectivity-def-index1}
Let $\Pi \in \cT$ and let $\de\Pi = !P !u_1 !Q !u_2$ as in Definition \ref{df:contiguity-subdiagram}.
By (**), for each $i=1,2$ the tiling side $!u_i^{-1}$ 
occurs in the boundary of a connected component of $\De_{\al-1} - \bigcup_{\Pi\in\cT} \Pi$. 
By (i), we already have
$$
  w(!u_i) = - w(!u_i^{-1}) = \frac13 (1 + \nu).
$$

We define $w(!P)$ (the definition of $w(!Q)$ is similar):
\begin{equation} \label{eq:contiguity-connection-def}
  w(!P) = 
  \begin{cases}
      0 & \text{if $!P$ has type I or II} \\
      \frac13 (1 - 2\nu) & \text{if $!P$ has type III and $!Q$ has type I} \\
      \frac16 (1 - 2\nu) & \text{if $!P$ has type III and $!Q$ has type II or III}
  \end{cases}
\end{equation}
\end{enum-par}

\begin{enum-par}
Let $!D$ be a cell of rank $\al$ of $\De$ and $!S$ be a tiling side occurring in $\de!D$.
The value of $w(!S)$ is already defined by (i) and (ii).
We have:
\begin{itemize}
\item 
If $!S^{-1}$ is the contiguity arc of a contiguity subdiagram $\Pi \in \cT$ of $!D$ to a side of $\De_{\al-1}$
of type ~I or II
then $w(!S) = -\frac13 (1 - 2\nu)$.
\item
If $!S^{-1}$ is the contiguity arc of a contiguity subdiagram $\Pi \in \cT$ of $!D$ to a side of $\De_{\al-1}$
of type ~III 
then $w(!S) = -\frac16 (1 - 2\nu)$.
\item
If $!S^{-1}$ occurs in the boundary of a connected component of $\De_{\al-1} - \bigcup_{\Pi\in\cT} \Pi$
then $w(!S) = -\ze\th |!S|_{\al-1}$.
\end{itemize}

\end{enum-par}

\medskip
We provide an upper bound for the curvature of any internal tile. 
For contiguity subdiagrams $\Pi \in \cT$ we immediately have $\ka(\Pi) \le 0$ by (ii).

Let $\Th$ be a connected component of $\De_{\al-1} - \bigcup_{\Pi \in \cT} \Pi$.
We have
$$
  \ka(\Th) = \chi(\Th) - \frac13(1+\nu) b(\Th) + \ze\th \sum_{!S} |!S|_{\al-1}
$$
where the sum is taken over the sides $!S$ of $\Th$.

If $\al = 1$ then $\sum |!S|_{\al-1} = 0$ (Proposition \ref{pri:small-rank0-diagram}).
If $\al \ge 2$ then by Proposition \ref{pr:principal-bound-sides}$_{\al-1}$,
$$
  \th \sum |!S|_{\al-1} \le \frac23 (1+\nu) b(\Th) - \chi(\Th)
$$
Using the fact that $c(\Th)>0$ it is easy to check that $\ka(\Th) \le 0$
in both cases $\al=1$ and $\al \ge 2$.
(The critical case is when $b(\Th) = 3$ and $\chi(\Th) = 1$; in this case we have $\ka(\Th) = -\nu$
if $\al=1$ and $\ka(\Th) = 0$ if $\al \ge 2$
by definition \eqref{eq:de-th} of $\nu$). 

Finally, let $!D$ be a cell of rank $\al$ of $\De$. We prove that $\ka(!D) \le - \th$.
By (*), $!D$ has at most one contiguity subdiagram to sides of $\De_{\al-1}$ of type I.
We consider first the case when $!D$ has one.
Let $r$ be the number of 
contiguity subdiagrams of ~$!D$ to sides of types II and III.
The remaining $r+1$ subpaths $!S_1, !S_2, \dots !S_{r+1}$ of $\de !D$ are tiling sides
such that $!S_i^{-1}$ belong to boundary loops of connected components of $\De_{\al-1} - \bigcup_{\Pi \in \cT} \Pi$; 
so we have
$$
  \ka(!D) \le 1 - \frac13 (1 - 2\nu) - r \left( \frac16 (1 - 2\nu) \right) - 
    \ze\th \sum_{i=1}^{r+1} |!S_i|_{\al-1}.
$$
By condition (S1) in \ref{ss:condition} and 
Lemmas \ref{lm:small-cancellation-diagram}, \ref{lm:no-folded-cells},
\begin{align*}
  \sum_{i=1}^{r+1} |!S_i|_{\al-1} \ge (1 - \rho - r \la) \Om = (9-r) \la\Om .
\end{align*}
Hence
\begin{equation} \label{eq:ICS-principal-bound-cell} 
  \ka(!D) \le \frac23(1+\nu) - r \left( \frac16 (1 - 2\nu) \right) 
    - \ze\th \la\Om \max (0,\ 9 - r). 
\end{equation}
If $r \ge 9$ then the coefficient before $r$ in the right-hand side of \eqref{eq:ICS-principal-bound-cell} is negative.
If $r \le 9$ then the coefficient is
$$
  - \frac16 (1 - 2\nu) + \ze\th \la \Om 
$$
which is positive since by the second inequality \eqref{eq:ISC-main} we have 
$\ze\th \la \Om \ge 20\ze \th = \th > \frac16$.
Hence the maximal value of the expression in \eqref{eq:ICS-principal-bound-cell}
is when $r = 9$.
Substituting $r=9$ into the right-hand side of \eqref{eq:ICS-principal-bound-cell} we obtain the expression
$$
  \frac23(1+\nu) - \frac96 (1 - 2\nu)
$$
which is equal $-\th$ by \eqref{eq:de-th}. This shows that $\ka(!D) \le -\th$.

Assume that $!D$ has no contiguity subdiagrams to sides of type I. 
Let, as above, $r$ be the number of 
contiguity subdiagrams of $!D$ to sides of types II and III and 
$!S_1, !S_2, \dots !S_{r}$ be the remaining $r$ tiling sides occurring in $\de !D$
such that $!S_i^{-1}$ belong to boundary loops of connected components of $\De_{\al-1} - \bigcup_{\Pi \in \cT} \Pi$.
Instead of \eqref{eq:ICS-principal-bound-cell}
we have 
\begin{equation} \label{eq:ICS-principal-bound-cell-case2}
  \ka(!D) \le 1 - r \left( \frac16 (1 - 2\nu) \right) 
    - \ze\th N \max (0,\ 1 - r \la) .
\end{equation}
If we allow $r$ to be a non-negative real then 
the maximal value of the right-hand side  is when 
$$
  1 - r \la  = 0. 
$$
Substituting $r = \frac{1}{\la}$ into the left-hand side of \eqref{eq:ICS-principal-bound-cell-case2} 
we obtain the expression
$$
  1 - \frac{1 - 2\nu}{6\la}
$$
which is less then $-\th$ since $\la \le \frac1{24}$.

Finally, we compute an upper bound for $\ka(\bd \De)$. For a tiling side $!S$ occurring in the
boundary loop of an external cell of $\De$ (the loop has the form $!L^{-1}$ where $!L$ is a boundary loop of ~$\De$)
we have three possibilities: either $!S^{-1}$ is a contiguity arc of a subdiagram $\Pi \in \cT$,
$!S^{-1}$ is a side of a component of $\De_{\al-1} - \bigcup_{\Pi\in\cT} \Pi$, 
or $!S^{-1}$ is a bridge of $\De_{\al-1}$
In the first two cases we have $w(!S) \le 0$ according to (ii) or (i) respectively.
If $!S^{-1}$ is a bridge of $\De_{\al-1}$ then by (**), $!S^{-1}$ is 
also a bridge of some component of $\De_{\al-1} - \bigcup_{\Pi\in\cT} \Pi$
and by (i),
$$
  w(!S) = \frac13 (1+\nu).
$$
Note that each bridge of $\De$ produces at most two bridges of $\De_{\al-1}$. 
Hence $b(\De_{\al-1}) \le 2 b(\De)$. We obtain
\begin{equation} \label{eq:ICS-principal-bound-boundary}
  \ka(\bd \De) \le \frac13 (1+\nu) b(\De_{\al-1}) \le \frac23 (1+\nu) b(\De)
\end{equation}
Application of Proposition \ref{pr:Gauss-Bonnet} gives
$$
  \frac23 (1+\nu) b(\De) - \th M  \ge \chi(\De)
$$
as required. The proof of Proposition \ref{pr:principal-bound-cells} is finished.

\begin{lemma} \label{lm:monogon-regularity}
Let $\De$ be a reduced disk diagram of rank $\al\ge1$.
If $\De$ has a single (cyclic or non-cyclic) side then $\De$ has no cells of rank $\al$.
\end{lemma}

\begin{proof}
Let $\De$ be a reduced disk diagram of rank $\al$ with a single side, i.e.\ $\De$ is 
of monogon or nullgon type. 
Assume that $\De$ has a cell of rank $\al$. 
We choose such $\De$ with minimal possible non-zero number $M$ of cells of rank $\al$.
We then have $\chi(\De_{\al-1}) \le 0$ and hence $c(\De_{\al-1}) > 0$.
We can assume that $\De$ is given a tight set $\cT$ of contiguity subdiagrams. 
If each cell of rank ~$\al$ of ~$\De$ has at most one contiguity subdiagram $\Pi \in \cT$ to the side of ~$\De$
then application of Proposition \ref{pr:principal-bound-cells} would give
$$
  \th M \le \frac23 (1+\nu) - 1 < 0.
$$
Therefore, $\De$ has a cell $!D$ of rank $\al$ having two contiguity subdiagram 
$\Pi_1,\Pi_2 \in \cT$ to the side of ~$\De$.
The union $!D \cup \Pi_1 \cup \Pi_2$ cuts off from $\De$ a disk diagram $\De'$ of rank $\al$ 
with a single side and a single bridge (Figure \ref{fig:monogon-regularity}).
\begin{figure}[h]
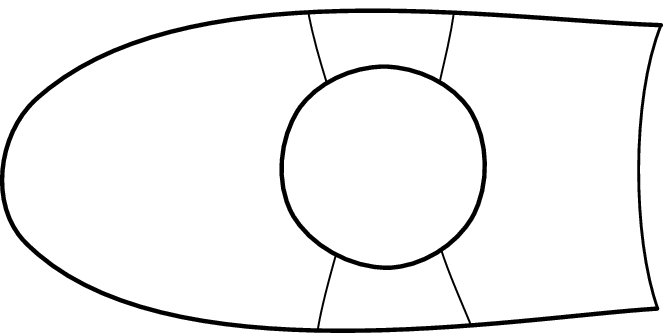
\caption{}  \label{fig:monogon-regularity}
\end{figure}
The assumption that $\De$ is reduced implies that $\De'$ is reduced as well.
By the choice of $\De$, $\De'$ has no cells of rank $\al$.
Then for some component ~$\Th$ of $\De_{\al-1} - \bigcup_{\Pi\in\cT} \Pi$
we have $c(\Th) = 0$ contrary to the choice of a tight set $\cT$ of contiguity subdiagrams of $\De$
(Definition \ref{df:tight-set}). 
\end{proof}

\begin{proposition} \label{pr:reduced-nontrivial}
If a non-empty word $X$ is reduced in $G_\al$ then $X \ne 1$ in $G_\al$.
\end{proposition}

\begin{proof}
Let $\al\ge1$.
Let $X$ be reduced in $G_\al$ and $X = 1$ in $G_\al$. Consider a reduced disk diagram $\De$ of rank ~$\al$
with one side labeled $X$ and one bridge labeled by the empty word. Lemma \ref{lm:monogon-regularity} says
that $\De$ has no cells of rank $\al$ and hence we have $X = 1$ in $G_{\al-1}$.
Since $\cR_\al \seq \cR_{\al-1}$, arguing by induction we conclude that $X = 1$ in the free group $G_0$. 
Since $X$ is freely reduced (definition \ref{ss:reduced-word}) we conclude that $X$ is empty.
\end{proof}

\begin{lemma} \label{lm:small-bond-regularity}
Let $\De$ be a reduced diagram of rank $\al\ge1$ and let $!u$ be a simple path in $\De$ 
homotopic rel endpoints
to a subpath $!S$ of a side of $\De$. Assume, moreover, that the label of ~$!u$ is equal in $G_{\al-1}$ to 
a word in $\cH_{\al-1}$. Then the subdiagram of $\De$ with boundary loop $!S !u^{-1}$ has no cells of rank $\al$.
\end{lemma}

\begin{proof}
Let $\De'$ be the subdiagram of $\De$ with boundary loop $!S !u^{-1}$ and let $w \in \cH_{\al-1}$
be a word such that $\lab(!u) = w$ in $G_{\al-1}$.
We attach to $\De'$ a diagram $\Th$ over the presentation of ~$G_{\al-1}$ with 
boundary loop $!u !w^{-1}$ where $\lab(!w) = w$. 
We consider $\De' \cup \Th$ as a diagram of rank ~$\al$ with one side $!S$ and one bridge $!w^{-1}$.
Note that any simple path in $\De' \cup \Th$ with endpoints in $\De'$ is homotopic rel endpoints
to a simple path in $\De'$.
Moreover, this holds also if $\De' \cup \Th$ is refined to a diagram $\Si$ and
we take a refinement of $\De'$ in $\Si$ instead of $\De'$.
This implies that 
$\De' \cup \Th$ is a reduced diagram of rank $\al$.
Then by Lemma \ref{lm:monogon-regularity}, $\De' \cup \Th$ has no cells of 
rank ~$\al$.
\end{proof}

\begin{proposition}[bounding sides of a small diagram, raw form]  \label{pr:principal-bound-sides}
Let $\De$ be a small diagram of rank $\al\ge1$.
Assume that $\De$ is not of bigon type and $c(\De_{\al-1}) > 0$.
Then
\begin{equation} \label{eq:ICS-principal-bound}
  \th \sum_{!S} |!S|_\al \le \frac23 (1+\nu) b(\De) - \chi(\De)
\end{equation}
where the sum is taken over all sides $!S$ of $\De$. 
\end{proposition}

\begin{proof}
We make $\De$ reduced and endow it
with a tight set ~$\cT$ of contiguity subdiagrams.
We assign to subpaths of sides of $\De_{\al-1}$ type I, II and III 
as in the proof of Proposition \ref{pr:principal-bound-cells}
and make several observations about $\cT$.

\begin{claim} \label{cl:principal-bound-no-I-I}
There are no contiguity subdiagrams $\Pi \in \cT$
between two (not necessarily distinct) sides of type I of $\De_{\al-1}$.
\end{claim}

Assume $\Pi$ is such a contiguity subdiagram. Let $\de\Pi = !P !u_1 !Q !u_2$
where $!P$ and $!Q$ are the contiguity arcs of $\Pi$.
According to Definition \ref{df:contiguity-subdiagram}
at least one of $!u_i$'s, say $!u_1$, is a bond in $\De_{\al-1}$
(otherwise $\Pi = \De_{\al-1}$ contrary to the assumption $c(\De_{\al-1}) > 0$).
Checking with Definition \ref{df:bond} we see that $!u_1$ is also a bond in $\De$
(condition (iii) of Definition \ref{df:bond} holds due to Lemma \ref{lm:small-bond-regularity}).
This contradicts the assumption that $\De$ is small.

\begin{claim} \label{cl:diagram-regularity}
Up to inessential change of $\De$ we may assume that condition (*) of 
Proposition ~\ref{pr:principal-bound-cells}
is satisfied, i.e.\ each cell of rank $\al$ of $\De$ has at most one contiguity subdiagram $\Pi \in \cT$ 
to sides of type I of $\De_{\al-1}$.
\end{claim}

Assume that a cell $!D$ of rank $\al$ has two contiguity subdiagrams $\Pi_i \in \cT$ $(i=1,2)$ 
to sides ~$!S_i$ 
of type I. Let $!P_i$ be the contiguity arc of $\Pi_i$ that occurs in $!S_i$.
The boundary loop of $!D \cup \Pi_1 \cup \Pi_2$ has the form $!P_1 !u_1 !P_2 !u_2$ where labels of $!u_i$
are in $\cH_\al$.
Since $\De$ is small, at least one of the conditions (iii) or (iv) 
of Definition \ref{df:bond} should be violated for each of the paths $!u_i$.
If $!S_1 = !S_2$ and some $!u_i$ (and hence both $!u_1$ and $!u_2$) are homotopic rel endpoints to a
subpath of  
$!S_1$ then $!D \cup \Pi_1 \cup \Pi_2$ cuts off a reduced disk subdiagram $\De'$ of $\De$ 
with one bridge ~$!u_1^{-1}$ or $!u_2^{-1}$.
By Lemma \ref{lm:monogon-regularity}, $\De'$ has no cells of rank $\al$.
Then either $\De'$ is a component of $\De_{\al-1} - \bigcup_{\Pi\in\cT} \Pi$ or $\De'$ 
contains a component $\Th$ of $\De_{\al-1} - \bigcup_{\Pi\in\cT} \Pi$ with $c(\Th) = 0$.
We come 
to a contradiction with the choice of a tight set $\cT$ of contiguity subdiagrams of $\De$.

Assume that condition (iv) of Definition \ref{df:bond} fails for both $!u_1$ and $!u_2$. Then,
up to renumeration of ~$\Pi_1$ and $\Pi_2$, $!D \cup \Pi_1 \cup \Pi_2$
cuts off a simply connected subdiagram $\De'$ with boundary loop $!u_1^{-1} !T_1 !v !T_2$ where $!P_1 !T_1$ is an ending subpath of $!S_1$,
$!v$ is a bridge of $\De$, $!T_2 !P_2$ is a starting subpath of ~$!S_2$ 
and labels of $!P_1 !T_1$ and $!T_2 !P_2$ are empty, see Figure ~\ref{fig:diagram-regularity-exception}a.
\begin{figure}[h]
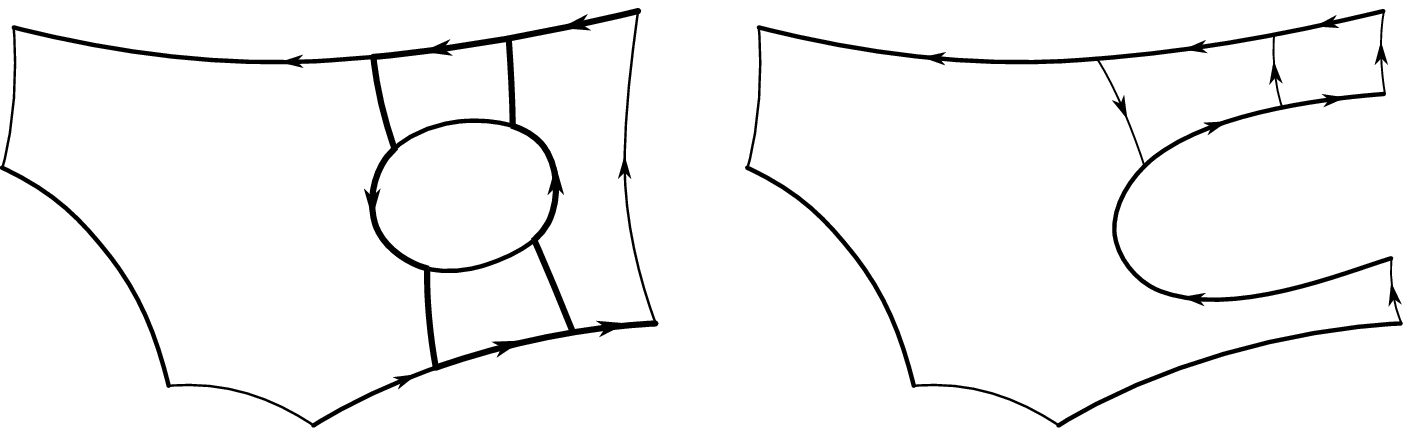
\caption{}  \label{fig:diagram-regularity-exception}
\end{figure}
In this case, we cut off the subdiagram $!D \cup \Pi_1 \cup \Pi_2 \cup \De'$ from $\De$. 
The operation does not change the values of $\sum |!S|_\al$, $b(\De)$ and $\chi(\De)$
in \eqref{eq:ICS-principal-bound} and preserves the assumption that $\De$ is small.
We have also $c(\De_{\al-1}) > 0$ for the modified $\De$ (otherwise $\De$ would be a monogon type 
contradicting Lemma~\ref{lm:monogon-regularity}).

\begin{claim} \label{cl:principal-bound-no-I-II}
Up to inessential change of $\De$ 
we may assume that there are no contiguity subdiagrams ~$\Pi \in \cT$
between sides of type I and ~II of $\De_{\al-1}$.
\end{claim}

Assume that $\Pi \in \cT$ is a contiguity subdiagram between sides of type I and II.
Let $\de\Pi = !P !u_1 !Q !u_2$ where $!P$ occurs in a side $!S$ of $\De$ and $!Q$ occurs in the central arc
$!R$ of a bridge $!v = !v_1 !R !v_2$.
Observe that any of the endpoints of $!P$ can be joined with any of the endpoints of $!v$
by a path labeled with a word in $\cH_\al$ in a graph composed from paths $!u_1$, $!u_2$ and $!v$,
see Figure ~\ref{fig:diagram-regularity-exception}b. 
Since $\De$ is small, this easily implies that $!v$
and $!S$ are adjacent in the boundary of $\De$.
Up to symmetry, assume that $!v !S$ occurs in a boundary loop of $\De$.
so $!R = !R_1 !Q !R_2$ and $!S = !S_1 !P !S_2$.
Note that $\lab(!S_1!P)$ is empty (otherwise $!v_1 !R_1 !u_1^{-1}$ would give a bond in $\De$ 
after refinement)
and $\lab(!Q !R_2)$ is nonempty (because $!u_1$ is a bond in $\De_{\al-1}$).
We cut off the subdiagram of $\De$  bounded by $!Q !R_2 !v_2 !S_1 !P !u_1$.
As in the proof of the previous claim, 
the operation does not change the values of terms in \eqref{eq:ICS-principal-bound},
the value of $c(\De_{\al-1})$ and keeps the assumption that $\De$ is small.
On the other hand, we decrease the total length of labels of sides $\De_{\al-1}$.
The claim is proved.

\medskip

We now define a discrete connection $w^*$ on $(\De,\cT)$ by changing the function $w$ defined in the proof 
of Proposition \ref{pr:principal-bound-cells}.
The new function $w^*$ differs from $w$ only on contiguity arcs of
contiguity subdiagrams $\Pi \in \cT$ as follows. 
Let $\de\Pi = !P !u_1 !Q !u_2$ where $!P$ and $!Q$ are the contiguity arcs of $\Pi$.
By Claims \ref{cl:principal-bound-no-I-I} and \ref{cl:principal-bound-no-I-II},
if $!P$ has type I then $!Q$ has necessarily type III. 
Instead of \eqref{eq:contiguity-connection-def} we define 
$$
  w^*(!P) = 
  \begin{cases}
      \th & \text{if $!P$ has type I} \\
      \frac13 (1 - 2\nu) - \th & \text{if $!P$ has type III and $!Q$ has type I} \\
      \frac16 (1 - 2\nu) & \text{in all other cases}
  \end{cases}
$$
For contiguity subdiagrams $\Pi \in \cT$ we immediately have $\ka^*(\Pi) \le 0$
where $\ka^*$ denotes the curvature function defined from $w^*$.
If $\Th$ is a connected component of $\De_{\al-1} - \bigcup_{\Pi\in\cT} \Pi$
then $\ka^*(\Th) = \ka(\Th) \le 0$.
Let $!D$ be a cell of rank $\al$ of $\De$. In view of Claim \ref{cl:diagram-regularity}
$$
  \ka^*(!D) \le \ka(!D) + \th \le 0.
$$
We provide a bound for $\ka^*(\bd \De)$. Let $t$ be the number of all 
contiguity subdiagrams $\Pi \in \cT$ between sides of type I and sides of type III.
Then 
\begin{align*}
  \ka^*(\bd \De) &\le \frac13 (1+\nu) b(\De_{\al-1}) - \th t - 
    \ze \th \!\!\!\! \sum_{!S\in\text{sides}(\Th)} \!\!\! |!S|_{\al-1} \\
    & \le \frac23 (1+\nu) b(\De) - \th \!\!\!\sum_{!S\in\text{sides}(\De)} \!\!\! |!S|_\al 
\end{align*}
where $\Th$ runs over all connected components of $\De_{\al-1} - \bigcup_{\Pi\in\cT} \Pi$.
Applying Proposition \ref{pr:Gauss-Bonnet} we obtain
$$
  \frac23 (1+\nu) b(\De) - \th \!\!\!\sum_{!S\in\text{sides}(\De)} \!\!\!|!S|_\al \ge \chi(\De)
$$
as required.
\end{proof}

Below we will often use Proposition \ref{pr:principal-bound-sides} in a slightly simplified form.
We introduce yet another numerical parameter
$$
  \eta = \frac{1+2\nu}{\th} = \frac{30}{17}.
$$

\begin{proposition}[bounding sides of a small diagram, simplified form] \label{pr:principal-bound}
If $\De$ is a small diagram of rank $\al$ of positive complexity then
\begin{equation} \label{eq:ICS-small-bound}
  \sum_{!S \in \,\text{\rm sides}(\De)} |!S|_\al \le \eta \, c(\De).
\end{equation}
\end{proposition}

\begin{proof}
By Proposition \ref{pri:small-rank0-diagram} we may assume that $\al\ge1$.
It remains to notice that if $c(\De) \ge 1$ then 
$$
  \frac1\th \left( \frac23 (1+\nu) b(\De) - \chi(\De) \right) \le  \eta \, c(\De).
$$
(The critical case is when $b(\De) = 3$ and $\chi(\De) = 1$. In this case we have 
the equality.)
\end{proof}

\begin{lemma} \label{lm:cell-regularity}
Let $\De$ be a reduced diagram of rank $\al\ge1$ and let $\cT$ be a tight set 
of contiguity subdiagrams of $\De$. 
Let $!D$ be a cell of rank $\al$ of $\De$. Then the following is true.
\begin{enumerate}
\item \label{lmi:cell-regularity-side}
Let $\Pi_1$ and $\Pi_2$ be two contiguity subdiagrams of $!D$ to
a side $!S$ of $\De_{\al-1}$. 
Then a subdiagram $\Th$ of $\De$ bounded by $\de !D$, $\Pi_1$, $\Pi_2$ and $!S$
(there are two of them if $!S$ is a cyclic side)
is not simply connected (see Figure~\ref{fig:cell-regularity}a).
\item
Let $\Pi$ be a contiguity subdiagram of $!D$ to itself. 
Then the subdiagram $\Th'$ of $\De$ bounded by $\de !D$ and $\Pi$ (see Figure~\ref{fig:cell-regularity}b) 
is not simply connected.
\item
If $\De$ is simply connected then any cell of rank $\al$ has at most one contiguity subdiagram
to each side of $\De_{\al-1}$ and has no contiguity subdiagrams to itself.
\end{enumerate}
\end{lemma}
\begin{figure}[h]
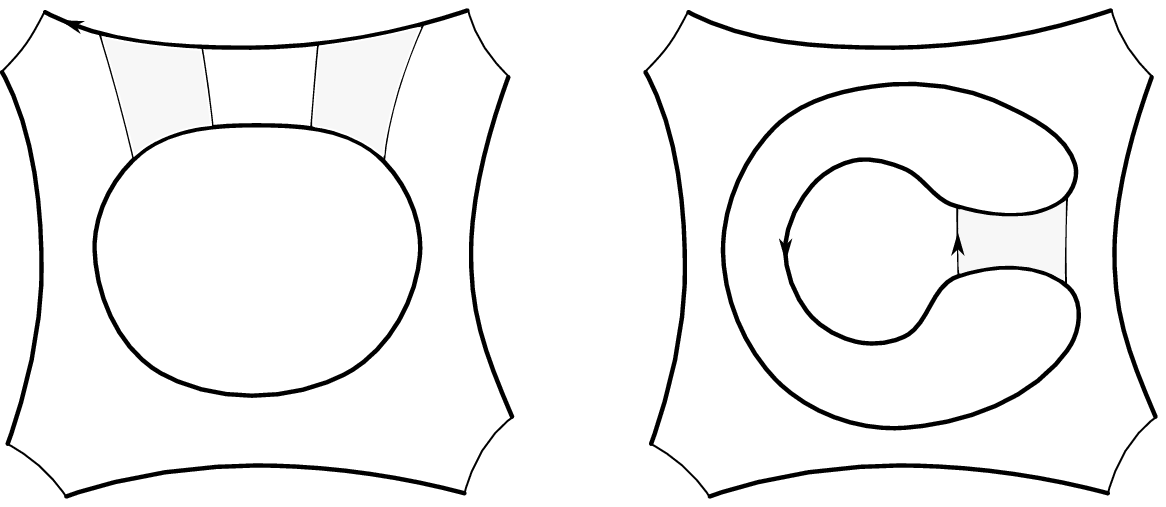
\caption{}  \label{fig:cell-regularity}
\end{figure}

\begin{proof}
%
(i) 
Assume that $\Th$ is simply connected.
We consider ~$\Th$ as a diagram of rank $\al$ with a single side that is a subpath of $!S$.
The assumption that $\De$ is reduced implies that $\Th$ is reduced.
By Lemma \ref{lm:monogon-regularity} $\Th$ has no cells of rank $\al$. 
Then we obtain a contradiction with the choice of a tight set $\cT$ of contiguity subdiagrams of $\De$.

(ii) 
Assume that $\Th'$ is simply connected. Let $\bd \Th' = !R !u$ where $!R^{-1}$
occurs in the boundary loop of $!D$ and $!u^{-1}$ is the bond in $\De_{\al-1}$ that occurs in $\bd \Pi$.
We consider $\Th'$ as a a diagram of rank $\al$ with one side $!S$ labeled by the empty word and one bridge $!R !u$
(formally, to fit the definition in \ref{df:boundary-marking} 
we have to take a copy of $\Th'$ and perform a refinement to make $!S$ a non-empty path).
By Lemma \ref{lm:monogon-regularity} $\Th'$ has no cells of rank $\al$ and we come to a contradiction since
in this case $!u^{-1}$ cannot be a bond in $\De_{\al-1}$ due to condition (iii)
of Definition \ref{df:bond}.

(iii) follows from (i) and (ii).
\end{proof}

%

\begin{proposition}[diagrams of small complexity are single layered]  \label{pr:single-layer}
Let $\De$ be a reduced diagram of rank $\al\ge1$ and let $\cT$ be a tight set 
of contiguity subdiagrams of $\De$.
\begin{enumerate}
\item 
If $\De$ is a disk diagram of bigon type then every cell of rank $\al$ of $\De$
has a contiguity subdiagram $\Pi \in \cT$ to each of the two sides of $\De$.
\item \label{pr:single-layer-3gon} \label{pr:single-layer-3-4-gon}
If $\De$ is a disk diagram of trigon or tetragon type then every cell 
of rank $\al$ of $\De$
has contiguity subdiagrams $\Pi\in \cT$ to at least two sides of $\De$.
\item \label{pri:annular-single-layer}
If $\De$ is an annular diagram with two cyclic sides 
then every cell of rank $\al$ of $\De$
has a contiguity subdiagram $\Pi \in \cT$ to each of the sides of $\De$.
\item \label{pri:annular-single-layer1}
If $\De$ is an annular diagram with one cyclic side and one non-cyclic side
then every cell $!D$ of rank $\al$ of $\De$
has at least two contiguity subdiagrams $\Pi, \Pi' \in \cT$ to sides of ~$\De$.
Here we admit the possibility that both $\Pi$ and $\Pi'$ are contiguity subdiagrams between ~$!D$ 
and the non-cyclic side of ~$\De$.
\end{enumerate}
\end{proposition}

\begin{proof}
Let $\De$ be a reduced diagram of rank $\al$ of a type listed in (i)--(iv).
We call a cell $!D$ of rank $\al$ of $\De$ {\em regular} if it
satisfies the conclusion of the corresponding statement (i)--(iv) and {\em exceptional} otherwise.
We need to prove that $\De$ has no exceptional cells. Observe that by Lemma \ref{lm:cell-regularity},
an exceptional cell has at most one contiguity subdiagram to sides of ~$\De$, i.e.\ such a cell
satisfies condition (*) of Proposition \ref{pr:principal-bound-cells}.
We use induction on the number $M$ of cells of rank $\al$ of $\De$.

(i) Let $\De$ be of bigon type, i.e.\ a disk diagram with two sides.
If $\De$ has no regular cells of rank $\al$ but has at least one exceptional cell then application
of Proposition ~\ref{pr:principal-bound-cells} gives a contradiction. 

Assume that $!D$ is a regular cell of $\De$. Let $\Pi_i$ $(i=1,2)$ be the contiguity subdiagram of $!D$
to $!X_i$. The complement of $!D \cup \Pi_1 \cup \Pi_2$ in $\De$ consists of two components $\De_1$ and $\De_2$ of bigon type 
with the induced boundary marking of rank $\al$ (see Figure~\ref{fig:single-layer-bigon}a).
\begin{figure}[h]
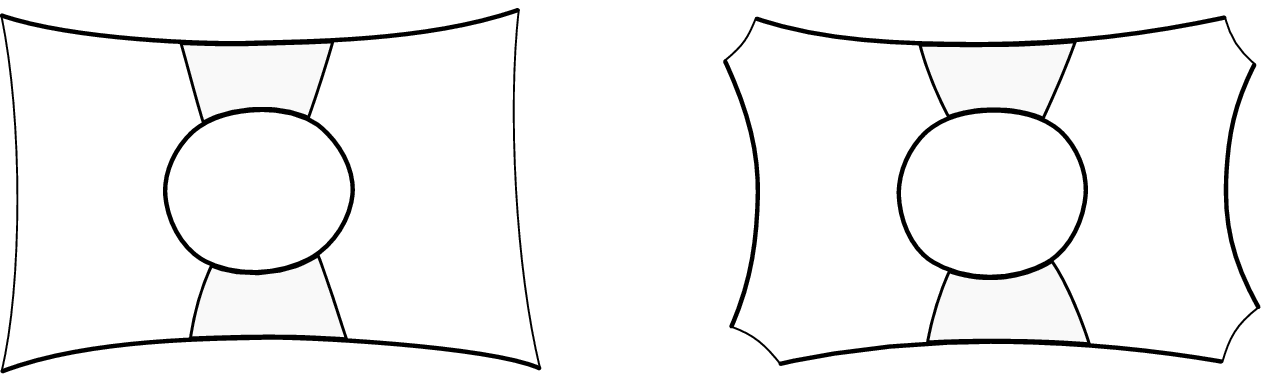
\caption{}  \label{fig:single-layer-bigon}
\end{figure}
The set of subdiagrams $\Pi \in \cT$
contained in $\De_i$ is a tight set of contiguity subdiagrams of $\De_i$.
Each of the subdiagrams $\De_i$ has a smaller number of cells of rank $\al$, so the statement follows by induction.

(ii)
Let $\De$ be of trigon or tetragon type.
Assume that $\De$ has a regular cell $!D$. Let $\Pi_i$ $(i=1,2)$ be contiguity subdiagrams of $!D$
to sides of $\De$. The complement of $\De - !D \cup \Pi_1 \cup \Pi_2$ consists of two components $\De_1$ and $\De_2$
with the induced boundary marking of rank $\al$ 
(Figure~\ref{fig:single-layer-bigon}b) making them diagrams of rank $\al$.
If $\De$ is of trigon type then $\De_1$ and $\De_2$ are
of trigon and bigon types. If $\De$ is of tetragon type then either
$\De_1$ and $\De_2$ are of tetragon and bigon types, or both $\De_i$
are of trigon type.
Then we can refer to (i) and the inductive hypothesis.

Assume that all cells of rank $\al$ of $\De$ are exceptional.
Then by Proposition \ref{pr:principal-bound-cells}
\begin{equation} \label{eq:single-layer-proof}
  \th M \le \frac83 (1 + \nu) -1
\end{equation}
which implies
$M \le 2$. 
Following the proof of Proposition \ref{pr:principal-bound-cells} we compute a better bound for ~$M$
and conclude that $M=0$.

Assume that $M \ge 1$ and let $!D$ be a cell of rank $\al$ of $\De$.
Consider the discrete connection $w$ on $(\De,\cT)$ defined in the proof of 
Proposition ~\ref{pr:principal-bound-cells}.
An upper bound for $\ka(!D)$ is
given by \eqref{eq:ICS-principal-bound-cell}.
The right-hand side of \eqref{eq:ICS-principal-bound-cell} is a linear expression on $r$ and,
as we have seen in the proof of Proposition ~\ref{pr:principal-bound-cells}, in the case $r \le 9$
the coefficient before ~$r$ is positive. To get a value for the upper bound, 
we compute the maximal possible value of ~$r$. 
Observe that by Lemma ~\ref{lm:cell-regularity},
$!D$ has no contiguity subdiagrams to itself, has at most one
contiguity subdiagram to another cell of rank $\al$ of $\De$ 
(if that cell exists) and 
the number of contiguity subdiagrams of $!D$ to sides of type II is at most 4; so $r \le 5$.
Then the maximal value of the right-hand side of  \eqref{eq:ICS-principal-bound-cell} is
achieved when $r=5$.
Substituting $r=5$ into \eqref{eq:ICS-principal-bound-cell} 
and using \eqref{eq:ISC-main} we obtain
\begin{align*}
  \ka(!D) &\le \frac23(1+\nu) - \frac56 (1 - 2\nu) 
    - 4 \ze\th \la\Om   \\
  &\le -\frac16 + \frac73 \nu - 4\th = - \frac{138}{54}.
\end{align*}
By \eqref{eq:ICS-principal-bound-boundary}
$$
    \ka(\bd \De) \le \frac83 (1+\nu) = \frac{152}{54}.
$$
Proposition \ref{pr:Gauss-Bonnet} gives
$$
    1 = \ka(\De) + \ka(\bd\De) \le \frac{14}{54}.
$$
The contradiction shows that the assumption $M \ge 1$ is impossible.

(iii): 
Similarly to the proof of (ii), assume first that $\De$ has a regular cell $!D$ of rank $\al$ 
with two contiguity subdiagrams 
$\Pi_1$ and $\Pi_2$ to sides of $\De$. By Lemma \ref{lm:cell-regularity}(i) these are contiguity subdiagrams
to distinct sides of $\De$. Then the complement $\De - (!D \cup \Pi_1 \cup \Pi_2)$ is a diagram of bigon type
and the statement follows directly from (i).

If all cells of rank $\al$ of $\De$ are exceptional and there is at least one cell of rank $\al$ then application of 
Proposition ~\ref{pr:principal-bound-cells} gives an immediate contradiction.

(iv): 
Assume that $\De$ has a regular cell $!D$ of rank $\al$ with two contiguity subdiagrams 
~$\Pi_i$ $(i=1,2)$ to sides of $\De$. There are two cases depending on whether or not $\Pi_1$ 
and $!\Pi_2$ are 
contiguity subdiagrams to distinct sides of $\De$ (see Figure~\ref{fig:annular-monogon}). 
\begin{figure}[h]
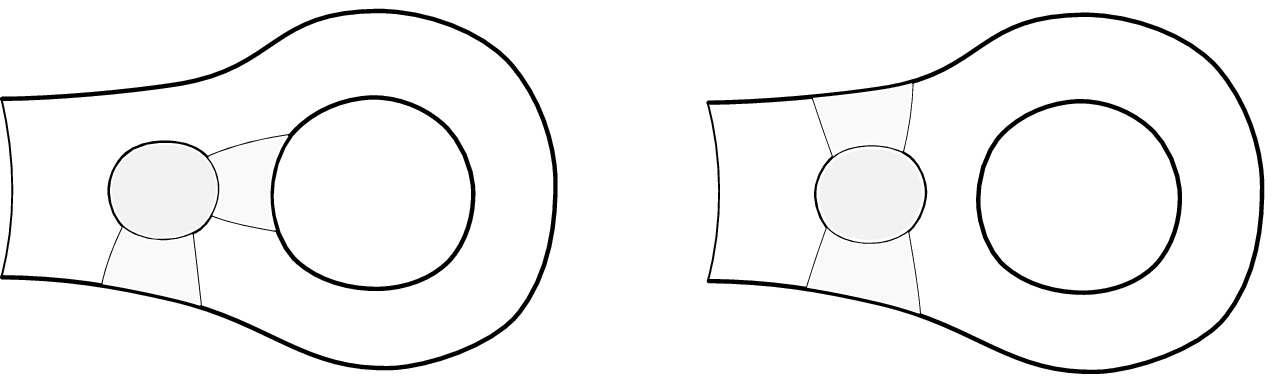
\caption{}  \label{fig:annular-monogon}
\end{figure}
In the first case,
the complement $\De - (!D \cup \Pi_1 \cup \Pi_2)$ is a diagram of trigon type 
and the statement follows from the already proved part (ii).
In the second case, $\De - (!D \cup \Pi_1 \cup \Pi_2)$ 
consists of a simply connected component ~$\De_1$ and 
and an annular component $\De_2$ with one non-cyclic side.
For cells of rank ~$\al$ in $\De_1$ the statement follows by (i)
and for cells of rank $\al$ in $\De_2$ we can apply induction since $\De_2$
has a strictly smaller number of cells of rank $\al$ than $\De$.

If all cells of rank $\al$ of $\De$ are exceptional 
then application of Proposition ~\ref{pr:principal-bound-cells} gives $M = 0$.
%
\end{proof}

\begin{proposition}[small diagrams of trigon or tetragon type] 
\label{pr:small-trigons-tetragons}
Let $\De$ be a small diagram of rank $\al$ of trigon or tetragon type 
with sides $!S_i$ ($1 \le i \le k$, $k = 3$ or $k = 4$). Then
$$
  \sum_{i=1}^3 |!S_i|_\al \le 4 \ze \eta
  \quad \text{or} \quad
  \sum_{i=1}^4 |!S_i|_\al \le 6 \ze \eta
$$
in the trigon and tetragon cases, respectively.
\end{proposition}

\begin{proof}
By Proposition \ref{pri:small-rank0-diagram} we may assume that $\al\ge1$.

We assume that $\De$ is reduced and is given a tight set $\cT$ of contiguity subdiagrams.
Following arguments from the proof of Proposition \ref{pr:principal-bound-sides}
we can assume that Claims \ref{cl:principal-bound-no-I-I}--\ref{cl:principal-bound-no-I-II}
from that proof hold in our case.
By Claim \ref{cl:diagram-regularity} and 
Proposition \ref{pr:single-layer-3gon}, $\De$ has no cells of rank ~$\al$.
By Claims \ref{cl:principal-bound-no-I-I} and \ref{cl:principal-bound-no-I-II},
$\cT$ has only contiguity subdiagrams between sides of $\De_{\al-1}$ of type II.
Hence any side of $\De$ occurs entirely in a boundary loop of a connected component $\Th$ of 
$\De_{\al-1} - \bigcup_{\Pi\in\cT} \Pi$.
By Lemma \ref{lm:filling-geometry}, $\sum_{\Th} c(\Th) = c(\De_{\al-1})$.
Applying Proposition \ref{pr:principal-bound}$_{\al-1}$ to components $\Th$
of $\De_{\al-1} - \bigcup_{\Pi\in\cT} \Pi$ we obtain
$$
  \sum_i |!S_i|_{\al-1} \le \eta c(\De_{\al-1}) \le (b(\De_{\al-1}) -2) \eta
$$
which gives the required inequality by \ref{cl:om-rank-increment}.
\end{proof}

\begin{proposition}[cell in a diagram of small complexity] \label{pr:small-complexity-cell}
Let $\De$ be a reduced diagram of rank $\al\ge1$ of one of the types listed
in Proposition \ref{pr:single-layer}.
Let $\cT$ be a tight set of contiguity subdiagrams on $\De$
and let $!D$ be a cell of rank $\al$ of $\De$.
Let $!P_i$, $i=1,2,\dots,r$ be the contiguity arcs of contiguity subdiagrams of $!D$ to sides of $\De$
that occur in $\de !D$.
Then:
\begin{enumerate} 
\item \label{pri:bigon-cell} 
If $\De$ has bigon type or is an annular diagram with two cyclic sides then $r=2$ and 
$$
  \mu(!P_1) + \mu(!P_2) \ge 1 - 2\la - 16\ze\eta\om.
$$
\item \label{pri:trigon-cell}
If $\De$ has trigon type then 
$2 \le k \le 3$ and 
$$
  \sum_{i=1}^k \mu(!P_i) \ge 1 - 3\la - 24\ze\eta\om.
$$
\item \label{pri:cyclic-monogon-cell}
If $\De$ is an annular diagram with one cyclic side and one non-cyclic side then
$2 \le k \le 3$ and 
$$
  \sum_{i=1}^k \mu(!P_i) \ge 1 - 4\la - 24\ze\eta\om.
$$
\end{enumerate}
\end{proposition}

\begin{proof}
Assume that $!C$ is another cell of rank $\al$ of $\De$. By Proposition \ref{pr:single-layer},
$!C$ has at least two contiguity subdiagrams $\Pi_1$, $\Pi_2$ to sides of $\De$.
Let $\De'$ be the connected component of $\De - !C - \Pi_1 - \Pi_2 $ containing $!D$.
Then $\De'$ inherits from $\De$ the boundary marking of rank $\al$ and the tight set of contiguity subdiagrams.
Observe also that $\De'$ is also a diagram of rank $\al$ of one of the types in cases (i)--(iii);
moreover, it is of the same type (i)---(iii) or has a smaller complexity. In this case the statement is reduced by
induction to the case of a diagram with a smaller number of cells of rank $\al$.

It remains to consider the case when $!D$ is a single cell of rank $\al$ of $\De$.
The equality $r=2$ in (i) and the bound $2 \le r \le 3$ in (ii) and (iii) follow from Lemma \ref{lm:cell-regularity}.
With bounds from Lemmas \ref{lm:small-cancellation-diagram}, \ref{lm:no-folded-cells},
Propositions \ref{pr:principal-bound}, \ref{pr:small-trigons-tetragons} for $\al := \al-1$
and inequality \eqref{eq:S2-piece-form},
an easy analysis shows that the worst cases for the lower bound on $\sum_i \mu(!P_i)$ are
as shown in Figure \ref{fig:small-complexity-cell}.
\begin{figure}[h]
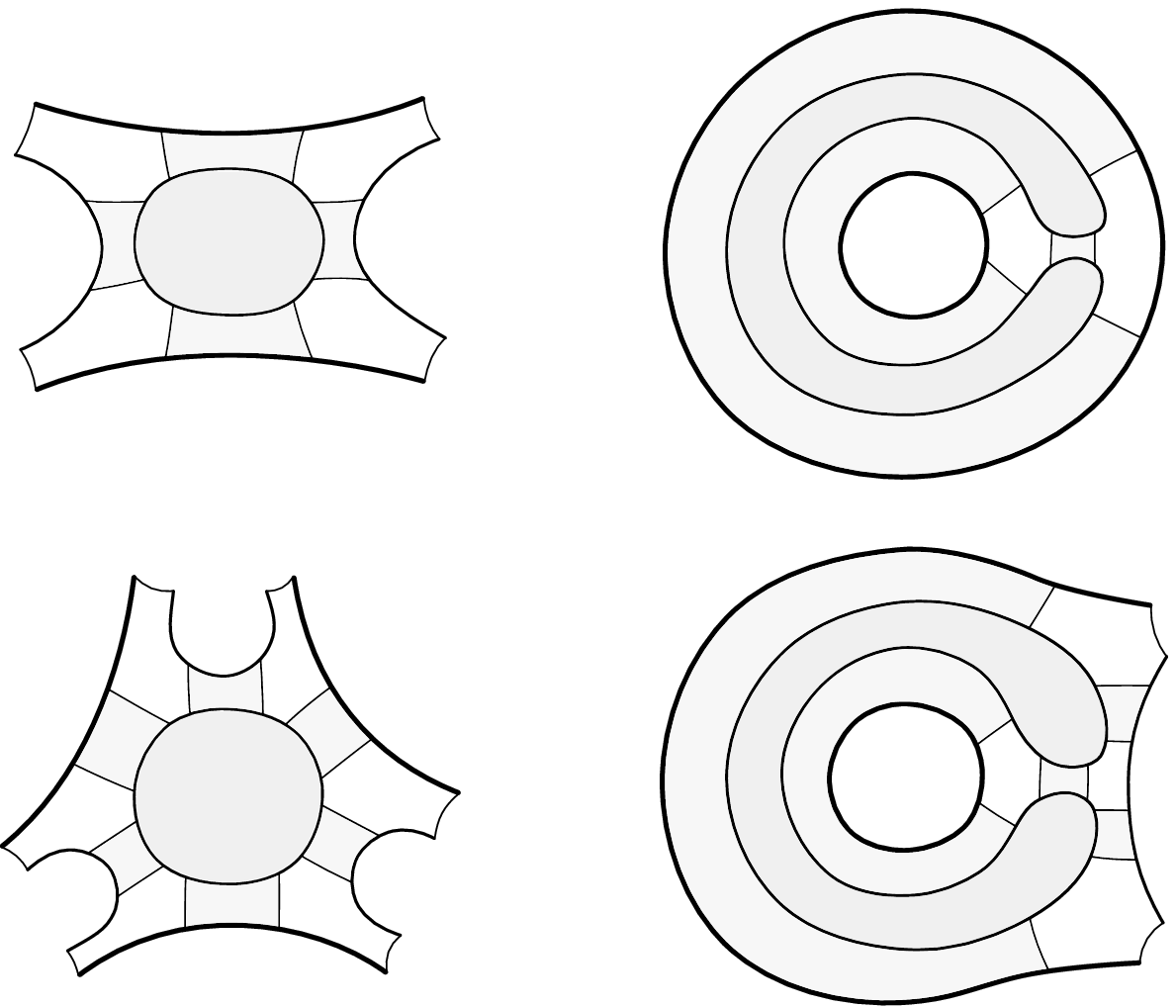
\caption{}  \label{fig:small-complexity-cell}
\end{figure}
We then get the corresponding inequality in (i)--(iii).
\end{proof}

\section{Fragments} \label{s:fragments}

In this section we establish several properties of fragments of rank $\al\ge1$. 
Most of them are proved using facts about relations in $G_{\al-1}$. 
Starting from this point we use extensively statements from subsequent 
Sections \ref{s:relations}--\ref{s:overlapping-periodicity}
for values of rank $\be<\al$.
We also switch our main action scene to Cayley graphs
$\Ga_{\al-1}$ and ~$\Ga_\al$.

All statements in this section are formulated and proved under assumption $\al\ge1$.

The following observation is a consequence of the assumption that 
the graded presentation of $G_\al$ is normalized, condition (S3) and the 
fact that centralizers of non-torsion elements of ~$G_{\al-1}$ are cyclic 
(Proposition \ref{pr:nontorsion-conjugation}$_{\al-1}$).
Recall that two periodic lines $!L_1$ and $!L_2$ in $\Ga_{\al-1}$ are called parallel if 
$s_{P_1,!L_1} = s_{P_2,!L_2}$
where $P_i$ is the period of ~$!L_i$ (see \ref{ss:Cayley}).

\begin{lemma} \label{lm:compatible-lines}
If $!L_1$ and $!L_2$ are two parallel periodic lines in $\Ga_{\al-1}$ whose periods
are relators of rank $\al$ then $!L_1 = !L_2$.
\end{lemma}

\begin{proof}
Let $!L_i$ $(i=1,2)$ be two parallel periodic lines in $\Ga_{\al-1}$ whose periods $R_i$ are
relators of rank $\al$. Up to cyclic shift of $R_i$ we can assume that $R_i \in \cX_\al^{\pm1}$
where $\cX_\al$ is the set of defining relators of rank $\al$ in the presentation \eqref{eq:G-presn}.
Let $!v_i$ be a vertex on $!L_i$ such that the label of $!L_i$
starts at ~$!v_i$ with ~$R_i$. Let $g = !v_1^{-1} !v_2 \in G_\al$ 
(recall that we identify vertices of $\Ga_\al$ with elements of $G_\al$).
Since $!L_1$ and $!L_2$ are parallel we have  $g R_2 g^{-1} = R_1$.
By (S3) we have either $R_1,R_2 \in \cX_\al$ or $R_1^{-1},R_2^{-1} \in \cX_\al$,
so according to Definition \ref{df:normalized-presentation}, we get $R_1 \greq R_2$ and 
$R_1 \greq R_0^t$ where $R_0$ it the root of $R_1$. 
Since the centralizer of $R_1$ is cyclic, we have $g = R_0^k$ for some integer $k$.
This implies $!L_1 = !L_2$.
\end{proof}

\begin{corollary}[Small cancellation in the Cayley graph] \label{co:small-overlapping-Cayley}
Let $!L_1$ and $!L_2$ be periodic lines in ~$\Ga_{\al-1}$ with periods $R_1$ and $R_2$, respectively, where both $R_i$ are relators of rank $\al$. 
Assume that $!L_1$ and ~$!L_2$ have close subpaths $!S_1$ and $!S_2$ 
such that $|!S_1|_{\al-1} \ge \la |R_1|_{\al-1}$. Then $!L_1 = !L_2$.
\end{corollary}

\begin{proof}
If $|!S_i| \le |R_i|$ for $i=1,2$ then the statement follows directly from 
condition (S2-Cayley) in \ref{ss:S2-S3-Cayley}. Let $|!S_1| > |R_1|$ or $|!S_2| > |R_2|$.
Using Proposition \ref{pr:fellow-traveling}$_{\al-1}$ and condition (S1) we find close
subpaths $!S_1'$ and $!S_2'$ of $!S_1$ and $!S_2$ with $|!S_i| \le |R_i|$, $i=1,2$
and $|!S_j|_{\al-1} \ge \la |R_j|_{\al-1}$ for $j=1$ or $j=2$.
This reduces the statement to the previous case.
\end{proof}

\begin{proposition} \label{pr:relator-strongly-reduced}
A relator of rank $\al$ is strongly cyclically reduced in $G_{\al-1}$.
\end{proposition}

\begin{proof}
Let $R$ be a relator of rank $\al$. Assume that some power $R^t$ is not reduced in $G_{\al-1}$.
According to definition \ref{ss:reduced-word}, for some $1 \le \be \le \al-1$ 
there exists a subword $S$ of $R^t$ which is close in $G_{\be-1}$ to a piece $P$ of rank $\be$
with $\mu(P) > \rho$. Since $R$ is cyclically reduced in $G_{\al-1}$ we have $|S| > |R|$.
Then according to the definition in \ref{ss:fragment} we have $|R^\circ|_\be \le 1$ and 
hence
$$
  |R^\circ|_{\al-1} \le \ze^{\al-\be-1} |R^\circ|_\be \le 1
$$
contradicting (S1) and \eqref{eq:ISC-main}.
\end{proof}

\subsection{} \label{ss:fragment-paths}
A {\em fragment path of rank ~$\al$} in ~$\Ga_{\al-1}$ is a path $!F$
labeled by a fragment of rank $\al$. We assume that $!F$ has an
associated $R$-periodic segment $!P$ with $R \in \cX_\al$ which is 
close to $!F$. We call $!P$ the {\em base} for $!F$.

Note that this agrees with the definition in \ref{ss:fragment}.
If $F$ is a fragment of rank $\al$ with associated triple $(P,u,v)$
and $!F$ is a path in ~$\Ga_{\al-1}$ with $\lab(!F) \greq F$ then
the loop $!F^{-1} !u !P !v$ with $\lab(!u!P!v) \greq uPv$ gives a base $!P$
for $!F$. Conversely, if $!F$ is a fragment of rank $\al$ in $\Ga_{\al-1}$
with base $!P$ then choosing a loop $!F^{-1} !u !P !v$ with 
$\lab(!u), \lab(!v) \in \cH_{\al-1}$ and denoting $F$, $P$, $u$ and ~$v$
the corresponding labels we obtain a fragment $F$ of rank $\al$ with associated
triple $(P,u,v)$.

If $\be \ge \al$ and paths $!F$ and $!P$ in $\Ga_\be$ are
obtained by mapping a fragment $\bar{!F}$ of rank $\al$ with base $\bar{!P}$ 
in $\Ga_{\al-1}$ then, by definition, we consider $!F$ as a fragment
of rank $\al$ with base $!P$ in $\Ga_\be$.

Abusing the language we will use the term `fragment' 
for both fragment words and fragment paths in $\Ga_\be$.

Recall that by a convention in \ref{ss:Cayley}, 
a base $!P$ for a fragment $!F$ of rank $\al$ in $\Ga_\be$ has an associated relator $R$ of rank $\al$ and 
the unique infinite $R$-periodic extension $!L$. 
If $\be=\al-1$ then $!L$ is a bi-infinite path
(which is simple by Proposition \ref{pr:relator-strongly-reduced}) that we call 
the {\em base axis} for ~$!F$.
If $\be > \al$ then $!L$ is winding over 
a relator loop labeled $R$ that we call the {\em base relator loop} for ~$!F$.


\subsection{} \label{ss:fragment-measuring}
We describe a way to measure fragments
of rank $\al$. 
If $P$ is a subword of a word $R^k$ where $R$ is a
relator of rank ~$\al$ then we define
\begin{equation} \label{eq:mu-def}
  \mu(P) = \frac{|P|_{\al-1}}{|R^\circ|_{\al-1}}.
\end{equation}
Note that this agrees with the definition in \ref{ss:mu-def} of 
the function $\mu(S)$ on the set of pieces $S$ of rank $\al$.
If $F$ is a fragment of rank $\al\ge 1$ then the size $\muf(F)$ of $F$
is defined to be equal to $\mu(P)$ where $P$ is the associated 
subword of $R^k$ and $R$ is the associated relator of rank $\al$.
Thus, for example, $\muf(F) = \frac12$
means approximately that $F$ is close in rank $\al-1$ to a ``half'' of its associated
relator of rank $\al$.

If $!F$ is a fragment of rank $\al$ in $\Ga_\be$ then we set
$\muf(!F) = \muf(\lab(!F))$. 
This means
that $\muf(!F)$ is given by the formula 
$$
  \muf(!F) = \frac{|!P|_{\al-1}}{|R^\circ|_{\al-1}}.
$$
where $!P$ is the base for $!F$ and $R$ is the relator
associated with $!P$.

Using Proposition \ref{pr:fellow-traveling}$_{<\al}$ 
we can easily reformulate the definition of a reduced in $G_\al$ word in \ref{ss:reduced-word} 
in the following way: 
a word ~$X$ is reduced in $G_\al$ if and only if $X$ is freely reduced and contains no
fragments $F$ of rank $1\le \be \le \al$ with $\muf(F) > \rho$.

\begin{definition}
 \label{df:fragment-compatibility}
Two fragments $!F$ and $!G$ of rank $\al$ in $\Ga_{\al-1}$ are {\em compatible} if their base axes are parallel.
Note that by Lemma \ref{lm:compatible-lines}, the base axes of fragments of rank $\al$
are parallel if and only if they coincide. 

In the case $\be \ge \al$, two fragments $!F$ and $!G$ of rank $\al$ in $\Ga_\be$ are defined to be 
compatible if they have compatible lifts in $\Ga_{\al-1}$, or, equivalently,
$!F$ and $!G$ have the same base relator loop.

It will be convenient to extend compatibility relation to fragments of rank $0$.
Recall that according to the definition in \ref{ss:fragment} 
fragments of rank 0 are letters in $\cA^{\pm1}$. Thus, fragments of rank $0$ in $\Ga_\be$
are paths of length 1. By definition, fragments $!F$ and $!G$ of rank $0$ in $\Ga_\be$ are
compatible if and only if $!F = !G$.
\end{definition}

We write compatibility of fragments as $!F \sim !G$. 
Note that we have in fact a family of relations with two parameters $\al \ge 0$ 
and $\be \ge \max(0,\al-1) $: compatibility of fragments of rank $\al$ in $\Ga_\be$.
The values of $\be$ and $\al$ will be always clear from the context.
Below we will use also ``compatibility up to invertion'' relation
on the set of fragments of rank $\al$ in $\Ga_\be$,
denoted $!F \sim !G^{\pm1}$ and meaning that $!F \sim !G$ or $!F \sim !G^{-1}$.
Both are obviously equivalence relations.

\begin{proposition}[fragment stability in bigon of the previous rank] 
\label{pr:fragment-stability-previous}
Let $\al\ge1$.
Let $!X$ and ~$!Y$ be reduced close paths in $\Ga_{\al-1}$.
Let $!K$ be a fragment of rank ~$\al$ in $!X$ with
$
  \muf(!K) \ge 2.3 \om.
$
Then there exists a fragment $!M$ of rank $\al$ in $!Y$ such that $!M \sim !K$ and
$$
  \muf(!M) > \muf(!K) - 2.6\om.
$$
\end{proposition}

\begin{proof}
Let $!P$ be the base for $!K$.
By \eqref{eq:S2-piece-form} and Proposition \ref{pr:closeness-stability}$_{\al-1}$ we have 
$!P = !z_1 !P' !z_2$ where $!P'$ is close to a subpath $!M$ of $!Y$ and $|!z_i|_{\al-1} < 1.3$
$(i=1,2)$. Then $!M$ is a fragment of rank ~$\al$ with base $!P'$, so $\muf(!M) = \mu(!P')$.
By \eqref{eq:S2-piece-form}
$$
   \mu(!z_1) + \mu(!z_2) < 2.6 \om
$$
and hence
$$
  \mu(!P') > \mu(!P) -  2.6\om = \muf(!K) - 2.6\om.
$$
\end{proof}

\begin{proposition}[fragment stability in trigon of the previous rank] 
\label{pr:fragment-trigon-previous}
Let $!X^{-1} * !Y_1 * !Y_2 *$ be a coarse trigon in $\Ga_{\al-1}$.
Let $!K$ be a fragment of rank $\al$ in $!X$ such that
$
  \muf(!K) \ge 2.5 \om.
$
Then at least one of the following statements holds:
\begin{itemize}
\item 
For $i=1$ or $i=2$ there is a fragment $!M_i$ of rank $\al$ in $!Y_i$ such that $!M_i \sim !K$ and
$$
  \muf(!M_i) > \muf(!K) - 2.8 \om.
$$
\item 
For each $i=1,2$ there is a fragments $!M_i$ of rank $\al$ in $!Y_i$ such that $!M_i \sim !K$ and 
$$
  \muf(!M_1) + \muf(!M_2) > \muf(!K) - 3\om.
$$
\end{itemize}
\end{proposition}

\begin{proof}
This follows from Proposition \ref{pr:stability-trigon}$_{\al-1}$
in a similar way as in the proof of 
Proposition ~\ref{pr:fragment-stability-previous}.
\end{proof}


\begin{proposition}[fragment stability in conjugacy relations of the previous rank] 
\label{pr:fragment-cyclic-monogon-previous}
Let $X$ be a word cyclically reduced in $G_{\al-1}$.
%
Let $Y$ be a word reduced in $G_{\al-1}$, $u \in \cH_{\al-1}$ and 
$Yu = z^{-1} X z$ in $G_{\al-1}$ for some $z$. We represent the conjugacy relation 
by two lines $\dots !Y_{-1} !u_{-1} !Y_{0} !u_{0} !Y_{1} !u_{1} \dots$
and $\bar{!X} = \dots !X_{-1} !X_{0} !X_{1} \dots$ 
in $\Ga_{\al-1}$ where $\lab(!X_i) \greq X$, $\lab(!Y_i) \greq Y$ and 
$\lab(!u_i) \greq u$ (see \ref{ss:relations}).
Let $!K$ be a fragment of rank $\al$ in $\bar{!X}$ with $|!K| \le |X|$ and
$\muf(!K) \ge 2.5\om$.
Then at least one of the following statements is true:
\begin{itemize}
\item 
For some $i$, there is a fragment $!M$ of rank $\al$ in $!Y_{i}$ such that $!M \sim !K$ and 
$$
  \muf(!M) > \muf(!K) - 2.9\om.
$$
\item 
For some $i$, there are fragments $!M_1$ and $!M_2$ of rank $\al$ in $!Y_{i}$ and $!Y_{i+1}$ respectively such that 
$!M_i \sim !K$ $(i=1,2)$ and 
$$
  \muf(!M_1) + \muf(!M_2) > \muf(!K) - 3\om.
$$
\end{itemize}
\end{proposition}

\begin{proof}
Follows from Proposition \ref{pr:stability-cyclic-monogon}$_{\al-1}$.
\end{proof}

\begin{proposition}[inclusion implies compatibility] 
\label{pr:inclusion-compatibility}
Let $!K$ and $!M$ be fragments of rank $\al$ in ~$\Ga_\be$, $\be \ge \al-1$. Assume that $!K$ is contained in $!M$ and $\muf(!K) \ge \la + 2.6\om$. 
Then $!K \sim !M$.
\end{proposition}

\begin{proof}
First consider the case $\be = \al-1$. Let $!P$ and $!Q$ be bases for $!K$ and $!M$, respectively.
By Proposition \ref{pr:closeness-stability}$_{\al-1}$, there are close subpaths $!P'$ of $!P$ 
and $!Q'$ of $!Q$ such that $\mu(!P') \ge \la$. 
Then by Corollary \ref{co:small-overlapping-Cayley} $!P$ and $!Q$
have the same infinite periodic extension and
we conclude that $!K$ and $!M$ are compatible.

If $\be \ge \al$ then we consider lifts $\ti{!K}$ and $\ti{!M}$ of $!K$ and $!M$ in $\Ga_{\al-1}$ 
such that $\ti{!K}$ is contained in $\ti{!M}$ and apply the already proved part.
\end{proof}

\begin{proposition}[dividing a fragment] \label{pr:dividing-fragment}
Let $!K$ be a fragment of rank $\al$ in  ~$\Ga_\be$, $\be \ge \al-1$. 
If $!K = !K_1 !K_2$ then either $!K_1$ or $!K_2$ contains a fragment $!F$ of rank $\al$
with $!F \sim !K$ and $\muf(!F) > \muf(!K) - \ze\om$, 
or $!K$ can be represented as $!K = !F_1 !u !F_2$
where $!F_i$ are fragments of rank $\al$, $!F_1$ is a start of $!K_1$, $!F_2$ is an end of $!K_2$, 
$!F_1 \sim !F_2 \sim !K$ and 
$$
  \muf(!F_1) + \muf(!F_2) > \muf(!K) - \ze\om .
$$
\end{proposition}

\begin{proof}
If $\al=1$ then $!u$ can be taken empty and the statement is trivial.
If $\be = \al-1 \ge 1$ then the statement follows from Proposition \ref{pr:fellow-traveling}$_{\al-1}$.
The case $\be > \al-1$ follows from the case $\be = \al-1$.
\end{proof}

As an immediate consequence of Propositions \ref{pr:inclusion-compatibility} and \ref{pr:dividing-fragment} we get:

\begin{proposition}[overlapping fragments] \label{pr:small-overlapping}
Let $!X$ be a reduced path in $\Ga_\be$, $\be \ge \al-1$. 
Let $!K$ and $!M$ be non-compatible fragments of rank $\al$ in $!X$.
Assume that $!K \le !M$ and $\muf(!K), \muf(!M) \ge \la + 2.7\om$.
Then there are a start $!K_1$ of $!K$ disjoint from $!M$ and an end $!M_1$ of $!M$ disjoint from ~$!K$ 
such that $!K_1$ and $!M_1$ are fragments of rank $\al$, $!K_1 \sim !K$, $!M_1 \sim !M$, 
$\muf(!K) - \muf(!K_1) < \la + 2.7\om$ and $\muf(!M) - \muf(!M_1) < \la + 2.7\om$.
\end{proposition}

\begin{proposition}[union of fragments] \label{pr:fragments-union}
Let $!X$ be a reduced path in $\Ga_{\al-1}$ and 
let $!K_i$ $(i=1,2)$ be compatible fragments of rank $\al$ in $!X$.
Assume that $\muf(!K_i) \ge 5.7\om$ for $i=1$ or $i=2$.
Then the union of $!K_1$ and ~$!K_2$ is a fragment of rank $\al$ with the same base axis.
Moreover, if $!K_1$ and $!K_2$ are disjoint then $\muf(!K_1 \cup !K_2) \ge \muf(!K_1) + \muf(!K_2) - 5.7\om$. 
\end{proposition}

\begin{proof}
By Lemma \ref{lm:compatible-lines}, $!K_1$ and $!K_2$ have a common base axis. 
If some of the $!K_i$'s is contained in the other  
then there is nothing to prove. Otherwise
the statement easily follows from Proposition \ref{pr:closeness-order}$_{\al-1}$.
\end{proof}

\begin{corollary}[compatibility preserves order] \label{co:compatibility-order}
Let $!X$ be a reduced path in $\Ga_{\al-1}$, 
let $!K_i, !M_i$ ($i=1,2$) be fragments of rank $\al$ in $!X$ and let
$\muf(!K_i), \muf(!M_i) \ge \la + 2.6\om$.
Assume that $!K_1 \sim !K_2$, $!M_1 \sim !M_2$ and $!K_1 \not\sim !M_1$.
Then $!K_1 < !M_1$ if and only if $!K_2 < !M_2$.
\end{corollary}

\begin{proof}
By Proposition \ref{pr:inclusion-compatibility}, for each $i=1,2$ 
neither of $!K_i$ or $!M_i$ can be contained in the other, so we have either $!K_i < !M_i$ or $!M_i < !K_i$.
It is enough to prove the statement in the case $!K_1 = !K_2$.
Assume, for example, that $!M_1 < !K_1 < !M_2$. 
Then by Proposition \ref{pr:fragments-union} $!M_1 \cup !M_2$ is a fragment of rank $\al$ with 
$!M_1 \cup !M_2 \not\sim !K_1$ and we get a contradiction with Proposition \ref{pr:inclusion-compatibility}.
\end{proof}

\begin{proposition}[no inverse compatibility] \label{pr:no-inverse-compatibility}
Let $!K$ and $!M$ be fragments of rank $\al$ in a reduced path $!X$ in $\Ga_{\al-1}$.
Let $\muf(!K), \muf(!M) \ge 5.7\om$. Then $!K \not\sim !M^{-1}$.
\end{proposition}

\begin{proof}
Follows from Lemma \ref{lm:compatible-lines} and Proposition \ref{pr:closeness-order}$_{\al-1}$.
\end{proof}

\begin{proposition} 
Let $!K$ be a fragment of rank $\be$ in $\Ga_\al$ where $1 \le \be \le \al$.
\begin{enumerate}
\item \label{pri:fragment-finite-order}
Let $!R$ be the base loop for $!K$ labeled by a relator ~$R$ of rank $\be$
and let $R_0$ be the root of $R$.
Then the subgroup $\setof{g \in G_\al}{g !K \sim !K}$ is finite
cyclic and conjugate to $\sgp{R_0}$.
\item \label{pri:fragment-in-periodic}
Let $X$ be a word representing an element of $G_\al$ which is not conjugate to 
a power of $R_0$.
Let $\bar{!X}$ be an $X$-periodic line in $\Ga_\al$ labeled $X^\infty$.
Then $s_{X,\bar{X}} !K \not \sim !K$.
\item \label{pri:fragment-in-periodic-small}
Under hypothesis of (ii), if $!K$ is a subpath of $\bar{!X}$ and $\muf(!K) \ge 2\la +5.3\om$
then $|!K| < 2 |X|$.
\end{enumerate}
\end{proposition}

\begin{proof}
(i) It follows from Lemma \ref{lm:compatible-lines}$_\be$ that 
$g !K \sim !K$ if and only if $g !R = !R$.
Since $\lab(!R) \greq R_0^t$ and $R_0$ is a non-power,
the stabilizer of $!K$ in $G_\al$ is a subgroup conjugate to $\sgp{R_0}$.

(ii) follows immediately from (i).

(iii) If $!K$ is a subpath of $\bar{!X}$, $\muf(!K) \ge 2\la +5.3\om$ and $|!K| \ge 2 |X|$ 
then using Propositions \ref{pr:dividing-fragment}$_\be$ 
and \ref{pr:inclusion-compatibility}$_\be$ we conclude
that either $s^{-1}_{X,\bar{!X}} !K \sim !K$ or $s_{X,\bar{!X}} !K \sim !K$, a contradiction with (ii).
\end{proof}

\section{Consequences of diagram analysis} \label{s:relations}


Following the terminology introduced in \ref{ss:polygon-relations}, a {\em coarse $r$-gon} in $\Ga_\al$
is a loop of the form
$$
  !P = !X_1 !u_1 !X_2 !u_2 ,\dots,!X_r !u_r
$$ 
where paths $!X_i$ are reduced and $!u_i$ are bridges of rank $\al$.

Let us assume that each bridge $!u_i$ of $!P$
is given an associate bridge partition of rank $\al$ 
(see \ref{ss:bridge-partition}) and 
consider a filling $\phi: \De^{(1)} \to \Ga_\al$ of $!P$ by a
disk diagram ~$\De$ over the presentation of ~$G_\al$, i.e.\ $\De$ has boundary loop 
$\ti{!X}_1 \ti{!u}_1 \ti{!X}_2 \ti{!u}_2 ,\dots, \ti{!X}_r \ti{!u}_r$ where $\phi(\ti{!X}_i) \greq !X_i$ and $\phi(\ti{!u}_i) \greq !u_i$.
We can assume that $\De$ has a boundary marking of rank $\al$ with
sides $\ti{!X}_i$ and bridges $\ti{!u}_i$ (see \ref{df:boundary-marking})
and that each $\ti{!u}_i$ has an induced bridge partition of rank $\al$.
Applying to ~$\De$ the reduction process described in \ref{ss:diagram-reduction}
we get a reduced diagram.
Note that during the process, bridges ~$\ti{!u}_i$ of $\De$ can be changed by switching. 
To keep the equality $\phi(\ti{!u}_i) \greq !u_i$ we have to perform
appropriated switching of bridges $!u_i$ (see \ref{ss:bridge-partition}).
As a consequence we obtain:

\begin{proposition}[filling coarse polygons by diagrams] \label{pr:filling-polygons}
Let $\al\ge 1$ and $!P = !X_1 !u_1 !X_2 !u_2 ,\dots,!X_r !u_r$ be a coarse $r$-gon in $\Ga_\al$
with fixed bridge partitions of all bridges $!u_i$.
Then, after possible switching of bridges ~$!u_i$, there exists a 
reduced disk diagram $\De$ of rank $\al$ which fills $!P$.
\end{proposition}

\begin{definition}
The {\em $\al$-area} of $!P$, denoted $\Area_\al(!P)$, 
is the number of cells of rank $\al$ of a filling diagram $\De$ as in Proposition \ref{pr:filling-polygons}.
To avoid correctness issues, we assume formally that $\Area_\al(!P)$ is defined with respect 
to a particular choice of $\De$.

The image $\phi(\de !D)$ in $\Ga_\al$ of the boundary loop of a cell of rank $\al$ of $\De$ is
an {\em active relator loop} for $!P$ for a particular choice $\De$.
Thus $\Area_\al(!P)$ is the number of active relator loops for ~$!P$.
Abusing the language,  we call the inverse 
loop $\phi(\de !D)^{-1}$ an active relator loop for $!P$ as well.
\end{definition}

\begin{remark} \label{rm:no-active-lift}
Equality $\Area_\al(!X_1 !u_1 !X_2 !u_2 ,\dots,!X_r !u_r)=0$ is equivalent 
to the assertion that $!X_1 !u_1 !X_2 !u_2 ,\dots,!X_r !u_r$ lifts to $\Ga_{\al-1}$
after possible switching of bridges ~$!u_i$.
\end{remark}

\subsection{} 
\label{ss:active-fragments}
As a special case of a coarse polygon, consider a coarse bigon $!X^{-1} !u !Y !v$ in $\Ga_\al$, $\al\ge 1$.
Up to switching of bridges ~$!u$ and ~$!v$ 
we can assume that there is a reduced diagram ~$\De$ of rank ~$\al$
which fills $!X^{-1} !u !Y !v$ via a map $\phi: \De^{(1)} \to \Ga_\al$.
We can assume also that $\De$ is given a tight set ~$\cT$ of contiguity subdiagrams. 
The boundary loop of $\De$ has the form 
$\ti{!X}^{-1} \ti{!u} \ti{!Y} \ti{!v}$ with sides ~$\ti{!X}^{-1}$ and ~$\ti{!Y}$
which are mapped onto $!X^{-1}$ and $!Y$ respectively.
%
By Proposition \ref{pr:single-layer}(i) each cell of rank ~$\al$ of $\De$
has a contiguity subdiagram to each of the sides $\ti{!X}^{-1}$ and $\ti{!Y}$.
The boundary loops of cells of rank $\al$ and the bridges of these contiguity subdiagrams
form a graph mapped in ~$\Ga_\al$ as in Figure ~\ref{fig:bigon-Cayley}.
\begin{figure}[h]
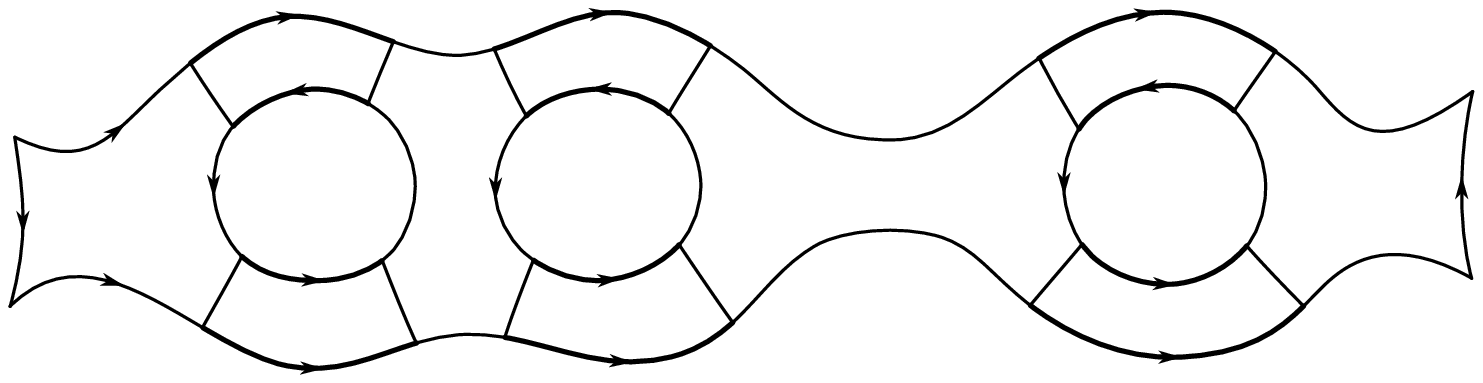
\caption{}
\label{fig:bigon-Cayley}
\end{figure}
Let $!R_i$ be images in $\Ga_\al$ of boundary loops of cells of rank $\al$ of $\De$
and let $!K_i$, $!M_i$, $!Q_i$ and $!S_i$ be subpaths of $!X$, $!Y$ and $!R_i$,
respectively, that are images of the corresponding contiguity arcs of 
contiguity subdiagrams of cells 
of rank $\al$ to $\ti{!X}^{-1}$ and $\ti{!Y}$, as shown in the figure.
According to the definition in ~\ref{ss:fragment-paths},
$!K_i$ and $!M_i$ are fragments of rank $\al$ with bases $!Q_i^{-1}$ and $!S_i$
and base relator loops $!R_i^{-1}$ and $!R_i$ respectively.
We call $!K_i$ and $!M_i$ {\em active fragments} of rank $\al$ 
of the coarse bigon $!X^{-1} !u !Y !v$.

Thus, if $\Area_\al(!X^{-1} !u !Y !v)=t$ then
there are precisely $t$ disjoint active fragments of rank $\al$ in each of the paths $!X$ and $!Y$.
Note again that the set of active relator loops and the set of 
active fragments formally depend on the choice of particular ~$\De$ and $\cT$.


\subsection{} \label{ss:active-loop-induction}
Let, as above, $!P = !X^{-1} !u !Y !v$ be a coarse bigon in $\Ga_\al$
and $\De$ a reduced diagram of rank ~$\al$ with $\de \De = \ti{!X}^{-1} \ti{!u} \ti{!Y} \ti{!v}$
filling $!P$ via a map $\phi: \De^{(1)} \to \Ga_\al$ 
(we assume that the switching operation is already applied to $!u$ and $!v$ if needed).
We assume that $\De$ has a tight set $\cT$ of contiguity subdiagrams.
Let $!R = \phi(\de !D)$ be an active relator loop of $!P$ and let 
$!Q^{-1} !w_1 !K^{-1} !w_2$ and $!S^{-1} !w_3 !M !w_4$ be images of boundary loop of 
contiguity subdiagrams in $\cT$ of the cell $!D$ to sides ~$\ti{!X}^{-1}$ and ~$\ti{!Y}$
respectively as in Figure ~\ref{fig:active-loop-induction-bigon}.
Then two loops $!P_1$ and $!P_2$ as shown in the figure
can be considered as coarse bigons in $\Ga_\al$ with sides that are subpaths of $!X$ and ~$!Y$.
\begin{figure}[h]
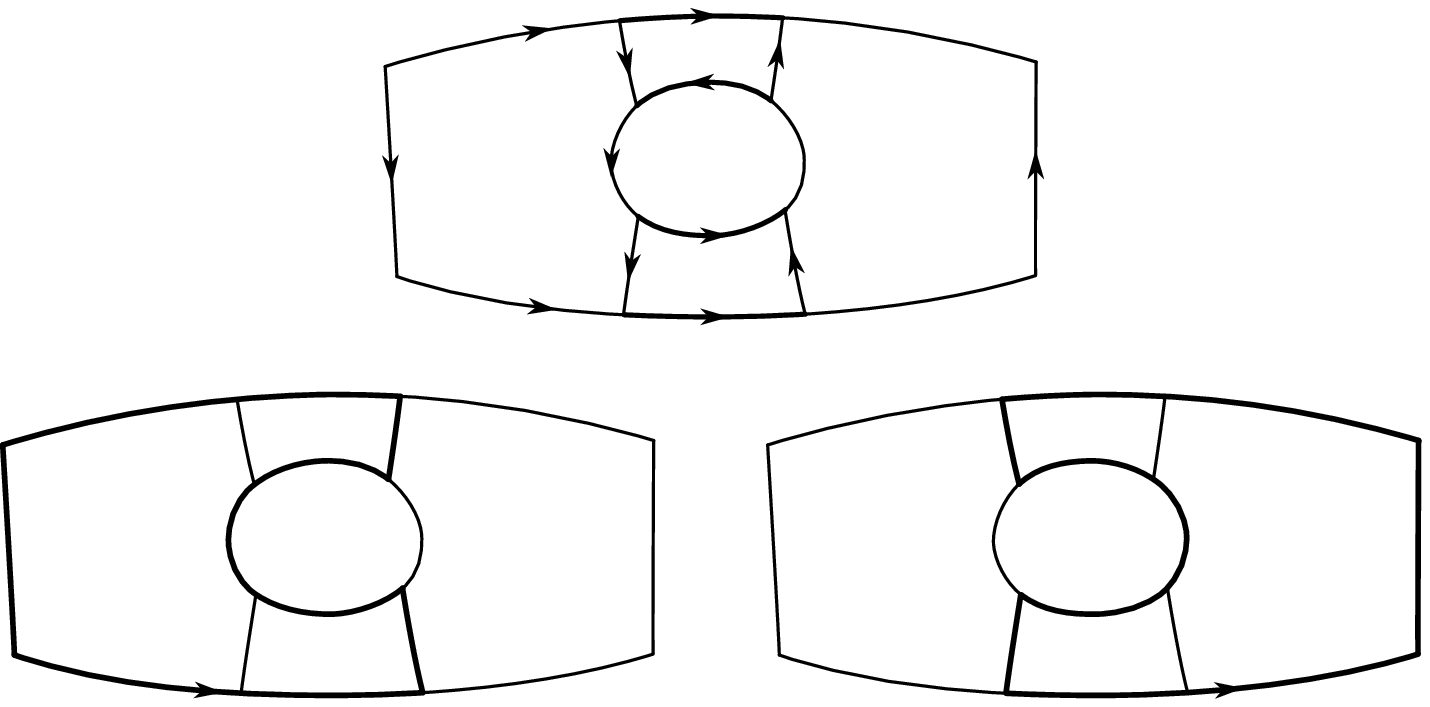
\caption{}  \label{fig:active-loop-induction-bigon}
\end{figure}
They are filled by reduced subdiagrams of ~$\De$, so we
have $\Area_\al(!P_1) + \Area_\al(!P_2) = \Area_\al(!P) - 1$. 
We will use this simple observation in inductive arguments.

\subsection{} \label{ss:active-loops-trigon}
In a similar way, let $!P = !X_1 !u_1 !X_2 !u_2 !X_3 !u_3$ be a coarse trigon in $\Ga_\al$.
After possible switching of bridges $!u_i$, 
we can find a reduced diagram $\De$ of rank ~$\al$
with boundary loop 
$\ti{!X}_1 \ti{!u}_1 \ti{!X}_2 \ti{!u}_2 \ti{!X}_3 \ti{!u}_3$ which fills $!P$ 
via a map $\phi: \De^{(1)} \to \Ga_\al$ of $!P$ 
where 
$\phi(\ti{!X}_i) = !X_i$ and $\phi(\ti{!u}_i) = !u_i$.
We can also assume that $\De$ has a tight set $\cT$ of contiguity subdiagrams.
By Proposition \ref{pr:single-layer-3gon} each cell of rank $\al$ of $\De$ has contiguity 
subdiagrams in $\cT$ to at least two sides ~$\ti{!X}_i$.
This implies that for any active relator loop $!R$ of $!P$ there are two or three 
fragments $!K_i$ ($i = 1,2$
or $i=1,2,3$) of rank $\al$ with base loop $!R$ that occur in distinct paths ~$!X_j$. 
Similarly to the bigon case, we call them {\em active fragments} of rank $\al$
of $!P$.

As in the bigon case, 
for any active relator loop $!R$ of $!P$ 
we can consider a coarse bigon $!P_1$ and a coarse trigon
~$!P_2$ respectively, as shown in Figure \ref{fig:active-loop-induction-trigon}, 
with $\Area_\al(!P_1) + \Area_\al(!P_2) = \Area_\al(!P) - 1$. 
\begin{figure}[h]
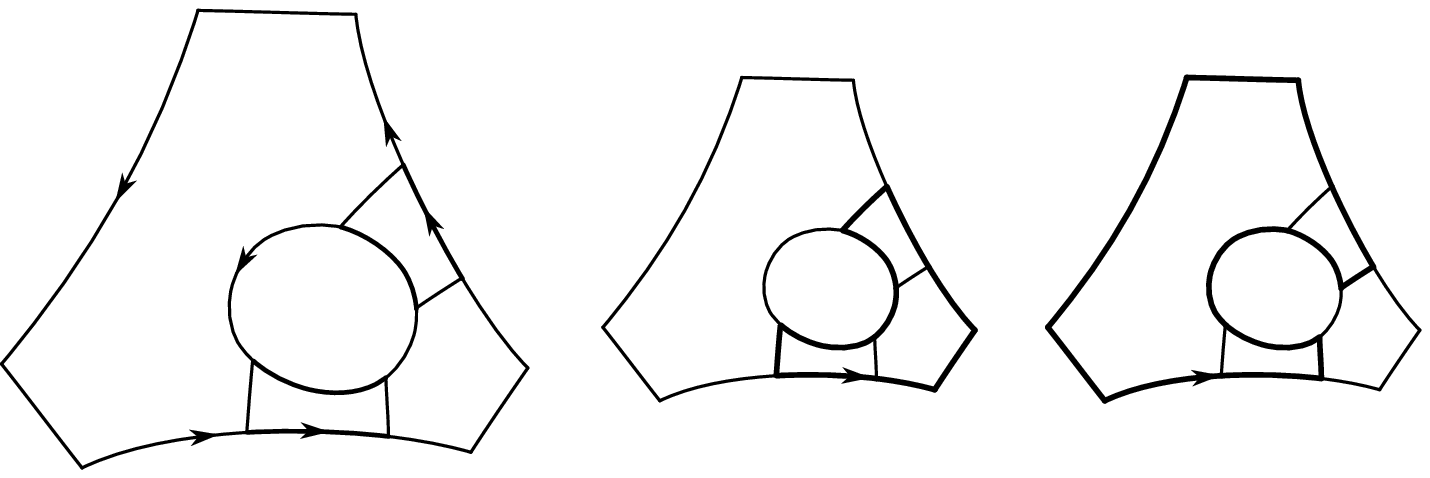
\caption{}  \label{fig:active-loop-induction-trigon}
\end{figure}

\begin{proposition}[active fragments in bigon] \label{pr:active-large} 
Let $!P = !X^{-1} !u !Y !v$ be a coarse bigon in ~$\Ga_\al$, $\al\ge 1$.
\begin{enumerate}
\item \label{pri:active-bound}
Let $!K$ and $!M$ be active fragments of rank $\al$ of $!P$ in $!X$ and $!Y$, respectively,
with mutually inverse base active relator loops.
Then $!K \sim !M^{-1}$,
$$
   \muf(!K) + \muf(!M) > 1 - 2\la - 1.5\om 
$$
and 
$$
  \muf(!K), \muf(!M) > 7\la - 1.5\om.
$$
\item \label{pri:active-non-compatible}
Let $!K$ and $!K'$ be two distinct active fragments of rank $\al$ in $!X$. Then $!K \not\sim !K'$.
\end{enumerate}


\end{proposition}

\begin{proof}
(i):
It follows directly from the construction that $!K \sim !M^{-1}$.
The first inequality follows from Proposition \ref{pri:bigon-cell}.
Since $!X$ and $!Y$ are reduced we have $\muf(!K) \le \rho$ and $\muf(!M) \le \rho$
which implies the lower bound on $\muf(!K)$ and $\muf(!M)$.

(ii):
Assume that $!K \sim !K'$.
Let $!M$ and $!M'$ be the corresponding active fragments of rank $\al$ in $!Y$. By (i), we have $!M \sim !M'$.
Then by Proposition \ref{pr:fragments-union} and the first inequality of (i),
$$
  \muf(!K \cup !K') + \muf(!M \cup !M') \ge 2 - 4\la - 17.4\om > 2 \rho
$$
which contradicts the hypothesis that $!X$ and $!Y$ are reduced.
\end{proof}

We introduce the notation for the lower bound on the size of active fragments in (i):
$$
  \xi_0 = 7\la - 1.5\om.
$$

\begin{definition} \label{df:close-in-rank}
 We say that paths $!X$ and $!Y$ in $\Ga_\al$ are {\em close in rank $\be \le \al$}
if there exist bridges $!u$ and $!v$ of rank $\be$ such that $!X^{-1} !u !Y !v$ is a loop that 
can be lifted to $\Ga_\be$.
(So `being close' for paths in $\Ga_\al$ means the same as 
`being close in rank $\al$'.)
\end{definition}

\begin{remark} \label{rm:close-in-rank-0}
If $!X$ and $!Y$ are labeled with freely reduced words then $!X$ and $!Y$ are close in rank $0$
if and only if $!X=!Y$.
\end{remark}

\begin{proposition}[lifting bigon]
\label{pr:lifting-bigon}
Let $0 \le \be < \al$ and $!X^{-1} !u !Y !v$ be a coarse bigon in ~$\Ga_\al$ where $!u$ and $!v$
are bridges of rank $\be$. Assume that for all $\ga$ 
in the interval $\be+1 \le \ga \le \al$
either $!X$ or $!Y$ has no fragments $!K$ of rank ~$\ga$
with $\muf(!K) \ge \xi_0$. 
Then $!X^{-1} !u !Y !v$ can be lifted to $\Ga_\be$ and, consequently, $!X$ and $!Y$ are close in 
rank $\be$.
\end{proposition}

\begin{proof}
This is a consequence of Proposition \ref{pr:active-large} and Remark \ref{rm:no-active-lift}.
\end{proof}

\begin{proposition}[no active relators] \label{pr:no-active-loops}
Let $\al\ge 1$,  $!X^{-1} !u !Y !v$ be a coarse bigon in ~$\Ga_\al$ and
$\Area_\al(!X^{-1} !u !Y !v) =0$.
Assume that $|!X|_\al > 2+6\ze^2\eta$.
Then $!X$ and $!Y$ can be represented as 
$!X = !w_1 !X_1 !w_2$ and $!Y =!z_1 !Y_1 !z_2$ where 
$!X_1$ and $!Y_1$ are close in rank $\al-1$ and $|!w_i|_\al, |!z_i|_\al \le 1+4\ze^2\eta$ $(i=1,2)$.
\end{proposition}

\begin{proof}
By Remark \ref{rm:no-active-lift} we can assume that $!X^{-1} !u !Y !v$ lifts
to $\Ga_{\al-1}$. 
To simplify notations, we assume that $!X^{-1} !u !Y !v$ is already in $\Ga_{\al-1}$. 
Let $!u = !u_1 !P !u_2$ and
$!v = !v_1 !Q !v_2$ where $!u_i$, $!v_i$ are bridges of rank $\al-1$ and $!P$, $!Q$ are paths labeled by pieces of rank ~$\al$. 
We apply Proposition \ref{pri:4gon-closeness}$_{\al-1}$ to the coarse 
tetragon $!X^{-1} !u_1 !P !u_2 !Y !v_1 !Q !v_2$. 
Observe that if a subpath of $!P$ or $!Q$ is close (in $\Ga_{\al-1}$) to a subpath $!S$ of $!X$ then
$|!S|_\al \le 1$.
Since  $|!X|_\al > 2+6\ze^2\eta$ we cannot get the first case of 
the conclusion of Proposition \ref{pri:4gon-closeness}$_{\al-1}$. Therefore, the second case holds:
we have $!X = !X_1 !z_1 !X_2 !z_2 !X_3$ where $!X_1$ is close to a start of $!P$, 
$!X_2$ is close to a subpath of $!Y$,
$!X_3$ is close to an end of $!Q$ and $|!z_i|_{\al-1} \le 4 \ze\eta$ $(i=1,2)$.
Then $|!X_1 !z_1|_\al \le 1 + 4\ze^2 \eta$, $|!z_2 !X_3|_\al \le 1 + 4\ze^2 \eta$ and we get the required bound.
\end{proof}

\begin{corollary}[no active fragments] \label{co:no-active-fragments}
Let $!X$ and $!Y$ be close reduced paths in $\Ga_\al$, $\al\ge 1$.
Assume that either $!X$ or $!Y$ has no fragments $!K$ of rank $\al$
with $\muf(!K) \ge \xi_0$. Assume also that $|!X|_\al > 2+6\ze^2\eta$.
Then $!X$ and $!Y$ can be represented as 
$!X = !w_1 !X_1 !w_2$ and $!Y =!z_1 !Y_1 !z_2$ where 
$!X_1$ and ~$!Y_1$ are close in rank $\al-1$ and 
$|!w_i|_\al, |!z_i|_\al  \le 1+4\ze^2\eta$ $(i=1,2)$.
\end{corollary}

\begin{corollary}[no active fragments, iterated] 
\label{co:no-active-fragments-iterated}
Let $!X$ and $!Y$ be close reduced paths in ~$\Ga_\al$. 
Let $0 \le \be < \al$ and assume that for all $\ga$ 
in the interval $\be+1 \le \ga \le \al$
either $!X$ or $!Y$ has no fragments $!K$ of rank ~$\ga$
with $\muf(!K) \ge \xi_0$. 
Let $|!X|_\al \ge 2 + 3\ze$.
Then $!X$ and $!Y$ can be represented as 
$!X = !w_1 !X_1 !w_2$ and $!Y =!z_1 !Y_1 !z_2$ where 
$!X_1$ and $!Y_1$ are close in rank $\be$ and 
$
  |!w_i|_\al < 1 + 5\ze^2 \eta
$
$(i=1,2)$.
\end{corollary}

\begin{proposition} \label{pr:nonidentity-active-piece}
Let $X$ be a nonempty freely reduced word equal 1 in $G_\al$. 
Then $X$ has a subword $P$ which is a piece of rank $\be$ where $1 \le \be \le \al$ and $\mu(P) > 136\om$.
\end{proposition}

\begin{proof}
By Proposition \ref{pr:reduced-nontrivial}, $X$ is not reduced in $G_\al$ 
and therefore contains a fragment $K$ of rank ~$\be$ where $1 \le \be \le \al$ and $\muf(K) \ge \rho$.
Let $\be\ge1$ be the minimal rank such that $X$ contains a fragment $K$ 
of rank $\be$ with $\muf(K) \ge \xi_0$. 
If $\be = 1$ then $K$ is already a piece of rank 1 with $\mu(K) \ge \xi_0 > 138\om$ by \eqref{eq:om-bounds}.
Let $\be > 1$. 
Let $!K$ be a fragment in $\Ga_{\be-1}$ with $\lab(!K) \greq K$
and $!S$ a base for $!K$.
By Corollary \ref{co:no-active-fragments-iterated}$_{\be-1}$ we have $!S = !w_1 !P !w_2$ where
$|!w_i|_{\be-1} < 1.03$ $(i=1,2)$ and $P = \lab(!P)$ occurs in $K$. 
By \eqref{eq:om-bounds}, $\mu(P) \ge \xi_0 - 2.06\om = 7\la - 3.56\om > 136\om$. 
\end{proof}

\begin{proposition}[active fragments in trigon] 
\label{pr:trigon-active-fragments}
Let $!P = !X_1 !u_1 !X_2 !u_2 !X_3 !u_3$ be a coarse trigon in ~$\Ga_\al$,
let $!R$ be an active relator loop for $!P$ and let $!K_i$ ($i = 1,2$
or $i=1,2,3$) be active fragments of rank $\al$ with base loop $!R$. Then $!K_i \sim !K_j$ for all $i,j$,
$$
  \sum_i \muf(!K_i) > 1 - 3\la - 2.2\om
$$
and 
$$
  \muf(!K_i) > 3\la - 1.1\om \quad\text{for at least two indices } i.
$$
\end{proposition}

\begin{proof}
We have $!K_i \sim !K_j$ by construction.
The first inequality follows from Proposition \ref{pri:trigon-cell}.
Since $!X_i$ is reduced in $G_\al$ we have $\mu(!K_i) \le \rho = 1 - 9\la$.
This implies the second inequality.
\end{proof}

\begin{proposition}[no active fragments in conjugacy relations] 
\label{pr:no-active-fragments-cyclic}
Let $X$ and $Y$ be words cyclically reduced in $G_\al$, $\al\ge1$.
Let $X = Z^{-1} Y Z$ in $G_\al$ for some $Z$.
Assume that no cyclic shift of ~$X$ contains a fragment $K$ of rank $\al$ with $\muf(K) \ge \xi_0$.
Then there exists a word $Z_1$ such that $Z_1 = Z$ in $G_\al$ and $X = Z_1^{-1} Y Z_1$ in $G_{\al-1}$.
\end{proposition} 

\begin{proof}
Let $\De_0$ be a disk diagram of rank $\al$ with boundary label $X^{-1} Z^{-1} Y Z$.
We produce an annular diagram $\De_1$ by gluing two boundary segments of $\De_0$
labeled $Z^{-1}$ and $Z$. The diagram $\De_1$ can be assigned a boundary
marking of rank $\al$ with two cyclic sides $!X^{-1}$ and ~$!Y$. 
We denote $!Z$ the path in $\De$ with $\lab(!Z) \greq Z$ that joins starting vertices of $!Y$ and $!X$.
Let $\De_2$ be a reduced diagram of rank $\al$ 
obtained from $\De_1$ by reduction process.
According to the remark in \ref{ss:reduction-transformations-good}, $\De_1$ 
and $\De_2$ have the same frame
type. It follows from Lemma \ref{lm:disk-annular-frame-type} that there exists a 
path $!Z_1$ in $\De_2$ joining starting vertices of 
boundary loops $!Y_1$ and $!X_1^{-1}$ such that $\lab(!X_1) \greq X$, 
$\lab(!Y_1) \greq Y$ and 
$Z_1 = Z$ in $G_\al$ where $Z_1 \greq \lab(!Z_1)$.
By Proposition \ref{pri:bigon-cell}, $\De_2$ has no cells of rank $\al$. 
Then $X = Z_1^{-1} Y Z_1$ in $G_{\al-1}$.
\end{proof}

\begin{proposition}[no active fragments in conjugacy relations, iterated] 
\label{pr:no-active-fragments-cyclic-iterated}
Let $X$ and $Y$ be cyclically reduced in $G_\al$ words which represent conjugate elements of $G_\al$,
$\al\ge 1$.
Let $\be \le \al$. Assume that at least one of the words ~$X$ or ~$Y$ has the property that 
no its cyclic shift contains a fragment $K$ of rank $\ga$ with $\muf(K) \ge \xi_0$ 
and $\be < \ga\le \al$.
Let $\bar{!X} = \dots !X_{-1} !X_0 !X_1 \dots$ and $\bar{!Y} = \dots !Y_{-1} !Y_0 !Y_1 \dots$ 
be parallel periodic lines in $\Ga_\al$ with $\lab(!X_i) \greq X$ and 
$\lab(!Y_i) \greq Y$ representing the conjugacy relation. 
Then some vertices on $\bar{!X}$ and $\bar{!Y}$ are joined by a bridge of rank $\be$.

Moreover, for any subpath $!Z$ of $\bar{!X}$ there exists a loop $!S^{-1} !u !T !v$ which can lifted to $\Ga_\be$ such that 
$!S$ and $!T$ are subpaths of $\bar{!X}$ and $\bar{!Y}$ respectively, 
$!u$ and $!v$ are bridges of rank $\be$
and $!Z$ is contained in $!S$.
\end{proposition}

\begin{proof}
Since $\bar{!X}$ and $\bar{!Y}$ are parallel, 
if vertices $!a$ on $\bar{!X}$ and $!b$ on $\bar{!Y}$ are joined by a path labeled ~$Z$ then the same is true
for all their translates $s_{X,\bar{!X}}^k !a$ and $s_{Y,\bar{!Y}}^k !b$. Then the second statement follows from the first.

Let $\De$ be an annular diagram of rank $\al$ with boundary loops $\hat{!X}^{-1}$ and $\hat{!Y}$
and $\phi: \ti{\De}^{(1)} \to \Ga_\al$ a combinatorially continuous
map of the 1-skeleton of the universal cover $\ti\De$ of $\De$ to $\Ga_\al$
sending lifts $\ti{!X}$ of $\hat{!X}$ and $\ti{!Y}$ of $\hat{!Y}$ 
to $\bar{!X}$ and $\bar{!Y}$ respectively.
We can assume that $\De$ is reduced and has a tight set of contiguity subdiagrams.
If $\be = \al$ and $\De$ has a cell of rank $\al$ then the statement follows from 
Proposition \ref{pri:annular-single-layer}.
If $\De$ has no cells of rank $\al$ 
then we can lift $\bar{!X}$ and $\bar{!Y}$ to $\Ga_{\al-1}$ and use induction
on $\al$.
If $\be < \al$ and at least one of the words ~$X$ or ~$Y$ has 
no cyclic shift containing a fragment $K$ of rank $\al$ with $\muf(K) > \xi_0$
then by Proposition \ref{pri:bigon-cell}, $\De$ has no cells of rank $\al$
and, again, the statement follows by induction.
\end{proof}

\begin{proposition}[small coarse polygons] \label{pr:sides-bound-Cayley}
Let $!P = !X_1*!X_2*\dots!X_r*$ be a coarse $r$-gon in $\Ga_\al$ where $r \ge 3$
and $!X_i$ are sides of $!P$.
Assume that there are no pairs of close vertices lying on distinct paths $!X_i$ and $!X_j$
except pairs $\set{\tau(!X_i),\io(!X_{i+1})}$ and $\set{\tau(!X_r),\io(!X_1)}$.
Then 
$$
  \sum_i |!X_i|_\al \le (r-2)\eta .
$$
If $r = 3$ or $r = 4$ then we have a stronger bound
$$
  \sum_i |!X_i|_\al \le 2(r-1)\ze\eta.
$$
\end{proposition}

\begin{proof}
Consider a filling $\phi: \De^{(1)} \to \Ga_\al$ of $!P$
by a reduced disk diagram $\De$ of rank ~$\al$.
Let $\de\De = \bar{!X}_1 !u_1 \bar{!X}_2 !u_2 \dots \bar{!X}_r !u_r$ where $!u_i$ are bridges and $!X_i$ are sides 
of $\De$ with $\phi(\bar{!X}_i) = !X_i$.
The hypothesis of the proposition 
implies that $\De$ is small. Then the statement follows from 
Propositions \ref{pr:principal-bound} and \ref{pr:small-trigons-tetragons}.
\end{proof}

\begin{proposition}[trigons and tetragons are thin] \strut\par 
\label{pr:3-4-gon-closeness}
\begin{enumerate}
\item \label{pri:3gon-closeness}
Let $!X^{-1} * !Y_1 * !Y_2 *$ be a coarse trigon in $\Ga_\al$. Then $!X$ can be represented as 
$!X = !X_1 !z !X_2$ where $!X_1$ is close to a start of $!Y_1$, $!X_2$ is close to an end of $!Y_2$
and $|!z|_\al \le 4 \ze\eta$.
\item \label{pri:4gon-closeness}
Let $!X^{-1} * !Y_1 * !Y_2 * !Y_3 *$ be a coarse tetragon in $\Ga_\al$. Then at least
one of the following possibilities holds:
\begin{itemize}
\item 
$!X$ can be represented as 
$!X = !X_1 !z !X_2$ where $!X_1$ is close to a start of $!Y_1$, $!X_2$ is close to an end of $!Y_3$
and $|!z|_\al \le 6\ze\eta$; or 
\item
$!X$ can be represented as 
$!X = !X_1 !z_1 !X_2 !z_2 !X_3$ where $!X_1$ is close to a start of $!Y_1$, $!X_2$ is close to a subpath of $!Y_2$,
$!X_3$ is close to an end of $!Y_3$ and $|!z_i|_\al \le 4 \ze\eta$ $(i=1,2)$.
\end{itemize}
\end{enumerate}
\end{proposition}

\begin{proof}
(i) We can represent $!X_1 = !X_1 !z !X_2$, $!Y_i = !Y_{i1} !w_i !Y_{i2}$ $(i=1,2)$
with close pairs $(!X_1,!Y_{11})$, $(!Y_{12},!Y_{21}^{-1})$ and $(!Y_{22},!X_2)$
where no vertices lying on distinct paths $!z$, $!w_1$ and $!w_2$ are close except appropriate endpoints (Figure~\ref{fig:side-bound}a). 
Then the statement follows by application of 
Proposition \ref{pr:sides-bound-Cayley} to $!z^{-1} * !w_1 * !w_2 *$. 
\begin{figure}[h] 
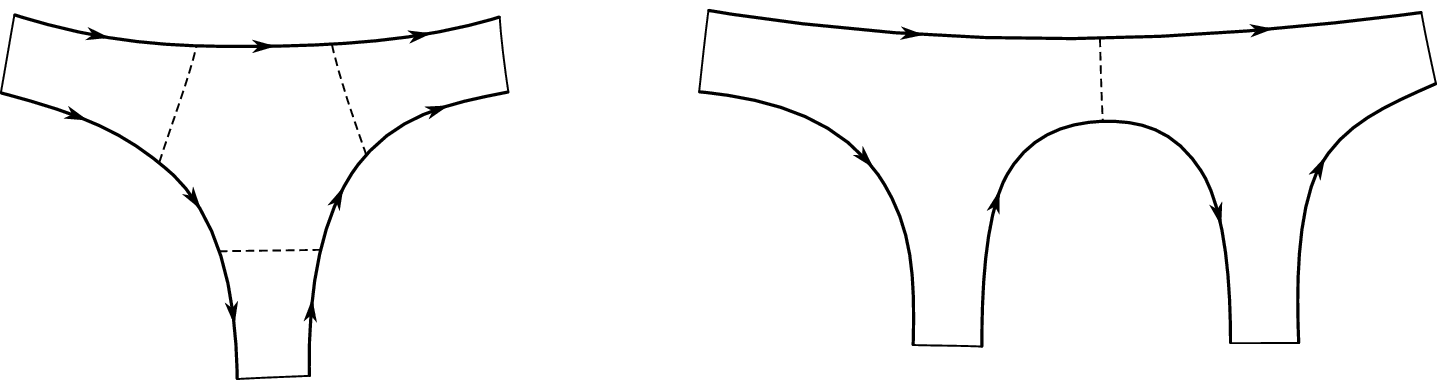
\caption{} \label{fig:side-bound}
\end{figure}

(ii) If there is a pair of close vertices on $!Y_1$ and $!Y_3$ then the statement follows from (i)
giving the first alternative. 
If there is a pair of close vertices on $!X$ and on $!Y_2$ then we 
represent $!X$ and ~$!Y_2$ as $!X = !X_1 !X_2$, $!Y_2= !Y_{21} !Y_{22}$ 
where $\tau(!X_1)$ and $\tau(!Y_{21})$
are close, and apply ~(i) to 
$!X_1^{-1} * !Y_1 * !Y_{21} *$ and $!X_2^{-1} * !Y_{22} * !Y_3 *$ (Figure~\ref{fig:side-bound}b).
We then come to the second alternative to the statement.
Otherwise we use an argument similar to the proof of (i) coming to the first alternative.
\end{proof}

\begin{proposition}[small cyclic monogon] \label{pr:small-cyclic-monogon}
Let $X$ be a word cyclically reduced in $G_\al$ and let $X$ be conjugate in $G_\al$ to a word $Yu$ 
where $Y$ is reduced in $G_\al$ and $u$ is a bridge of rank ~$\al$.
Let $\bar{!X} = \prod_{i \in \Z} !X_i$ and $\prod_{i \in \Z} !Y_i !u_i$ be lines in 
$\Ga_\al$ representing the conjugacy relation. 
Assume that no vertex on $!X_0$ is close to a vertex on $!Y_i$.
Then $|X|_\al \le \eta$.
\end{proposition}

\begin{proof}
Let $\De$ be an annular diagram of rank $\al$ with boundary loops $\hat{!X}$
and $\hat{!Y}^{-1} \hat{!u}^{-1}$ representing the conjugacy relation. We
consider $\De$ as having a cyclic side $\hat{!X}$, a non-cyclic side 
$\hat{!Y}^{-1}$ and a bridge $\hat{!u}^{-1}$. 
Up to switching of $\hat{!u}^{-1}$ we can assume that $\De$ is reduced.
The hypothesis implies that $\De$ cannot have a bond between $\hat{!X}$
and $\hat{!Y}^{-1}$ after any refinement. Assume that $\De$ has a bond $!v$
(possibly after refinement) joining 
two vertices on the same side $\hat{!Y}^{-1}$. Then $!v$ cuts off 
from $\De$ a simply connected subdiagram $\Si$ with boundary 
loop $!Z_1 \hat{!u}^{-1} !Z_2 !v^{\pm1}$ where $\hat{!Y}^{-1}= !Z_2 !W !Z_1$
for some $!W$. According to Definition \ref{df:bond}, at least one of the words
$\lab(!Z_i)$ $(i=1,2)$ is nonempty. Removing $\Si$ from $\De$ we obtain a diagram $\De'$
with a shorter total label of its two sides. Hence, by induction, we can assume
that $\De'$ is small. Then $|X|_\al = |\hat{!X}|_\al \le \eta$ by
Proposition \ref{pr:principal-bound}.
\end{proof}

\begin{proposition}[closeness fellow traveling] \label{pr:fellow-traveling}
Let $!X$ and $!Y$ be close reduced paths in ~$\Ga_\al$, $\al\ge1$. 
Then $!X$ and $!Y$ can be represented as
$!X = !U_1 !U_2 \dots !U_k$ and $!Y = !V_1 !V_2 \dots !V_k$ ($!U_i$ and $!V_i$ can be empty)
where the starting vertex of each $!U_i$ is close to the starting vertex of ~$!V_i$ and
$ |!U_i|_\al, |!V_i|_\al \le \ze$ for all $i$.
\end{proposition}

\begin{proof}
Observe that the statement of the lemma holds in the case $\al=0$ with
$|!U_i|_0, |!V_i|_0 = 1$. 
Thus we may refer to the statement of the lemma in rank $\al-1$ with bounds
$|!U_i|_{\al-1}, |!V_i|_{\al-1} \le 1$ which imply $|!U_i|_{\al}, |!V_i|_{\al} \le \ze$.
Observe also that if $!X = !X_1 !X_2 \dots !X_r$ and $!Y = !Y_1 !Y_2 \dots !Y_r$ 
where for each $i$, $!X_i$ and $!Y_i$ are close then the statement of the lemma for each pair $(!X_i,!Y_i)$
implies the statement of the lemma for $!X$ and $!Y$.
By \ref{ss:active-loop-induction} we represent $!X$ and $!Y$ as $!X = !X_1 !X_2 \dots !X_r$ 
and $!Y = !Y_1 !Y_2 \dots !Y_r$ where pairs $(!X_i,!Y_i)$ satisfy the following conditions (1) or ~(2)
in the alternate way:
(1) for some bridges $!u_i$ and $!v_i$ of rank ~$\al$ the loop $!X_i^{-1} !u_i !Y_i !v_i$ lifts to $\Ga_{\al-1}$
or (2) there are loops $!X_i^{-1} !w_{i1} !R_i !w_{i2}$ and 
$!Y_i !w_{i3} !S_i !w_{i4}$ which can be lifted to $\Ga_{\al-1}$
such that $!S_i$ and $!R_i$ occur in one relation loop of rank $\al$ and $!w_{ij}$ are bridges of rank $\al-1$
(see Figure ~\ref{fig:fellow-traveling}).
\begin{figure}[h]
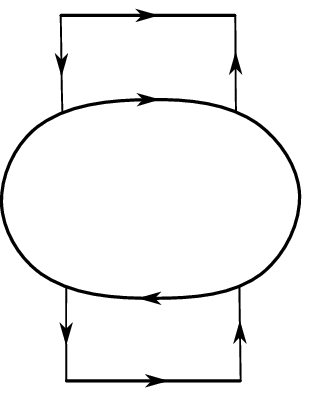
\caption{}  \label{fig:fellow-traveling}
\end{figure}
We can assume that pairs $(!X_1,!Y_1)$ and $(!X_r,!Y_r)$ satisfy (2) and that in the case of ~(2),
subpaths $!X_i$, $!Y_i$ 
of ~$!X$, $!Y$ and $!S_i$, $!R_i$ of the appropriate relation loop cannot be extended.
We prove the statement for each of the pair $(!X_i,!Y_i)$.

{\em Case of\/} (1):
Omitting the index $i$ for $!X_i$ and $!Y_i$, assume that a loop 
$!X^{-1} !w_1 !P !w_2 !Y !w_3 !Q !w_4$ lifts to $\Ga_{\al-1}$
where $!w_i$ are bridges of rank $\al-1$ and $!P$ and $!Q$ are labeled by pieces of rank ~$\al$. 
Without changing notations, we assume that $!X^{-1} !w_1 !P !w_2 !Y !w_3 !Q !w_4$ 
is already in $\Ga_{\al-1}$.
By the maximal choice of  $!X_i$, $!Y_i$, $!S_i$ and $!R_i$ in the case of (2),
there are no close vertices on pairs $(!X,!P)$, $(!X,!Q)$, $(!Y,!P)$ and $(!Y,!Q)$ except 
appropriate endpoints (i.e.\ except $\io(!X)$ and $\io(!P)$ for $(!X,!P)$ etc.).
Depending on existence of close vertices on pairs $(!P,!Q)$ and $(!X,!Y)$ we consider three cases 
(a)--(c) as in Figure ~\ref{fig:fellow-traveling-subcases}.
\begin{figure}[h]
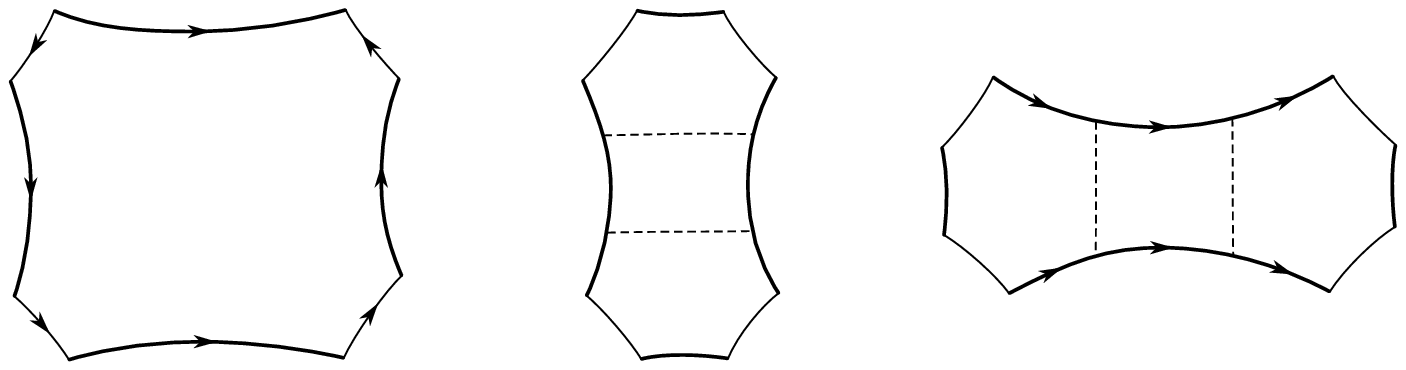
\caption{}  \label{fig:fellow-traveling-subcases}
\end{figure}
In case (a) we have $|!X|_\al, |!Y|_\al \le 6\ze^2\eta < \ze$ 
by Proposition \ref{pr:sides-bound-Cayley}$_{\al-1}$.
In case (b) taking the maximal pair of close subpaths of $!P$ and $!Q$ we 
get $|!X|_\al, |!Y|_\al \le 4\ze^2\eta < \ze$ again by Proposition \ref{pr:sides-bound-Cayley}$_{\al-1}$.
In case (c) we have $!X = !X_1 !X_2 !X_3$ and $!Y = !Y_1 !Y_2 !Y_3$ where $!X_2$ and $!Y_2$
are close. Taking $!X_2$ and $!Y_2$ maximal possible we get $|!X_i|_\al, |!Y_i|_\al \le 4\ze^2\eta $
for $i=1,3$ by Proposition \ref{pr:sides-bound-Cayley}$_{\al-1}$. For $!X_2$ and $!Y_2$
we can apply the statement for $\al := \al-1$.

{\em Case of\/} (2):
In the second case by the statement of the lemma for $\al := \al-1$ we have 
$!X = !U_1 !U_2 \dots !U_k$ and $!Y = !W_1 !W_2 \dots !W_l$ where 
$|!U_i|_\al, |!W_i|_\al \le \ze$, the starting vertex of each $!U_i$
can be joined by a bridge of rank $\al-1$ with a vertex on $!R$ and the starting vertex of each $!W_i$
can be joined by a bridge of rank $\al-1$ with a vertex on ~$!S$. Then each $\io(!U_i)$ is close to $\io(!Y)$
and each $\io(!W_i)$ is close to $\tau(!X)$. 
We take $!X = !U_1 !U_2 \dots !U_{k+l}$ and $!Y = !V_1 !V_2 \dots !V_{k+l}$ where $!U_{k+1}$, $\dots$,
$!U_{k+l}$, $!V_1$, $\dots$, $!V_k$ are empty and $!V_j = !W_{j-k}$ for $k +1 \le j \le k+l$.
\end{proof}

\begin{lemma} \label{lm:monogon-graph-version}
Let $!X$ be a reduced path and $!R$ a relation loop of rank $\al$ in $\Ga_\al$, $\al\ge 1$.
Let $!u_i$ $(i=1,2)$ be a path labeled by a word in $\cH_{\al-1}$ and joining vertices $!a_i$ on $!X$ and $!b_i$ on $!R$.
Let $!Y$ be the subpath of $!X^{\pm1}$ that starts at $!a_1$ and ends at $!a_2$, and let $!R = !R_1 !R_2$
where $!R_i$ starts at ~$!b_i$ (Figure ~\ref{fig:monogon-graph-version}).
Then one of the two loops $!Y !u_2 !R_1^{-1} !u_1^{-1}$ or $!Y !u_2 !R_2 !u_1^{-1}$ lifts to $\Ga_{\al-1}$.
\end{lemma}

\begin{figure}[h]
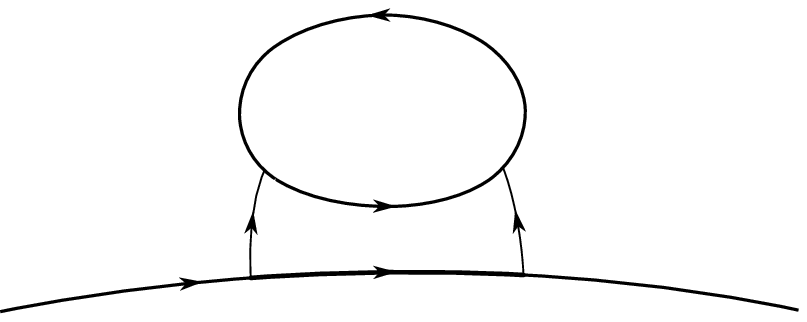
\caption{}  \label{fig:monogon-graph-version}
\end{figure}

\begin{proof}
We fill the loop $!Y !u_2 !R_1^{-1} !u_1^{-1}$ by 
a disk diagram $\De$ of rank $\al$ with boundary loop $\bar{!Y} \bar{!u}_2 !S \bar{!u}_1^{-1}$ where $\lab(!S) \greq \lab(!R_1^{-1})$.
We take $\bar{!Y}$ as a side and $\bar{!u}_2 !S \bar{!u}_1^{-1}$ as a bridge of $\De$ with bridge partition $\bar{!u}_2 \cdot !S \cdot \bar{!u}_1^{-1}$.
Then we apply the reduction process making $\De$ reduced. After reduction, we get either 
 $\lab(!S) \greq \lab(!R_1^{-1})$ or $\lab(!S) \greq \lab(!R_2)$.
By Lemma \ref{lm:monogon-regularity}, $\De$ has no cells of rank ~$\al$.
Depending on the case, this implies that either
$!Y !u_2 !R_1^{-1} !u_1^{-1}$ or $!Y !u_2 !R_2 !u_1^{-1}$ lifts to ~$\Ga_{\al-1}$.
\end{proof}

\begin{proposition}[compatibility lifting] \label{pr:compatibility-reduction} 
Let $1 \le \be \le \al$. Let $!K$ and $!M$ be fragments of rank ~$\be$ which occur in a reduced path $!X$ in $\Ga_{\al}$.
Let $\hat{!X}$ be a lift of $!X$ in $\Ga_{\be-1}$ and $\hat{!K}$ and $\hat{!M}$ 
be the subpaths of ~$\hat{!X}$ which are projected onto $!K$ and $!M$ respectively.
Then $!K \sim !M$ implies $\hat{!K} \sim \hat{!M}$
and $!K \sim !M^{-1}$ implies $\hat{!K} \sim \hat{!M}^{-1}$.
\end{proposition}

\begin{proof}
Assume that $!K \sim !M^\ep$ where $\ep =\pm1$.
Let $!R$ be the common base loop for $!K$ and $!M^\ep$. 
Lemma \ref{lm:monogon-graph-version} implies that $!R$ can be lifted to a line $\hat{!R}$
which is the common base axis for both $\hat{!K}$ and $\hat{!M}^\ep$.
This implies $\hat{!K} \sim \hat{!M}^\ep$.
\end{proof}

\begin{corollary} \label{co:compatibility-same-rank} 
Let $1 \le \be \le \al$.
Then statements of Proposition \ref{pr:fragments-union}, Corollary \ref{co:compatibility-order} and
Proposition \ref{pr:no-inverse-compatibility} hold for fragments of rank $\be$ in 
a reduced path $!X$ in $G_\al$.

More precisely, let $!X$ be a reduced path in $\Ga_{\al}$. Then the following is true. 
\begin{enumerate}
\item \label{coi:fragments-union-alpha}
Let $!K_i$ $(i=1,2)$ be fragments of rank $\be$ in $!X$, $!K_1 \sim !K_2$ and 
$\muf(!K_i) \ge 5.7\om$ for $i=1$ or $i=2$.
Then $!K_1 \cup !K_2$ is a fragment of rank $\be$ with $!K_1 \cup !K_2 \sim !K_1$.
If $!K_1$ and $!K_2$ are disjoint then $\muf(!K_1 \cup !K_2) \ge \muf(!K_1) + \muf(!K_2) - 5.7\om$.
\item \label{coi:compatibility-order-alpha}
Let $!K_i, !M_i$ ($i=1,2$) be fragments of rank $\be$ in $!X$ with
$\muf(!K_i), \muf(!M_i) \ge \ga+2.6\om$.
Assume that $!K_1 \sim !K_2$, $!M_1 \sim !M_2$ and $!K_1 \not\sim !M_1$.
Then $!K_1 < !M_1$ if and only if $!K_2 < !M_2$.
\item \label{coi:no-inverse-compatibility-alpha}
If $!K$ and $!M$ are fragments of rank $\be$ in $!X$ and 
$\muf(!K), \muf(!M) \ge 5.7\om$ then $!K \not\sim !M^{-1}$.
\end{enumerate}
\end{corollary}

\section{Stability}

Let $F_A$ be a free group with basis $A$ and let $X^{-1} Y_1 Y_2 \dots Y_{k+1} = 1$ be a relation in $F_A$
where $X$, $Y_1$, $\dots$, $Y_k$ are freely reduced words in the generators $A$.
Then for any occurrence of a letter $a^\ep \in A^{\pm1}$ in $X$ there is a unique occurrence of 
the same letter $a^\ep$ in some $Y_i$ which cancels with $a^{-\ep}$ in $X^{-1} Y_1 Y_2 \dots Y_{k+1}$.
The main goal of this section is to establish an analog of this statement for relations in $G_\al$.
The role of letters $a^\ep$ will be played by fragments of rank $\al$ and instead of relation
$X^{-1} Y_1 Y_2 \dots Y_{k+1} = 1$ we consider coarse polygons $!X^{-1} * !Y_1 * \dots !Y_k *$ in ~$\Ga_{\al}$ (for our considerations, it is enough to consider cases $k=1,2,3$).
The role of correspondence of canceled letters will be played by equivalence relation `$!K \sim !L^{\pm1}$'.


There are two essential differences of the case of groups $G_\al$ from the case of a free group $F_A$.
One is a ``fading effect'': a fragment in $!Y_i$ can be of a ``smaller size'' than an initial fragment in $!X$. 
Another difference is that bridges of the coarse polygon can produce exceptions for stability (to describe such situations we introduce a special relation
between fragments and bridges of the same rank $\be$, 
see Definition \ref{def:interference}).

We start with a statement which shows how closeness 
is propagated in coarse tetragons in $\Ga_{\al-1}$. This is essentially 
a consequence of inductive hypotheses.


\begin{definition}[uniformly close] 
For $\al\ge 1$, we say that vertices $!a_1$, $!a_2$, $\dots$, $!a_r$ of $\Ga_\al$ are {\em uniformly close} if at least
one of the following is true:
\begin{itemize}
\item 
they are pairwise close in rank $\al-1$; or 
\item
there exists a relator loop $!R$ of rank $\al$ such that each 
$!a_i$ is close in rank $\al-1$ to a vertex on $!R$. 
\end{itemize}
We cover also the case $\al=0$: vertices $!a_1$, $!a_2$, $\dots$, $!a_r$ of $\Ga_0$
are said to be uniformly close if $!a_1 = !a_2 = \dots = !a_r$.
\end{definition}

Note that uniformly close vertices are pairwise close. 
If $r=2$ then being uniformly close and being close is equivalent. 

\begin{lemma} \label{lm:stability-tetragon-previous}
Let $\al\ge1$, $!X$ and $!Y$ be close reduced paths in $\Ga_{\al-1}$, and let 
$!S^{-1} * !T_1 * !T_2 * !T_3 *$ be a coarse tetragon in $\Ga_{\al-1}$ such that $!Y$ is a subpath of ~$!S$.
Assume that $|!X|_{\al-1} \ge 5.2$.
Then $!X$ can be represented as 
$!z_0 !X_1 !z_1 \dots !X_r !z_r$ $(1 \le r \le 3)$ where $!X_i$ is close 
to a subpath $!W_i$ of some ~$!T_{j_i}$, $j_1 < \dots < j_r$ and 
\begin{equation} \label{eq:stability-tetragon-previous-bound}
  \sum_i |!X_i|_{\al-1} > |!X|_{\al-1} - 5.8.
\end{equation}
Moreover:
\begin{enumerate}
\item 
if $r=3$ then we have a stronger bound
$$
  \sum_i |!X_i|_{\al-1} > |!X|_{\al-1} - 3.4.
$$
\item
There is a subpath $!Y_1$ of $!Y$ such that the starting vertices $\io(!X_1)$, $\io(!Y_1)$ and $\io(!W_1)$
are uniformly close and the same is true for the ending vertices $\io(!X_r)$, $\io(!Y_1)$ and $\io(!W_r)$.
\end{enumerate}
\end{lemma}

\begin{proof}
If $\al = 1$ the statement is obvious 
(see Remark \ref{rm:stability-tetragon-previous-rank0} below).
Let $\al > 1$.
Let $!Z$ be a reduced path joining $\io(!S)$ and $\tau(!T_2)$ which exists by Proposition \ref{pr:reduction}$_{\al-1}$
(see Figure~\ref{fig:stability-tetragon-previous}).
\begin{figure}[h]
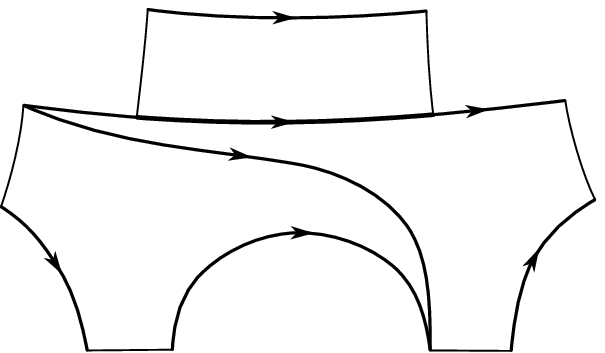
\caption{}  \label{fig:stability-tetragon-previous} 
\end{figure}
We apply Proposition \ref{pr:stability-trigon}$_{\al-1}$ 
first to the coarse trigon $!S^{-1} * !Z * !T_3*$ and then, possibly, to 
the coarse trigon $!Z^{-1} * !T_1 * !T_2$. 
Since $|!X|_{\al-1} \ge 5.2$, after the first application of Proposition \ref{pr:stability-trigon}$_{\al-1}$,
we find either a subpath ~$!X_3$ of $!X$ that is close to a subpath of $!T_3$ 
or a subpath ~$!X'$ of $!X$ that is close to a subpath of ~$!Z$ with 
$|!X'|_{\al-1} > |!X|_{\al-1} - 2.75 > 2.45$. 
In the latter case, the second application of \ref{pr:stability-trigon}$_{\al-1}$
gives the remaining ~$!X_1$ and/or $!X_2$. 
If $r < 3$ then for the bound \eqref{eq:stability-tetragon-previous-bound},
the worst cases are when we get two $!X_i$'s after double application 
of \ref{pr:stability-trigon}$_{\al-1}$. 
In those cases we have once case (iii) of \ref{pr:stability-trigon}$_{\al-1}$
and another time case (i) or (ii). Hence $\sum_i |!X_i|_{\al-1} >  |!X|_{\al-1} - 3 - 2.75$.
Statement ~(ii) follows from the appropriate part of Proposition \ref{pr:stability-trigon}$_{\al-1}$.

Assume that $r=3$ and therefore 
$!X = !z_0 !X_1 !z_1 !X_2 !z_2 !X_3 !z_3$ where each $!X_i$ is close 
to a subpath of $!T_i$. From application of Proposition 
\ref{pr:stability-trigon}$_{\al-1}$ we have
$|!z_0|_{\al-1}, |!z_3|_{\al-1} < 1.3$. 
Then using Proposition \ref{pri:3gon-closeness}$_{\al-1}$ we extend all $!X_i$
to get $|!z_1|_{\al-1}, |!z_2|_{\al-1} \le 4\ze\eta < 0.4$.
This proves ~(i).
\end{proof}

\begin{remark} \label{rm:stability-tetragon-previous-rank0}
If $\al=1$ then hypotheses of Lemma \ref{lm:stability-tetragon-previous} say that $!X = !Y$
and $!S^{-1} !T_1 !T_2 !T_3$ is a loop in the Cayley graph $\Ga_0$ of the free group $G_0$.
Then the statement of the lemma holds without the assumption $|!X|_{\al-1} \ge 5.2$.
Furthermore, in the conclusion we have $\sum_i |!X_i|_{\al-1} = |!X|_{\al-1}$.
\end{remark}

\begin{definition}[independence] \label{def:interference}
Let $1 \le \be \le \al$, $!K$ be a fragment of rank $\be$ in ~$\Ga_\al$ 
and $!u$ be a bridge of rank $\be$ in $\Ga_\al$.
Recall that $!K$ is considered with the associated base loop ~$!R$ of rank ~$\be$.
We say that $!K$ is {\em independent of} $!u$ if either $\lab(!u) \in \cH_{\be-1}$
or $!u$ possesses a bridge partition $!u = !v \cdot !S \cdot !w$ of rank $\be$
where $!S$ occurs in a relator loop $!L$ of rank $\be$ such that 
$!L \ne !R^{\pm1}$.
%
\end{definition}


It follows from the definition that if $!K$ is independent of $!u$ and 
$!M \sim !K^{\pm1}$ then $!M$ is also independent of $!u$.

\begin{proposition}[non-active fragment in bigon] 
\label{pr:fragment-nonactive-bigon}
Let $\al\ge 1$, $!X^{-1} !u !Y !v$ be a coarse bigon in ~$\Ga_\al$
and let $!X = !F_0 !K_1 !F_1 \dots !K_r !F_r$ where $!K_i$ are the associated active fragments of rank $\al$.
Let $!K$ be a fragment of rank $\al$ in $!X$ with $\muf(!K) \ge 2\la +5.8\om$.
Assume that $!K \not\sim !K_i$ for all $i$ and that $!K$ 
is independent of $!u$ and ~$!v$.
Then there exists a fragment of rank $\al$ in $!Y$ such that $!M \sim !K$ and
$$
  \muf(!M) \ge \muf(!K) - 2 \la - 3.4\om.
$$
\end{proposition}

\begin{proof}
By Proposition \ref{pr:inclusion-compatibility}
$!K$ is a subpath of one of the paths $!F_0 !K_1$, $!K_1 !F_1 !K_2$, $\dots$, $!K_r !F_r$.
We consider the case when $!K$ is a subpath of some $!K_i !F_i !K_{i+1}$ 
(the cases when $!K$ is a subpath of $!F_0 !K_1$ or $!K_r !F_r$ are similar; see also the remark in the end of the proof).
Let $!Y = !H_0 !M_0 !H_1 \dots !M_r !H_r$ where $!M_i$ are the corresponding active fragments of rank $\al$ in $!Y$.

As we can see from \ref{ss:active-fragments}, there is a loop
$!T = (!K_i !F_i !K_{i+1})^{-1} !w_1 !S_1 !w_2 !H_i  !w_3 !S_2 !w_4$ which can be lifted to $\Ga_{\al-1}$
and where $!w_j$ are bridges of rank $\al-1$ and $!S_1$ and $!S_2$
occur in base loops for $!K_i$ and $!K_{i+1}$ respectively
(see Figure \ref{fig:fragment-stability-bigon}).
\begin{figure}[h]
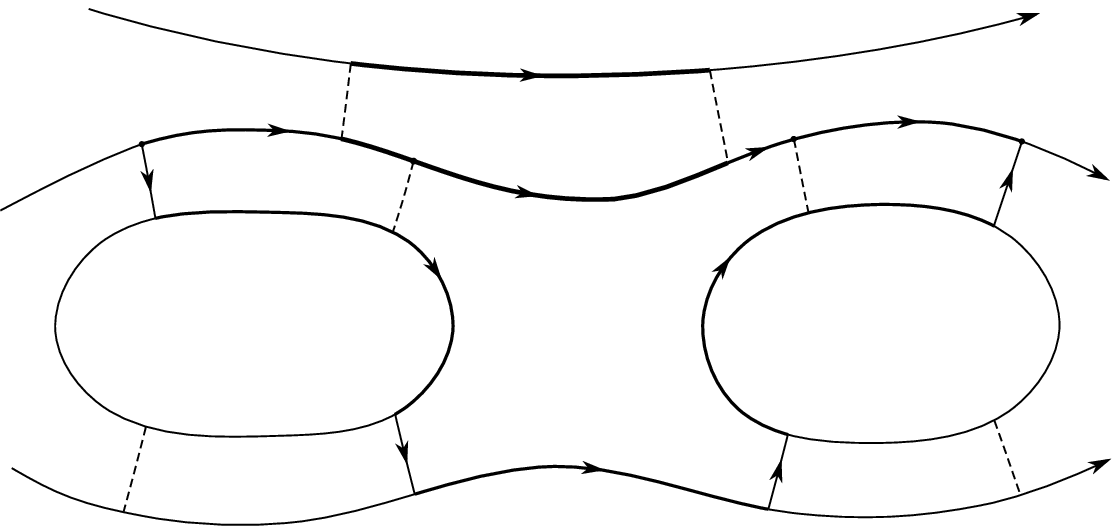
\caption{}  \label{fig:fragment-stability-bigon}
\end{figure}
Abusing notation we assume that $!T$ is already in ~$\Ga_{\al-1}$.
Then, instead of base loops,  $!S_1$ and $!S_2$ occur in base axes $!L_1$ and $!L_2$ for $!K_i$ and 
$!K_{i+1}$ respectively.

Let $!L$ be the base axis for $!K$ and $!S$ the base for $!K$ (which is contained in $!L$ by definition).
Assumptions $!K \not\sim !K_i$ and $!K \not\sim !K_{i+1}$ imply $!L \ne !L_i$ $(i = 1,2)$.
By Corollary \ref{co:small-overlapping-Cayley}, if a subpath ~$!P$ of $!S$ is close to a subpath of $!S_i$
then $\mu(!P) < \la$.
Then by Lemma \ref{lm:stability-tetragon-previous} we find a subpath $!Q$ of ~$!S$ which is close 
to a subpath $!M$ of $!H_i$ and satisfies 
$$
  \mu(!Q) > \mu(!S) - 2\la - 3.4\om.
$$
Then $!M$ is a fragment of rank $\al$ with base $!Q$. Clearly, $!M$ satisfies the conclusion of the proposition.

If $!K$ is a subpath of $!F_0 !K_1$ or $!K_r !F_r$, a similar argument applies.
For example, assume that $!K$ is a subpath of $!F_0 !K_1$.
As above, we assume that all paths are in $\Ga_{\al-1}$ not changing
their notations. Let $!L$ be a base axis for $!K$.
By hypothesis, either $\lab(!u) \in \cH_{\al-1}$ or 
$!u = !u_1 !V  !u_2$ where $!V$
occurs in a line $!L_1$ labeled by the infinite power $R^\infty$ of a relator $R$
of rank $\al$ and $!L_1$ is distinct from $!L$.
In the case $\lab(!u) \in \cH_{\al-1}$ we apply Proposition \ref{pr:stability-trigon}$_{\al-1}$. 
Otherwise the argument is the same as in the case when $!K$ is a subpath 
of $!K_i !F_i !K_{i+1}$.
The case when $!K$ is a subpath of $!K_r !F_r$ is similar.

Finally, there is a ``degenerate'' case when $\Area_\al(!X^{-1} !u !Y !v) = 0$ and both $!u$ and ~$!v$
are bridges of rank $\al-1$. In this case, the statement follows directly from 
Proposition \ref{pr:fragment-stability-previous}.
\end{proof}

\begin{proposition}[fragment stability in bigon] \label{pr:fragment-stability-bigon}
Let $\al\ge 1$,  $!X^{-1} !u !Y !v$ be a coarse bigon in ~$\Ga_\al$
and let $!K$ be a fragment of rank $\al$ in $!X$ with $\muf(!K) \ge 2\la +5.8\om$.
Assume that $!K$ is independent of ~$!u$ and ~$!v$.
Then there exists a fragment $!M$ of rank $\al$ in $!Y$ such that $!M \sim !K^{\pm1}$ and
$$
  \muf(!M) \ge \min\set{\muf(!K) - 2 \la - 3.4\om, \ \xi_0}
$$
\end{proposition}

\begin{proof}
Let $!X = !F_0 !K_1 !F_1 \dots !K_r !F_r$ and $!Y = !H_0 !M_0 !H_1 \dots !M_r !H_r$ where
$!K_i$ and $!M_i$ are the associated active fragments of rank $\al$. 
If $!K \sim !K_i$ for some $i$ then we can take $!M = !M_i$ due to Proposition ~\ref{pr:active-large}.
Otherwise we apply Proposition \ref{pr:fragment-nonactive-bigon}.
\end{proof}

\begin{proposition}[fragment stability in trigon] \label{pr:fragment-stability-trigon}
Let $\al\ge 1$, $!X^{-1} !u_1 !Y_1 !u_2 !Y_2 !u_3$ be a coarse trigon in ~$\Ga_\al$ 
and let $!K$ be a fragment of rank $\al$ in $!X$
with $\muf(!K) \ge 3\la +10\om$.
Assume that $!K$ is independent of any of $!u_i$.
Then there is a fragment $!M$ of rank $\al$ in $!Y_1$ or $!Y_2$ such that $!M \sim !K^{\pm1}$ and
$$
  \muf(!M) > \min\bbset{3\la - 1.1\om, \ \frac12 (\muf(!K) - 3\la - 6.8\om)}.
$$
\end{proposition}

\begin{proof}
The idea of the proof is the same as in the proof of Proposition \ref{pr:fragment-nonactive-bigon}.
To avoid complicated notations, we proceed by induction on the $\al$-area of 
$!P = !X^{-1} !u_1 !Y_1 !u_2 !Y_2 !u_3$
as described in \ref{ss:active-loops-trigon}. Assume that $!R$ is an active relator loop of rank $\al$
of $!P$. As observed in \ref{ss:active-loops-trigon}, there are two or three fragments $!N_i$ ($i=1,2$ or 
$i=1,2,3$) of rank $\al$ with base loop ~$!R$ that occur in distinct paths $!X^{-1}$, $!Y_1$ or $!Y_2$.
By Proposition \ref{pr:trigon-active-fragments} we can assume that $\muf(!N_i) \ge 3\la - 1.1\om$
for $i=1,2$.
If $!K \sim !N_1^{\pm1}$ then we for the required $!M$ we take that $!N_i$
which occurs in $!Y_1$ or $!Y_2$. Let $!K \not\sim !N_1^{\pm1}$.

If $!N_1$ and $!N_2$ occur in $!Y_1$ and $!Y_2$ then we can replace $!P$ by a coarse trigon 
with smaller $\al$-area
and use induction (see Figure \ref{fig:fragment-stability-trigon}a). 
\begin{figure}[h]
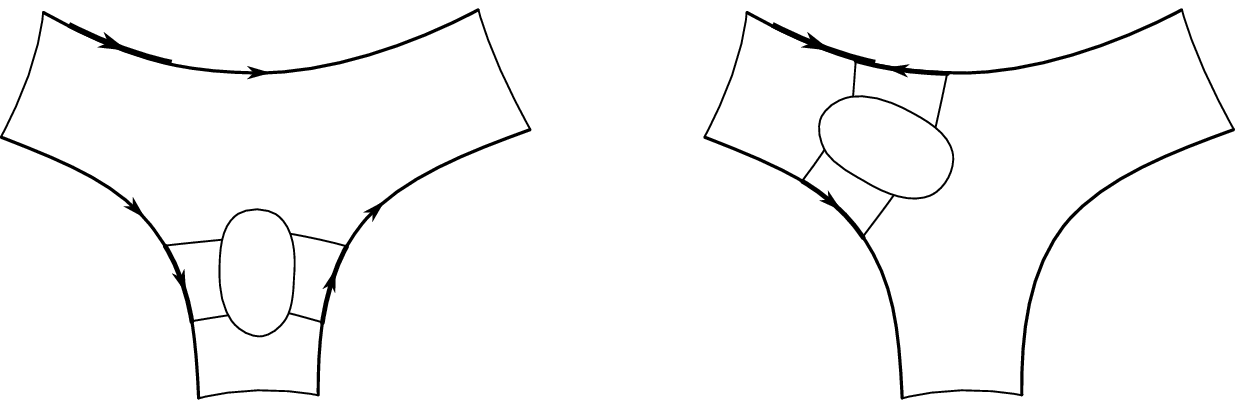 
\caption{}  \label{fig:fragment-stability-trigon}
\end{figure}
(In this case $!u_2$ is replaced by a new bridge $!u_2'$ and the 
assumption $!K \not\sim !N_1^{\pm1}$ implies that $!K$ is independent of ~$!u_2'$.)
Otherwise, assume that $!N_1$ occurs in $!X^{-1}$ and $!N_2$ occurs in $!Y_1$
(the case when $!N_2$ occurs in $!Y_2$ is symmetric).

Since $!K \not\sim !N_1^{-1}$ we have either $!K < !N_1^{-1}$
or $!K > !N_1^{-1}$. In the first case, we reduce the statement to the case of a coarse bigon 
as in Figure \ref{fig:fragment-stability-trigon}b and apply Proposition ~\ref{pr:fragment-nonactive-bigon}.
In the second case, the statement follows by inductive hypothesis.

It remains to consider the case $\Area_\al(!P) = 0$. Then the loop $!P$ can be lifted to $\Ga_{\al-1}$
and we assume that $!P$ is already in $\Ga_{\al-1}$.
Let $!L$ be the base axis for $!K$ and $!S$ the base for $!K$.
Since $!K$ is independent of $!u_i$ (when viewed in $\Ga_\al$), we have
either $\lab(!u_i) \in \cH_{\al-1}$ or
 $!u_i = !v_i !Q_i !w_i$ where $\lab(!v_i), \lab(!w_i) \in \cH_{\al-1}$ and $!Q_i$ 
occurs in a line $!L_i$ labeled by the infinite power $R_i^\infty$ of a relator $R_i$
of rank $\al$ such that $!L_i \ne !L$.
We obtain a coarse $r$-gon with sides $!X^{-1}$, $!Y_1$, $!Y_2$ and $!Q_i$ where $3 \le r \le 6$
(see Figure \ref{fig:fragment-stability-trigon-non-active}).
We consider the ``worst'' case $r=6$ (the other cases are similar, with application of
Propositions \ref{pr:stability-trigon}$_{\al-1}$  or \ref{pr:fragment-stability-previous}$_{\al-1}$
where needed).  
\begin{figure}[h]
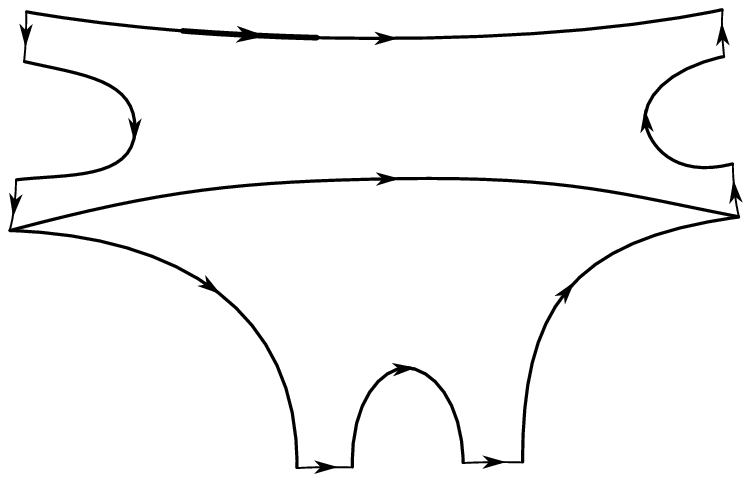
\caption{}  \label{fig:fragment-stability-trigon-non-active}
\end{figure}
Let $!Z$ be a reduced path joining $\tau(!u_1)$ and $\io(!u_3)$ existing by 
Proposition \ref{pr:reduction}$_{\al-1}$.
By Corollary ~\ref{co:small-overlapping-Cayley}, if a subpath ~$!P$ of $!S$ is close to a subpath of $!Q_i$
then $\mu(!P) < \la$.
Then the statement easily follows by applying Lemma \ref{lm:stability-tetragon-previous} twice
to coarse tetragons $!X^{-1} !v_1 !Q_1 !w_1 !Z !v_3 !Q_3 !w_3$ and 
$!Z^{-1} !Y_1 !v_2 !Q_2 !w_2 !Y_2 $.
\end{proof}

\begin{lemma} 
\label{lm:piece-fragment-no-higher-fragments}
Let $\al\ge1$, $X$ be a piece of rank $1 \le \be < \al$ or a fragment
of rank $\be<\al$. Then $X$ contains no fragment $K$ of rank $\al$ with $\muf(K) \ge 3.2\om$.

In particular, any fragment $K$ of rank $\al$ with $\muf(K) \ge 3.2\om$ is a nonempty word
(since otherwise it would occur in a fragment of rank 0).
\end{lemma}

\begin{proof}
We consider the case when $X$ is a fragment of rank $\be<\al$. 
We represent $X$ by a path ~$!X$ in $\Ga_{\al-1}$. Assume that $!X$ contains a fragment $!K$  
of rank $\al$ with $\muf(!K) \ge 3.2\om$. 
Let $!S$ be a base for $!K$ with $|!S|_{\al-1} \ge 3.2$.
By Lemma \ref{lm:piece-fragment-no-higher-fragments}$_{\le\al-1}$
and Corollary \ref{co:no-active-fragments-iterated} we have $!S = !w_1 !S_1 !w_2$ and $!K = !z_1 !K_1 !z_2$
where $!S_1$ and $!K_1$ are close in rank $\max(0,\be-1)$ and 
$|!S_1|_{\al-1} > |!S|_{\al-1}- 2-10\ze^2\eta > 1.15$.
If $\be=0$ we already get a contradiction since in this case 
$|!K_1|\le 1$ but $|!S_1| \ge |!S_1|_{\al-1} > 1$.
Let $\be\ge 1$.
Up to change of notation, we assume that $!X$, $!K_1$ and $!S_1$ are lifted to $\Ga_{\be-1}$.
Let $!T$ be a base for $!X$. By Proposition \ref{pr:closeness-stability}$_{\be-1}$
a subpath $!T_1$ of ~$!T$ is close to a subpath $!S_2$ of $!S$ with 
$|!S_2|_{\al-1} > |!S_1|_{\al-1} - 2.6\ze > 1$.  Then $!S_2$ is a fragment of rank $\be$ with base $!T_1$
and we should have $|!S_2|_{\al-1} \le 1$, a contradiction.

In the case when $X$ is a piece of rank $\al$ a similar argument works with skipping application of
Proposition \ref{pr:closeness-stability}$_{\be-1}$.
\end{proof}

\begin{lemma} \label{lm:strongly-cyclically-reduced-previous}
Let $\al\ge 1$ and $X$ be a word cyclically reduced in $G_{\al-1}$. 
Assume that a cyclic shift of $X$ contains a fragment $K$ of rank $\al$ with $\muf(K) \ge 6.5\om$.
Then $X$ is strongly cyclically reduced in $G_{\al-1}$. 
\end{lemma}

\begin{proof}
Let $F$ be a fragment of rank $1 \le \be \le \al-1$ in a word  $X^t$.
Assume that $|F| > |X|$.
Using Proposition \ref{pr:dividing-fragment} represent $K$ as 
$K \greq K_1 u K_2$ where $\muf(K_1), \muf(K_2) > 3.2\om$.
Since $|K| \le |X|$,  $F$ should contain a translate of $K_1$ or $K_2$.
But this is impossible by Lemma ~\ref{lm:piece-fragment-no-higher-fragments}.
Hence $|F| \le |X|$ and then $\muf(F) \le \rho$ since $X$ is cyclically reduced in  $G_{\al-1}$.
This shows that any power $X^t$ is reduced in $G_{\al-1}$, i.e.\  
$X$ is strongly cyclically reduced in ~$G_{\al-1}$.
\end{proof}

\begin{proposition} [fragment stability in conjugacy relations with cyclic sides] 
\label{pr:fragment-stability-cyclic}
Let $\al\ge 1$ and $X$ and ~$Y$ be words which are cyclically reduced in $G_\al$ and represent conjugate 
elements of $G_\al$.
Let $\bar{!X} = \prod_{i \in \Z} !X_i$ and $\bar{!Y} = \prod_{i \in \Z} !Y_i$ be parallel lines in $\Ga_\al$ 
representing the conjugacy relation.
Let $!K$ be a fragment of rank ~$\al$ in $\bar{!X}$ with $\muf(!K) \ge 2\la +5.8\om$
and $|!K| \le |X|$.
Then there is a fragment $!M$ of rank $\al$ in $\bar{!Y}$ such that $!M \sim !K^{\pm1}$ and
$$
  \muf(!M) \ge \min\set{\muf(!K) - 2 \la - 3.4\om, \ \xi_0}
$$
\end{proposition}

\begin{proof}
By Lemma \ref{lm:strongly-cyclically-reduced-previous} $X$ is 
strongly cyclically reduced in $G_{\al-1}$. 
We claim that a cyclic shift of $Y$ also contains a fragment $F$ of rank $\al$ with $\muf(F) \ge 6.5$
and thus $Y$ is strongly cyclically reduced in $G_{\al-1}$ as well.
Indeed, by Proposition \ref{pr:no-active-fragments-cyclic-iterated} with $\be :=\al-1$
we may assume for some cyclic shifts $X'$ and $Y'$ of $X$ and $Y$ we have $Y' = w^{-1} X' w$
in $G_{\al-1}$ where $w \in \cH_{\al-1}$. Then existence of $F$ easily follows by 
Propositions \ref{pr:dividing-fragment} and \ref{pr:fragment-stability-previous}.

Consider a reduced annular diagram $\De$ of rank $\al$ with boundary loops $\hat{!X}$ and $\hat{!Y}^{-1}$ 
representing the conjugacy relation given in the proposition.
Let $\ti\De$ be the universal cover of $\De$ and let $\phi: \ti\De^{(1)} \to \Ga_\al$ 
be a combinatorially continuous 
map which sends lifts of $\hat{!X}$ 
and $\hat{!Y}$ to $\bar{!X}$ and $\bar{!Y}$ respectively.

Assume that $\De$ has a cell of rank $\al$. Let $!D$ be some lift of this cell in $\ti\De$.
By Proposition \ref{pri:bigon-cell}, $\phi(\de !D)$ and $\phi(\de !D)^{-1}$ are base loops 
for fragments $!N_i$ 
$(i=1,2)$ of rank $\al$ in $\bar{!X}$ and ~$\bar{!Y}$ respectively, 
such that $\muf(!N_1) + \muf(!N_2) \ge 1 - 2\la - 1.5\om$.
Since $X$ and $Y$ are cyclically reduced in $G_\al$ we have
$\muf(!N_i) \le \rho$ and hence 
$\muf(!N_i) \ge 1 - \rho - 2\la - 1.5\om = \xi_0$. 
By construction, we have $!N_1 \sim !N_2^{-1}$. Since $\bar{!X}$ and $\bar{!Y}$ are parallel, 
we have $s^k_{X,\bar{!X}} !N_1 \sim s^k_{Y,\bar{!Y}} !N_2^{-1}$ for any $k \in \Z$.
If $!K \sim s^k_{X,\bar{!X}} !N_1$ for some $k$ then we can take $s^k_{Y,\bar{!Y}} !N_2$ for $!M$.
Otherwise we have $s^k_{X,\bar{!X},} !N_1 < !K < s^{k+1}_{X,\bar{!X}} !N_1$ for some ~$k$ and the rest 
of the argument is the same as in the proof of Proposition \ref{pr:fragment-nonactive-bigon}. 

Now assume that $\De$ has no cells of rank $\al$. 
We can assume that $\De$ is a reduced diagram of rank ~${\be}$
for some $\be \le \al-1$ and in case $\be \ge 1$, $\De$ has at least one cell of rank $\be$.
If $\be = 0$ then $\bar{!X} = \bar{!Y}$ and there is nothing to prove.
Let $\be \ge 1$.
Up to change of notations, we assume that $!K$, $\bar{!X}$ and $\bar{!Y}$ are lifted to $\Ga_{\al-1}$.
Proposition \ref{pri:bigon-cell}$_\be$ implies that some vertices $!a$ on $\bar{!X}$
and ~$!b$ on $\bar{!Y}$ are joined by a bridge of rank $\be$.
This is true also for any translates $s^i_{X,\bar{!X}} !a$ and $s^i_{Y,\bar{!Y}} !b$.
Then the statement follows by Proposition \ref{pr:fragment-stability-previous}
(here we use that $X$ and $Y$ are strongly cyclically reduced in $G_{\al-1}$).
\end{proof}

\begin{lemma} \label{lm:stability-cyclic-monogon-previous}
Let $\al\ge 1$ and $S$ be a word cyclically reduced in $G_{\al-1}$. 
Assume that $S$ is conjugate in $G_{\al-1}$ to a word $T_1v_1T_2v_2$ 
where $T_i$ are reduced in $G_{\al-1}$ and $v_i$ are bridges of rank $\al$.
Let $\bar{!S} = \prod_{i \in \Z} !S_i$ and 
$\prod_{i \in \Z} !T_1^{(i)} !v_1^{(i)} !T_2^{(i)} !v_2^{(i)}$ be parallel lines in $\Ga_{\al-1}$ 
representing the conjugacy relation. Denote $!U_{2i} = !T_1^{(i)}$ and 
$!U_{2i+1} = !T_2^{(i)}$.

Assume that a reduced path $!X$ in $\Ga_{\al-1}$ is close to a subpath $!Y$ of $\bar{!S}$ with $|!Y| \le |S|$. 
Let $|!X|_{\al-1} \ge 8$.
Then $!X$ can be represented as $!z_0 !X_1 !z_1 \dots !X_r !z_r$ $(1 \le r \le 4)$ where each $!X_i$ is close 
to a subpath of some $!U_{j_i}$, $j_1 < \dots < j_r$, $j_r - j_1 \le 3$ and 
$$
  \sum_i |!X_i|_{\al-1} \ge |!X|_{\al-1} - 9.
$$
\end{lemma}

\begin{proof}
Let $Z$ be a word reduced in $G_{\al-1}$ such that $Z = T_1 v_1 T_2$ in $G_{\al-1}$.
We join $\io(!T_1^{(i)})$ and $\tau(!T_2^{(i)})$ with the path $!Z_i$ labeled $Z$.
Since  $|!X|_{\al-1} \ge 8$, application of Propositions \ref{pr:stability-cyclic-monogon}$_{\al-1}$
gives $!X = !w_1 !X' !w_2$ or $!X = !w_1 !X' !w_2 !X'' !w_3$ where, respectively, $!X'$ is close to 
a subpath of some ~$!Z_i$ and $|!X'|_{\al-1} \ge |!X|_{\al-1} - 2.9$ or for some $i$, $!X'$  is close to 
a subpath of $!Z_i$, $!X''$  is close to a subpath of $!Z_{i+1}$ and 
$|!X'|_{\al-1} + |!X''|_{\al-1} \ge |!X|_{\al-1} - 3$. 
Then a single or double application of  Proposition \ref{pr:stability-trigon}$_{\al-1}$ gives
the required $!X_i$'s. 
\end{proof}

\begin{proposition} [fragment stability in conjugacy relations with non-cyclic side] 
\label{pr:fragment-stability-cyclic1}
Let $\al\ge 1$ and $X$ be a word cyclically reduced in $G_\al$. 
Assume that $X$ is conjugate in $G_\al$ to a word $Yu$ 
where $Y$ is reduced in $G_\al$ and $u$ is a bridge of rank $\al$.
Let $\bar{!X} = \prod_{i\in\Z} !X_i$ and $\prod_{i\in\Z} !Y_i !u_i$ be 
parallel lines in $\Ga_\al$ 
representing the conjugacy relation. 
Let $!K$ be a fragment of rank ~$\al$ in $\bar{!X}$ with $\muf(!K) \ge 3\la +9\om$ and $|!K| \le |X|$.
Assume that $!K$ is independent of any of the bridges ~$!u_i$.
Then there is a fragment $!M$ of rank $\al$ in some $!Y_k$ such that $!M \sim !K^{\pm1}$ and
$$
  \muf(!M) > \min\bbset{\frac52 \la - 1.1\om, \ \frac12 (\muf(!K) - 3\la - 6.8\om)}.
$$
\end{proposition}

\begin{proof}
Let $\De$ be an annular diagram of rank $\al$ with boundary loops $\hat{!X}^{-1}$ and $\hat{!Y} \hat{!u}$
representing the conjugacy relation. Let $\ti{\De}$ be the universal cover of $\De$ and
$\phi: \ti{\De}^{(1)} \to \Ga_\al$ a combinatorially continuous 
map sending lifts $\ti{!X}_i$, $\ti{!Y}_i$ 
and $\ti{!u}_i$ of $\hat{!X}$, $\hat{!Y}$ and $\hat{!u}$
to $!X_i$, $!Y_i$ and $!u_i$ respectively.
Up to switching of $\hat{!u}$, 
we assume that $\De$ is reduced and has a tight set $\cT$ of contiguity subdiagrams.

{\em Case\/} 1: $\De$ has no cells of rank $\al$. 
Then parallel lines $\bar{!X} = \prod_{i\in\Z} !X_i$ and $\prod_{i\in\Z} !Y_i !u_i$
can be lifted to $\Ga_{\al-1}$; we assume that they and the subpath $!K$ of $\bar{!X}$ are
already lifted to $\Ga_{\al-1}$.
If $u \in \cH_{\al-1}$ then the statement follows by Proposition \ref{pr:stability-cyclic-monogon}$_{\al-1}$,
so we assume that $u \notin \cH_{\al-1}$.
Let $!L$ be the base axis for $!K$ and $!S$ the base for $!K$.
Since $!K$ is independent of $!u_i$ (when viewed in $\Ga_\al$) we have 
$!u_i = !w_1^{(i)} !Q_i !w_2^{(i)}$ where $\lab(!w_j^{(i)}) \in \cH_{\al-1}$ 
and $!Q_i$ occurs in a line $!L_i$ labeled by the infinite power $R_i^\infty$ of a relator $R_i$
of rank $\al$ such that $!L_i \ne !L$.
By Corollary \ref{co:small-overlapping-Cayley}, if a subpath ~$!P$ of $!S$ is close to a subpath of $!Q_i$
then $\mu(!P) < \la$.
Applying Lemma \ref{lm:stability-cyclic-monogon-previous} we conclude 
that either there exists a fragment $!M$ of rank $\al$ in some 
$!Y_k$ such that $!M \sim \bar{!K}$ and 
$\muf(!M) > \muf(!K) - 2\la - 9\om$ or there exist 
fragments $!M_1$ and $!M_2$ of rank $\al$ in some $!Y_k$
and $!Y_{k+1}$ respectively such that 
$!M_1 \sim !M_2 \sim !K$ and 
$$
  \muf(!M_1) + \muf(!M_2) > \muf(!K) - 2\la - 9\om.
$$
In the latter case, for at least one $!M_i$ we have 
$\muf(!M_i) > \frac12 (\muf(!K) - 2\la - 9\om)$ and we can take its image in  $\Ga_\al$
for the required $!M$.

{\em Case\/} 2: 
$\De$ has at least one cell of rank $\al$.
Let $!D$ be such a cell and let $\ti{!D}$ be a lift of $!D$ in ~$\ti{\De}$.
By Proposition \ref{pri:annular-single-layer1} and Lemma \ref{lmi:cell-regularity-side}, $!D$ has 
two or three contiguity subdiagrams $\Pi_i \in \cT$ to sides of ~$\De$, 
at most two to $\hat{!Y}$ and at most one to ~$\hat{!X}^{-1}$. 
By Proposition \ref{pri:cyclic-monogon-cell},
$\phi(\de \ti{!D})$ is the base loop for two or three fragments $!N_i$
($i=1,2$ or $i=1,2,3$) of rank $\al$ in two or three of the paths
$\bar{!X}^{-1}$, $!Y_j$ and $!Y_{j+1}$ for some ~$j$, respectively, with
\begin{equation} \label{eq:fragment-stability-cyclic1}
    \sum_i \muf(!N_i) > 1 - 4\la - 2.2\om.
\end{equation}
Since $\muf(!N_i) \le \rho$ for each $i$, for at least two indices $i$ we have 
$$
    \muf(!N_i) > \frac12 (1 - 4\la - 2.2\om - \rho) = \frac52 \la - 1.1\rho.
$$
Note that all $!N_i$ are pairwise compatible.
If $!K \sim !N_1^{\pm1}$ then for the required $!M$ we can take that $!N_i$
which occurs in $!Y_i$ or in $!Y_{j+1}$ and has a larger $\muf(!N_i)$.
Hence we can assume that $!K \not\sim !N_i^{\pm1}$ for all $!N_i$
produced by all lifts $\ti{!D}$ of all cells $!D$ of rank $\al$ of $\De$.

Assume that $!D$ has two contiguity subdiagrams $\Pi_i \in \cT$ $(i=1,2)$
to $\hat{!Y}$, i.e.\ the corresponding fragments $!N_1$ and $!N_2$
of rank $\al$ occur in $!Y_{k}$ and $!Y_{k+1}$ respectively.
Then we cut off from $\De$ the subdiagram $\De \cup \Pi_1 \cup \Pi_2$ 
and the remaining simply connected component. This replaces 
$\De$ with a new diagram $\De'$ with a smaller number of cells of rank $\al$,
$!Y_i$ with a subpath of $!Y_i$, bridges $!u_i$ with another bridges $!u_i'$
and the assumption that $!K \not\sim !N_i^{\pm1}$ for $!N_i$
produced by all lifts $\ti{!D}$ of $!D$ implies that $!K$ is
independent of all new bridges $!u_i'$.
In this case we can apply induction on the number of the cells of 
rank $\al$ of $\De$.

We may assume now that each cell $!D$ of rank $\al$ of $\De$
has precisely two contiguity subdiagrams $\Pi_i \in\cT$ to sides of $\De$, 
one to $\hat{!X}^{-1}$ and another one to $\hat{!Y}$. This implies that each
lift of $!D$ produces two fragments $!N_i$, one in $\bar{!X}^{-1}$ 
and one in some $!Y_j$. 
Let $\set{!D_1,!D_2,\dots,!D_k}$ be the set of all cells of 
rank $\al$ of $\De$. For each lift $\ti{!D}_i^{(j)}$ ($t \in \Z$) of $!D_i$,
denote $!N_{i,1}^{(j)}$ and $!N_{i,2}^{(j)}$ 
the corresponding fragments of rank $\al$ that occurs in 
$\bar{!X}^{-1}$ and $!Y_j$ respectively (the requirement that 
$!N_{i,2}^{(j)}$ occurs in $!Y_j$ determines uniquely the lift $\ti{!D}_i^{(j)}$
and the fragment $!N_{i,1}^{(j)}$).
Note that \eqref{eq:fragment-stability-cyclic1} implies 
$$
    \muf(!N_{i,k}^{(j)})
    > 1 - 4\la - 2.2\om - \rho  = 5\la - 2.2\om.
$$
We order cells $!D_i$ to get $!N_{i,2}^{(j)}$ ordered in $!Y_j$ as
$!N_{1,2}^{(j)} \ll \dots \ll !N_{k,2}^{(j)}$.
Consequently, in $\bar{!X}$ we have
$ \cdots {!N_{1,1}^{(j)}}^{-1} \ll \dots \ll {!N_{k,1}^{(j)}}^{-1}
\ll {!N_{1,1}^{(j+1)}}^{-1} \ll \dots \ll {!N_{k,1}^{(j+1)}}^{-1} \cdots$
(Figure \ref{fig:fragment-stability-cyclic1}). 
\begin{figure}[h]
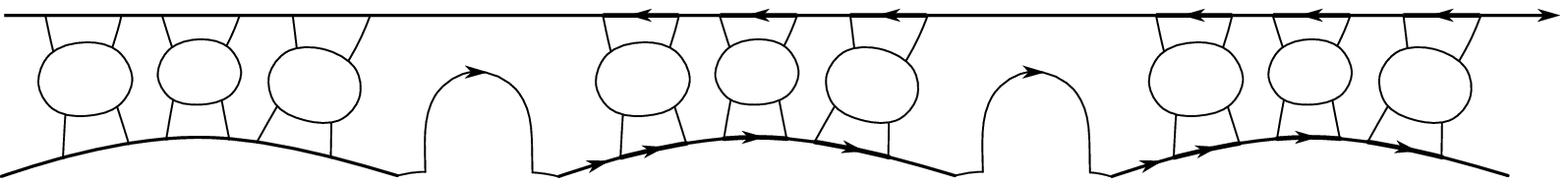
\caption{}  \label{fig:fragment-stability-cyclic1}
\end{figure}
By the assumption above, we have $!K \not\sim {!N_{i,1}^{(j)}}^{-1}$
for all $i,j$.
Then by Proposition \ref{pr:inclusion-compatibility} we 
have either 
${!N_{i,1}^{(j)}}^{-1} < !K < {!N_{i+1,1}^{(j)}}^{-1}$
for some $i,j$ or 
${!N_{k,1}^{(j)}}^{-1} < !K < {!N_{1,1}^{(j+1)}}^{-1}$
for some $i$. In each of these cases, we find the required $!M$ by
applying an appropriate part of the proof of 
Proposition ~\ref{pr:fragment-nonactive-bigon} or 
Proposition ~\ref{pr:fragment-stability-trigon}.
\end{proof} 

We will use the following observation.

\begin{lemma} 
\begin{enumerate}
\item 
\label{lmi:compatible-close}
Let $!K$ be a fragment of rank $1 \le \be \le \al$ in $\Ga_\al$.
Let $!M$ be either another fragment of rank $\be$ in $\Ga_\al$ 
such that $!K \sim !M^{\pm1}$ or a bridge of rank $\be$ 
such that $!K$ is not independent of ~$!M$.
Then any of the endpoints of $!K$ can be joined with any of the endpoints 
of $!M$ by a bridge $!w$ of rank ~$\be$.

Moreover, $!w$ can be chosen with the following property.
If $!N$ is any other fragment of rank ~$\be$ such that 
$!N \not\sim !M^{\pm1}$ then $!N$ is independent of $!w$.
\item
\label{lmi:compatible-uniformly-close}
Let $!K_1$, $!K_2$, $\dots$, $!K_r$ be fragments of rank $\be \le \al$ in $\Ga_\al$
such that $!K_1 \sim !K_i^{\pm1}$ for all $i$. Then all endpoints of all $!K_i$ are uniformly close.
\end{enumerate}
\end{lemma}

\begin{proof}
Follows from  definitions in \ref{ss:fragment-paths} and 
Definition \ref{def:interference}.
\end{proof}

\begin{lemma} \label{lm:closeness-2bigons}
Let $(!X_i, !Y_i)$ $(i=1,2)$ be two pairs of close reduced paths in $\Ga_\al$ where $!X_1$ and ~$!X_2$
are subpaths of a reduced path $\bar{!X}$. Assume that for the common subpath $!Z$ of $!X_1$ and ~$!X_2$
we have $|!Z|_\al \ge 2.2$. Then  there exists a triple $!a_i$ $(i=1,2,3)$ of uniformly close vertices 
on $!Z$, $!Y_1$ and ~$!Y_2$ respectively.
\end{lemma}

\begin{proof}
If $\al=0$ there is nothing to prove. Let $\al\ge 1$.
Let $!X_i^{-1} !u_i !Y_i !v_i$ $(i=1,2)$ be a coarse bigon where $!u_i$ and $!v_i$
are bridges of rank $\al$. 

{\em Case\/} 1: $\Area_\al(!X_i^{-1} !u_i !Y_i !v_i) = 0$ for both $i=1,2$.
We apply Proposition \ref{pr:no-active-loops} and find loops 
$!X_i'^{-1} !u_i' !Y_i' !v_i'$ that can be lifted to $\Ga_{\al-1}$ where
$!X_i'$ and $!Y_i'$ are subpaths of $!X_i$ and $!Y_i$ respectively. 
For the common part $!Z'$ of $!X_1'$ and $!Z_2'$ we have 
$|!Z'|_\al \ge |!Z|_\al - 2.04 \ge 0.16$ and hence $|!Z'|_{\al-1} \ge 3.2$.
Then the statement follows by induction.

{\em Case\/} 2: $\Area_\al(!X_i^{-1} !u_i !Y_i !v_i) > 0$ for $i=1$ or $i=2$.
Without loss of generality, assume that $!K$ and $!M$ are active fragments of rank $\al$ in 
$!X_1$ and in $!Y_1$, respectively, such that $!K \sim !M^{-1}$. 
Let $!X_1 = !S_1 !K !S_2$ and $!Y_1 = !T_1 !M !T_2$. If $!S_1 !K$ contains $!Z$ then 
we shorten $!X_1$ and $!Y_1$ replacing them with $!S_1 !K$ and $!T_1$ thereby decreasing 
$\Area_\al(!X_1^{-1} !u_1 !Y_1 !v_1)$ as described in \ref{ss:active-loop-induction}.
Similarly, if $!K !S_2$ contains $!Z$ then we can replace $!X_1$ and $!Y_1$ 
with $!K !S_2$ and $!T_2$. Therefore, we can assume that $!K$ is contained in $!Z$. 
We take $!a_1 = \io(!K)$ and $!a_2 = \io(!M)$. If $!K$ is not independent of $!u_2$
or from $!v_2$ then for $!a_3$ we can take $\io(!Y_2)$ or $\tau(!Y_2)$ 
respectively. Otherwise by Proposition \ref{pr:fragment-stability-bigon} there exists
a fragment $!N$ of rank $\al$ in $!Y_2$ such that $!N \sim !K^{\pm1}$ and we 
can take $!a_3 = \io(!N)$.
\end{proof}

\begin{lemma} \label{lm:closeness-bigon-1side}
Let $(!S,!T)$ and $(!X,!Y)$ be pairs of close reduced paths in $\Ga_\al$ where $!Y$ is an end of $!S$
and the ending vertices $\tau(!X)$, $\tau(!Y)=\tau(!S)$ and $\tau(!T)$ are uniformly close.
Then there exists a triple $!a_i$ $(i=1,2,3)$ of uniformly close 
vertices on $!X$, $!Y$ and $!T$ respectively,
such that $!a_1$ cuts off a start $!X_1$ of $!X$ with 
$|!X_1|_\al < 1.3$ and $!a_2$ cuts off a start $!Y_1$ of $!Y$ with 
$|!Y_1|_\al < 1.15$.
\end{lemma}

\begin{proof}
We can assume $\al\ge1$. 
We use induction on $|!X| + |!Y| + |!T|$. If $|!X|_\al < 1.3$ and 
$|!Y|_\al < 1.2$ there is nothing to prove.
We assume that $|!X|_\al \ge 1.3$ or $|!Y|_\al \ge 1.15$. 
It is enough to find a triple $!a_i$ $(i=1,2,3)$ of uniformly close vertices on 
$!X$, $!Y$ and $!T$ respectively, such that at least one $!a_i$ cuts off a proper start of appropriate
path $!X$, $!Y$ or $!T$.

Let $!X^{-1} !u_1 !Y !u_2$ and $!S^{-1} !v_1 !T !v_2$ be coarse bigons in ~$\Ga_\al$
where $!u_i$ and $!v_i$ are bridges of rank ~$\al$. 

{\em Case\/} 1: 
$\Area_\al(!X^{-1} !u_1 !Y !u_2) = \Area_\al(!S^{-1} !v_1 !T !v_2) = 0$.
We assume that $!u_2$ and $!v_2$ are defined from the condition that $\tau(!X)$, $\tau(!Y)$ and $\tau(!T)$
are uniformly close;
that is, either $!u_2$ and $!v_2$ are bridges of rank $\al-1$ or
have the form $!u_2 = !w_1 !P_1 !w_2$ and $!v_2 = !w_3 !P_2 !w_4$ where 
$!w_i$ are bridges of rank $\al-1$ and $!P_i^{\pm1}$ are subpaths of a relator loop $!R$ of rank $\al$.
We consider the second case (the case when $!u_2$ and $!v_2$ are bridges of rank $\al-1$ is treated in
a similar manner).

Without changing notations, we assume that loops $!X^{-1} !u_1 !Y !u_2$ and $!S^{-1} !v_1 !T !v_2$
are lifted to ~$\Ga_{\al-1}$ and, consequently, all paths introduced are in $\Ga_{\al-1}$
(the only change is that $!P_i^{\pm1}$ become subpaths of an $R$-periodic line $\ti{!R}$ where 
$R$ is a relator of rank $\al$). After choosing $!a_i$ ($i=1,2,3$) in $\Ga_{\al-1}$ 
we pass on to their images in $\Ga_\al$.

{\em Case\/} 1a: $|!X|_\al \ge 1.3$.
If a vertex $!b_1 \ne \tau(!X)$ on $!X$ is close in rank $\al-1$ to a vertex 
$!b_2$ on ~$!P_1$ then we can take $!a_1 := !b_1$, $!a_2 := \tau(!Y)$ and 
$!a_3 := \tau(!T)$. We assume that no such $!b_1$ and $!b_2$ exist.
Then application of Proposition \ref{pri:4gon-closeness}$_{\al-1}$  shows that 
$!X = !z_1 !X' !z_2$ where $!X'$ is close to a subpath $!Y'$ of $!Y$,
$|!z_1|_\al \le 1 + 4\ze^2\eta$, $|!z_2|_\al \le 4\ze^2\eta$ and hence 
$  |!X'|_\al \ge  0.3 - 8\ze^2\eta $.

Assume first that $\al\ge 2$. Then shortening $!X'$ 
from the end by Proposition \ref{pr:fellow-traveling}$_{\al-1}$ we can assume that $!z_1 !X'$ 
is a proper start of $!X$ (and that $!X'$ is still close
to a subpath $!Y'$ of $!Y$). For the shortened $!X'$, we have 
$
    |!X'|_{\al} > 0.3 - 8\ze^2\eta - \ze^2 > 0.26
$ 
which implies 
$|!X'|_{\al-1} \ge \frac1\ze |!X'|_{\al} > 5.2$. 
Let $!v_1 = !w_5 !Q !w_6$ where $!w_5, !w_6$ are bridges of rank $\al-1$ and 
$!Q$ is labeled by a piece of rank ~$\al$.
Application of Lemma \ref{lm:stability-tetragon-previous}
gives a triple of uniformly close vertices $!a_i$ $(i=1,2,3)$ where $!a_1$ lies on $!X'$, $!a_2$ lies on $!Y'$ and 
$!a_3$ lies either on $!Q$ or $!T$. If $!a_3$ lies on $!Q$ then we replace it with $\io(!T)$.
In the case $\al=1$ we shorten $!X'$ by one edge and for the new $!X'$ we have 
$ |!X'|_{\al} > 0.3 - 8\ze^2\eta - \ze > 0$.
We can still apply Lemma \ref{lm:stability-tetragon-previous} 
due to Remark \ref{rm:stability-tetragon-previous-rank0}, so the argument remains the same.

{\em Case\/} 1b: $|!Y|_\al \ge 1.15$.
Similarly to Case 1, we can assume that there is no vertex $!b \ne \tau(!Y)$ on $!Y$ 
(and hence on $!S$ since $|!Y|_{\al-1} \ge \frac{1.15}\ze = 23$)
close in rank $\al-1$ to a vertex on $!P_1$ or on ~$!P_2$.
Applying Proposition \ref{pri:4gon-closeness}$_{\al-1}$ we represent $!Y$ and $!S$
as $!Y = !z_1 !Y' !z_2$,  $!S = !z_3 !S' !z_4$ where $!Y'$ is close (in rank $\al-1$) to 
a subpath $!X'$ of $!X$, $!S'$ is close to a subpath $!T'$ of $!T$ and $|!z_1|_\al,  |!z_3|_\al < 1 + 4\ze^2\eta$, 
$|!z_2|_\al,  |!z_4|_\al < 4\ze^2\eta$. 
In the case $\al=1$ there is a common subpath $!Z$ of $!X'$, $!Y'$, $!S'$ and $!T'$ of
size $|!Z|_\al \ge |!Y|_\al - 1 - 8\ze^2\eta > 0$ and we can take $\io(!Z)$ for all ~$!a_i$.
In the case $\al \ge 2$, 
shortening $!Y'$ from the end by Proposition \ref{pr:fellow-traveling}$_{\al-1}$ 
we can assume that $!z_1 !Y'$ is a proper start of $!Y$. 
Let $!Z$ be the common subpath of $!Y'$ and $!S'$.
We have 
$
    |!Z|_\al > |!Y|_\al - 1 - 8\ze^2\eta - \ze^2 > 0.11 
$
and hence  $|!Z|_{\al-1} > 2.2$. 
Then the statement follows by Lemma \ref{lm:closeness-2bigons}$_{\al-1}$.

{\em Case\/} 2: $\Area_\al(!S^{-1} !v_1 !T !v_2) > 0$. Let $!K$ and $!M$ be active fragments of rank $\al$ in 
$!S$ and in $!T$, respectively, such that $!K \sim !M^{-1}$. 
Let $!S = !G_1 !K !G_2$ and $!T = !H_1 !M !H_2$.
Note that $|!K|,|!M| > 0$ by Lemma \ref{lm:piece-fragment-no-higher-fragments}.
If $!K$ is not contained in $!Y$ then we replace $!S$ and $!T$ with $!K !G_2$ and $!H_2$
respectively and use induction. Assume that $!K$ is contained in $!Y$. 
We first take $!a_2 := \io(!K)$, $!a_3 := \io(!M)$. If $!M$ is not independent
on ~$!u_1$ or from $!u_2$ then we take $!a_1 := \io(!X)$ or $!a_1 := \tau(!X)$ respectively.
Otherwise by Proposition \ref{pr:fragment-stability-bigon}
there exits a fragment $!N$ of rank $\al$ in $!X$ such that $!N \sim !M^{\pm1}$. In this case
we take $!a_1 := \io(!N)$ by Lemma \ref{lmi:compatible-uniformly-close}.

{\em Case\/} 3: $\Area_\al(!X^{-1} !u_1 !Y !u_2) > 0$. Let $!K$ and $!M$ be active fragments of rank $\al$ in 
$!X$ and $!Y$ respectively such that $!K \sim !M^{-1}$. 
Then take $!a_1 := \io(!K)$, $!a_2 := \io(!M)$. Depending on whether $!M$ 
is not independent of $!v_1$ or $!v_2$
we find $!a_3$ similarly to the case 2 using Proposition \ref{pr:fragment-stability-bigon}
and Lemma \ref{lmi:compatible-uniformly-close}.
\end{proof}

\begin{proposition}[closeness transition in bigon] 
\label{pr:closeness-stability} 
Let $(!X,!Y)$ and $(!S,!T)$ be pairs of close reduced paths in $\Ga_\al$ where $!Y$ is a subpath of $!S$. 
Assume that $|!X|_\al \ge 2.3$.
Then $!X = !z_1 !X' !z_2$ where 
$!X'$ is close to a subpath $!W$ of ~$!T$ and 
$|!z_i|_\al < 1.3$ $(i=1,2)$.

Moreover, we have $!Y = !t_1 !Y' !t_2$ where $|!t_1|_\al, |!t_2|_\al < 1.15$ and 
triples $(\io(!X'), \io(!Y'), \io(!W))$ and $(\tau(!X'), \tau(!Y'), \tau(!W))$ are uniformly close.
\end{proposition}

\begin{proof}
We can assume that $\al\ge 1$.
Let $!X^{-1} !u_1 !Y !u_2$ and $!S^{-1} !v_1 !T !v_2$ be coarse bigons in ~$\Ga_\al$
where $!u_i$ and $!v_i$ are bridges of rank ~$\al$. 
By Lemma \ref{lm:closeness-bigon-1side} it is enough to find a triple $!a_i$ $(i=1,2,3)$
of uniformly close vertices on $!X$, $!Y$ and $!T$ respectively.
An easy analysis involving Proposition ~\ref{pr:fragment-stability-bigon} shows how to do this 
in the case when $\Area_\al(!X^{-1} !u_1 !Y !u_2) > 0$ or $\Area_\al(!S^{-1} !v_1 !T !v_2)> 0$.
It remains to consider the case when $\Area_\al(!X^{-1} !u_1 !Y !u_2) = \Area_\al(!S^{-1} !v_1 !T !v_2) = 0$.
Let $!v_i = !v_{i1} !R_i !v_{i2}$ $(i=1,2)$ where $!v_{ij}$ is a bridge of rank $\al-1$ and 
$!R_i$ is labeled by a piece of rank ~$\al$.
By Proposition ~\ref{pr:no-active-loops} 
we have $!X = !w_1 !X_1 !w_2$ where endpoints of $!X_1$ and a subpath $!Y_1$ of $!Y$ 
can be joined by bridges $!u_1'$ and 
$!u_2'$ of rank $\al-1$, so that the loop $!X_1^{-1} !u_1' !Y_1 !u_2'$ can be lifted to ~$\Ga_{\al-1}$ and 
$|!w_i|_\al \le 1 + 4\ze^2\eta$ $(i=1,2)$. 
Without changing notations, we assume that loops $!X_1^{-1} !u_1' !Y_1 !u_2'$ and $!S^{-1} !v_1 !T !v_2$
are already lifted to $\Ga_{\al-1}$ (and $!Y_1$ is still a subpath of $!S$). 
We have 
$$
  |!X_1|_\al \ge |!X|_\al - |!w_1|_\al - |!w_2|_\al > 0.3 - 8 \ze^2\eta > 0.26
$$
and, consequently, $|!X_1|_{\al-1} > 5.2$. 
By Lemma \ref{lm:stability-tetragon-previous} there is a triple of uniformly close vertices 
~$!b_1$ on $!X$, $!b_2$ on $!Y$ and $!b_3$ on one of the paths $!R_1$, $!T$ or ~$!R_2$.
For $!a_1$ and $!a_2$ we take images of $!b_1$ and $!b_2$ in $\Ga_\al$.
Depending on the location of $!b_3$ we take for $!a_3$ the image of either $\io(!T)$, $!b_3$
or $\tau(!T)$ as shown in Figure \ref{fig:bigon-stability}. 
\begin{figure}[h]
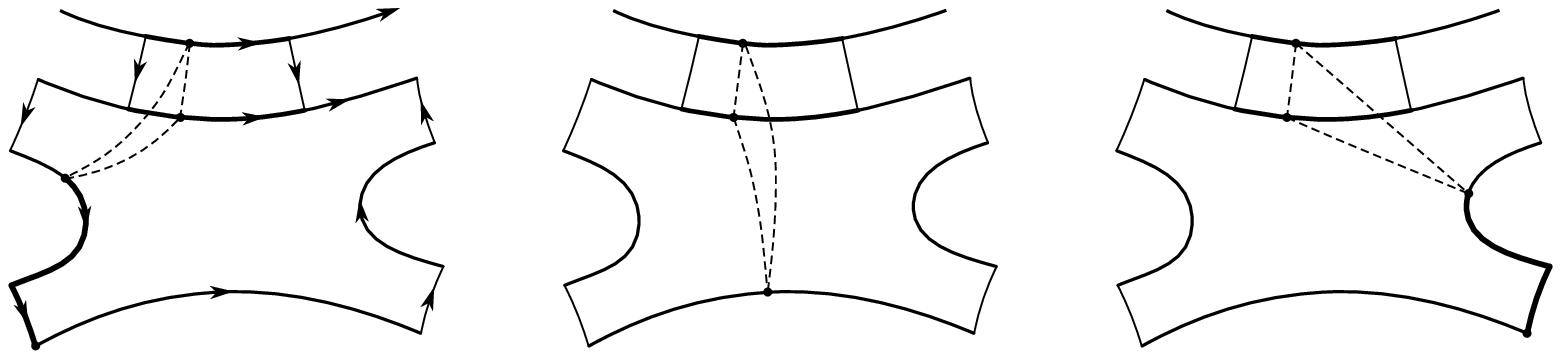
\caption{}  \label{fig:bigon-stability}
\end{figure}
\end{proof}

\begin{lemma} \label{lm:closeness-trigon-1side}
Let $(!X,!Y)$ be a pair of close reduced paths in $\Ga_\al$, and let 
$!S^{-1} * !T_1 * !T_2 *$ be a coarse trigon in $\Ga_\al$ where $!Y$ is an end of $!S$ and ending vertices $\tau(!X)$, $\tau(!Y)$ and $\tau(!T_2)$ are uniformly close.   
Then either
\begin{enumerate}
\item 
there exists a triple $!a_i$ $(i=1,2,3)$ of uniformly close vertices on $!X$, $!Y$ and $!T_1$ respectively,
such that $!a_1$ cuts off a start $!X_1$ of $!X$ with 
$|!X_1|_\al < 1.3$;
\item 
there exists a triple $!a_i$ $(i=1,2,3)$ of uniformly close vertices on $!X$, $!Y$ and $!T_2$ respectively,
such that $!a_1$ cuts off a start $!X_1$ of $!X$ with 
$|!X_1|_\al \le 1.45$.
\end{enumerate}
\end{lemma}

\begin{proof}
We can assume $\al\ge 1$.
We use the same strategy as in the proof of Lemma \ref{lm:closeness-bigon-1side}
and proceed by induction on $|!X| + |!Y| + |!T_2|$.
In view of Lemma \ref{lm:closeness-bigon-1side}, it is enough to prove that 
if $|!X| \ge 1.45$ 
then there exists a triple $!a_i$ of uniformly close vertices 
on $!X$, $!Y$ and some ~$!T_i$ respectively
such that $!a_1$ or $!a_2$ cuts off a proper start of the 
appropriate path $!X$ or $!Y$.

Let $!u_i$ ($i=1,2$) and $!v_j$ ($j=1,2,3$) be bridges of rank $\al$ in $\Ga_\al$ 
such that $!u_1 !X !u_2 !Y^{-1}$ 
is a coarse bigon and  $!S^{-1} !v_1 !T_1 !v_2 !T_2 !v_3$ is a coarse trigon.

{\em Case\/} 1: $\Area_\al(!X^{-1} !u_1 !Y !u_2) = \Area_\al(!S^{-1} !v_1 !T_1 !v_2 !T_2 !v_3) = 0$.
We assume that $!u_2$ and $!v_3$ are defined from the condition that $\tau(!X)$, $\tau(!Y)$ and $\tau(!T_2)$
are uniformly close;
that is, either $!u_2$ and $!v_3$ are bridges of rank $\al-1$ or
have the form $!u_2 = !u_{21} !Q !u_{22}$ and $!v_3 = !v_{31} !P_3 !v_{32}$ where 
$!u_{2i}, !v_{3i}$ are bridges of rank $\al-1$ and $!Q^{\pm1}, !P_3^{\pm1}$ 
are subpaths of a relator loop $!R$ of rank $\al$.
We consider the second case (in the first case the argument is similar).
Let $!v_i = !v_{i1} !P_i !v_{i2}$ 
($i=1,2$) where $!v_{ij}$ is a bridge of rank $\al-1$ and $\lab(!P_i)$ is a piece of rank $\al$.

We can assume that there is no vertex on $!X$ other than $\tau(!X)$ which is close in rank $\al-1$ to a vertex on $!R$ 
(otherwise we can take those for $!a_1$ and $!a_2$ as
in the proof of Lemma \ref{lm:closeness-bigon-1side}).
By Remark \ref{rm:no-active-lift}, we can assume that loops 
$!X^{-1} !u_1 !Y !u_2$ and $!S^{-1} !v_1 !T_1 !v_2 !T_2 !v_3$ can be lifted to $\Ga_{\al-1}$. 
Abusing notations, we assume that they are already in $\Ga_{\al-1}$.
Application of Proposition \ref{pri:4gon-closeness}$_{\al-1}$  shows that 
$!X = !w_1 !X' !w_2$ where $!X'$ is close to a subpath $!Y'$ of $!Y$,
$|!w_1|_\al \le 1 + 4\eta\ze^2$, $|!w_2|_\al \le 4\eta\ze^2$ and hence
$ |!X'|_\al \ge 0.45 - 8\eta\ze^2$.

As in the proof of Lemma \ref{lm:closeness-bigon-1side} the proof slightly differs in 
cases $\al\ge2$ and $\al\ge1$. In the case $\al\ge2$, shortening $!X'$ 
from the end by Proposition \ref{pr:fellow-traveling}$_{\al-1}$ we can assume that $!w_1 !X'$ is a proper start of $!X$, with 
a new bound $|!X'|_{\al} > 0.45 - 8\eta\ze^2 - \ze^2 > 0.41$ which implies 
$|!X'|_{\al-1} > 8.2$. 
If there is a triple of uniformly close vertices on $!X'$, $!Y'$ and some $!P_i$ then we are done.
We assume that no such triple exists.
Let $!S_1$ be a reduced path joining $\io(!T_1)$ and $\tau(!T_2)$ 
(see Figure~\ref{fig:trigon-stability-claim1}). 
\begin{figure}[h]
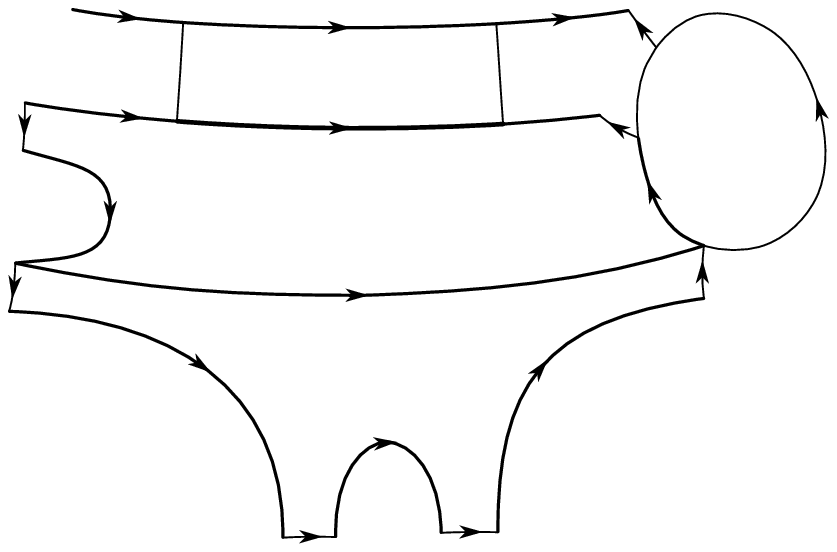
\caption{}  \label{fig:trigon-stability-claim1}
\end{figure}
By Lemma \ref{lm:stability-tetragon-previous} we have $!X' = !z_1 !X'' !z_2$ where 
$!X''$ is close to a subpath of $!S_1$. Moreover, the lemma says that 
there exists a triple 
of uniformly close vertices on $!X'$, $!Y'$ and $!S_1$
and then applying Lemma \ref{lm:closeness-trigon-1side}$_{\al-1}$ 
we may assume that $|z_i|_{\al-1} < 1.45$.
Then 
$$
   |!X''|_{\al-1} \ge |!X'|_{\al-1} - |!z_1|_{\al-1} - |!z_2|_{\al-1} > 5.3.
$$
Another application of Lemma \ref{lm:stability-tetragon-previous}
gives a triple of uniformly close vertices $!b_i$ $(i=1,2,3)$ where $!b_1$ lies on $!X'$, 
$!b_2$ lies on $!Y'$ and $!b_3$ lies either on $!T_1$ or on $!T_2$. For $!a_i$ we take the images 
of $!b_i$ in $\Ga_\al$.

In the case $\al=1$ the argument is similar 
(see Case 1a in the proof of Lemma  \ref{lm:closeness-bigon-1side}) with  no need
for a lower bound on $|!X''|_{\al-1}$ for application of Lemma  \ref{lm:stability-tetragon-previous}.

{\em Case\/} 2: $r = \Area_\al(!S^{-1} !v_1 !T_1 !v_2 !T_2 !v_3) > 0$.
Let $!L$ be an active relator loop for $!S^{-1} !v_1 !T_1 !v_2 !T_2 !v_3$
and $!K_i$ ($i=1,2$ or $i=1,2,3$) be the associated active 
fragments of rank ~$\al$ occurring in $!S$, $!T_1$ or $!T_2$. 
If some $!K_i$ occurs in $!T_1$ and some $!K_j$ occur in $!T_2$ 
then we can shorten $!T_1$ and $!T_2$ decreasing $r$ as described in 
\ref{ss:active-loops-trigon}. 
A similar inductive argument works in the case when some ~$!K_i$ occurs in 
$!S$ and is not contained in $!Y$. 
Thus we may assume that there are only $!K_1$ and $!K_2$,, 
$!K_1$ is contained in $!Y$ and $!K_2$ occurs in $!T_1$ or $!T_2$. 
By Proposition \ref{pr:trigon-active-fragments},
$\muf(!K_i) \ge 3\la - 1.1\om$.
The rest of the argument is the same as in the Case 2 of
the proof of Lemma \ref{lm:closeness-bigon-1side}.

{\em Case\/} 3: $\Area_\al(!X^{-1} !u_1 !Y !u_2) > 0$.
Let $!K$ and $!M$ be active fragments of rank $\al$ in 
$!X$ and in $!Y$ respectively such that $!K \sim !M^{-1}$. 
We take $!a_1 := \io(!K)$, $!a_2 := \io(!M)$ and
define $!a_3$ according to the following cases: 
\begin{itemize}
\item 
If $!M$ is not independent of $!v_1$ then $!a_3 := \io(!T_1)$;
\item
If $!M$ is not independent of $!v_2$ then $!a_3 := \tau(!T_1)$; 
\item
If $!M$ is not independent of $!v_3$ then $!a_3 := \tau(!T_2)$; 
\item
Otherwise by Proposition \ref{pr:fragment-stability-trigon} 
applied to $!M$ there exists
a fragment $!N$ or rank $\al$ in $!T_1$ or $!T_2$ such that $!M \sim !N^{\pm1}$. 
Then $!a_3 := \io(!N)$.
\end{itemize}
\end{proof}

\begin{proposition}[closeness transition in trigon] 
\label{pr:stability-trigon}
Let $(!X,!Y)$ be a pair of close reduced paths in $\Ga_\al$, and let 
$!S^{-1} * !T_1 * !T_2 *$ be a coarse trigon in $\Ga_\al$ where $!Y$ is a subpath of $!S$.
Assume that $|!X|_\al \ge 2.45$.
Then $!X$ can be represented as in one of the following three cases:
\begin{enumerate}
\item 
$!X = !z_1 !X_1 !z_2$ where $!X_1$ is close to a subpath $!W_1$ of $!T_1$ and $|!z_1|_\al < 1.3$, 
$|!z_2|_\al < 1.45$.
\item
$!X = !z_1 !X_2 !z_2$ where $!X_2$ is close to a subpath $!W_2$ of $!T_2$ and $|!z_1|_\al < 1.45$, 
$|!z_2|_\al < 1.3$.
\item
$!X = !z_1 !X_1 !z_3 !X_2 !z_2$ where $!X_i$ is close to a subpath $!W_i$ of $!T_i$ ($i=1,2$), 
$|!z_1|_\al,|!z_2|_\al < 1.3$ and $|!z_3|_\al < 0.4$.
\end{enumerate}
Moreover, we can assume that there exists a subpath $!Y'$ of $!Y$ such that 
triples $(\io(!X_p), \io(!Y'), \io(!W_p))$ and $(\tau(!X_q), \tau(!Y'), \tau(!W_q))$ are uniformly close
where $p$ and $q$ are the smallest and the greatest indices of $!X_i$ in (i)--(iii),
i.e.\ $p=q=1$ in (i), $p=q=2$ in (ii) and $p=1$, $q=2$ in (iii).
\end{proposition}

\begin{proof}
Let $!u_i$ ($i=1,2$) and $!v_j$ ($j=1,2,3$) be bridges of rank $\al$ such that $!u_1 !X !u_2 !Y^{-1}$ 
is a coarse bigon and  $!S^{-1} !v_1 !T_1 !v_2 !T_2 !v_3$ is a coarse trigon.
In view of Lemmas \ref{lm:closeness-bigon-1side} and \ref{lm:closeness-trigon-1side},
finding a triple $!a_i$ $(i=1,2,3)$
of uniformly close vertices on $!X$, $!Y$ and some $!T_i$
implies the conclusion of the proposition except the bound 
$|!z_3|_\al < 0.4$ in (iii). The latter follows from
Proposition ~\ref{pri:3gon-closeness}.
An easy analysis as in Cases 2 and 3 of the proof of 
Lemma \ref{lm:closeness-trigon-1side} shows how to find the vertices $!a_i$
in the case when 
$\Area_\al(!X^{-1} !u_1 !Y !u_2) > 0$ or $\Area_\al(!S^{-1} !v_1 !T !v_2 !T_2 !v_3)> 0$.
It remains to consider the case when 
$\Area_\al(!X^{-1} !u_1 !Y !u_2) = \Area_\al(!S^{-1} !v_1 !T !v_2 !T_2 !v_3) = 0$.
Let $!v_i = !w_{i1} !R_i !w_{i2}$ $(i=1,2,3)$ where $\lab(!w_{ij}) \in \cH_{\al-1}$ and the label of $!R_i$
is a piece of rank ~$\al$. 
By Proposition \ref{pr:no-active-loops} 
we have $!X = !w_1 !X_1 !w_2$ where endpoints of $!X_1$ and a subpath $!Y_1$ of $!Y$ 
can be joined by bridges $!u_1'$ and 
$!u_2'$ of rank $\al-1$ and the loop $!X_1 !u_1' !Y_1^{-1} !u_2'^{-1}$ can be lifted to ~$\Ga_{\al-1}$ and 
$|!w_i|_\al \le 1 + 4\ze^2\eta$ $(i=1,2)$. 
Without changing notations, we assume that loops $!X_1^{-1} !u_1' !Y_1 !u_2'$ and $!S^{-1} !v_1 !T !v_2$
are already in $\Ga_{\al-1}$ (and $!Y_1$ is still a subpath of $!S$). 
We have 
$$
  |!X_1|_\al \ge |!X|_\al - |!w_1|_\al - |!w_2|_\al > 0.41
$$
and, consequently, $|!X_1|_{\al-1} > 8.2$. 
Then we find $!a_i$ applying 
Lemmas \ref{lm:closeness-trigon-1side}$_{\al-1}$ and \ref{lm:stability-tetragon-previous} as in the proof
of Lemma \ref{lm:closeness-trigon-1side}.
\end{proof}

\begin{proposition}[closeness transition in conjugacy relations] 
\label{pr:stability-cyclic-monogon}
Let $S$ be a word cyclically reduced in ~$G_\al$.
Assume that $S$ is conjugate in ~$G_\al$ to a word $Tv$ 
where $T \in \cR_\al$ and $v \in \cH_\al$.
Let $\bar{!S} = \prod_{i \in \Z} !S_i$ and $\prod_{i \in \Z} !T_i !v_i$ be lines in $\Ga_\al$ 
representing the conjugacy relation.


Assume that a reduced path $!X$ in $\Ga_\al$ is close to a subpath $!Y$ of $\bar{!S}$ with $|!Y| \le |S|$.
Let $|!X|_\al \ge 2.45$.
Then either:
\begin{enumerate}
\item 
$!X$ can be represented as 
$!X = !z_1 !X_1 !z_2$ where $!X_1$ is close to a subpath $!W_1$ of $!T_i$ for some ~$i$ and 
$|!z_1|_\al,|!z_2|_\al < 1.45$.
\item
$!X$ can be represented as 
$!X = !z_1 !X_1 !z_3 !X_2 !z_2$ where for some $i$, $!X_1$ is close to a subpath $!W_1$ of $!T_i$,
$!X_2$ is close to a subpath $!W_2$ of $!T_{i+1}$, $|!z_1|_\al,|!z_2|_\al < 1.3$
and $|!z_3|_\al \le 0.4$. 
\end{enumerate}
Moreover, we can assume that there exists a subpath $!Y'$ of $!Y$ such that 
triples $(\io(!X_1), \io(!Y'), \io(!W_1))$ and $(\tau(!X_q), \tau(!Y'), \tau(!W_q))$ are uniformly close
where $q=1$ in (i) and $q=2$ in (ii).
\end{proposition}

\begin{proof}
It is enough to find a uniformly close triple of vertices $!a_i$ $(i=1,2,3)$ on $!X$, 
$!Y$ and some ~$!T_i$
and then use Lemmas \ref{lm:closeness-trigon-1side} or \ref{lm:closeness-bigon-1side}.
Let $!X^{-1} !u_1 !Y !u_2$ be a coarse bigon where $!u_1$ and $!u_2$ are bridges of rank $\al$.
If $\Area_\al(!X^{-1} !u_1 !Y !u_2) > 0$ then we reach the goal 
using Proposition ~\ref{pr:fragment-stability-cyclic1} and Lemma \ref{lmi:compatible-uniformly-close}.
Assume that $\Area_\al(!X^{-1} !u_1 !Y !u_2) = 0$.

Let $\De$ be an annular diagram of rank $\al$ with boundary loops $\hat{!S}^{-1}$ and $\hat{!T} \hat{!v}$
representing the conjugacy relation. Let $\ti{\De}$ be the universal cover of $\De$ and
$\phi: \ti{\De}^{(1)} \to \Ga_\al$ the combinatorially continuous 
map sending lifts $\ti{!S}_i$, $\ti{!T}_i$ and $\ti{!v}_i$
to $!S_i$, $!T_i$ and $!v_i$ respectively.
We assume that $\De$ is reduced and has a tight set of contiguity subdiagrams.
Let $r$ be the number of cells of rank ~$\al$ of $\De$.

Assume that $r > 0$ and let $!D$ be a cell of rank $\al$ of $\De$.
By Proposition \ref{pri:annular-single-layer1} and Lemma \ref{lmi:cell-regularity-side}, $!D$ has 
two or three contiguity subdiagrams $\Pi_i \in \cT$ to sides of ~$\De$, 
at most two to $\hat{!T}$ and at most one to ~$\hat{!S}^{-1}$. 
If there are two contiguity subdiagrams $\Pi_i$ $(i=1,2)$ of $!D$ 
to $\hat{!T}$ then we consider a new annular diagram $\De'$ obtained by
cutting off $!D \cup \Pi_1 \cup \Pi_2$ and the remaining simply connected component from $\De$, 
and new words $T'$ and $v'$ where $T'$ is 
a subword of $T$. In this case, the statement follows by induction on $r$.

We can assume now that $!D$ has 
one contiguity subdiagram to $\hat{!S}^{-1}$ and one to $\hat{!T}$.
Let $\ti{!D}_i$ $(i \in \Z)$ be the lifts of $!D$ in $\ti\De$.
With an appropriate numeration of $\ti{!D}_i$'s, each 
relation loop $\phi(\de\ti{!D}_i)$
is a base loop for a fragment $!K_i$ in $\bar{!S}^{-1}$ 
and a fragment $!M_i$ in $!T_i$.
By Proposition \ref{pri:cyclic-monogon-cell},
$$
  \muf(!K_i^{-1}) + \muf(!M_i) > 1 - 4\la - 2.2\om.
$$
Since $T$ is reduced in $G_\al$, we have $\muf(!M_i) \le \rho$ and hence
$$
  \muf(!K_i^{-1}) > 5 \la - 2.2\om.
$$
If none of $!K_i^{-1}$'s is contained in $!Y$ then we can apply Proposition \ref{pr:stability-trigon}.
Otherwise we use an argument similar to one in Case 2 of the proof 
of Lemma \ref{lm:closeness-bigon-1side}.

Now assume that $\De$ has no cells of rank $\al$.
Without changing notations, we assume that parallel lines 
$\bar{!S} = \prod_{i \in \Z} !S_i$, $\prod_{i \in \Z} !T_i !v_i$
and paths $!X$ and $!Y$ are lifted to $\Ga_{\al-1}$ so that $!Y$ is still a subpath of $\bar{!S}$.
Let $v \greq w_1 R w_2$
where $w_i \in \cH_{\al-1}$ and $R$ is a piece of rank $\al$. 
We represent $!v_i$ accordingly as $!v_i = !w_1^{(i)} !R_i !w_2^{(i)}$.
Let $Z$ be a word reduced in $G_{\al-1}$ such that $Z = T w_1 R$ and let $!Z_i$ $(i \in \Z)$
be appropriate paths in ~$\Ga_{\al-1}$ with $\lab(!Z_i) \greq Z$
(Figure \ref{fig:stability-cyclic-monogon-claim1}).   
\begin{figure}[h]
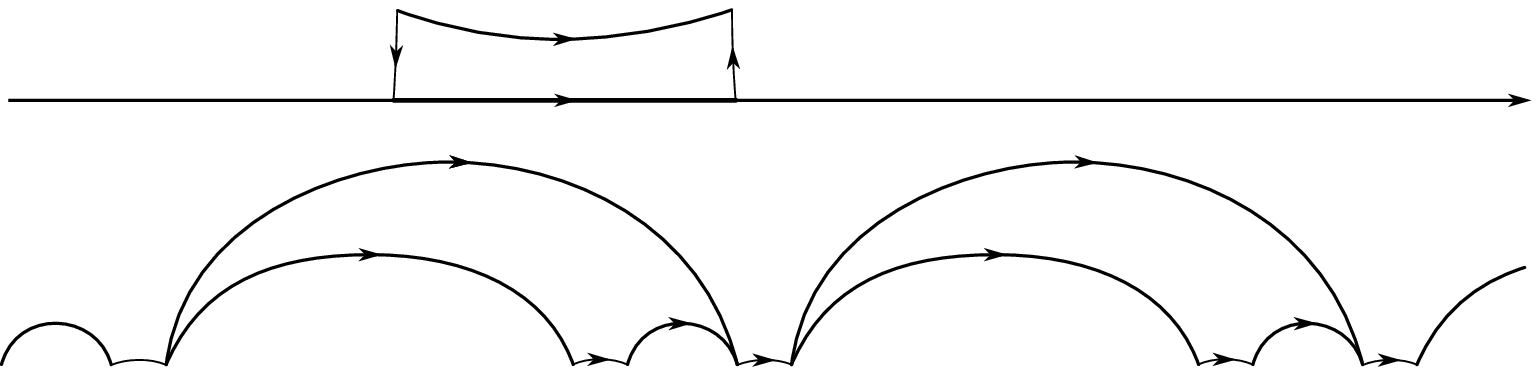
\caption{}  \label{fig:stability-cyclic-monogon-claim1}
\end{figure}
Since $|!X|_\al \ge 2.45$ we have $|!X|_{\al-1} \ge \frac1\ze |!X|_\al \ge 49$.
By Proposition \ref{pr:stability-cyclic-monogon}$_{\al-1}$, 
a subpath $!X'$ of $!X$ with $|!X'|_{\al-1} > 23$ is close to a subpath of some $!Z_i$.
Then using Proposition \ref{pr:stability-trigon}$_{\al-1}$ we find 
a triple $!b_i$ of uniformly close vertices on $!X'$, $!Y$ and $!T_i$ or ~$!R_i$
respectively. If $!b_3$ lies on $!T_i$ then for the desired $!a_i$
we take images of $!b_i$ in $\Ga_\al$. If $!b_3$ lies on ~$!R_i$ then 
for $!a_i$ $(i=1,2,3)$ we take images of $!b_1$, $!b_2$ and $\tau(!T_i)$,
respectively.
\end{proof}

\begin{lemma} \label{lm:dimmed-fragment}
Let $1 \le \be \le \al$ and 
$!X$ be a reduced path in $\Ga_\al$. Let $!K_1$ and $!K_2$ be fragments 
of rank $\be$ in $!X$ such that $\muf(!K_i) \ge \la + 2.6\om$ $(i=1,2)$,
$!K_1 < !K_2$ and $!K_1 \not\sim !K_2$.
If a bridge of rank $\be$ starts or ends at $\io(!X)$
then $!K_2$ is independent of $!u$. Similarly, if a 
bridge of rank ~$\be$ starts or ends at $\tau(!X)$
then $!K_1$ is independent of $!u$.
\end{lemma}

\begin{proof}
We consider the case when $\io(!u) = \io(!X)$ (all other cases are similar).
Assume that $!K_2$ is not independent of $!u$.
By Definition \ref{def:interference}, $!u = !v !S !w$ where $!S$ occurs in a 
relation loop $!R$ of rank $\be$, $!v$ and $!w$ are bridges of rank $\be-1$ and $!R^{\pm1}$ is the base relation loop for $!K$. 
Let $\ti{!R}$ and $\ti{!X}$ be lifts of $!R$ and $!X$ in $\Ga_{\be-1}$ so that 
$\ti{!R}^{\pm1}$ is the base axis for $\ti{!K}_2$. 
Lemma \ref{lm:monogon-graph-version} implies that the starting vertex of $\ti{!X}$ is close to a vertex on $\ti{!R}$.
Then using Proposition \ref{pr:closeness-order}$_{\al-1}$ we conclude that
the starting segment 
$\ti{!X}_1 \ti{!K}_2$ of $\ti{!X}$ is a fragment of rank $\al$ 
with base axis $\ti{!R}$.
Since $!K_1$ is contained in $\ti{!X}_1 \ti{!K}_2$, 
Proposition \ref{pr:inclusion-compatibility} gives $!K_1 \sim !K_2$, 
a contradiction. 
\end{proof}

\begin{proposition}[closeness preserves order] \label{pr:closeness-order}
Let $!X_1 !X_2$ and $!Y_1 !Y_2$ be reduced paths in $\Ga_\al$ such that endpoints of $!X_i$ and $!Y_i$
are close in the order as 
in Figure \ref{fig:closeness-order}.  
Then $|!X_1|_\al, |!Y_2|_\al < 5.7$.
\end{proposition}

\begin{figure}[h]
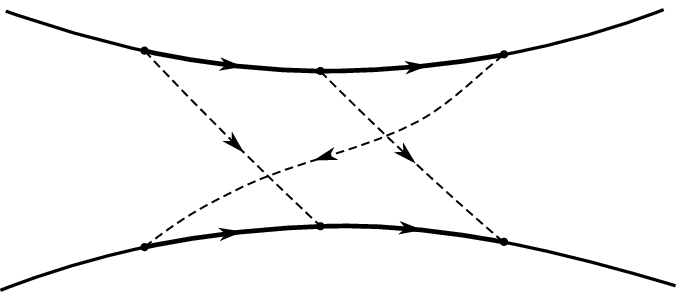
\caption{}  \label{fig:closeness-order}
\end{figure}

\begin{proof}
We can assume that $\al\ge1$.
Due to symmetry, it is enough to show that $|!X_1|_\al  < 5.7$. 
Denote $!u_i$ $(i=1,2,3)$ bridges of rank $\al$ joining endpoints of $!X_i$ and $!Y_i$ as 
 shown in Figure ~\ref{fig:closeness-order}.

\begin{claim} \label{cl:closeness-order-claim1}
$\Area_\al(!X_1^{-1} !u_1 !Y_2 !u_2^{-1}) \le 1$.
\end{claim}

\begin{proof}[Proof of Claim 1]
Assume that $\Area_\al (!X_1^{-1} !u_1 !Y_2 !u_2^{-1}) \ge 2$. Let $!K_i$ and $!M_i$ $(i=1,2)$ be  
active fragments of rank $\al$ in $!X_1$ and $!Y_2$, respectively, 
such that $!K_1 < !K_2$ and $!K_i \sim !M_i^{-1}$.
By Proposition \ref{pri:active-non-compatible} and Lemma \ref{lm:dimmed-fragment}, $!K_2$ is independent of ~$!u_1$. 
Similarly, $!M_2$ and hence ~$!K_2$, are independent of $!u_3$.
By Propositions \ref{pr:active-large} and \ref{pr:fragment-nonactive-bigon} 
applied to $(!X_1 !X_2)^{-1} !u_1 !Y_1^{-1} !u_3^{-1}$,
there is a fragment $!N$ of rank $\al$ in $!Y_1$ such that $!N \sim !K_2^{\pm 1}$ and 
$\muf(!N) \ge 5\la - 4.9\om$.
We obtain a contradiction with Corollary \ref{co:compatibility-same-rank}(ii),(iii).
\end{proof}

\begin{claim} \label{cl:closeness-order-claim2}
If $\Area_\al(!X_1^{-1} !u_1 !Y_2 !u_2^{-1}) = 0$ and $\lab(!u_1), \lab(!u_2) \in \cH_{\al-1}$ then 
$|!X_1|_\al < 1 + 6.1\ze$.
\end{claim}

\begin{proof}[Proof of Claim 2]
If $r = \Area_\al(!X_2 !u_3 !Y_1 !Y_2 !u_2^{-1}) >0$ then we can 
reduce the statement to the case of a smaller $r$ as explained in \ref{ss:active-fragments}.
So we can assume that $\Area_\al(!X_2 !u_3 !Y_1 !Y_2 !u_2^{-1}) =0$.
Then loops $!X_1^{-1} !u_1 !Y_2 !u_2^{-1}$ and $!X_2 !u_3 !Y_1 !Y_2 !u_2^{-1}$ can be
lifted to $\Ga_{\al-1}$ (up to possible switching of ~$!u_3$).
To simplify notations, we assume that these loops are already in $\Ga_{\al-1}$.
Let $!u_3 = !v_1 !Q !v_2$ where $\lab(!v_i) \in \cH_{\al-1}$ and $\lab(!Q)$ is
a piece of rank $\al$.
We obtain a coarse trigon in $\Ga_{\al-1}$ with sides $!X_1 !X_2$, $!Q$ and $!Y_1$,
see Figure \ref{fig:closeness-order-prev}.
Applying Propositions \ref{pri:3gon-closeness}$_{\al-1}$ and \ref{pr:closeness-order}$_{\al-1}$ we obtain
$$
  |!X_1 !X_2|_\al < 1 + 4\ze^2 \eta + 5.7\ze < 1 + 6.1\ze.
$$
\end{proof}  

\begin{figure}[h]
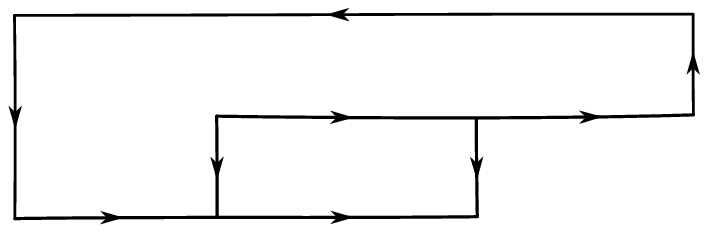
\caption{}  \label{fig:closeness-order-prev}
\end{figure}

{\em The rest of the proof:}
If $\Area_\al(!X_1^{-1} !u_1 !Y_2 !u_2^{-1}) = 0$ then the statement follows from Claim 2 and
Proposition \ref{pr:no-active-loops}.
By Claim 1,
it remains to consider the case $\Area_\al(!X_1^{-1} !u_1 !Y_2 !u_2^{-1}) = 1$.
Then $!X_1$ can be represented as $!R_1 !S_1 !R_2 !S_2 !R_3$ (see Figure~\ref{fig:closeness-order-active})
where each $!R_i$ is a fragment of rank ~$\al$ and by Claim 2 
and Proposition \ref{pri:4gon-closeness}$_{\al-1}$ each $!S_i$ satisfies
$|!S_i|_\al < 1 + 6.1\ze + 8\ze^2\eta$. 
We obtain
$$
  |!X_1|_\al < 3 + 2(1+6.1\ze + 8\ze^2\eta) < 5.7.
$$
The proof is completed.  
\begin{figure}[h]
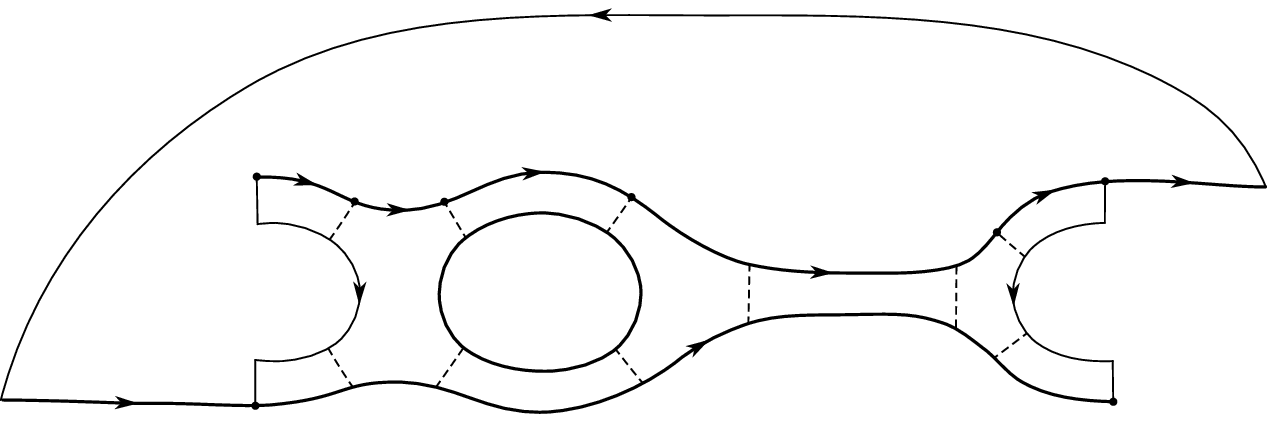
\caption{}  \label{fig:closeness-order-active}
\end{figure}
\end{proof}


In the end of the section we formulate several statements about stability of fragments in a 
more general setup when fragments have arbitrary rank $\be$ in the interval $0 \le \be \le \al$.

\begin{proposition} \label{pr:closeness-stability-beta} 
Let $!S$ and $!T$ be close reduced paths in $\Ga_\al$.
Let $0 \le \be < \al$ and let $!X$ and ~$!Y$ be close 
in rank $\be$ reduced paths in $\Ga_\al$ 
such that $!Y$ is a subpath of $!S$.
Assume that $|!X|_\al \ge 2.3$ 
and $!Y$ contains no fragments $!K$ of rank $\ga$ 
with $\be < \ga \le \al$ and $\muf(!K) \ge \xi_0$.
Then $!X$ can be represented as $!X = !w_1 !X' !w_2$ where 
$!X'$ is close in rank $\be$ to a subpath of ~$!T$ and $|!w_i|_\al < 1.2$ $(i=1,2)$.
\end{proposition}

\begin{proof}
Let $!S^{-1} !u_1 !T !u_2$ and $!X^{-1} !v_1 !Y !v_2$ be corresponding coarse bigons. 
If $\Area_\al(!S^{-1} !u_1 !T !u_2) > 0$ 
then by the argument from \ref{ss:active-loop-induction} we reduce the statement to a new pair $(!S,!T)$
and a coarse bigon $!S^{-1} !u_1 !T !u_2$ with a smaller value of $\Area_\al(!S^{-1} !u_1 !T !u_2)$.
Hence we can assume that $\Area_\al(!S^{-1} !u_1 !T !u_2) = 0$. 
Without changing notations, we assume that both loops $!S^{-1} !u_1 !T !u_2$ and 
$!X^{-1} !v_1 !Y !v_2$ are in $\Ga_{\al-1}$.
Let $!u_i = !u_{i1} !P_i !u_{i2}$ where $\lab(!u_{ij}) \in \cH_{\al-1}$
and $\lab(!P_i)$ is a piece of rank $\al$.
Observe that if a subpath $!X'$ is close to a subpath of $!P_1$ or $!P_2$
then $|!X'|_\al \le 1$. 
Since $|!X|_\al \ge 2.3$ applying Lemma \ref{lm:stability-tetragon-previous} 
we find a subpath of $!X$ close to a subpath of $!T$. 
We consider the case when 
$!X = !z_0 !X_1 !z_1 !X_2 !z_2 !X_3 !z_3$
where $!X_i$ $(i=1,2,3)$ are close to subpaths of $!P_1$, $!T$ and ~$!P_2$ respectively 
(the other cases from Lemma \ref{lm:stability-tetragon-previous}
give a better lower bound on $|!X_2|_{\al}$).
By Lemma \ref{lm:closeness-bigon-1side} we can assume that $|!z_0|_{\al-1}, |!z_3|_{\al-1} < 1.3$
and by Proposition \ref{pri:3gon-closeness}$_{\al-1}$ we can assume that 
$|!z_1|_{\al-1}, |!z_2|_{\al-1} < 0.4$. We have $|!X_1|_\al, |!X_3|_\al \le 1$, 
so $|!X_2|_{\al} > 2.3 - 2 - 3\ze = 0.15$ and hence
$|!X_2|_{\al-1} > 3$. 
Then by Corollary \ref{co:no-active-fragments-iterated}$_{\al-1}$
we have $!X_2 = !t_1 !X' !t_2$ where $!X'$ is close in rank $\be$ to a subpath of $!T$
and $|!t_i|_{\al-1} < 1.03$. We have $!X = !z_1 !X_1 !z_2 !t_1 !X' !t_2 !z_3 !X_3 !z_4$
where $|!z_1 !X_1 !z_2 !t_1|_\al < 1 + 2.73\ze < 1.2$ and a similar bound holds for
$|!t_2 !z_3 !X_3 !z_4|_\al$.
\end{proof}

\begin{proposition} \label{pr:fragment-correspondence-between-fragments}
Let $!X$ and $!Y$ be reduced paths in $\Ga_\al$. 
Let $1 \le \be \le \al$ and assume that either $!X$ or $!Y$ contains 
no fragments $!N$ of rank $\ga$ with $\be < \ga \le \al$
and $\muf(!N) \ge \xi_0$.

Let $!K_i$ $(i=1,2)$ be fragments of rank $\be$ in $!X$ such that 
$!K_1 \not\sim !K_2$ and $!K_1 < !K_2$. Assume that at least one 
of the following conditions holds:
\begin{itemize}
\item[(*)] 
there exist fragments $!M_i$ $(i=1,2)$ of rank $\be$
in $!Y$ such that $\muf(!M_i) \ge \la + 2.7\om$, $!K_i \sim !M_i^{\pm1}$ and $!M_1 < !M_2$; or
\item[(**)]
$!X$ and $!Y$ are close in rank $\be$.
\end{itemize}
Then the following is true:

\begin{enumerate}
\item 
\label{pri:fragment-correspondence-between-fragments}
Let $!N$ be a fragment of rank $\be$ 
in $!X$ with $\muf(!N) \ge 2\la + 9.1\om$
such that $!K_1 < !N < !K_2$ and $!N \not\sim !K_i$ for $i=1,2$.
Then there exists a fragment $!N'$ 
of rank $\be$ in $!Y$ such that $!N' \sim !N^{\pm1}$, 
$!M_1 < !N' < !M_2$ in case (*)  and 
\begin{equation} \label{eq:fragment-correspondence-between-fragments}
  \muf(!N') \ge \min\set{\muf(!N_i) - 2 \la - 3.4\om, \ \xi_0}
\end{equation}
In case (*), if $!M_1$ and $!M_2$ are disjoint then we can assume that $!M_1 \ll !N' \ll !M_2$.
This is the case (that is, $!M_1$ and $!M_2$ are necessarily disjoint) if $\muf(!N) \ge 4\la + 9\om$.
\item 
\label{pri:fragment-ordering-between-fragments}
Assume that $\muf(!K_i) \ge 2\la + 9.1\om$ and  in case (*),
$\muf(!M_i) \ge 2\la + 9.1\om$.
Let $!K_i'$ $(i=1,2)$ be a pair of another fragments of rank $\be$ in $!X$ and $!M_i'$ $(i=1,2)$ a pair 
of another fragments of rank $\be$ in ~$!Y$ such that $\muf(!K_i'), \muf(!M_i') \ge 2\la+9.1\om$, 
$!K_i' \sim !M_i'^{\pm1}$ $(i=1,2)$ and $!K_1' \not\sim !K_2'$. Then $!K_1' < !K_2'$ if and only if $!M_1' < !M_2'$.
\end{enumerate}
Furthermore, the statement of the proposition is true also in the case $\be=0$ if we drop all conditions
of the form $\muf(\cdot) \ge \dots$ for fragments of rank $\be$.
\end{proposition}

\begin{proof}
If $\be=0$ then by  Proposition \ref{pr:lifting-bigon}
we have 
$!M_i = !K_i$ $(i=1,2)$, $!M_1 \cup !M_2 = !K_1\cup !K_2$ in case (*) and $!X = !Y$
in case (**). Then the statement is trivial. 
We assume that $\be\ge 1$.

(i): Assume that (*) holds. 
First assume that $!M_1$ and $!M_2$ are disjoint. 
Let $!X_1 = !K_1 \cup !K_2$ and $!Y_1$ be the subpath of $!Y$ between $!M_1$ and $!M_2$, i.e.\ 
$!Y = * !M_1 !Y_1 !M_2 *$. 
By Lemma \ref{lmi:compatible-close} and Proposition \ref{pr:lifting-bigon}
we have a loop $!X_1^{-1} !u !Y_1 !v$ that can be lifted to $\Ga_\be$
where $!u$ and $!v$ are bridges of rank $\be$.
Up to change of notation, we assume that $!X_1^{-1} !u !Y_1 !v$ is already in $\Ga_\be$.
Again by Lemma \ref{lmi:compatible-close}$_\be$, $!N$ is independent of ~$!u$ and $!v$.
By Proposition \ref{pr:fragment-stability-bigon}$_\be$, there exists $!N'$ in $!Y_1$
satisfying \eqref{eq:fragment-correspondence-between-fragments} such that $!N' \sim !N^{\pm1}$,
i.e.\  we have $!M_1 \ll !N' \ll !M_2$ as required.

Assume that $!M_1$ and $!M_2$ have a nonempty intersection. 
By Proposition \ref{pr:small-overlapping}$_\be$ there exist fragments $!M_1'$ and $!M_2'$ of
rank $\be$ such that $!M_i' \sim !M_i$, $!M_1'$ is a start of $!M_1$ disjoint from $!M_2$ and 
$!M_2'$ is an end of $!M_2$ disjoint from $!M_1$. Let $!Y_2 = !M_1 \cup !M_2$.
Using the argument above with $!Y_2$ instead of $!Y_1$ and $!M_1'$ instead of $!M_1$
we find $!N_1$ in $!Y_2$ disjoint from $!M_2$ 
such that $\muf(!N_1) > 5.7\om$ and  $!N_1 \sim !N^{\pm1}$.
Similarly, using $!Y_2$ instead of $!Y_1$ and $!M_2'$ instead of $!M_2$ we find $!N_2$ in $!Y_2$
disjoint from $!M_1$ 
such that $\muf(!N_2) > 5.7\om$ and  $!N_2 \sim !N^{\pm1}$.
Then we can take $!N' = !N_1 \cup !N_2$ by 
Corollary \ref{coi:fragments-union-alpha},  \itemref{coi:no-inverse-compatibility-alpha}.

If $\muf(!N) \ge 4\la + 9\om$ then $\muf(!N') > 2\la + 5.6\om$ and using 
Propositions \ref{pr:dividing-fragment}$_\be$ and \ref{pr:inclusion-compatibility}$_\be$ we
conclude that $!M_1$ and $!M_2$ cannot cover $!N'$ together, i.e.\ $!M_1 \ll !M_2$.

In case  (**) we already have a loop $!X^{-1} !u !Y !v$
with bridges $!u$ and $!v$ of rank $\be$. We lift it to ~$\Ga_\be$ and then apply 
Lemma \ref{lm:dimmed-fragment}$_\be$ to see that 
the lift of $!N$ is independent of the lifts of ~$!u$ and ~$!v$. 
Then application of  Proposition \ref{pr:fragment-stability-bigon}$_\be$ gives the 
required $!N'$.

(ii): An easy analysis with a help of 
Propositions \ref{coi:compatibility-order-alpha} 
and \ref{pr:inclusion-compatibility}$_\be$
shows that it is enough to prove the following: 
{\em Let $!X$ and $!Y$ be reduced paths in $\Ga_\al$.
Let $!K_i$ $(i=1,2,3)$ be fragments of rank $\be$ in $!X$, $!M_i$ $(i=1,2,3)$ be fragments 
of rank $\be$ in $!Y$, $\muf(!K_i), \muf(!M_i) \ge \la + 9.1\om$,
$!K_i \sim !M_i^{\pm1}$ for all $i$ and $!K_i \not\sim !K_j$ for $i \ne j$.
If $!K_1 < !K_2 < !K_3$ and $!M_1 < !M_3$ then $!M_1 < !M_2 < !M_3$.}

Assume that this is not the case, that is, we have $!K_1 < !K_2 < !K_3$, $!M_1 < !M_3$ and 
either $!M_2 < !M_1$ or $!M_3 < !M_2$.
By (i), there exists a fragment $!N$ of rank $\al$ in $!Y$ such that 
$!K_2 \sim !N^{\pm1}$ and $!M_1 < !N < !M_3$. Then by Propositions \ref{coi:fragments-union-alpha} and
\ref{pr:inclusion-compatibility}$_\be$ we obtain $!M_1 \sim !N$ or $!M_3 \sim !N$, a contradiction.
\end{proof}

\begin{proposition} \label{pr:fragment-correspondence-cyclic-beta}
Let $X$ and $Y$ be words strongly cyclically reduced in $G_\al$,
representing conjugate elements of $G_\al$.
Let $\bar{!X}$ and $\bar{!Y}$ be lines in $\Ga_\al$ representing the conjugacy relation.
Let $1 \le \be \le \al$.
Assume that at least one of the words ~$X$ or ~$Y$ has the property that 
no its cyclic shift contains a fragment $K$ of rank $\ga$ with $\muf(K) > \xi_0$ 
and $\be < \ga\le \al$.
Let $\bar{!X} = \dots !X_{-1} !X_0 !X_1 \dots$ and $\bar{!Y} = \dots !Y_{-1} !Y_0 !Y_1 \dots$ 
be lines in $\Ga_\al$ representing the conjugacy relation. 
\begin{enumerate}
\item 
\label{pri:fragment-correspondence-cyclic-beta}
Then for any fragment $!K$ of rank $\be$ in $\bar{!X}$ with $\muf(!K) \ge 2\la+9.1\om$ 
there exists a fragment $!M$ of rank $\be$
in $\bar{!Y}$ such that $!M \sim !K^{\pm1}$ and 
$$
  \muf(!M) \ge \min\set{\muf(!K) - 2 \la - 3.4\om, \ \xi_0}
$$
\item
\label{pri:fragment-correspondence-cyclic-ordering}
If $X$ and $Y$ are strongly cyclically reduced in $G_\al$ then 
the correspondence between fragments of rank $\be$ in $\bar{!X}$ and in $\bar{!Y}$ 
preserves the ordering in the following sense: if $!K_i$ $(i=1,2)$ are fragments of rank $\be$ in $\bar{!X}$,
$!M_i$ $(i=1,2)$ are fragments of rank $\be$ in ~$\bar{!Y}$, $\muf(!K_i), \muf(!M_i) \ge 2\la+9.1\om$,
$!K_i \sim !M_i^{\pm1}$ $(i=1,2)$ and $!K_1 \not\sim !K_2$. Then $!K_1 < !K_2$ if and only if $!M_1 < !M_2$.
\end{enumerate}
Furthermore, the statement of the proposition is true also in the case $\be=0$ if we drop all conditions
of the form $\muf(\cdot) \ge \dots$ for fragments of rank $\be$.
\end{proposition}

\begin{proof}
By Proposition ~\ref{pr:no-active-fragments-cyclic-iterated} every subpath of $\bar{!X}$ can be extended
to be close in rank $\be$ to a subpath of $\bar{!Y}$.
Then (i) follows from Proposition \ref{pri:fragment-in-periodic} and 
Proposition \ref{pri:fragment-correspondence-between-fragments}
with $!K_1 = s_{X,\bar{!X}}^{-1} !K$ and $!K_2 = s_{X,\bar{!X}} !K$.
Statement (ii) follows by Proposition \ref{pri:fragment-ordering-between-fragments}.
In the case $\be=0$ the statement becomes trivial after application of 
Proposition ~\ref{pr:no-active-fragments-cyclic-iterated}.
\end{proof}

\section{Reduced representatives} \label{s:reduction}

The main goal of this section is to prove that any element of $G_\al$ can be represented by a reduced word and to prove a cyclic analog of this statement 
(Proposition \ref{pr:cyclic-reduction}).

\begin{proposition}[reduced representative] \label{pr:reduction}
Every element of $G_\al$ can be represented by a reduced in $G_\al$ word
which contains no fragments $F$ of rank $1\le \be\le \al$ with 
$\muf(!F) \ge \frac12 + 2\la + 15\om$.
\end{proposition}

\begin{lemma} \label{lm:reduction-contiguity}
Let $m \ge 3$ and $!X^{-1} * !Y_1 * !Y_2* \dots * !Y_m * $ be a coarse $(m+1)$-gon in $\Ga_{\al-1}$.
Assume that there are indices $1 \le t_1 < t_2 < \dots < t_k \le m$ $(k \ge 1)$ such that 
$$
  t_1 \le 3, \quad t_k \ge m-2, \quad t_j - t_{j-1} \le 2 \text{ for all } j 
$$
and 
$$
  |!Y_{t_j}|_{\al-1} > 4\eta \quad \text{for all } j.
$$
Assume further that there are no close vertices in each of the pairs 
$(!Y_i,!Y_{i+1})$,
$(!Y_1,!Y_{t_1})$, $(!Y_{t_j}, !Y_{t_j+1})$, $(!Y_{t_k}, !Y_m)$ except appropriate endpoints
(i.e.\ except $\tau(!Y_i)$ and $\io(!Y_{i+1})$). Then each of the paths ~$!Y_{t_j}$
has a vertex close to a vertex $!a_j$ on $!X$ and these vertices $!a_j$ are in $!X$ 
in the (non-strict) order from start to end.
\end{lemma}

\begin{proof}
We first claim that there are no close vertices in pairs $(!Y_i,!Y_j)$ for $j-i > 1$.
Assume there are. We choose such a pair with minimal possible $j-i$.
Then an ending segment ~$!Y_i'$ of $!Y_i$, paths $!Y_{i+1}$, $\dots$, $!Y_{j-1}$ and
a starting segment $!Y_j'$ of $!Y_j$ form a coarse $r$-gon with $r = j-i+1 \ge 3$.
Applying Proposition \ref{pr:sides-bound-Cayley}$_{\al-1}$ we get
$$
  \sum_{k=i+1}^{j-1} |!Y_i|_{\al-1} \le (r-2) \eta.
$$
On the other hand, it follows from the hypothesis of the lemma that
there are at least $\min(1,\frac12 (r-3))$ paths $!Y_{t_k}$ among $!Y_{i+1}$, $\dots$, $!Y_{j-1}$
and hence
$$
  \sum_{k=i+1}^{j-1} |!Y_i|_{\al-1} > 4\eta \min \left(1,\frac12 (r-3) \right).
$$
We get a contradiction since the right hand side of the inequality is at least $(r-2)\eta$.
This proves the claim.

Shortening if necessary $!Y_1$ and $!X$ we can assume that there is no pair of close vertices 
on ~$!Y_1$ and $!X$
other that $(\io(!Y_1),\io(!X))$. Similarly, we can assume that
there is no pair of close vertices on $!Y_m$ and $!X$ 
other than $(\tau(!Y_m),\tau(!X))$. 
Now we claim that there is a pair of close vertices on $!Y_i$ and $!X$ for some $2 \le i \le m-1$.
Indeed, otherwise we can apply Proposition ~\ref{pr:sides-bound-Cayley}$_{\al-1}$
to the whole coarse $(m+1)$-gon
$!X^{-1} * !Y_1 * !Y_2* \dots * !Y_m *$ and obtain a contradiction since $4k\eta \ge (m-1)\eta$.

Let $(!b,!c)$ be a pair of close vertices on $!X$ and $!Y_{i_0}$ where $2 \le i_0 \le m-1$.
Let $!b$ divide ~$!X$ as $!X_1 !X_2$ and $!c$ divide $!Y_{i_0}$ as $!Z_1 !Z_2$
If there is at least one index ~$t_j$ in the interval $2 \le t_j \le i_0-1$ then the conditions of the 
lemma are satisfied for the coarse $(i_0+1)$-gon $!X_1^{-1} * !Y_1 * \dots !Y_{i_0-1} * !Z_1 *$
and we conclude by induction that every $!Y_{t_j}$ with $t_j < i_0$ 
has a vertex close to a vertex $!a_j$ on $!X$ and the vertices $!a_j$ occur in $!X$ in the appropriate order.
Similarly, we conclude the same for every path $!Y_{t_j}$ with $t_j > i_0$. 
This implies the statement for all ~$!Y_{t_j}$.
\end{proof}

\begin{lemma} \label{lm:reduction-step}
Let $X$ be a word reduced in $G_{\al-1}$. Assume that 
for any fragment $K$ of rank ~$\al$ in $X$ we have
$$
  \muf(K) \le 1 - 3\la - 5\om.
$$
Then there exists a word $Y$ equal to $X$ in $G_\al$ which is reduced in $G_{\al-1}$ and such that for any
fragment $M$ of rank $\al$ in $Y$ we have
$$
  \muf(M) < \frac12 + 2 \la + 15\om.
$$
In particular, $Y$ is reduced in $G_\al$
(note that $\frac12 + 2 \la + 15\om < \rho = 1 - 9\la$ by 
\eqref{eq:ISC-main} and \eqref{eq:om-bounds}.)
\end{lemma}

\begin{proof}
We represent $X$ by a reduced path $!X$ in $\Ga_{\al-1}$. 
Denote
$$
  t = \frac12 + 11\om.
$$
Let $!K_1$, $\dots$, $!K_r$ be a maximal set of pairwise non-compatible fragments of rank $\al$ in $!X$ with 
$\muf(!K_i) \ge t$.
We assume that each $!K_i$ has maximal size $\muf(!K_i)$ in its equivalence class of compatible fragments 
of rank ~$\al$ occurring in $!X$.
Using Proposition \ref{pr:small-overlapping} we shorten each $!K_i$ from the start 
obtaining a fragment $\bar{!K}_i$ of rank $\al$ 
so that $\bar{!K}_i$ do not intersect pairwise; we have
$\muf(\bar{!K}_i) > \muf(!K_i) - \la - 2.7\om$.
Let 
$$
  !X = !S_0 \bar{!K}_1 !S_1 \dots \bar{!K}_r !S_r.
$$ 
Let $!P_i$ be a base for $\bar{!K}_i$; 
for each $i$, we have a coarse bigon $\bar{!K}_i^{-1} !u_i !P_i !v_i$
with bridges $!u_i$ and ~$!v_i$.
Let $P_i \greq \lab(!P_i)$ and $P_iQ_i^{-1}$ be the associated relator of rank $\al$.
We consider a path in ~$\Ga_{\al-1}$
$$
  !Z = !S_0^* !u_1^* !Q_1! v_1^* !S_1^* \dots !u_r^* !Q_r !v_r^* !S_r^*
$$ 
where labels of $!S_i^*$, $!u_i^*$ and $!v_i^*$ are equal to corresponding labels of 
~$!S_i$, $!u_i$ and $!v_i$ and $\lab(!Q_i) \greq Q_i$.
Note that $\lab(!Z) = X$ in $G_\al$.
We perform the following procedure:
\begin{enumerate}
\item 
if a pair of vertices on $!Q_i$ and $!S_i^*$ are close and is distinct from $(\tau(!Q_i), \io(!S_i^*))$ 
then we choose a bridge $!w$ of rank $\al-1$ joining these vertices,
replace $!v_i^*$ with $!w$ and shorten ~$!Q_i$ from the end and $!S_i^*$ from the start; 
similarly, if a pair of vertices on $!Q_i$ and $!S_{i-1}^*$ are close 
and is distinct from $(\io(!Q_i), \tau(!S_{i-1}^*))$ 
then we choose a bridge $!w$ of rank $\al-1$ joining them and
replace ~$!u_i^*$ with $!w$ shortening ~$!Q_i$ from the start and $!S_{i-1}^*$ from the end;
we apply recursively the operation until possible;
\item
if a vertex on $!Q_i$ is close to a vertex on $!Q_{i+1}^*$ then we choose a bridge $!w$ 
of rank $\al-1$ joining these vertices,
shorten $!Q_i$ from the end and $!Q_{i+1}$ from the end and join then by $!w$ 
(so ~$!S_i^*$ is eliminated and $!v_i^* !S_i^* !u_i^*$ is replaced with a bridge $!w$ of rank $\al-1$);
we apply recursively the operation until possible;
\end{enumerate}
After the procedure, we obtain a path
$$
  !Z_1 = !T_0 !U_0 !R_1 !U_1 \dots !R_r !U_r !T_r
$$
where for each $i$, $!R_i$ is a subpath of $!Q_i$ and $!U_i$ either is a bridge of rank $\al-1$ 
or has the form $!w_i !T_i !z_i$ where 
$!T_i$ is a subpath of $!S_i^*$ and $!w_i$ and $!z_i$ are bridges of rank $\al-1$.
Let $!Y$ be a reduced path with the same endpoints as $!Z_1$. 
Our goal is to prove that the label $Y$ of $!Y$ satisfies the requirement of the lemma,
that is, for any fragment $!N$ of rank $\al$ in $!Y$ we have 
$\muf(!N) < \frac12 + 2\la + 15\om$.

We compute a lower bound for $\mu(!R_i)$. Fix $i$ and let $!Q_i = !Q' !R_i !Q''$. 
At step (i) of the procedure, we do not shorten $!Q_i$ more than this would give
a fragment of rank $\al$ in $!X$ with a base that is a proper extension of $!P_i$,
so we get $\mu(!Q_i) \ge 1 - \muf(!K_i) \ge 3\la + 5\om$.
At step (ii) we shorten $!Q_i$ from each side by less than $\la + 0.4\om$ 
(this follows from Proposition \ref{pri:3gon-closeness}$_{\al-1}$, 
Proposition \ref{pr:no-inverse-compatibility} and Corollary \ref{co:small-overlapping-Cayley}).
This implies $\mu(!R_i) > \la + 4\om$ and, in particular, $|!R_i|_{\al-1} > 4 \eta$.

We apply Lemma \ref{lm:reduction-contiguity} with $!X :=!Y$ where $!R_i$ and $!T_i$
play the role of $!Y_i$'s 
and $!R_i$ are taken as ~$!Y_{t_i}$. 
The lemma says that each path $!R_i$ has a vertex close to a vertex on $!Y$
and these vertices on $!Y$ are appropriately ordered.
We can write 
$$
    !Y = !V_0 !M_1 !V_1 \dots !M_r !V_r
$$ 
where each $!M_i$ is close to a subpath of $!Q_i$
(at the moment each $!M_i$ is empty because it is represented by a vertex on ~$!Y$).
Extending $!M_i$'s we make them maximal so that no vertex on ~$!W_i$ except $\io(!V_i)$ is close to
a vertex on $!Q_i$ and no vertex on $!V_i$ except $\tau(!V_i)$ is close to a vertex on ~$!Q_{i+1}$. 
Up to location of $!Z$ in $\Ga_{\al-1}$ we can assume that it starts at $\io(!X)$.
Combining the two graphs shown in Figure \ref{fig:reduction-step-graphs}a  
and mapping them to $\Ga_\al$ we obtain a graph as shown in Figure  ~\ref{fig:reduction-step-graphs}b.
\begin{figure}[h]
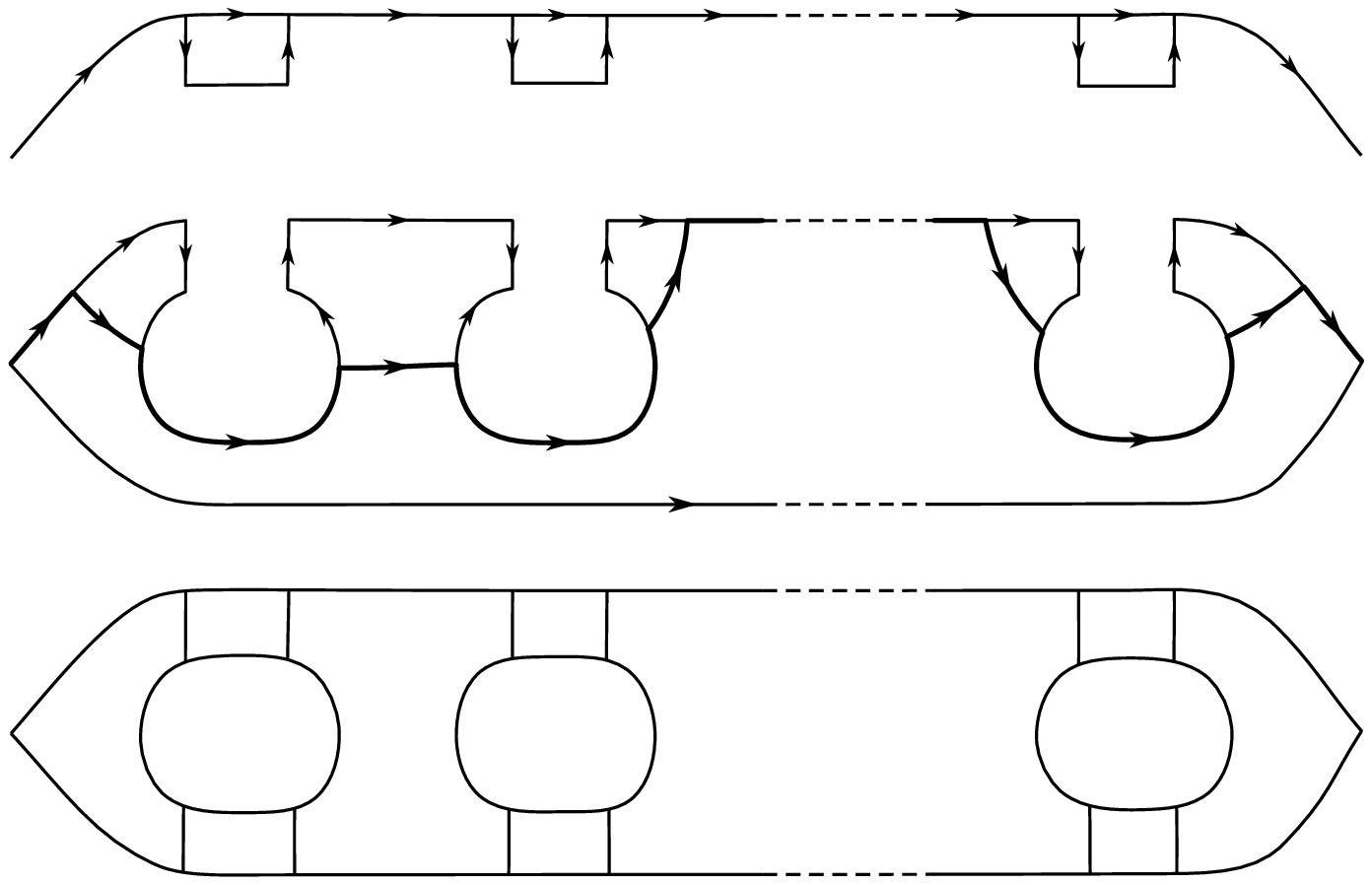
\caption{}  \label{fig:reduction-step-graphs}
\end{figure}
This graph is similar to one obtained from a single-layer diagram (as in Figure ~\ref{fig:bigon-Cayley}).
An easy analysis with use of Proposition \ref{pr:3-4-gon-closeness}$_{\al-1}$,
Proposition ~\ref{pr:no-inverse-compatibility} and Corollary ~\ref{co:small-overlapping-Cayley}
shows that $!M_i$ and some extension $\ti{!K}_i$ of $\bar{!K}_i$ satisfy the bound as
in Proposition ~\ref{pr:active-large}, i.e.\
$$
  \muf(!M_i) + \muf(\ti{!K}_i) > 1 - 2\la - 1.5\om.
$$
Since $\muf(\ti{!K}_i) \le \muf(!K_i) \le 1 - 3\la - 5\om$ we obtain that for all $i$,
$$
  \muf(!M_i) > \la + 3.5\om.
$$

Let $!N$ be a fragment of rank $\al$ in $!Y$. 
By Proposition \ref{pr:inclusion-compatibility},
we have either $!N \sim !M_i$ or $!N \seq !M_i \cup !M_{i+1}$ for some $i$.
In the case when $!N \seq !M_i \cup !M_{i+1}$, $!N \not\sim !M_i$ and $!N \not\sim !M_{i+1}$
we can apply the argument
from the proof of Proposition \ref{pr:fragment-nonactive-bigon} and find a fragment $!N'$ in $!X$ 
such that 
$$
  \muf(!N') > \muf(!N) - 2 \la - 3.4\om.
$$
We have also $!N' \not\sim !K_i, !K_{i+1}$ and hence $!N' \not\sim !K_j$ for all $j$.  
By the choice of the $!K_i$'s, we have $\muf(!K') < t$ and hence  
$$
  \muf(!N) < t + 2\la + 3.4\om < \frac12 + 2\la + 15\om.
$$

Assume that $!N \sim !M_i$ for some $i$. 
Let $\bar{!Q}$ and $\bar{!P}$ be bases for $!N$ and $!K_i$ respectively. Images of $\bar{!Q}^{-1}$ and $\bar{!P}$
in $\Ga_\al$ are subpaths of a relator loop and have at most two overlapping parts. We give
an upper bound for $\mu(\bar{!Q}) + \mu(\bar{!P})$ by finding an upper bound for the size of each
overlapping part. Assume, for example, that an end of the image of $\bar{!P}$ in $\Ga_\al$
overlaps with a start of the image of $\bar{!Q}^{-1}$.
Changing the location of $!Z$ in $\Ga_{\al-1}$ we can assume that $\bar{!P}$ and ~$\bar{!Q}^{-1}$
overlap on a subpath $!W$ of the same size already in $\Ga_{\al-1}$. 

We consider the case $i < r$ (see Figure \ref{fig:reduction-active-upper-bound}; the case $i=r$ is similar with a better upper bound on $\mu(!W)$).  
\begin{figure}[h]
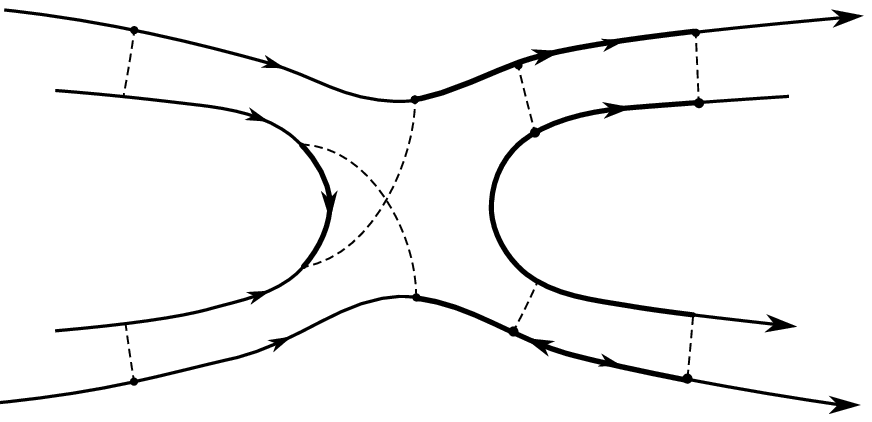
\caption{}  \label{fig:reduction-active-upper-bound}
\end{figure}
We apply Proposition \ref{pri:4gon-closeness}$_{\al-1}$ to a coarse tetragon with one side $!W$
and other sides which are an end $!S$ of $!S_{i} \bar{!K}_{i+1}$, 
a start $!V$ of $!M_{i+1}^{-1} !V_{i}^{-1}$ and a subpath of a common base axis ~$!L$ for $!K_{i+1}^{-1}$
and $!N_{i+1}$. In the worst case we have $!W = !W_1 !z_1 !W_2 !z_2 !W_3$ where $!W_1$ is close to a subpath of $!V^{-1}$,
$!W_2$ is close to a subpath of $!L^{-1}$, 
$!W_3$ is close to a subpath of ~$!S^{-1}$ and $|!z_i|_{\al-1} \le 4\eta\ze$.
Proposition \ref{pr:closeness-order}$_{\al-1}$ implies $|!W_1|_{\al-1} < 5.7$ and $|!W_3|_{\al-1} < 5.7$.
Since $!K_i \not\sim !K_{i+1}$ we obtain $\mu(!W_2) < \la$.
Hence
$$
  \mu(!W) < \la + 2\om(5.7+4\eta\ze) < \la + 13\om.
$$
We obtain
$$ 
  \muf(!N) + \muf(!K_i) < 1 + 2\la + 26\om.
$$
Since $\muf(!K_i) \ge t$ this implies the required bound $\muf(!N) < \frac12 + 2\la + 15\om$.
\end{proof}

\begin{lemma} \label{lm:reduction-adding-letter}  
Let $\al\ge1$ and $X$ be a word reduced in $G_\al$ and 
$a \in\cA^{\pm1}$ a letter in the generators of $G_\al$. 
Let $Y$ be a word reduced in $G_{\al-1}$ such that $Y = Xa$ in $G_{\al-1}$.
Then $Y$ has no fragments ~$K$ of rank $\al$ with $\muf(K) \ge \rho + 6.2\om$.
\end{lemma}

\begin{proof}
Follows from Lemma \ref{lm:piece-fragment-no-higher-fragments} and 
Proposition \ref{pr:fragment-trigon-previous}.
\end{proof}

\begin{proof}[Proof of Proposition \ref{pr:reduction}]
It is trivial if $\al=0$.
In the case $\al\ge 1$
Proposition \ref{pr:reduction} follows by induction from Lemmas \ref{lm:reduction-step}
and \ref{lm:reduction-adding-letter} since
$
  \rho + 6.2\om < 1 - 3\la -5\om.
$
\end{proof}

We turn to the cyclic analogue of Proposition \ref{pr:reduction}:

\begin{proposition}[cyclically reduced representative] \label{pr:cyclic-reduction}
Every element of $G_\al$ of finite order is conjugate to a cyclically reduced word of the form $R_0^k$ where $R_0$
is the root of a relator of rank $\be$, $1 \le \be \le \al$.

Every element of $G_\al$ of infinite order is conjugate to a 
strongly cyclically reduced word in ~$G_\al$.
\end{proposition}

\begin{lemma}[a cyclic version of Lemma \ref{lm:reduction-contiguity}] \label{lm:reduction-contiguity-cyclic}
Let $X$ be a word cyclically reduced in $G_{\al-1}$ representing an element of $G_{\al-1}$ of infinite order. 
Let $m \ge 2$, 
$Y_1, \dots, Y_m$ be words reduced in $G_{\al-1}$, $u_1,\dots,u_m$ be bridges of rank $\al-1$
and let $X$ be conjugate to $Y_1 u_1 \dots Y_m u_m$ in $G_{\al-1}$. 
Let $\prod_{i \in \Z} !Y_1^{(i)} !u_1^{(i)} \dots !Y_m^{(i)} !u_m^{(i)}$ and $\bar{!X} = \prod_{i \in \Z} !X^{(i)}$
be lines in $\Ga_{\al-1}$ labeled $(Y_1 u_1 \dots Y_m u_m)^\infty$ and ~$X^\infty$
respectively representing the conjugacy relation.

Assume that there are indices $1 \le t_1 < t_2 < \dots < t_k \le m$ $(k \ge 1)$ such that 
$$
   m+ t_1 - t_m \le 2, \quad t_j - t_{j-1} \le 2 \quad\text{ for all } j, 
$$
and 
$$
  |Y_{t_j}|_{\al-1} > 4\eta \quad \text{for all } j.
$$
Assume that there are no close vertices in each of the pairs $(!Y_i^{(0)},!Y_{i+1}^{(0)})$, $(!Y_m^{(0)}, !Y_1^{(1)})$,
$(!Y_{t_j}^{(0)}, !Y_{t_j+1}^{(0)})$, $(!Y_{t_k}^{(0)}, !Y_{t_1}^{(1)})$ except appropriate endpoints
(i.e.\ except pairs $(\tau(!Y_i^{(0)}), \io(!Y_{i+1}^{(0)}))$ and 
$(\tau(!Y_m^{(0)}), \io(!Y_{1}^{(1)}))$). 
Then each of the paths $!Y_{t_j}^{(0)}$, $j=1,\dots,k$
has a vertex close to a vertex $!a_j$ on $\bar{!X}$ and these vertices $!a_j$ are in the (non-strict) order
corresponding to the order of the $!Y_j^{(0)}$'s 
(and $!a_k$ is located non-strictly before $s_{X,\bar{!X} }!a_0$).
\end{lemma}

\begin{proof}
The proof follows the proof of Lemma \ref{lm:reduction-contiguity}
with appropriate changes. 

\begin{claim}
There are no close vertices in pairs $(!Y_i^{(0)},!Y_j^{(0)})$ with $j-i > 1$
and $(!Y_i^{(0)},!Y_j^{(1)})$ with $j+m -i > 1$.
\end{claim}

The proof repeats the argument from the proof of Lemma \ref{lm:reduction-contiguity}.

\begin{claim}
For some $i$, there are close vertices in the pair $(!Y_i^{(0)}, \bar{!X})$.
\end{claim}

Assume this is not true. Consider an annular diagram $\De$ of rank ${\al-1}$ with boundary loops
$\hat{!X}^{-1}$ and $\hat{!Y}_1 \hat{!u}_1 \dots \hat{!Y}_m \hat{!u}_m$ 
and a combinatorially continuous map $\phi: \ti{\De} \to \Ga_{\al-1}$
such that $\phi$ maps the boundary of $\ti\De$ to 
$\bar{!X}^{-1}$ and $\prod_i !Y_1^{(i)} !u_1^{(i)} \dots !Y_m^{(i)} !u_m^{(i)}$.
The assumption, Claim 1 and the hypothesis of the lemma imply that $\De$ is small. 
Application of Proposition \ref{pr:principal-bound}$_{\al-1}$ gives
$$
  \sum_i |Y_i|_{\al-1} \le \eta m.
$$
On the other hand, from the hypothesis of the lemma we have $\sum_i |Y_i|_{\al-1} \ge 4k \eta > \eta m$,
a contradiction. This proves the claim.

By Claim 2, assume 
without loss of generality 
that there is a vertex $!b$ on $!Y_1^{(0)}$ which is close to a vertex $!c$
on $\bar{!X}$.
Let $!b$ divide $!Y_1^{(0)}$ as $!Y_1^{(0)} = !Z_1 !Z_2$ and up to cyclic shift of $X$, assume that $!X^{(0)}$
starts at $!c$. Now we can directly apply Lemma \ref{lm:reduction-contiguity} to the coarse $(m+2)$-gon
$$
  (!X^{(0)})^{-1} * !Z_2 !u_1^{(0)} !Y_2^{(0)} \dots !u_{m-1}^{(0)} !Y_m^{(0)} !u_m^{(0)} !Z_1 *
$$
and get the required conclusion.
\end{proof}



\begin{lemma}[a cyclic version of Lemma \ref{lm:reduction-step}] 
\label{lm:cyclic-reduction-step}
Let $X$ be a word strongly cyclically reduced in $G_{\al-1}$.
Assume that $X$ is not conjugate in $G_{\al}$ to a power of the root of a relator of rank $\be\le\al$.
Next, assume that for any fragment $K$ of rank $\al$ in a cyclic shift 
of $X$ we have
$$
  \muf(K) \le 1 - 4\la - 8\om.
$$
Then there exists a word $Z$ conjugate to $X$ in $G_\al$ which is strongly cyclically reduced in $G_{\al-1}$ and
such that no power $Z^k$ contains a fragment $L$ of rank $\al$ with 
$$
  \muf(L) < \frac12 + 2\la + 15\om.
$$
In particular, $Z$ is strongly cyclically reduced in $G_\al$.
\end{lemma}

\begin{proof}
The general scheme is the same as in the proof of Lemma \ref{lm:reduction-step}.
Let $\bar{!X} = \prod_{i\in \Z} !X_i$ be a line in $\Ga_{\al-1}$ labeled $X^\infty$.
First we note that for any fragment $!K$ of rank $\al$ in $\bar{!X}$
we have $s_{X,\bar{!X}} !K \not \sim !K$ by Proposition \ref{pri:fragment-in-periodic}.
By Propositions \ref{pr:inclusion-compatibility} and \ref{pr:dividing-fragment}
there exists a starting segment $!K'$ of $!K$ that is a fragment
of rank $\al$ with $\muf(!K') > \muf(!K) - \la - 3\om$ and $|!K'| \le |X|$,
i.e.\ $\lab(!K')$ occurs in a cyclic shift of $X$.
Then the hypothesis of the lemma implies that $\bar{!X}$ contains
no fragments $!K$ of rank $\al$ with $\muf(!K) \ge 1 - 3\la - 5\om$.


Denote $t = \frac12 + 11\om$.
We can assume that there is at least one fragment $!K$ of rank ~$\al$ in $\bar{!X}$ with $\muf(!K) \ge t$
(otherwise we can take $Z := X$).
We choose a maximal set $!K_1$, $\dots$, $!K_r$ of pairwise non-compatible fragments of rank $\al$ in $\bar{!X}$ with 
$\muf(!K_i) \ge t$ such that $!K_1 < \dots < !K_r < s_{X,\bar{!X}} !K_1$ 
and $!K_r \not\sim s_{X,\bar{!X}} !K_1$
(after choosing $!K_1$ we use Proposition \ref{pri:fragment-in-periodic} 
to get $s_{X,\bar{!X}} !K_1 \not \sim !K_1$).
We assume that each $!K_i$ has maximal size $\muf(!K_i)$ in its class of compatible fragments of rank ~$\al$ in $\bar{!X}$.
Using Proposition \ref{pr:small-overlapping} we shorten each $!K_i$ from its start obtaining a fragment $\bar{!K}_i$ 
of rank $\al$ so that all $\bar{!K}_i$ do not intersect pairwise and $|!K_1 \cup !K_r| \le |X|$;
we have $\muf(\bar{!K}_i) > \muf(!K_i) - \la - 2.7\om$.
Passing to a cyclic shift of $X$ (and changing all $!X_i$ accordingly) we may
assume also that 
$$
  !X_0 = \bar{!K}_1 !S_1 \dots \bar{!K}_r !S_r.
$$

Let $!P_i$ be the base for $\bar{!K}_i$ and $\bar{!K}_i^{-1} !u_i !P_i !v_i$
a loop in $\Ga_{\al-1}$ with bridges $!u_i$ and $!v_i$.
Denote $S_i \greq \lab(!S_i)$, $P_i \greq \lab(!P_i)$, $u_i \greq \lab(!u_i)$, 
$v_i \greq \lab(!v_i)$
and let $P_i Q_i^{-1}$ be the associated relator of rank ~$\al$.
Let 
$$
  Z = u_1 Q_1 v_1 S_1 u_2 Q_1 v_2 S_2 \dots u_r Q_r v_r S_r.
$$
Let $Y$ be a word strongly cyclically reduced in $G_{\al-1}$ that is conjugate to $Z$ in $G_{\al-1}$.
We prove that $Y$ satisfies the requirements of the lemma.
Note that $Y$ and hence $Z$ are conjugate to $X$ in $G_\al$.

We transform $Z$ using a procedure analogous to the procedure described in the proof of Lemma \ref{lm:reduction-step}.
At any moment, we will have a word $Z_1$ of the form
$$
  Z_1 = R_1 U_1 \dots R_r U_r,
$$
conjugate to $Z$ in $G_{\al-1}$ where each $R_i$ is a subword of $Q_i$ and each $U_i$ either is a bridge 
of rank $\al-1$ or has the form $w_i T_i z_i$ where $w_i$, $z_i$ are bridges of rank $\al-1$ and $T_i$
is a subword of $S_i$. At the start, we have $R_i = Q_i$ and $U_i = v_i S_i u_{i+1}$ (here and
below $i+1$ is taken modulo ~$r$).
The transformation procedure consists of the following steps applied recursively until possible.
\begin{enumerate}
\item 
Suppose that $U_i$ has the form $w_i T_i z_i$ above. If $R_i = R' R''$, $T_i = T' T''$ where $|R''| + |T'| > 0$ and 
$R'' w_i T'$ is equal in $G_{\al-1}$
to a bridge $w$ of rank $\al-1$ then replace $R_i$, $w_i$ and $T_i$ with $R'$, $w$ and $T''$ respectively;
similarly, if $T_i = T' T''$, $R_{i+1} = R' R''$ where $|T''| + |R'| >0$ and 
$T'' z_i R'$ is equal in $G_{\al-1}$ 
to a bridge $w$ of rank $\al-1$ then replace $T_i$, $z_i$ and $R_{i+1}$ with $T'$, $w$ and $R''$ respectively.
\item
If $R_i = R'R''$ and $R_{i+1} = R^* R^{**}$ where $|R''| + |R^*| >0$ and $R'' U_i R^*$ is equal in $G_{\al-1}$
to a bridge $w$ of rank $\al-1$ then replace $R_i$, $U_i$ and $R_{i+1}$ with $R'$, $w$ and $R^{**}$ respectively.
\end{enumerate}
Similar to the proof of Lemma \ref{lm:reduction-step}, after 
performing the procedure we obtain $|R_i|_{\al-1} > 4\eta$
for all ~$i$. 

Let $\bar{!Z} = \prod_{i \in \Z} !Z^{(i)}$ be a line in $G_{\al-1}$ labeled $Z^\infty$
and let $!Q_j^{(i)}$ denote the appropriate subpath of $!Z^{(i)}$ labeled $Q_j$.
We can implement the procedure above on the line $\bar{!Z}$ instead of a word ~$Z$
by changing appropriate paths instead of words (to each change of words 
in (i) or (ii) there corresponds infinitely many changes of paths translated
by $s_{X,\bar{!X}}$).
As a result, we get a line $\prod_{i \in \Z} !Z_1^{(i)}$ so that the corresponding 
subpath $!R_j^{(i)}$ of $!Z_1^{(i)}$ is also a subpath of $!Q_j^{(i)}$.
Denote also ~$!T_j^{(i)}$ the appropriate subpath of $!Z_1^{(i)}$ labeled $T_j$.
Let $\bar{!Y} = \prod_{i \in \Z} !Y^{(i)}$ be the line in $G_{\al-1}$ such that 
$\bar{!Z}$ and $\bar{!Y}$ are associated with conjugate words $Z$ and ~$Y$.
We apply Lemma \ref{lm:reduction-contiguity-cyclic} with $\bar{!X} :=\bar{!Y}$ 
where $!R_j^{(i)}$ and $!T_j^{(i)}$ play the role of $!Y_j^{(i)}$'s 
and $!R_j^{(i)}$ are taken as $!Y_{t_j}^{(i)}$. 
According to the lemma, each path $!R_j^{(0)}$ has a vertex close to a vertex 
on $\bar{!Y}$,
these vertices on $\bar{!Y}$ are ordered along $\bar{!Y}$ in the increasing
order of the index $j$, and the length of the segment of $\bar{!Y}$ between 
the first and the last one is not more that $|Y|$.
Up to cyclic shift of $Y$, we can write
$$
  !Y^{(0)} = !W_0 !M_1 !W_1 \dots !M_r !W_r
$$ 
where each $!M_j$ is close to a subpath of $!Q_j^{(0)}$.
Taking $!M_j$ maximal with these properties we obtain, as in the proof of Lemma \ref{lm:reduction-step},
$$
  \muf(!M_i) > \la + 3.5\om \quad\text{for all }j.
$$
The rest of the proof is similar to the proof of Lemma \ref{lm:reduction-step}.
\end{proof}

\begin{lemma} \label{lm:close-distance-bound}
If $!X$ is a reduced path in $\Ga_\al$ and the endpoints of $!X$ are close then $|!X|_\al \le 1$.
\end{lemma}

\begin{proof}
For $\al\ge 1$ this follows from 
Lemma \ref{lm:monogon-graph-version}. 
\end{proof}

\begin{lemma} \label{lm:piece-fragment}
If $P$ is a piece of rank $\al$ then for any fragment $K$ of rank $\al$ in $P$ we have
$\muf(K) \le \max\set{\la, \mu(P) + 2\om}$. 
\end{lemma}

\begin{proof}
Let $!P$ be a path in $\Ga_{\al-1}$ with $\lab(!P) \greq P$, let $R$ be the associated relator of rank ~$\al$
and let $!L$ be the line labeled $R^\infty$ extending $!P$.
Assume that $!K$ is a fragment of rank $\al$ contained in $!P$.
If the base axis for $!K$ is distinct from $!L$ then $\muf(!K) < \la$ by Corollary \ref{co:small-overlapping-Cayley}.
Otherwise the base $!Q$ for $!K$ is contained in $!L$ and
Lemma \ref{lm:close-distance-bound}$_{\al-1}$ implies
$$
    \muf(!K) = \mu(!Q) \le \mu(!K) + 2\om \le \mu(!P) + 2\om.
$$
\end{proof}

\begin{proposition} \label{pr:piece-reduction}
Let $P$ be a piece of rank $1\le \be \le \al$ with $\mu(P) \le \rho - 2\om$.
Then $P$ is reduced in $G_\al$. 
If $R\greq QS$ where $R$ is a relator of rank $\be$ then either $Q$ or $S$ 
is reduced in ~$G_\al$. 
\end{proposition}

\begin{proof}
The first statement follows from Lemmas \ref{lm:piece-fragment-no-higher-fragments} 
and \ref{lm:piece-fragment}.
If $R$ is a relator of rank $\be$ and $R\greq QS$ then 
by \ref{ss:alpha-length-properties}(ii),
we have either $\mu(Q) \le \frac12 + \om$ or $\mu(S) \le \frac12 + \om$.
It remains to note that $\frac12 + \om <  \rho - 2\om$.
\end{proof}

\begin{proof}[Proof of Proposition \ref{pr:cyclic-reduction}]
Let $X$ be a word representing an element of $G_\al$. We may assume that $X$ is reduced in $G_\al$ as 
a non-cyclic word. We perform a ``coarse cyclic cancellation'' in ~$X$:
represent $X$ as $U X_1 V$ where $V U$ is equal in $G_\al$ 
to a bridge $u$ of rank $\al$ and $X_1$ has the minimal possible length.
Let $u \greq v_1 P v_2$ where $P$ is a piece of rank $\al$.
We can assume that $\mu(P) \le \frac12 + \om$.
Let $Y$ be a word cyclically reduced in $G_{\al-1}$ and conjugate to $X_1 u$ in ~$G_{\al-1}$. 
Note that $X_1 u$ and hence $Y$ are conjugate to $X$ in $G_\al$.
We show that either $Y$ is conjugate in $G_{\al-1}$
to a power $R_0^t$ of the root $R_0$ of a relator of rank $\be\le\al$
or no cyclic shift of $Y$ contains a fragment $K$ of rank $\al$
with $\muf(!K) \ge \rho + 2\la + 16\om$.
In the first case,
by Proposition ~\ref{pr:piece-reduction} we can assume that $R_0^k$ is cyclically reduced in $G_\al$
and we come to the first alternative of Proposition \ref{pr:cyclic-reduction}.
Otherwise, according to Proposition \ref{pr:cyclic-reduction}$_{\al-1}$ we can assume
that $Y$ is strongly cyclically reduced in $G_{\al-1}$.
Then we apply Lemma \ref{lm:cyclic-reduction-step} to
find a strongly cyclically reduced in $G_\al$ word $Z$ conjugate to $Y$ in $G_\al$
(note that $\rho + 2\la + 16\om < 1 - 4\la - 8\om$), coming to the second alternative.


Let $\bar{!Y} = \prod_{i \in \Z} !Y_i$ and $\prod_{i \in \Z} !X_1^{(i)} !v_1^{(i)} !P_i !v_2^{(i)}$
be lines in $\Ga_{\al-1}$ representing the conjugacy relation. We observe that
\begin{enumerate}
\item
{\em The base axis of any fragment $!N$ of rank $\al$ in $!P_i$ with $\muf(!N) \ge \la$ is the 
infinite periodic extension of 
$!P_i$. In particular, If $!N_1$ and $!N_2$ are fragments of rank $\al$ in $!P_i$ with $\muf(!N_j) \ge \la$
then $!N_1 \sim !N_2$.} 
(This follows from Corollary \ref{co:small-overlapping-Cayley}.)
\end{enumerate}
Now formulate some consequences of the choice of $X_1$ of minimal possible length:
\begin{enumerate}
\setcounter{enumi}{1}
\item 
{\em There exist no fragments $!N_1$ and $!N_2$ of rank $\al$ in $!X_1^{(i)}$ and in $!X_1^{(i+1)}$, respectively,
such that $!N_1 \sim !N_2$ and $\muf(!N_i) \ge 3.2\om$.}  
\end{enumerate}
Indeed, assume that such $!N_1$ and $!N_2$ do exist. Note that both $!N_1$ and $!N_2$ are nonempty
by Lemma \ref{lm:piece-fragment-no-higher-fragments}. 
By Lemma  \ref{lmi:compatible-close}, any two of the endpoints of the images of $!N_1$ and $!N_2$
in ~$\Ga_\al$ are close. Then we can shorten $X_1$ to its subword $X_2$ so that $X_2 u'$ is conjugate
to $X$ in ~$G_\al$ for some $u' \in \cH_\al$ contrary to the choice of $X_1$ 
(see Figure~\ref{fig:cyclic-reduction}a; in the figure we have 
$!N_2 \ll s_{Y,\bar{!Y}} !N_1$ in $!X_1^{(i+1)}$ but in all other cases we can easily find an 
appropriate path $!X_2$ with $|!X_2| < !X_1$ and take $X_2 := \lab(!X_2)$). 
\begin{figure}[h]
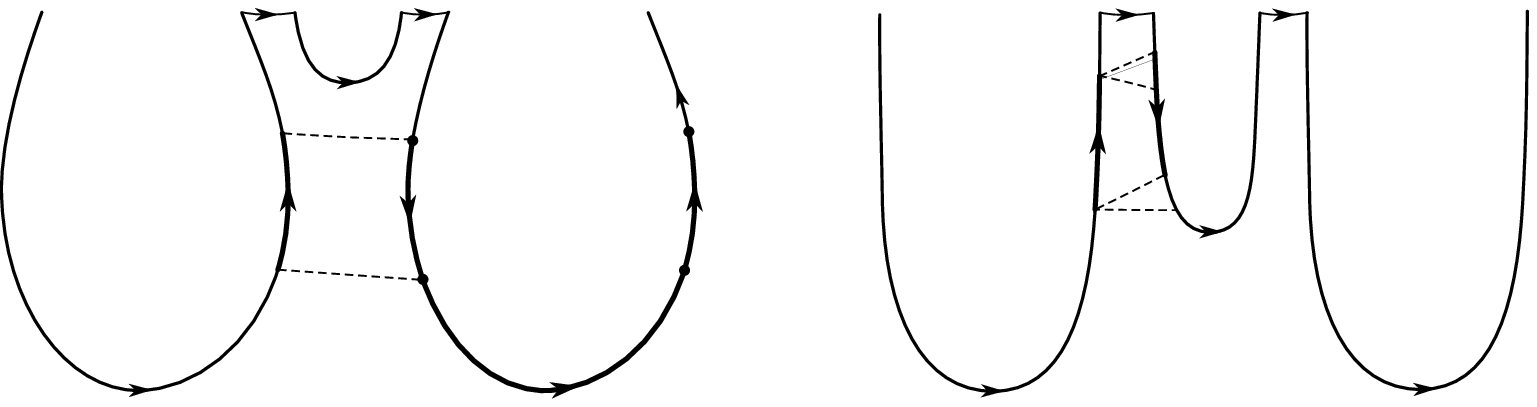
\caption{}  \label{fig:cyclic-reduction}
\end{figure}
\begin{enumerate}
\setcounter{enumi}{2}
\item
{\em There exist no fragments $!N_1$ and $!N_2$ of rank $\al$ in $!X_1^{(i)}$ and in $!P_i$ or $!P_{i-1}$, 
respectively, such that $!N_1 \sim !N_2$,  $\muf(!N_1) \ge 3.2\om$ and $\muf(!N_2) \ge \la$.} 
(Otherwise using (i) we can shorten ~$X_1$ to $X_2 := \lab(!X_2)$ as shown in 
Figure~\ref{fig:cyclic-reduction}b.)
\end{enumerate}
Let $Q$ be a word reduced in $G_{\al-1}$ which is equal 
to $X_1 v_1 P$ in $G_{\al-1}$.
We denote $!Q_i$ the corresponding path in $\Ga_{\al-1}$ joining $\io(!X_1^{(i)})$ with $\tau(!P_i)$.
Using (iii), Proposition \ref{pr:fragment-trigon-previous} and Lemma ~\ref{lm:piece-fragment} we conclude that
\begin{enumerate}
\setcounter{enumi}{3}
\item
There are no fragments $!M$ of rank $\al$ in $!Q_i$ with $\muf(!M) \ge \rho + \la + 6.2\om$.
\end{enumerate}

Assume that $!K$ is a fragment of rank $\al$ in $\bar{!Y}$ with $\muf(!K) \ge \rho + 2\la + 16\om$
and $|!K| \le |Y|$.
By (iv) and Proposition \ref{pr:fragment-cyclic-monogon-previous},
for some $i$ there are fragments $!M_1$ and $!M_2$ of rank $\al$ in $!Q_i$ and $!Q_{i+1}$ respectively such that 
$!M_j \sim !K$ $(i=1,2)$ and $\muf(!M_j) > \la +6.8\om$.
By Proposition ~\ref{pr:fragment-trigon-previous} there is a fragment $!N_1$ of rank $\al$ 
such that $!M_1 \sim !N_1$
and either $!N_1$ occurs in $!X_1^{(i)}$ and $\muf(!N_1) > 3.2\om$ or 
$!N_1$ occurs in $!P_i$ and $\muf(!N_1) > \la$. Similarly, there 
is a fragment $!N_2$ of rank $\al$ such that $!M_2 \sim !N_2$
and either $!N_2$ occurs in $!X_1^{(i+1)}$ and $\muf(!N_2) > 3.2\om$ or 
$!N_2$ occurs in $!P_{i+1}$ and $\muf(!N_2) > \la$.
If $!N_1$ occurs in $!X_1^{(i)}$ and $!N_2$ occurs in $!X_1^{(i+1)}$
we get a contradiction with (ii). 
If $!N_1$ occurs in $!P_i$ and $!N_2$ occurs in $!X_1^{(i+1)}$
or $!N_1$ occurs in $!X_1^{(i)}$ and $!N_2$ occurs in $!P_{i+1}$ we get 
a contradiction with ~(iii). Finally, if $!N_1$ occurs in $!P_i$ 
and $!N_2$ occurs in ~$!P_{i+1}$ then by (i), 
we have $s_{Y,\bar{!Y}} !N_1 \sim !N_2$ and hence $!K \sim s_{Y,\bar{!Y}} !K$. 
By Proposition \ref{pri:fragment-finite-order}$_{\al-1}$ this implies that $Y$ is 
conjugate in $G_{\al-1}$ to a power of the root of a relator of rank $\al$.
This finishes the proof.
\end{proof}

\begin{proposition} \label{pr:relator-root}
Let $R$ be a relator of rank $\be \le \al$ and let $R \greq R_0^n$ where $R_0$ is the root of $R$.
Then $R_0$ has order $n$ in $G_\al$.
\end{proposition}

\begin{proof}
Let $k$ be a proper divisor of $n$.
By Lemma \ref{lm:piece-fragment-no-higher-fragments}, $R_0^k$ contains no fragments $K$ 
of rank ~$\ga$ with 
$\muf(K) \ge 3.2\om$, for all $\ga=\be+1,\dots,\al$. 
By Proposition \ref{pr:piece-reduction}$_\be$,
$R_0^k$ is cyclically reduced in ~$G_\be$ and hence also in rank $\al$. Hence $R_0^k \ne 1$ in $G_\al$.
\end{proof}

\begin{proposition}[conjugate powers of relator roots] 
\label{pr:conjugate-relator-roots}
Let $R$ be a relator of rank $1\le \be \le \al$ and let $R \greq R_0^n$ where $R_0$ is the root of $R$.
If $R_0^k = g^{-1} R_0^l g$ in $G_\al$ for some $k,l \not\equiv 0 \pmod n$ then $g \in \sgp{R_0}$ 
and $k \equiv l \pmod n$.
\end{proposition}

\begin{proof}
By Proposition \ref{pr:relator-root}, 
if $R_0^k = g^{-1} R_0^l g$ in $G_\al$ and $g \in \sgp{R_0}$
then $k \equiv l \pmod n$.
It remains to prove that equality 
$R_0^k = g^{-1} R_0^l g$ for $k,l \not\equiv 0 \pmod n$
implies $g \in \sgp{R_0}$.

By Proposition \ref{pr:piece-reduction} we can assume that $R_0^k$
and $R_0^l$ are cyclically reduced in $G_\al$. 
We represent $g$ by a word $Z$ and consider an annular diagram $\De$ 
of rank $\al$ with two cyclic sides ~$!X_1$ and $!X_2$ 
labeled $R_0^{-k}$ and $R_0^{l}$ 
which is obtained from a disk diagram with boundary label 
$R_0^{-k} Z^{-1} R_0^l Z$ by gluing two boundary segments
labeled $Z^{-1}$ and $Z$.
Let $!Z$ be the path in $\De$ with $\lab(!Z) \greq Z$ that joins starting vertices of $!X_2$ and $!X_1$.

We apply to $\De$ the reduction process \ref{ss:reduction-transformations-good}.  
By Lemma \ref{lm:disk-annular-frame-type},
we can replace $!Z$ by a new path ~$!Z_1$ with the same endpoints 
such that $\lab(Z_1) = Z$ in $G_\al$ (so $\lab(Z_1)$ represents $g$
in $G_\al$).
We can assume also that $\De$ has a tight set $\cT$ of contiguity subdiagrams.

{\em Case\/} 1: $\De$ has a cell $!D$ of rank $\al$. 
By Proposition \ref{pri:bigon-cell}, $!D$ has a contiguity subdiagram
$\Pi_i \in \cT$ to each of the sides $!X_i$ of $\De$.
Moreover, if $\de\Pi_i = !S_i !u_i !Q_i !v_i$ where $!S_i^{-1}$
is a contiguity arc occurring in $\de!D$ then $\mu(!S_i) > \la$.
By Lemma \ref{lm:piece-fragment-no-higher-fragments} this implies $\be=\al$.
Let $\lab(\de\De) \greq R'$ where $R'$ is a relator of rank ~$\al$.
Consider lines $\bar{!X}_1$, $\bar{!X}_2$ and $\bar{!R}$ in $\Ga_{\al-1}$
labeled $R^{\pm \infty}$, $R^{\pm \infty}$ and $R'^{\infty}$
which are obtained by mapping the universal cover of the subgraph 
of $\De$ shown in Figure \ref{fig:conjugate-relator-roots}.    
\begin{figure}[h]
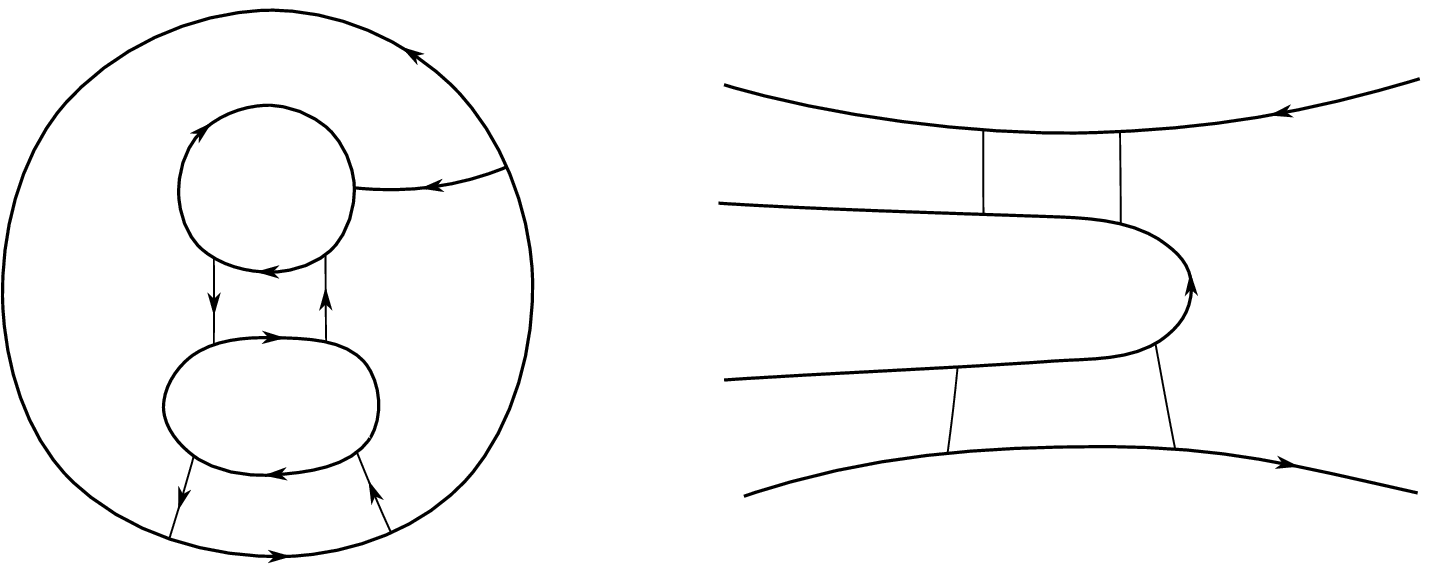
\caption{}  \label{fig:conjugate-relator-roots}
\end{figure}
By Corollary \ref{co:small-overlapping-Cayley} we get $\bar{!X}_1 = \bar{!X}_2 = \bar{!R}$. 
This implies that $\lab(!Z_1)$ is equal in ~$G_{\al-1}$ to a power of $R_0$, as required. 

{\em Case\/} 2: $\De$ has no cells of rank $\al$. Then we have equality $R_0^k = Z_1^{-1} R_0^l Z_1$
in $G_{\al-1}$. If $\be < \al$ then the statement follows from
Proposition \ref{pr:conjugate-relator-roots}$_{\al-1}$.
Let $\be = \al$. 
If $kl > 0$ then the statement follows from 
Proposition \ref{pr:nontorsion-conjugation}$_{\al-1}$.
If $kl < 0$ then by Corollary \ref{coi:root-unique}$_{\al-1}$ we obtain 
$R_0 = g^{-1} R_0^{-1} g$ which contradicts our condition (S3) on the presentation of $G_\al$.
\end{proof}

\begin{proposition} \label{pr:root-existence}
Every element of $G_\al$ of infinite order has the form $h^m$ 
where $h$ is a non-power.
\end{proposition}

\begin{proof}
We need to prove this only in the case $\al\ge1$.
Let $g \in G_\al$ be an element of infinite order.
It is enough to find an upper bound on $|m|$ in the equality of the form $g = h^m$.
Up to conjugation, we represent $g$ and $h$ by a strongly cyclically reduced in $G_\al$ words $X$ and ~$Y$
by Proposition \ref{pr:cyclic-reduction}.
Let $\be$ be the maximal rank with $1 \le \be \le \al$ 
such that a cyclic shift of $X$ contains a fragment $K$ of 
rank $\be$ with $\muf(K) \ge \xi_0$. (It there is no such $K$ then by Proposition \ref{pr:no-active-fragments-cyclic}
$X$ in conjugate to $Y^m$ in the free group $G_0$ and then $|m| \le |X|$.)
Using Propositions \ref{pri:fragment-correspondence-cyclic-beta} and \ref{pri:fragment-in-periodic}
we find $m$ pairwise non-compatible fragments $M$ of rank ~$\be$ with 
$\muf(M) \ge \xi_0 - 2\la - 3.4\om$ in
a cyclic shift of $X$. This again implies $|m| \le |X|$.
\end{proof}


\section{Coarsely periodic words and segments over $G_\al$}
\label{s:coarsely-periodic}

In this section we analyze words which are ``geometrically close'' in $G_\al$ to periodic words.
In Sections \ref{s:coarsely-periodic} and \ref{s:overlapping-periodicity}
we use the following notation for numeric parameters:
$$
  \xi_1 = \xi_0 - 2.6\om, \quad  \xi_2 = \xi_1 - 2\la - 3.4\om.
$$

\begin{definition}
A {\em simple period over $G_\al$} is 
a strongly cyclically reduced word representing a non-power element of $G_{\al}$.
\end{definition}

According to \ref{ss:reduced-word}, 
if $A$ is a simple period over $G_\al$ then any word $A^n$ is reduced over $G_\al$.
Proposition \ref{pr:reduced-nontrivial} implies that $A$ has
infinite order in $G_\al$.


\begin{definition}  \label{df:activity-rank}
Let $A$ be a simple period over $G_\al$.
The {\em activity rank} of $A$ is the maximal rank $\be$ 
such that an $A$-periodic word contains a fragment $K$ of rank $\be\ge 1$ 
with $\muf(K) \ge \xi_1$ or it is $0$ if no such fragments exist.
\end{definition}

\subsection{Case of activity rank 0} \label{ss:activity-rank-0}
The arguments below differ depending on whether the activity rank $\be$
of a simple period over $G_\al$ is positive or $0$. However, the difference is only that 
in the case $\be\ge 1$ we use various conditions on the size $\muf(!F)$ of fragments $!F$ of 
rank $\be$. {\em All definitions, statements and proofs in Sections \ref{s:coarsely-periodic} and \ref{s:overlapping-periodicity} apply
in cases when the activity rank $\be$ of a simple period over $G_\al$ is 0 
simply ignoring conditions of the form $\muf(\cdot) \ge \dots$ for fragments of rank ~$\be$
(i.e.\ assuming that these conditions are all formally true in case $\be=0$).}
Below we do not distinguish this special case $\be=0$.

\medskip
We will use the following notations. If $!K$ and $!M$ are fragments of 
the same rank $0 \le \be \le \al$ occurring in a reduced path $!X$ in $\Ga_\ga$
then $!K \lesssim !M$ means $!K < !M$ or $!K \sim !M$;
similarly, $!K \lnsim !M$ means $!K < !M$ and $!K \not\sim !M$ .
Note that by Corollary \ref{coi:compatibility-order-alpha}, 
for fragments $!K$, $!M$ of rank $\be\ge1$ with $\muf(!K),\muf(!L) \ge \ga + 2.6\om$
the relation `$!K \lesssim !M$' depends only on their equivalence classes with respect to compatibility. Thus, for fixed $!X$ and $\be$
it induces the linear order on the set of equivalence classes of `$\sim$'
of fragments $!N$ of rank $\be$ in $!X$ with $\muf(!N) \ge \ga + 2.6\om$.
(In case $\be=0$ 
relation $!K \lesssim !M$ is defined on subpaths on length 1 and means $!K \ll !M$ or $!K = !M$.)

\begin{definition} \label{df:coarsely-periodic-segment}
Let $A$ be a simple period over $G_\al$ and $\be$ the activity rank of $A$.
A reduced path $!S$ in $\Ga_\al$ is a {\em coarsely periodic segment with period $A$} (or a {\em coarsely $A$-periodic segment} for short)
if there exists a path $!P$ labeled by an $A$-periodic word, 
fragments $!K_0$, $!K_1$ of rank ~$\be$ in $!P$ 
and fragments $!M_0$, $!M_1$ of rank $\be$ in $!S$ such that:
\begin{itemize}
\item 
$!P$ starts with $!K_0$ and ends with $!K_1$; $!S$ starts with $!M_0$ and ends with $!M_1$;
\item
$!K_0 \sim !M_0^{\pm1}$, $!K_1 \sim !M_1^{\pm1}$ and $!K_0 \not\sim !K_1$;
\item 
$\muf(!K_i) \ge \xi_1$,  $\muf(!M_i) \ge \xi_2$ ($i=0,1$); 
\item
$s_{A,!P} !K_0 \lesssim !K_1$ (informally, $!P$ ``contains at least one period $A$'').
\end{itemize}
The path $!P$ is a {\em periodic base} for $!S$. The infinite $A$-periodic 
extension of $!P$ is an {\em axis} for $!S$.

Note that the starting fragment $!M_0$ and the ending fragment $!M_1$ of $!S$ are defined up 
to compatibility.

Note also that by Lemma \ref{lmi:compatible-close} and Proposition \ref{pr:lifting-bigon}, 
$!P$ and $!S$ are close in rank $\be$. In particular, if $\be=0$ then $!P=!Q$ and 
thus $!P$ is an $A$-periodic segment.

We will be assuming that a coarsely $A$-periodic segment is always considered with a fixed associated axis. (In fact, we prove later that the axis of a coarsely $A$-periodic segment 
is defined in a unique way, see Corollary \ref{co:compatible-periodic-axes}).
Note that under this assumption, 
the periodic base $!P$ for $!S$ is defined up to changing the starting and the ending 
fragments $!K_0$ and $!K_1$ of rank $\be$ with compatible ones. 

The label of a coarsely $A$-periodic segment in $\Ga_\al$ is a {\em coarsely $A$-periodic word over $G_\al$}.

Note that a simple period $A$ over $G_0$ is any cyclically freely reduced word that is not a proper power.
A coarsely $A$-periodic word over $G_0$ is simply any $A$-periodic word $P$ with $|P| > |A|$.
\end{definition}

\begin{definition} \label{df:coarsely-periodic-size}
We measure the size of a coarsely $A$-periodic segment $!S$,
which roughly corresponds to the number of periods $A$, 
in the following way. 
Let $!P$ be the periodic base for ~$!S$ and $!K_0$, $!K_1$ as in Definition \ref{df:coarsely-periodic-segment}.
Then we write $\ell_A(!S) = t$ where $t$ is the maximal integer such that 
$s_{A,!P}^t !K_0 \lesssim !K_1$.
Thus, we always have $\ell_A(!S) \ge 1$.

Since we consider a fixed associated axis for $!S$, the number $\ell_A(!S)$
does not depend on the choice of a periodic base $!P$.

If $S$ is a coarsely $A$-periodic word over $G_\al$ then we formally 
define $\ell_A(S)$ to be
the maximal possible value of $\ell_A(!S)$ where $!S$ is a coarsely $A$-periodic segment labeled ~$S$.
\end{definition}

\begin{remark}
(i) It immediately follows from the definition that $t$ is also the maximal integer such that 
$!K_0 \lesssim s_{A,!P}^{-t}  !K_1$. Thus, $\ell_A(!S) = \ell_{A^{-1}} (!S^{-1})$.

(ii) To compute $\ell_A(S)$
we have to take a path $!S$ in $\Ga_\al$ with $\lab(!S) \greq S$ and then choose
a periodic base $!P$ for $!S$ so that $\ell_A(!S)$ is maximal possible;
it will follow from Proposition \ref{pr:strictly-close-parallel}
that any choice of $!P$ gives in fact the same value $\ell_A(!S)$.
\end{remark}

\begin{remark} \label{rm:base-whole-periods}
Up to changing the periodic base $!P$,
we can always assume in Definition \ref{df:coarsely-periodic-size} \
that both $!K_0$ and its translation $s_{A,!P}^t !K_0$ occur in $!P$.
In this case we have $|!P| \ge \ell_A(!S) |A|$.
\end{remark}

\begin{definition} \label{df:coarsely-periodic-compatibility}
Let $!S_1$ and $!S_2$ be coarsely $A$-periodic segments in $\Ga_\al$.

We say that $!S_1$ and $!S_2$ are {\em compatible} if 
they have the same axis
and {\em strongly compatible} if they share a common periodic base.

We use notations $!S_1 \sim !S_2$ and $!S_1 \approx !S_2$ for compatibility and strong compatibility
respectively.
\end{definition}

Note that in the case $!S_1 \approx !S_2$ any periodic base for $!S_1$ is a periodic base for $!S_2$ and vice versa. This easily follows from Definition \ref{df:coarsely-periodic-segment}.

If $!S_1$ and $!S_2$ are coarsely $A$-periodic segments in $\Ga_0$ then $!S_1 \sim  !S_2$
if and only if they have a common periodic extension and $!S_1 \approx !S_2$ if and only if $!S_1 = !S_2$.

\begin{proposition}
Let $!S_1$ and $!S_2$ be coarsely $A$-periodic segments in $\Ga_\al$.
\begin{enumerate}
\item 
If $!S_1 \approx !S_2$ then $\ell_A(!S_1) = \ell_A(!S_2)$.
\item \label{pri:coarsely-periodic-union}
Assume that $!S_1$ and $!S_2$ occur in a reduced path $!X$
in $\Ga_\al$ and $!S_1 \sim !S_2$. Then the union of $!S_1$ and $!S_2$ in $!X$ is an $A$-coarsely periodic segment
where a periodic base for $!S_1 \cup !S_2$ is the union of periodic bases f
or $!S_1$ and $!S_2$
in their common infinite $A$-periodic extension.
\end{enumerate}
\end{proposition}

\begin{proof}
(i) is immediate consequence of Definition \ref{df:coarsely-periodic-compatibility}.

(ii) follows from Proposition \ref{pri:fragment-ordering-between-fragments}.
\end{proof}

\subsection{} \label{ss:coarse-periodic-cutting}
We describe a procedure of shortening
a coarsely $A$-periodic segment $!S$ by a ``given number $k$ of periods''.
Let $k\ge 1$ and $\ell_A(!S) \ge k+1$.
Let $\be$ be the activity rank of $!S$, let $!P$ a periodic base for $!S$
and let $!K_i$ and $!M_i$ $(i=0,1)$ be starting and ending fragments of rank ~$\be$ 
of ~$!P$ and $!S$ respectively as
in Definition \ref{df:strictly-close}.
We have $!K_0 < s_{A,!P}^k !K_0 \lesssim  s_{A,!P}^{-1} !K_1 < !K_1$ and it follows from 
Proposition \ref{pri:fragment-in-periodic} that $s_{A,!P}^k !K_0 \not\sim !K_0$
and $s_{A,!P}^k !K_0 \not\sim !K_1$.
By Proposition \ref{pri:fragment-correspondence-between-fragments} 
there exists a fragment $!N$ of rank $\be$ 
in $!S$ with $\muf(!N') \ge \xi_2$ such that $s_{A,!P}^k !K_0 \sim !N^{\pm1}$.
Then $!S_1 = !N \cup !M_1$ is an end of $!S$ which is a coarsely $A$-periodic segment with periodic
base $!P_1 = s_{A,!P}^k !K_0 \cup !K_1$ and $\ell_A(!S_1) = \ell_A(!S) - k$.
We note that:
\begin{enumerate}
\item 
The result of the operation is defined up to the strict compatibility.
\item \label{ssi:cutting-periods-count}
We have $!P = !X !P_1$ where $|!X| = k|A|$.
\item \label{ssi:cutting-dividing}
If $k\ge 2$ then 
by Proposition \ref{pri:fragment-correspondence-between-fragments} we can find also 
a fragment $!N'$ of rank ~$\be$
in $!S$ with $\muf(!N') \ge \xi_2$ such that $s_{A,!P}^{k-1} !K_0 \sim !N'^{\pm1}$ and $!N'$ and $!N$
are disjoint. Then $!S = !S_0 !u !S_1$ where $!S_0 = !M_0 \cup !N'$ is a coarsely $A$-periodic segment 
with periodic base $!K_0 \cup s_{A,!P}^{k-1} !K_0$ and $\ell_A(!S_0) = k-1$.
\item \label{ssi:cutting-invariance}
The starting position of $!S_1$ depends only on the starting position 
of $!S$; more precisely, if $!S'$ is a start of $!S$ 
and $!S_1$ and $!S_1'$ are obtained from $!S$ and $!S'$ as above
then $!S_1'$ is a start of $!S_1$ up to strict compatibility of $!S_1'$;
if $!S \approx !S'$ then $!S_1 \approx !S_1'$.
\end{enumerate}

\begin{definition} \label{df:t-cut}
If $!S_1$ is obtained from $!S$ by the procedure 
in \ref{ss:coarse-periodic-cutting}
then we say that $!S_1$ is obtained by {\em shortening of $!S$ by $t$ periods from the start}.
In the symmetric way, we define  {\em shortening of $!S$ by $t$ periods from the end}.

If $\ell_A(!S) \ge 2t+1$ and $!S'$ is obtained from $!S$ by applying the operation from both sides
then $!S'$ is the result of {\em truncation of $!S$ by $t$ periods}.
\end{definition}

\begin{definition} \label{df:period-stability-parameter} \label{df:stable-part}
We define two numeric parameters associated with a simple period $A$ over $G_\al$: 
the {\em stable size $[A]_\al$ of $A$ in rank $\al$},
$$
  [A]_\al = \inf_{m \ge 1} \frac{|(A^m)^\circ|_\al}{m}
$$
and {\em the stability decrement $h_\al(A)$}: 
$$
  h_\al(A) = \bbceil{\frac{1.2}{[A]_\al}} + 1.
$$
If $\ell_A(!S) \ge 2h_\al(A)+1$ then the result of truncation of $!S$ by $h_\al(A)$ periods 
is {\em the stable part of ~$!S$}.
By claim \ref{ssi:cutting-invariance} and its symmetric version,
the function `$!S \to \text{stable part of } !S$' respects
strict compatibility: if $!S_1 \approx !S_2$ 
and $!S_i^*$ is the stable part of $!S_i$ then $!S_1^* \approx !S_2^*$.
\end{definition}

The basic fact about $[A]_\al$ and $h_\al(A)$ is the following observation.

\begin{lemma} \label{lm:stability-parameters}
If $X$ is an $A$-periodic word and $|X| \ge m|A|$ then $|X|_\al \ge m [A]_\al$.
In particular, if $|X| \ge (h_\al(A) - 1)|A|$ then $|X|_\al \ge 1.2$.
\end{lemma}

\begin{proof}
We have
$$
    |X|_\al \ge |A_1^m|_\al \ge |(A^m)^\circ|_\al \ge m [A]_\al
$$
where $A_1$ is the cyclic shift of $A$ at which $X$ starts. The second statement follows from the first.
\end{proof}

The principal role of the stable part is described by the following proposition.

\begin{proposition}[stability of coarsely periodic words] 
\label{pr:coarse-periodic-stability}
Let $!S$ be a coarsely $A$-periodic segment in $\Ga_\al$ with 
$\ell_A(!S) \ge 2h_\al(A)+1$ and let $!S^*$ be the stable part of $!S$.
If $!X$ and $!Y$ are close reduced paths in $\Ga_\al$ and $!S$ is a subpath of $!X$
then $!Y$ contains a coarsely $A$-periodic segment $!T$ such that $!T \approx !S^*$.
\end{proposition}

\begin{proof}
Let $!P$ and $!P^*$ be periodic bases for $!S$ and $!S^*$ respectively. 
Let $\be$ be the activity rank of $A$ and
let $!K_i$ and $!M_i$ $(i=0,1)$ be fragments of rank $\be$ in $!P$ and in $!S$, respectively, from
Definition \ref{df:strictly-close} applied to $!P$ and $!S$. Denote $t = h_\al(A)$.

Let $!X$ and $!Y$ be as in the proposition.
If $\al=0$ then $!X = !Y$ and there is nothing to prove. Let $\al>0$.
We claim that $!P = !z_1 !P' !z_2$ where $!P'$ is close in rank ~$\be$ 
to a subpath of $!Y$ and $|!z_i|_\al < 1.2$. 
Indeed, if $\be=\al$ then it 
easily follows from Proposition ~\ref{pr:fragment-stability-bigon} 
and Lemma \ref{lmi:compatible-close}
that $!P$ is already close to a subpath of $!Y$.
If $\be < \al$ then we observe that $!S$ contains
no fragments $!K$ of rank $\ga$ with $\be < \ga \le \al$
and $\muf(!K) \ge \xi_0$ due to the definition of the activity rank and
Proposition \ref{pr:fragment-stability-previous}$_{\le\al}$.
Then the claim follows by Proposition \ref{pr:closeness-stability-beta}.

By Lemma \ref{lm:stability-parameters} we have $|!z_i| < (t-1) |A|$.
This implies that $s_{A,!P}^{t-1} !K_0 \cup s_{A,!P}^{-t+1} !K_1$ is contained in $!P'$.
Note that $!P^* = s_{A,!P}^{t} !K_0 \cup s_{A,!P}^{-t} !K_1$ where $\muf(!K_0), \muf(!K_1) \ge \xi_1$.
Then by Proposition \ref{pri:fragment-correspondence-between-fragments} we find
a subpath $!T$ which is
a coarsely $A$-periodic segment with periodic base $!P^*$ and, consequently, we have $!T \approx !S^*$.
\end{proof}

We use parameter $h_\al(A)$ also in several other situations.

\begin{proposition} \label{pr:coarsely-periodic-from-close}
Let $!P$ be a periodic segment in $\Ga_\al$ with a simple period $A$ over $G_\al$.
Assume that $|!P| \ge m |A|$ where $m \ge 2h_\al(A)+3$. 
Let $!X$ be a reduced path in $\Ga_\al$ such that $!P$ and $!X$ are close.
Then there exist a subpath $!P_1$ of $!P$ and a subpath $!X_1$ of $!X$ such that 
$!X_1$ is a coarsely $A$-periodic segment with periodic base $!P_1$ 
and $\ell_A(!X_1) =  m - 2h_\al(A)-2$. 
\end{proposition}

\begin{proof}
Let $\be$ be the activity rank of $A$.
Using Corollary \ref{co:no-active-fragments-iterated} and Lemma \ref{lm:stability-parameters}
we find close in rank $\be$ subpaths $!P_2$ of $!P$ and $!X_2$ of $!X$ with $|!P_2| \ge m - 2h_\al(A) +2$.
By Proposition \ref{pri:fragment-in-periodic-small} any fragment $!K$ of rank $\be$ 
in $!P$ with $\muf(!K) \ge 2\la + 5.3\om$
satisfies $|!K| < 2|A|$, so according to 
Definition \ref{df:coarsely-periodic-segment} 
there exists a fragment $!K$ of rank $\be$ in $!P$ with $\muf(!K) \ge \xi_1$.
Shortening $!K$ from the end by Proposition \ref{pr:small-overlapping} if $\be \ge 1$ and 
using again Proposition \ref{pri:fragment-in-periodic} 
we find a fragment ~$!K_1$ of rank $\be$ with $\muf(!K_1) > \xi_1- \la-2.7\om$ that is a start of $!K$
disjoint from $s_{A,!P} !K$; hence $|!K_1| \le |A|$.
We can assume that $!K$ occurs in $!P_2$ and is 
closest to the start of $!P_2$. Then
$!P_2$ contains $m - 2h_\al(A)$ translates $s_{A,!P}^i !K$ of $!K$ 
for $i=0,\dots,m - 2h_\al(A)-1$ and 
contains also $s_{A,!P}^{m - 2h_\al(A)} !K_1$.
Applying Proposition  \ref{pri:fragment-correspondence-between-fragments} we find 
fragments $!M_i$ $(i=1,\dots,m - 2h_\al(A)-1)$ of rank $\be$
in $!X_2$ with $\muf(!M_i) \ge \xi_2$ such that $s_{A,!P}^i !K \sim !M_i^{\pm1}$.
Then $!X_1 = !M_1 \cup !M_{m - 2h_\al(A)-1}$ is a coarsely $A$-periodic segment with periodic base
$s_{A,!P} !K \cup s_{A,!P}^{m - 2h_\al(A)-1} !K$ and we have
$\ell_A(!X_1) = m - 2h_\al(A)-2$.
\end{proof}

\begin{proposition} \label{pr:changing-period}
Let $S$ be a coarsely $A$-periodic word over $G_\al$ and $B$ a simple period over $G_\al$ conjugate to $A$.
Let $\ell_A(S) \ge 2h_\al(A) + 3$.
Then a subword $T$ of $S$ is a coarsely $B$-periodic word over $G_\al$ with 
$\ell_B(T) \ge \ell_A(S) - 2h_\al(A) - 2$.
\end{proposition}

\begin{proof}
We represent $S$ by a coarsely $A$-periodic segment ~$!S$ in $\Ga_\al$. 
Let $!P$ a periodic base for ~$!S$,
let $!L_1$ be the axis of $!S$ and let 
$!L_2$ be the $B$-periodic line parallel to $!L_1$.
Denote $\be_1$ and ~$\be_2$ activity ranks of $A$ and $B$ respectively.

According to Definition \ref{df:activity-rank}, 
either $!L_1$ or $!L_2$ contains 
no fragments $!K$ of rank $\ga$ with $\be_1 < \ga \le \al$
and $\muf(!K) \ge \xi_1$.
Let $!K_0$ and $!K_1$ be fragments of rank $\be_1$ 
with $\muf(!K_i) \ge \xi_1$ that are a start and an end of $!P$ respectively.
We have $s_{A,!L_1}^{\ell_A(!S)} !K_0 \lesssim !K_1$.
By Proposition \ref{pri:fragment-correspondence-cyclic-beta}, there exist 
fragments $!M_0$ and $!M_1$ of rank $\be_1$ 
in $!L_2$ with $\muf(!M_i) \ge \xi_2$ such that 
$!K_i \sim !M_i^{\pm1}$.
Since $!L_1$ and $!L_2$ are parallel, we have $s_{A,!L_1} = s_{B,!L_2}$ and hence 
$s_{B,!L_2}^{\ell_A(!S)} !M_0 \lesssim !M_1$ by 
Proposition \ref{pri:fragment-correspondence-cyclic-ordering}.
Then  $!Q = !M_0 \cup s_{B,!L_2}^{\ell_A(!S)} !M_0 \cup !M_1$ is close in rank $\be_1$ to $!P$,
 $|!Q| \ge \ell_A(!S)$  and the statement follows by
Proposition \ref{pr:coarsely-periodic-from-close}.
\end{proof}

\section{Overlapped coarse periodicity} \label{s:overlapping-periodicity}

The main result of this section is
Proposition \ref{pr:coarsely-periodic-overlapping-general} which can be thought as an analog
of a well known property of periodic words: if two periodic words
have a sufficiently large overlapping then they have a common period. 
We need such an analog in 
a more general context where closeness plays the role of overlapping.
As a main technical tool, instead of coincidence of letters in the overlapping case we use correspondence of fragments of rank $\be\le\al$ in strictly close in  rank $\be$ segments in ~$\Ga_\al$ given by
Proposition ~\ref{pr:fragment-correspondence-between-fragments}.
A difficulty is caused by the ``fading effect'' of this correspondence: a fragment size can decrease
when passing from one segment to the other.
To overcome this difficulty, we use a special combinatorial argument  \cite[Lemma ~6.4]{Lys96}.

\begin{lemma}[penetration lemma, {\cite[Lemma 6.4]{Lys96}}] \label{lm:penetration-lemma}
Let $S_0$, $S_1$, $\dots$, $S_k$ be a finite collection of disjoint sets. 
Assume that the following assertions hold:
\begin{enumerate}
\item
Each $S_i$ is pre-ordered, i.e.\ endowed with a transitive relation `$<_i$'.
\item 
There is an equivalence relation $a \sim b$ on the union $\bigcup_i S_i$
such that for any $a,b$ in the same set $S_i$ we have either $a <_i b$, $b <_i a$ or $a \sim b$;
in other words, we have an induced linear ordering on the set of equivalence classes on each $S_i$.

\item
We assume that the equivalence preserves the pre-ordering 
in neighboring sets: if $a,b \in S_i$, $a',b' \in S_{i+1}$, $a \sim a'$ and $b \sim b'$
then $a <_i b \lequiv a' <_{i+1} b'$.

If $c \in S_i$, $a,b \in S_j$ and $a \lesssim_j b$ (where $a \lesssim_j b$ denotes `$a <_j b$ or $a \sim b$')
then we say that $c$ {\em penetrates} between $a$ and $b$ if there exists $c' \sim c$
such that $a \lesssim_j c' \lesssim_j b$.
\item
There is a subset of $\bigcup_i S_i$ of {\em stable} elements that have the following property: 
if $c \in S_i$ is stable, $a \lesssim_i c \lesssim_i b$, $a', b' \in S_j$, $a' \lesssim_j b'$, 
$a \sim a'$ and $b \sim b'$ then $c$ penetrates between $a'$ and $b'$.
\item
For each $i \le k-1$, there are stable elements $a_i, b_i \in S_i$ and $a_i',b_i' \in S_{i+1}$ 
such that $a_i \sim a_i'$, $b_i \sim b_i'$ and $a_i <_i b_i$.
\end{enumerate}

Finally, let $c_0 \in S_0$ be stable and $a_0 \lesssim_0 c_0 \lesssim_0 b_0$. Assume that $c_0$ 
penetrates between $a_i$ and $b_i$ for each $i=1,2,\dots,k-1$.
Then $c_0$ penetrates between $a_k$ and $b_k$.
\end{lemma}

%

The following observation is a special case of \cite[Lemma 6.2]{Lys96}.

\begin{lemma} \label{lm:overlapping-by-action}
Suppose a group $G$ acts on set $X$. Let $g,h\in G$, $x_0,x_1,\dots,x_t \in X$ and for some $r,s \ge 0$
with $\gcd(r,s) = 1$ and $r+s \le t$,
$$
  gx_i = x_{i+r} \ (i=0,1,\dots,t-r), \quad hx_i = x_{i+s} \ (i=0,1,\dots,t-s).
$$
Assume that the stabilizer $H$ of $x_0$ is malnormal in $G$. 
Then either $g,h \in H$ (and hence $x_0=x_1=\dots=x_t$) or there exists 
$d \in G$ such that $g = d^r$ and $h = d^s$.
\end{lemma}

\begin{proof}
Induction on $r+s$. We can assume that $r \le s$. If $r > 0$ then we have 
$g^{-1} h x_i = x_{i+s-r}$
for $0 \le i \le t-s$ and the statement follows from the inductive hypothesis with $h: = g^{-1} h$,
$s := s-r$ and $t := t-r$. Otherwise we have $r=0$ and $s=1$. 
Then $h^{-1} g h x_0 = g x_0 = x_0$
and by malnormality of $H$, we have either $g,h \in H$ or $g = 1$ 
(and then $g = h^0$ and $h = h^1$).
\end{proof}



\begin{definition} \label{df:strictly-close}
Let $!X$ and $!Y$ be reduced paths in $\Ga_\al$. 
We say that $!X$ and $!Y$ are {\em strictly close in rank $\be \le \al$}
if there are fragments $!K_0$, $!K_1$ of rank $\be$ in $!X$ and fragments $!M_0$, $!M_1$ of rank $\be$ in $!Y$ such that:
\begin{itemize}
\item
$\muf(!K_i), \muf(!M_i) \ge \xi_2$ ($i=0,1$). 
\item 
$!X$ starts with $!K_0$ and ends with $!K_1$; $!Y$ starts with $!M_0$ and ends with $!M_1$;
\item
$!K_0 \sim !M_0^{\pm1}$, $!K_1 \sim !M_1^{\pm1}$ and $!K_0 \not\sim !K_1$.
\end{itemize}
\end{definition}

By Lemma \ref{lmi:compatible-close}, paths which are strictly close in rank $\be$ are also 
close in rank $\be$. 
One of the advantages of strict closeness is that this relation is transitive (this follows
immediately from Definition \ref{df:strictly-close}).
Note that a coarsely periodic segment $!P$ in $\Ga_\al$ and its periodic base $!S$
are strictly close according to Definition \ref{df:coarsely-periodic-segment}
(and the condition in Definition \ref{df:coarsely-periodic-segment} is slightly stronger
because of the lower bound on the size of the starting and the ending fragments of ~$!S$).

\begin{proposition}  \label{pr:coarsely-periodic-overlapping-general} 
Let $A$ be a simple period over $G_\al$, $\be$ the activity rank of $A$ and 
$!P_i$ $(i=0,1)$ be two $A$-periodic segments in $\Ga_\al$.
Let $!S_i$ $(i=0,1)$ be a reduced path in $\Ga_\al$ 
which is strictly close to ~$!P_i$.
Assume that $!S_0$ is contained in $!S_1$.
Assume also that $!P_0$ contains at least one period $A$ in the sense 
that there exist fragments $!K$ and $!K'$ of rank $\be$ 
in ~$!P_0$ such that $\muf(!K), \muf(!K') \ge \xi_2$
and $!K' \sim s_{A,!P_0} !K$.
Then $!P_0$ and $!P_1$ have a common periodic extension.
\end{proposition}

\begin{proof}
Denote
$$
    \xi_3 = \xi_2 - 2\la - 3.4\om = 3\la - 10.9 \om.
$$
Throughout the proof, ``fragment $!M$'' means 
``fragment $!M$ of rank $\be$ with $\muf(!M) \ge \ze_3$'' 
(or simply  ``fragment $!M$ of rank 0'' if $\be=0$, see \ref{ss:activity-rank-0}).

Let a line $!L_i$ be the infinite periodic extension of $!P_i$ and let $g$ be an element
of $G_\al$ such that $!L_1 = g !L_0$, so $s_{A,!P_1} = g s_{A,!P_0} g^{-1}$.
Our argument relies on establishing a correspondence between 
fragments of rank $\be$ in $!P_i$ and $!S_i$.
It will be convenient to consider fragments of rank $\be$
in four paths $!P_i$ and ~$!S_i$ as four disjoint
sets, i.e.\ we will formally consider pairs $(!M,!X)$ where 
$!X \in \set{!P_0,!P_1,!S_0,!S_1}$
and $!M$ is a fragment occurring in $!X$. 
We will refer to $!M$ as a 
``fragment belonging to $!X$'' or simply as a ``fragment in $!X$''.

We introduce two operations on fragments in $!P_i$ and $!S_i$.
Let $!M$ and $!N$ be fragments 
each belonging to some $!P_i$ or ~$!S_i$. 
\begin{enumerate}
\item 
If $!M$ belongs to $!P_i$, $!N$ belongs to $!S_i$ and $!M \sim !N^{\pm1}$ 
then either of $!M$ and $!N$ {\em jumps} to the other.
\item
$!M$ {\em translates} to $!N$ in the following cases (a)--(d):
\begin{enumerate}
\item
$!M$ and $!N$ belong to the same $!P_i$ and $!N \sim s_{A,!P_i}^k !M$ for some $k \in \Z$; or
\item
$!M$ belongs to $!P_0$, $!N$ belongs to $!P_1$ and $!N \sim g s_{A,!P_0}^k !M$ for some $k \in \Z$; or
\item
$!M$ belongs to $!P_1$, $!N$ belongs to $!P_0$ and $!N \sim g^{-1} s_{A,!P_1}^k !M$ for some $k \in \Z$.
\end{enumerate}
(In other words, $!M$ translates to $!N$ in cases (a)--(c) if they have the same position in their corresponding 
periodic lines $!L_i$ with respect to the period $A$ up to compatibility.)
\begin{enumerate} \setcounter{enumii}{3}
\item 
An ``identical'' case: $!M \sim !N$ and they belong to some $!S_i$ and $!S_j$ respectively.
\end{enumerate} 
\end{enumerate}
Note that the two operations are reversible and are defined up to compatibility.

Let $!K$ and $!K'$ be fragments in $!P_0$ such that $\muf(!K), \muf(!K') \ge \xi_1$ and
$!K' \sim s_{A,!P_0} !K$, as assumed in the proposition.
Let $\cM$ be a maximal set of pairwise non-compatible fragments which can be obtained
by operations (i) and (ii) starting from $!K$. 
By Proposition \ref{pr:inclusion-compatibility},
neither of any two fragments in $\cM$ is contained in the other, so $\cM$ is a finite set.

The following assertion is the principal step of the proof.
\begin{claim*}
The jump operation is always possible inside $\cM$; 
that is, for any $!M \in \cM$ in ~$!P_i$ or in $!S_i$, $i \in \set{0,1}$, there 
exists a fragment $!N$ of rank $\al$ in $!S_i$ or, respectively, in $!P_i$
such that $!M \sim !N^{\pm1}$.
\end{claim*}

{\em Proof of the claim.}
We assume that some $!M \in \cM$ is given and prove existence of the required ~$!N$.
The proof will consist of application of Lemma \ref{lm:penetration-lemma}.
We do a necessary preparation.

According to the definition of $\cM$, there is a sequence $!T_0 = !K$, 
$!T_1$, $\dots$, $!T_l = !M$
of fragments $!T_j \in \cM$ such that $!T_{j+1}$ is obtained from $!T_j$ by one of the operations (i) or (ii).
We can assume that the sequence has no two translations in a row (otherwise we can replace them by a single
translation) and has no two jumps in a row (otherwise they eliminate).
Assume also for convenience that $!T_0 \to !T_1$ is a translation (by inserting a trivial translation if needed).
Thus for each $i$, $!T_{2j}$ translates to $!T_{2j+1}$ and $!T_{2j+1}$ jumps to $!T_{2j+2}$.
We can assume that the last step $!T_{l-1} \to !T_l$ is a translation, so $l = 2k-1$ for some $k$. 

Now roughly speaking, we move all fragments $!T_j$ 
along with the corresponding paths $!P_i$ or $!S_i$ belonging
them, to the same location up to compatibility. 
We define a sequence $!Y_0$, $!Y_1$, $\dots$, $!Y_k$ of paths in $\Ga_\al$
and a sequence $!W_j$ of fragments in $!Y_j$ for $j=0,1,\dots,k-1$. 
For each ~$j$ we will have $!W_j = f_j !T_{2j+1}$ for some $f_j \in G_\al$.
The definition of $!Y_j$ and $f_j$ goes as follows.

Denote $(!X_1,!X_2,!X_3,!X_4) = (!P_0,!S_0,!P_1,!S_1)$ and let $J(i)$ denote the index such that
a fragment in $!X_i$ jumps to a fragment in $!X_{J(i)}$ (i.e.\ $(J(1),J(2),J(3),J(4)) = (2,1,4,3)$).
Denote also ~$I(j)$ the index such that $!T_{2j-1}$ belongs to $!X_{I(j)}$. 
Thus, $!T_{2j}$ belongs to $!X_{J(I(j))}$.

We start with $!Y_0 = !X_{I(0)}$ and $!W_0 = !T_1$, so $f_0 = 1$.
Assume that $j < k-1$ and $!Y_j$ and $f_j$ are already defined.
If $!T_{2j} \to !T_{2j+1}$ is a translation by (a)--(c) then there exists $f_{j+1} \in G_\al$ such that 
$f_{j+1} !X_{I(j+1)}$ and $f_j !X_{J(I(j))}$
belong to the same $A$-periodic line and $f_{j+1} !T_{2j+1} \sim f_j !T_{2j}$.
We take $!Y_{j+1} =  f_{j+1} !X_{I(j+1)} \cup f_j !X_{J(I(j))}$.
Otherwise $!T_{2j} \to !T_{2j+1}$ is a translation by (d), i.e.\ $!X_{J(I(j))}$ is either $!S_0$ or $!S_1$.
In this case we take $f_{j+1} = f_j$ and $!Y_{j+1} = f_j !S_1$.
Finally, define $!Y_k = f_k !X_{J(I(k-1))}$. 
We have $f_{j+1} !T_{2j+2}^{\pm1} \sim f_{j+1} !T_{2j+1} \sim f_j !T_{2j}$ for all $j=0,1,\dots,k-2$
and hence $!W_0 \sim !W_1^{\pm1} \sim \dots \sim !W_{k-1}^{\pm1}$.
Figure \ref{fig:coarsely-periodic-overlapping} 
illustrates the construction.
\begin{figure}[h]
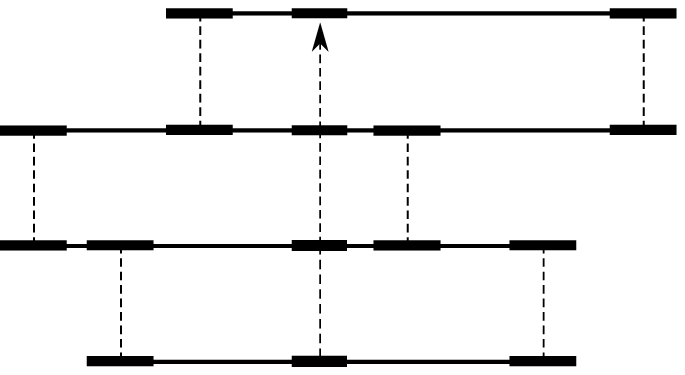
\caption{}  \label{fig:coarsely-periodic-overlapping}
\end{figure}

By strict closeness of pairs $(!P_0,!S_0)$ and $(!P_1,!S_1)$, each 
$!X_i$ starts with a fragment $!U_i$ and ends with a fragment $!V_i$ 
such that $\muf(!U_i), \muf(!V_i) \ge \xi_2$, 
$!U_i \not\sim !V_i$ and we have $!U_i \sim !U_{J(i)}^{\pm1}$ and $!V_i \sim !V_{J(i)}^{\pm1}$

We now apply Lemma \ref{lm:penetration-lemma} where:
\begin{itemize}
\item 
$S_j$ is the set of all fragments $!N$ in $!S_j$ with $\muf(!N) \ge \xi_3$.
\item
$!N <_i !N'$ is defined as `$!N \not\sim !N'$ and $!N < !N'$ in $!S_j$'.
\item
Equivalence of $!N,!N' \in \bigcup_j S_j$ is defined as $!N \sim !N'^{\pm1}$.
\item
$!N \in \bigcup_j S_j$ is defined to be stable iff $\muf(!N) \ge \xi_2$.
\item
For $a_j$, $b_j$, $a_j'$ and $b_j'$ we take appropriate translates of $!U_i$ and $!V_i$, namely,
$f_j !U_{I(j)}$, $f_j !V_{I(j)}$, $f_j !U_{J(I(j))}$ and $f_j !V_{J(I(j))}$ respectively.
\end{itemize}
We have conditions (i)--(v) of Lemma \ref{lm:penetration-lemma} satisfied: 
condition (i) holds in case $\be\ge1$ by Corollary \ref{coi:compatibility-order-alpha}, 
condition (ii) holds by Proposition \ref{pr:inclusion-compatibility}$_\be$,
conditions (iii) and (iv) hold by Proposition \ref{pr:fragment-correspondence-between-fragments}
in view of the inequality
$
  \xi_3 \ge 2\la + 9.1\om
$
and, finally, condition (v) holds immediately by construction. 

For $c_0$ in Lemma \ref{lm:penetration-lemma} we take $!T_1$.
Note that up to compatibility, we can assume that $\muf(!T_1) \ge \xi_2$, so $!T_1$ is stable.
(By construction, $!T_1$ is obtained from $!T_0 = !K$ 
by translation to ~$!X_{I(0)}$;
if $!T_1$ is compatible with the starting or the ending fragment of $!X_{I(0)}$
then we can assume $\muf(!T_1) \ge \xi_2$ due to Definition \ref{df:strictly-close};
otherwise we can assume that $!T_1$ is a literal translation of ~$!K$ and then 
$\muf(!T_1) = \muf(!K) \ge \xi_1$.)
Since $!T_1 = !W_0 \sim !W_1^{\pm1} \sim \dots \sim !W_{k-1}^{\pm1}$ 
and each $!W_j$ occurs in $f_j !X_{I(j)}$, $!T_1$ penetrates between each pair $f_j !U_{I(j)}$ and $f_j !V_{I(j)}$
for $j=0,1,\dots,k-1$. All the hypotheses of Lemma \ref{lm:penetration-lemma} are satisfied
and applying it we find a fragment $!W_k$
in $f_{k-1} !X_{J(I(k-1))}$ such that $!W_{k}^{\pm1} \sim !W_{k-1} = f_{k-1} !M$. 
Then $!M \to f_{k-1}^{-1} !W_k$
is the required jump. This finishes the proof of the claim.

We finish the proof of the proposition. 
Let $!K_0 = !K$, $!K_1$, $\dots$, $!K_m \sim s_{A,!P_0} !K$ be all 
fragments in $\cM$ between
$!K$ and $s_{A,!P_0} !K$ in their natural order, i.e.\ we have 
$!K_0 < !K_1 < \dots < !K_m$.
 Let $!M_0, \dots, !M_m \in \cM$ be fragments in $!P_1$ such that $!M_i \sim !K_i^{\pm1}$ for all $i$
(each $!M_i$ is obtained from ~$!K_i$ by two jumps).
Note that $!M_0 < !M_1 < \dots < !M_m$ 
by Proposition \ref{pr:fragment-correspondence-between-fragments}.
Since $\cM$ is closed under translations, the number of fragments in $\cM$ between
$!M_0$ and $s_{A,!P_1} !M_0$ is the same as the number of fragments in ~$\cM$ 
between $!K$ and $s_{A,!P_0} !K$, 
i.e.\ we have $!M_m \sim s_{A,!P_1} !M_0$.
This implies that $!K_0$ translates to some ~$!M_q$, i.e.\ 
$!M_q \sim g s_{A,!P_0}^t !K_0^{\pm1}$ for some $t$
and hence 
$$
    !M_{i+q} \sim g s_{A,!P_0}^t !K_i^{\pm1} \ \text{for} \ i =0,1,\dots,m-q, \quad
    !M_{i+q-m} \sim g s_{A,!P_0}^{t-1} !K_i^{\pm1}  \ \text{for} \ i = m-q+1,\dots,m.
$$
Note that $\gcd(m.q) = 1$ since $\cM$ is generated by a single fragment $!K$.
By Propositions ~\ref{pri:fragment-finite-order}, \ref{pr:conjugate-relator-roots}
and Corollary \ref{coi:no-inverse-compatibility-alpha},
the subgroup $\setof{g \in G_\al}{g !M_0 \sim !M_0^{\pm1}}$ is malnormal in ~$G_\al$.
We now apply Lemma \ref{lm:overlapping-by-action} where for $x_i$ we
take the equivalence class of $!M_i$ in the set of fragments of rank $\be$
in $\Ga_\al$ under compatibility up to invertion.
By the lemma, $\sgp{g,s_{A,!P_0}}$ is cyclic. 
Since $A$ is a non-power, we get $g \in \sgp{s_{A,!P_0}}$
which means that $!L_1 = !L_2$.
\end{proof}

As an immediate consequence of Proposition \ref{pr:coarsely-periodic-overlapping-general} we get:

\begin{corollary} [overlapping coarse periodicity] 
\label{co:coarsely-periodic-overlapping}
Let $!S_0$ and $!S_1$ be coarsely periodic segments in $\Ga_\al$ with the same simple period $A$ over $G_\al$. 
If $!S_0$ is contained in $!S_1$ then $!S_0 \sim !S_1$.
\end{corollary}

\begin{corollary} \label{co:coarsely-periodic-cut}
Let $!S$ and $!T$ be non-compatible coarse periodic segments in $\Ga_\al$ with the same simple period $A$
which occur in a reduced path $!X$. Let $\ell_A(!S) \ge 3$. 
Assume that $!S_1$ is obtained from $!S$ by shortening by 2 periods from the end if $!S < !T$
or by shortening by 2 periods from the start if $!S > !T$. Then $!S_1$ and $!T$ are
disjoint.
\end{corollary}

\begin{proof}
Without loss of generality, we assume that $!S < !T$ and $!S_1$
is obtained from $!S$ by shortening by 2 periods from the end.
By \ref{ssi:cutting-dividing} we have $!S = !S_1 !u !S_2$ where $!S_2$ is a coarsely $A$-periodic segment with $!S_2 \sim !S$. 
B
y hypothesis we have $!S_2 \not\sim !T$ and 
then by Corollary \ref{co:coarsely-periodic-overlapping}, neither of $!S_2$ or $!T$
is contained in the other. This implies that $!S_1$ and $!T$ are disjoint.
\end{proof}

\begin{proposition} [strictly close periodic paths with one period] \label{pr:strictly-close-parallel}
Let $A$ be a simple period over $G_\al$ and $\be$ the activity rank of $A$.
Let $!P_0$ and $!P_1$ be strictly close in rank $\be$ paths in $\Ga_\al$ labeled by periodic words with period $A$.
Assume that there exist fragments $!K, !K'$ of rank $\be$ 
in ~$!P_0$ such that $\muf(!K), \muf(!K')  \ge \xi_2$
and $s_{A,!P_0} !K \sim !K'$.
Then $!P_0$ and $!P_1$ have a common periodic extension.
\end{proposition}

\begin{proof}
This is a special case of Proposition \ref{pr:coarsely-periodic-overlapping-general}
with $!S_0 = !S_1 = !P_1$.
\end{proof}

\begin{proposition} \label{pr:nontorsion-conjugation} 
Let $g \in G_\al$ be a non-power of infinite order and let $h \in G_\al$.
If $g^k = h^{-1} g^l h$ for some $k,l > 0$ then $h \in \sgp{g}$ and $k=l$.

\end{proposition}

\begin{proof}
By Proposition \ref{pr:cyclic-reduction}, up to conjugation we can assume that $g$ is represented by
a simple period $A$ over $G_\al$. 
It is enough to prove that $h \in \sgp{A}$.

Consider two periodic lines $!L_0$ and $!L_1$ in $G_\al$ with period $A$ which represent the conjugacy relation.
We have $h \in \sgp{A}$ if and only if $!L_0 = !L_1$.
Let $\be$ be the activity rank of $A$. 
By Proposition \ref{pr:fragment-correspondence-cyclic-beta} we find strictly close in rank $\be$ subpaths $!P_i$ of $!L_i$
with any desired bound $|!P_0| \ge t |A|$.
Then the statement follows from Proposition \ref{pr:strictly-close-parallel}.
\end{proof}

As an immediate consequence we get:

\begin{corollary} \label{co:compatible-periodic-axes}
Let $!S_0$ and $!S_1$ be coarsely $A$-periodic segments in $\Ga_\al$
and $!L_i$ $(i=1,2)$ be an axis for $!S_i$. If $!S_0 \sim !S_1$ then $!L_1 = !L_2$.
\end{corollary}

\begin{corollary} 
Let $g \in G_\al$ be an element of infinite order. Then the following is true.

\begin{enumerate}
\item \label{coi:root-unique} 
$g$ has the unique root; i.e.\
there exists a unique non-power element $g_0 \in G_\al$ such that $g = g_0^t$ for some $t \ge 1$.
\item \label{coi:nontorsion-relations} 
If $h^r \in \sgp{g}$ and $h^r \ne 1$ then $h \in \sgp{g_0}$ where $g_0$ is the root of $g$.
\item \label{coi:inverse-conjugate}
If $g$ is conjugate to $g^{-1}$ then $g$ is the product of two involutions.
\end{enumerate}
\end{corollary}

\begin{proof}
(i) is direct consequence of Propositions \ref{pr:root-existence} and \ref{pr:nontorsion-conjugation}.

(ii) follows from (i) and Proposition \ref{pr:nontorsion-conjugation} because $g_0^t = h^r$ implies 
$g_0^t = h^{-1} g_0^t h$.

(iii) Assume that $g = h^{-1} g^{-1} h$. From $g = h^{-2} g h^2$ we conclude that $h^2 = 1$ by (ii).
Similarly, we have $(hg)^2 = 1$ and then $g = h \cdot hg$.
\end{proof}

\begin{corollary} \label{co:no-involutions-no-inverse-conjugates}
Assume that each relator $R$ of each rank $\be \le \al$ has the form $R = R_0^n$ where $R_0$ is the root
of $R$ and $n$ is odd ($n$ can vary for different relators $R$). Then $G_\al$ has no involutions and no
element of $G_\al$ is conjugate to its inverse.
\end{corollary}

\begin{proof}
By Proposition \ref{pr:cyclic-reduction}, any element of finite order of $G_\al$ is conjugate to some power $R_0^t$
of the root $R_0$ of a relator $R$ of rank $\be \le \al$. By Proposition \ref{pr:relator-root}, $R_0^t$ has an odd 
order and cannot be an involution.
The second statement follows from the first by Corollary \ref{coi:inverse-conjugate}.
\end{proof}

\begin{lemma} \label{lm:coarsly-periodic-in-periodic}
Let $!P$ be an $A$-periodic segment in $\Ga_\al$ with a simple period $A$ over ~$G_\al$.
Let $!S$ be a coarsely periodic segment in $!P$ with another simple period $B$ over $G_\al$ and assume that $A$ and $B$ are not conjugate in ~$G_\al$. 
Then the following is true.
\begin{enumerate}
\item
$!S \not\sim s_{A,!P}^t !S$ for any $t \ne 0$.
\item
If $\ell_B(!S) \ge 3$ then $|!S| < 2|A|$.
\end{enumerate}
\end{lemma}

\begin{proof}
(i) Assume that $!S \sim s_{A,!P}^t !S$ for some $t \ne 0$.
Let $!L_1$ be the infinite periodic extension of $!P$, and let $!L_2$ be the axis for $!K$.  
By Corollary \ref{co:compatible-periodic-axes} we have $!L_2 =  s_{A,!P}^t !L_2$, so
$s_{A,!P}^t = s_{B,!Q}^r$ for some $r \ne 0$. Since $A$ and $B$ are non-powers,
by Corollary \ref{coi:nontorsion-relations} $s_{A,!P}^\ep = s_{B,!Q}$ for $\ep=\pm1$
and hence $L_1^\ep$ and $L_2$ are parallel. From the fact that $!S$ is a subpath of $!P$
we easily deduce by Proposition \ref{pr:fragment-correspondence-between-fragments}
(taking for $\be$ the activity rank of $A$) that $\ep=1$.
We obtain a contradiction with the assumption that $A$ and $B$ are not conjugate in $G_\al$.

(ii) By \ref{ssi:cutting-dividing} we represent $!S$ as $!S = !S_1 !u !S_2$ where $!S_1$ and $!S_2$
are coarsely periodic segments with period $B$ and 
$\ell_B(!S_1) \ge \ell_B(!S) -2$.
By (i) and Corollary \ref{co:coarsely-periodic-overlapping}, $s_{A,!P}^{-1} !S$ does not contain $!S_1$
and $s_{A,!P} !S$ does not contain $!S_2$. This implies $|!S| < 2|A|$.
\end{proof}

\begin{proposition} \label{pr:close-parallel-one-period}
Let $!P$ and $!Q$ be close periodic segments in $\Ga_{\al}$ with the same simple period $A$ over $G_\al$.
If $|!P| \ge (2h_\al(A)+1) |A|$ (where $h_\al(A)$ is defined in \ref{df:period-stability-parameter})
then $!P$ and $!Q$ belong to the same $A$-periodic line.
\end{proposition}

\begin{proof}
Follows from Propositions \ref{pr:coarsely-periodic-from-close} and \ref{pr:strictly-close-parallel}.
\end{proof}

We finish the section by formulating technical statements which we will need 
in the construction of relations of Burnside groups.
We use notation $!S \lessapprox !T$ for `$!S < !T$ or $!S \approx !T$'.

\begin{lemma} \label{lm:strict-compatibility-ordering}
Let $!S$ and $!T$ be coarsely $A$-periodic segments occurring in a reduced path $!X$ in $\Ga_\al$.
Assume that some periodic bases for $!S$ and $!T$ have the same label.
If $!S$ is contained in $!T$ then $!S \approx !T$.
\end{lemma}

\begin{proof}
Assume that $!S$ is contained in $!T$.
Let $!P_i$ $(i=1,2)$ be periodic bases for $!S$ and $!T$ respectively, with 
$\lab(!P_1) \greq \lab(!P_2)$. Let $\be$ be the activity rank of $A$.
By Proposition ~\ref{pr:coarsely-periodic-overlapping-general},
$!P_1$ and $!P_2$ have a common periodic extension. Let $!K_i$ and $!M_i$
$(i=0,1,2,3)$ be fragments of rank $\be$ with $\muf(!K_i),\muf(!M_i) \ge \xi_2$
such that $!P_1 = !K_0 \cup !K_1$, $!P_2 = !K_2 \cup !K_3$, $!S = !M_0 \cup !M_1$,
$!T = !M_2 \cup !M_3$ and $!K_i \sim !M_i$ for all $i$.
We have $!M_2 \lesssim !M_0 \lnsim !M_1 \lesssim !M_3$
which by Proposition \ref{pr:coarsely-periodic-overlapping-general} 
implies $!K_2 \lesssim !K_0$ and $!K_1 \lesssim !K_3$. 
Now from $\lab(!P_1) \greq \lab(!P_2)$ we conclude that $!K_2 \sim !K_0$
and $!K_1 \sim !K_3$, i.e.\ $!S \approx !T$.
\end{proof}

\begin{lemma} \label{lm:coarse-periodic-oredering-close}
Let $!X$ and $!Y$ be close reduced paths in $\Ga_\al$. Let $!S_0, !S_1$ 
be coarsely $A$-periodic segments in $!X$ and 
$!T_0, !T_1$ be coarsely $A$-periodic segments in $!Y$ 
such that $\ell(!S_i) \ge 2h_\al(A) + 1$,  $!S_i \approx !T_i$ for $i=0,1$  and $!S_0 \not\sim !S_1$.
Then $!S_0 < !S_1$ if and only if $!T_0 < !T_1$.
\end{lemma}

\begin{proof}
By Corollary \ref{co:coarsely-periodic-overlapping}, none of $!S_0$ and $!S_1$ is contained in the other
and the same is true for $!T_0$ and $!T_1$.
Assume, for example, that $!S_0 < !S_1$ and $!T_1 < !T_0$. 
Let $!X_1$ and $!Y_1$ be the starting segments of $!X$ and $!Y$ ending with $!S_1$ and $!S_2$
respectively. 
By Proposition \ref{pr:coarse-periodic-stability} with $!X:=!X_1$ and $!Y := !Y_1$ 
there exists $!U$ in $Y_1$ such that $!U \approx !S_0^*$ where $!S_0^*$ is the stable part of $!S_0$.
Then $!U \cup !T_0$ is a coarsely $A$-periodic segment containing $!T_1$ and we get
a contradiction with Corollary \ref{co:coarsely-periodic-overlapping}.
\end{proof}

\begin{lemma} \label{lm:coarse-periodic-stability-between}

%
Let $!X$ and $!Y$ be reduced paths in $\Ga_\al$. Let $!S_0, !S_1$ 
be coarsely $A$-periodic segments in $!X$ and 
$!T_0, !T_1$ be coarsely $A$-periodic segments in $!Y$ 
such that $!S_0 \lessapprox !S_1$, $!T_0 \lessapprox !T_1$ 
and $!S_i \approx !T_i$, $i=0,1$.
\begin{enumerate} 
\item \label{lmi:coarse-periodic-stability-penetration}
Let $!U$ be a coarsely $A$-periodic segment in $!X$ 
such that $!S_0 \lessapprox !U \lessapprox !S_1$, $\ell_A(!U) \ge h_\al(A) + 1$
and $!U$ is the stable part of some other coarsely $A$-periodic segment in $!X$.
Then there exists a coarsely $A$-periodic segment ~$!V$ in $!Y$ 
such that $!T_0 \lessapprox !V \lessapprox !T_1$ and $!U \approx !V$.
\item \label{lmi:coarse-periodic-ordering}
Let $!U_i$ $(i=1,2)$ be coarsely $A$-periodic segments  in $!X$
and $!V_i$ $(i=1,2)$ be coarsely $A$-periodic segments in $!Y$ 
such that $\ell_A(!U_i) \ge 2 h_\al(A) + 1$ $(i=1,2)$,
$!S_0 \lessapprox !U_i \lessapprox !S_1$, 
$!T_0 \lessapprox !V_i \lessapprox !T_1$ 
and $!U_i \approx !V_i$ for $i=1,2$.
Assume that $!U_2 \approx g !U_1$ for some $g \in G_\al$, i.e.\ $!U_1$ and $!U_2$
have periodic bases with the same label.
Then $!U_1 \lessapprox !U_2$ if and only if $!V_1 \lessapprox !V_2$.
\end{enumerate}
\end{lemma}

\begin{proof}
Let $\be$ be the activity rank of $A$.

(i): 
Let $!U$ be the stable part of $\bar{!U}$ and $\bar{!U} = !Z_1 !U !Z_2$.
We consider several cases.

{\em Case\/} 1: $!U \not\sim !S_i$ for $i=0,1$.
Then by Corollary \ref{co:coarsely-periodic-overlapping} we have
$!S_0 < \bar{!U} < !S_1$. Since $!S_0 \cup !S_1$ and $!T_0 \cup !T_1$ 
are close, existence of $!V$ follows from 
Proposition \ref{pr:coarse-periodic-stability}.

{\em Case\/} 2: Exactly one of the relations $!U \sim !S_i$ $(i=0,1)$ holds.
Without loss of generality, assume that $!U \sim !S_0$ and $!U \not\sim !S_1$.
By Corollary \ref{co:coarsely-periodic-overlapping} we have 
$\bar{!U} < !S_1$. 
If $!U \approx !S_0$ there is nothing to prove. 
Assume that $!U \not\approx !S_0$ and hence $!U !Z_2$ is contained in 
$!S_0 \cup !S_1$.

By the construction of the stable part, 
$!U !Z_2$ is a coarsely $A$-periodic segment with 
$\ell_A(!U !Z_2) = \ell(!U) + h_\al(A) \ge 2h_\al(A) + 1$.
Let $!W$ be the stable part of $!U !Z_2$.
Using Proposition \ref{pr:coarse-periodic-stability} with 
$!X := !S_0 \cup !S_1$ and $!Y := !T_0 \cup !T_1$ 
we find a coarsely $A$-periodic segment $!W'$ in $!T_0 \cup !T_1$ such 
that $!W \approx !W'$. 
By Proposition \ref{pri:coarsely-periodic-union},

$!S_0 \cup !U$ is a coarsely $A$-periodic segment and since
$!W' \sim !T_0$, $!T_0 \cup !W'$ is a coarsely $A$-periodic segment as well.
By \ref{ssi:cutting-invariance}
(more formally, by the symmetric version of \ref{ssi:cutting-invariance})
$!W$ is an end of $!U$ which implies $!S_0 \cup !U \approx !T_0 \cup !W'$.

Now let $!P$ be a periodic base for $!U$. 
By the construction of the stable part,  $!P$ starts with a fragment $!N$
of rank $\be$ with $\muf(!N) \ge \xi_1$.
Since $!P$ is contained in a periodic base for $!T_0 \cup !W'$, 
by Proposition \ref{pr:fragment-correspondence-between-fragments} we
find a fragment $!N'$ of rank $\be$ in $!T_0 \cup !W'$ such
that $\muf(!N') \ge \xi_2$ and $!N' \sim !N$.
Then for the desired $!V$ we can take the end of $!T_0 \cup !W'$
starting with $!N'$.

{\em Case\/} 3: $!U \sim !S_0 \sim !S_1$. Then a periodic base $!P$ for $!U$
is contained in a periodic base for $!S_0 \cup !S_1$. By the construction of 
the stable part, $!P$ starts and ends with fragments $!N_0$ and $!N_1$ 
of rank $\be$ with $\muf(!N_i) \ge \xi_1$. 
Then using Proposition \ref{pr:fragment-correspondence-between-fragments}
we find fragments $!N_i'$ $(i=0,1)$ of rank $\be$ in $!T_0 \cup !T_1$
such that $\muf(!N_i') \ge \xi_2$ and $!N_i' \sim !N_i$ $(i=1,2)$.
We can take $!V = !N_0' \cup !N_1'$.

(ii): We consider two cases.

{\em Case\/} 1: $!U_1 \sim !U_2$.
Let $!P_1$ and $!P_2$ be periodic bases for $!U_1$ and $!U_2$ with 
$\lab(!P_1) \greq \lab(!P_2)$ which have a common periodic extension. 
It easily follows from Proposition \ref{pri:fragment-ordering-between-fragments} that $!U_1 < !U_2 \lequiv !P_1 < !P_2$ and 
$!U_1 \approx !U_2 \lequiv !P_1 = !P_2$. Since $!P_i$ is also a periodic 
base for $!V_i$, a similar statement holds for $!V_i$'s which clearly 
implies the required conclusion.

{\em Case\/} 2: $!U_1 \not\sim !U_2$. Without loss of generality,
we assume that $!U_1 < !U_2$, $!V_1 > !V_2$ and come to a contradiction.
We can assume also that $!X = !S_0 \cup !S_1$, $!Y = !T_0 \cup !T_1$
and hence $!X$ and $!Y$ are close in rank $\al$.
Let $!U_i^*$ and $!V_i^*$ be
stable parts of $!U_i$ and $!V_i$. By Corollary \ref{co:coarsely-periodic-cut},
$!U_1$ is disjoint from $!U_2^*$. Let $!X = !X_1 !U_1 !X_2 !U_2^* !X_3$ and
$!Y = !Y_1 !V_2^* !Y_2 !V_1 !Y_3$. 
By Proposition \ref{pr:coarse-periodic-stability} with $!X = !X_1 !U_1 !X_2$
and $!Y := !Y_1$ there exists a coarsely $A$-periodic segment $!W$ in $!Y_1$
such that $!W \approx !U_1^*$. Then $!W \sim !U_1 \sim !V_1$
and by Proposition \ref{pri:coarsely-periodic-union}
and Corollary \ref{co:coarsely-periodic-overlapping} we get 
$!U_1 \sim !W \sim !W \cup !V_1 \sim !V_2 \sim !U_2$, the desired
contradiction.
\end{proof}

\section{Comparing $\al$-length of close words}
\label{s:compare-close}

In this section, we prove the following proposition.

\begin{proposition} \label{pr:length-compare-close}
Let $X,Y \in \cR_\al$ be close in rank $\al$. Then
$$
   |Y|_\al < 1.3 |X|_\al + 2.2.
$$
\end{proposition}

Recall that a fragment word $F$ of rank $\al$ is considered with fixed associated words $S$, $u$, $v$ 
and a relator $R$ of rank $\al$ 
such that $F = uSv$ in $G_{\al-1}$, $u,v \in \cH_{\al-1}$ and $S$ is a subword of $R^k$ for some $k > 0$.
If $!F$ is a path in $\Ga_{\al-1}$ labeled $F$ then this uniquely defines the base $!S$ for ~$!F$.

Let $F$ and $G$ be fragments of rank $\al$ in a word $X$. Let $!X$ be a path in $\Ga_{\al-1}$ labeled $X$
and $!F$, $!G$ the corresponding subpaths of $!X$. We write $F \sim G$ if $!F \sim !G$ (so the relation
is formally defined for the occurrences of $F$ and $G$ in $X$).

Recall that the size $|X|_\al$ of a word $X$ in rank $\al$ is the minimal
possible value of $\text{weight}_\al(\cF)$ 
of a fragmentation $\cF$ of rank $\al$ of $X$.
A fragmentation $\cF$ of rank $\al$ of $X$ is a partition 
$X \greq F_1 \cdot F_2 \cdots F_k$ where $F_i$ is a nonempty subword of a fragment of rank 
$\be \le \al$. 
Assuming that each $F_i$ is assigned a unique value of $\be$,
the weight in rank $\al$ of $\cF$ is defined by formula
$$
  \text{weight}_\al(\cF) = m_\al + \ze m_{\al-1} + \ze^2 m_{\al-2} + \dots + \ze^\al m_0 
$$
where $m_\be$ is the number of subwords of fragments of rank $\be$ in $\cF$.

We call a fragmentation $\cF$ of $X$ {\em minimal} if 
$\text{weight}_\al(\cF) = |X|_\al$.

We call a subword $F$ of a fragment of rank $\be \ge 1$ 
a {\em truncated fragment of rank $\be$}. We will be assuming that with a truncated 
fragment $F$ of rank $\al$ there is an associated genuine fragment ~$\bar{F}$ of rank $\be$
such that $F$ is a subword of $\bar{F}$. 
If $!F$ is a path in $\Ga_\al$ with $\lab(!F) \greq F$ then we have the
associated fragment $\bar{!F}$ in $\Ga_\al$ such that $!F$ is a 
subpath of $\bar{!F}$.  Note a truncated fragment of rank 1 is simply a fragment of rank 1.

We extend the compatibility relation to truncated fragments of rank $\be$ in a word $X$ in 
the following natural way.
If $F$ and ~$G$ are truncated fragments of rank $\be$ in $X$ and $\bar{F}$ and $\bar{G}$
their associated fragments of rank ~$\be$ in $\Ga_\al$ then $F \sim G$ if and only if 
$\bar{F} \sim \bar{G}$.

\subsection{} \label{ss:fragmentation-extending-fragment}
Let $\cF =  F_1 \cdot F_2 \cdot \ldots \cdot F_k$ be a fragmentation of rank $\al$ of a word $X$.
Let $F_i$ be a truncated fragment of rank $\be \ge 1$ in $\cF$. Assume that $F_i$ can be
extended in $X$ to a larger truncated fragment $G$ of rank $\be$, i.e.\ 
$$
    X \greq F_1 F_2 \dots F_p'  F_p'' \dots F_i \dots F_q' F_q'' \dots F_k
$$
where 
$F_p \greq F_p' F_p''$, $F_q \greq F_q' F_q''$ and $G \greq F_p'' \dots F_i \dots F_q'$
(here we consider the case $1 < i < k$; cases $i=1$ and $i=k$ differ only in notation).
Then we can produce a new fragmentation $\cF'$ of rank $\al$, 
$X \greq F_1 \cdots F_{p-1} \cdot [F_p'] \cdot G \cdot [F_q''] \cdot F_{q+1} \cdots F_k$
where square brackets mean that $F_p'$ and $F_q''$ are absent if empty.
We say that $\cF'$ is obtained from $\cF$ by {\em extending $F_i$ to $G$}.
Note that if $\cF$ is minimal then in the case $i>1$, we necessarily have $p=i-1$ and nonempty $F_p'$ 
and in the case $i < k$ we necesarily have $q=i+1$ and nonempty $F_q''$.

\begin{lemma} \label{lm:fragment-in-fragmentation}
Let $\cF =  F_1 \cdot F_2 \cdot \ldots \cdot F_k$ be a minimal fragmentation of rank $\al\ge1$ of 
a word $X \in \cR_\al$. 
\begin{enumerate}
\item 
\label{lmi:fragment-in-fragmentation-large}
Let $F_i$ be a truncated fragment of rank $\al$ in $\cF$. Then 
$|F_i|_{\al-1} \ge \frac1\ze$ and $F_i = u Fv$ where $F$ is a fragment of rank $\al$, $F_i \sim F$,
$|u|_{\al-1}, |v|_{\al-1} < \ze$
and the base $!P$ for the corresponding fragment $!F$ in $\Ga_{\al-1}$ satisfies 
$|!P|_{\al-1} > 13$.
\item
If $K$ is a fragment of rank $\al$ in $X$ and 
$\muf(K) \ge 3\la + 15\om$ 
then $F_i \sim K$ for some $i$.
\item
\label{lmi:fragment-in-fragmentation-large-fragments}
Let $X = P_0 K_1 P_1 \dots K_r P_r$ where $K_i$ are fragments of rank $\al$ with 
$\muf(K_i) \ge 3\la + 13\om$ for all $i$. Then there exists another minimal fragmentation $\cF'$
of rank $\al$ of $X$ such that each $K_i$ is contained in a compatible truncated fragment 
of rank $\al$ in $\cF'$.
\end{enumerate}
\end{lemma}

\begin{proof}
(i) If $|F_i|_{\al-1} < \frac1\ze$
then we could replace ~$F_i$ by its fragmentation of rank $\al-1$ which would decrease
the weight of $\cF$. By Proposition \ref{pr:fellow-traveling}$_{\al-1}$ in the case $\al\ge 2$ 
(in the case $\al=1$ we take $u$ and $v$ empty) we have $F_i = uFv$ where
$F$ is a fragment of rank $\al$, $F_i \sim F$ and $|u|_{\al-1}, |v|_{\al-1} <\ze$.
If $!F$ is the corresponding fragment of rank $\al$ in $\Ga_{\al-1}$ and 
$!P$ is the base for ~$!F$ then by Proposition \ref{pr:length-compare-close}$_{\al-1}$
$$
    |!P|_{\al-1} > \frac{1}{1.3} \left(\frac{1}{\ze} - 2\ze - 2.2 \right) > 13.
$$

(ii) Let $K$ be a fragment of rank $\al$ in $X$ and $\muf(K) \ge 3\la + 15\om$.
We assume that there is no truncated fragment $F_i$ of rank $\al$ such that $F_i \sim K$.

By Proposition \ref{pr:inclusion-compatibility} and
the assumption, if $H$ is a common part
of $K$ and some $F_i$ of rank ~$\al$ then $H$ contains no fragment $K'$
of rank ~$\al$ with $\muf(K') \ge \la + 2.6\om$. 
By Lemma \ref{lm:piece-fragment-no-higher-fragments}, 
if $H$ is a common part
of ~$K$ and some $F_i$ of rank $\be < \al$ then $H$ contains no fragment $K'$
of rank $\al$ with $\muf(K') \ge 3.2\om$. 
In particular, $K$ is not contained in any $F_i$ .
Let 
$$
    X \greq F_1 F_2 \dots F_p'  F_p'' \dots F_q' F_q'' \dots F_k
    \quad\text{where}\quad 
    F_p \greq F_p' F_p'', \quad F_q \greq F_q' F_q'', \quad K \greq F_p'' F_{p+1} \dots F_q'.
$$
If some $F_i$ is contained in $K$ and has rank $\al$ then 
by the remark above and ~\ref{ss:fragmentation-extending-fragment}, 
$K$ is covered by at most three of the $F_j$'s. 
In this case, by Proposition ~\ref{pr:dividing-fragment} we would have
$$
    \muf(K) \le 3(\la+2.6\om) + 2\ze\om < 3\la + 15\om
$$ 
contrary to the hypothesis.
Therefore, each $F_i$ that contained in $K$ has rank $\be < \al$.
Now by Proposition ~\ref{pr:dividing-fragment}, $F_p F_{p+1} \dots F_q$
contains a fragment $K'$ of rank $\al$ with 
$$
    \muf(K') \ge \muf(K) - 2(\la+2.6\om) - 2\ze\om > 29 \om.
$$
For a base $P$ of $K'$ we have $|P|_{\al-1} > 29$
and by Proposition \ref{pr:length-compare-close}$_{\al-1}$,
$|K'|_{\al-1} > 20$. This implies that
$
     \text{weight}_\al(F_p \cdot F_{p+1} \cdot \ldots \cdot F_q) > 1 
$
and we get a contradiction with minimality of ~$\cF$
since we can replace $F_p F_{p+1} \ldots F_q$ in $\cF$ by a single truncated fragment of rank ~$\al$.
This finishes the proof.

(iii) By (ii), for each $i=1,2,\dots,r$ there exists a truncated fragment $F_{t_i}$ of rank $\al$ in $\cF$ 
such that $K_i \sim F_{t_i}$. Proposition \ref{pr:fragments-union} 
easily implies that $F_{t_i} \cup K_i$ is a truncated fragment of rank $\al$.
For each $i=1,2,\dots,r$ we consequently replace $F_{t_i}$ in $\cF$ by $F_{t_i} \cup K_i$.
Since we do not increase $\text{weight}_\al(\cF)$, the resulting fragmentation $\cF'$ of $X$ is 
also minimal.
\end{proof}

\begin{lemma} \label{lm:length-compare-close-lm}
Let $\al\ge1$ and $X,Y \in \cR_\al$ be close in rank $\al-1$. Then
$$
   |Y|_\al < 1.3 |X|_\al + 2.2 \ze . 
$$
\end{lemma}

\begin{proof}
Let $\cF$ be a minimal fragmentation of $X$. 
We represent $X$ and $Y$ by close paths ~$!X$ and ~$!Y$ in $\Ga_{\al-1}$. Then $\cF$
induces the partition of $!X$, denoted $\bar\cF$,
into (path) truncated fragments of ranks $\le\al$.
 
Let 
$$
    !X = !P_0 !H_1 !P_1 \dots !H_r !P_r 
$$
where $!H_1,\dots,!H_r$ are all truncated fragments of rank $\al$ in $\bar\cF$. 
If $r=0$ then $|X|_\al = \ze |X|_{\al-1}$, $|Y|_\al \le \ze |X|_{\al-1}$ and the statement simply follows
from Proposition \ref{pr:length-compare-close}$_{\al-1}$. We assume $r > 0$.
By Lemma \ref{lmi:fragment-in-fragmentation-large},  for each $i$ we have $!H_i = !u_i !H_i' !v_i$
where $!H_i'$ is a fragment of rank $\al$, $!H_i' \sim !H_i$,
$|!u|_{\al-1}, |!v|_{\al-1} < \ze$, and 
the base $!S_i$ for $!H_i$ satisfies $|!S_i|_{\al-1} > 13$.
Using Proposition \ref{pr:closeness-stability}$_{\al-1}$ we find fragments $!H''_i$ and $!G_i$
of rank $\al$ in $!X$ and $!Y$ respectively where $!H'_i = !w_i !H_i'' !z_i$, 
$|!w_i|_{\al-1}, |!z_i|_{\al-1} < 1.15$, $!H_i \sim !H_i'' \sim !G_i$ and 
$!H_i''$ and $!G_i$ are close in rank $\al-1$.
Using Lemma \ref{lmi:compatible-close}$_{\al-1}$ after each 
application of Proposition \ref{pr:closeness-stability}$_{\al-1}$
we can assume that $!G_i$ are disjoint, i.e.\
$$
    !Y = !Q_0 !G_1 !Q_1 \dots !G_r !Q_r.
$$
By Proposition \ref{pr:length-compare-close}$_{\al-1}$ we have
\begin{gather*}
    |!Q_0|_{\al-1} < 1.3 |!P_0 !u_1 !w_1|_{\al-1} + 2.2, \\
    |!Q_i|_{\al-1} < 1.3 |!z_i !v_i !P_i !u_{i+1} !w_{i+1}|_{\al-1} + 2.2 \quad (i=1,\dots,r-1), \\
    |!Q_k|_{\al-1} < 1.3 |!z_k !v_k !P_k|_{\al-1} + 2.2.
\end{gather*}
We have also 
$$
    |X|_\al = r + \ze\sum_{i=1}^r |!P_i|_{\al-1} \quad\text{and}\quad
    |Y|_\al \le r + \ze\sum_{i=1}^r |!Q_i|_{\al-1}.
$$
Then
\begin{align*}
    |Y|_\al 
        &< r + 1.3\ze\sum_{i=1}^r |!P_i|_{\al-1} + 1.3r\ze(2.3 + 2 \ze) + 2.2\ze(r+1) \\
        & = (1 + 1.3\ze(4.5 + 2\ze))r + 1.3\ze\sum_{i=1}^r |!P_i|_{\al-1} + 2.2\ze \\
        & < 1.3 |X|_\al + 2.2\ze. 
\end{align*}
\end{proof}

\begin{proof}[Proof of Proposition \ref{pr:length-compare-close}]
Let $X,Y \in \cR_\al$ be close in rank $\al$. Let $\cF$ be a minimal 
fragmentation of $X$. 
We consider close  paths $!X$ and $!Y$ in $\Ga_{\al}$ labeled $X$ and $Y$ respectively.
Then $\cF$ induces the partitions of $!X$
into (path) truncated fragments of ranks $\le\al$,
$$
    !X =  !F_1 \cdot !F_2 \cdot \ldots \cdot !F_k.
$$
Let $!X^{-1} !u !Y !v$ be a coarse bigon.
We fix some bridge partitions of ~$!u$ and $!v$.
Let $\De$ be a filling diagram of rank $\al$ with boundary loop $\ti{!X}^{-1} \ti{!u} \ti{!Y} \ti{!v}$. Up to switching of $!u$ and $!v$ we can assume that $\De$ is
reduced and has a tight set ~$\cT$ of contiguity subdiagrams. Let $!D_1$, $\dots$, $!D_r$
be all cells of rank $\al$ of $\De$. In the process of forming $\cT$ we assume that we pick first
the contiguity subdiagrams of $!D_i$ to $\ti{!X}^{-1}$ choosing them with maximal possible 
contiguity segment occurring in $\ti{!X}^{-1}$. Let 
$$
    !X = !P_0 !K_1 !P_1 \dots !K_r !P_r \quad\text{and}\quad 
    !Y = !Q_0 !M_1 !Q_1 \dots !M_r !Q_r .
$$
where $!K_i$ and $!M_i$ are the corresponding active fragments of rank $\al$ in $!X$ and $!Y$.
By the way we produce $\cT$ and by Proposition \ref{pr:fellow-traveling}$_{\al-1}$ in the case $\al\ge 2$ 
we have the following:

(*) {\em For all $i$, the fragment $!K_i$ cannot be extended in $!P_{i-1} !K_i !P_i$.
In particular, if $!F$ is a truncated fragment of rank $\al$ contained in  $!P_{i-1} !K_i !P_i$ and 
containing $!K_i$ then $!F = !w_1 !K_i !w_2$ where $|!w_i|_{\al-1} < \ze$ $(i=1,2)$}

By Lemma \ref{lmi:fragment-in-fragmentation-large-fragments} we can assume that each $!K_i$
is contained in a compatible truncated fragment $!F_{t_i}$ of rank $\al$. Let 
$$
    !X = !P_0' !F_{t_1} !P_1' \dots !F_{t_r} !P_r'.
$$
Note that 
$$
    |X|_\al = r + \sum_i |!P_i'|_{\al} \quad\text{and}\quad
    |Y|_\al \le r + \sum_i |!Q_i|_{\al}.
$$
By (*), 
$$
    |!P_i'|_\al \ge |!P_i|_\al - \ze^2 \text{ for } i=0,r, \quad 
    |!P_i'|_\al \ge |!P_i|_\al - 2\ze^2 \text{ for } 1 \le i \le r-1.
$$
Hence
\begin{equation} \label{eq:length-compare-close}
    |X|_\al \ge r + \sum_i |!P_i|_{\al} - 2r\ze^2 .
\end{equation}

We give an upper bound on $|!Q_i|_{\al}$ in terms of $|!P_i|_\al$.
First we consider the case $1 \le i \le r-1$. 
There are three possibilities for the subdiagram of $\De$ surrounded by $!D_i$ and $!D_{i+1}$
and contiguity subdiagrams of $!D_i$ and $!D_{i+1}$ to $\ti{!X}^{-1}$ and $\ti{!Y}$,
depending on the presence of contiguity subdiagrams from $\cT$ 
(see Figure \ref{fig:length-compare-close}).   
\begin{figure}[h]
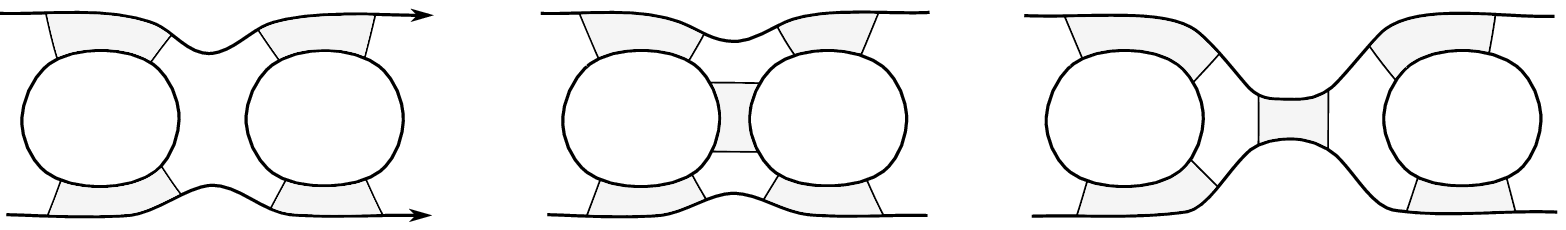
\caption{}  \label{fig:length-compare-close}
\end{figure}
Note that according to Definition \ref{df:tight-set}, all the components of $\De - \cup_{\Pi \in \cT}$
are small diagrams of rank $\al-1$, so we can use bounds from 
Proposition \ref{pr:small-trigons-tetragons}$_{\al-1}$.
In cases (a) and (b) we have $|!Q_i|_{\al} \le 6\ze^2\eta < 0.6\ze$ and 
$|!Q_i|_{\al} \le 4\ze^2\eta < 0.4\ze$ respectively. Assume that case (c) holds. 
Then $!P_i = !u_1 !S !u_2$ and $!Q_i = !v_1 !T !v_2$ where 
$!S$ and ~$!T$ are close in rank $\al-1$ and 
$|!u_i|_{\al}, |!v_i|_{\al} \le 4\ze^2\eta < 0.4\ze$.
Using Lemma \ref{lm:length-compare-close-lm}, we get
$$
    |!Q_i|_{\al} < 1.3 |!P_i|_{\al} + 3\ze 
$$
Note that this inequality holds also in cases (a) and (b).

Now let $i=0$ or $i=r$. If $r >0$ then the difference of the case $i=0$ from the case $1 \le i \le r-1$ is that
we can have an extra contiguity subdiagram between $!Y$ and the central arc of $\ti{!u}$
(see Figure \ref{fig:length-compare-close-ii}). 
\begin{figure}[h]
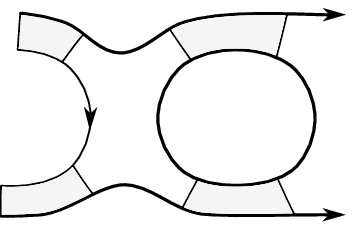
\caption{}  \label{fig:length-compare-close-ii}
\end{figure}
We then have  
$$
    |!Q_0|_{\al} < 1+ 1.3 |!P_0|_{\al} + 3\ze 
$$
and, similarly,
$$
    |!Q_r|_{\al} < 1+ 1.3 |!P_r|_{\al} + 3\ze.
$$
If $r=0$ we have a single bound instead,
$$
    |!Q_0|_{\al} < 2+ 1.3 |!P_0|_{\al} + 3\ze.
$$

Summarizing, with \eqref{eq:length-compare-close} we get
\begin{align*}
    |Y|_\al 
    &\le r + \ga\sum_i |!P_i|_{\al} + 2+ 3\ze(r+1) \\
    &= (1+ 3\ze)r + 1.3 \sum_i |!P_i|_{\al} + 2 + 3\ze \\
    &< 1.3 |X|_\al + 2.2.
\end{align*}

\end{proof}
 
\begin{corollary} \label{co:fragment-length-bound}
If $F$ is a fragment of rank $\al$ and $\muf(F) \ge t \om$ then 
$|F|_{\al-1} > \frac{1}{1.3} (t - 2.2)$. In particular, $|F| > \frac1{1.3}{\ze^{1-\al}} (t - 2.2)$.
\end{corollary}

\begin{corollary} \label{co:trigon-side-length}
Let $Y = u_1 X_1 u_2 X_2 u_3$ in $\Ga_\al$ where $X_i, Y \in \cR_\al$
and $u_i \in \cH_\al$. Then $|Y|_\al \le 1.3(|X_1|_\al + |X_2|_\al) + 4.8$.
\end{corollary}

\begin{proof}
Follows from Propositions \ref{pri:3gon-closeness} and \ref{pr:length-compare-close}.
\end{proof}

The following two statements are proved under the assumption that 
a normalized presentation \eqref{eq:G-presn} of $G$ satisfies the iterated small cancellation
condition (S0)--(S3) for all $\al\ge 1$. We therefore will be assuming
that all statements starting from Section \ref{s:diagrams} hold for all values of $\al$.

\begin{proposition}
Let $W$ be a word with $|W| \le \al$ and let $W = X$ in $G_\al$ where $X \in \cR_\al$.
Then $|X|_\al < 0.3$, $X$ contains no fragments $F$ of rank $\be > \al$ with $\muf(F) \ge 3\om$
and, in particular, $X \in \cap_{\al\ge1} \cR_\al$.
\end{proposition}

By Corollary \ref{co:fragment-length-bound} it is enough to prove that $|X|_\al < 0.3$. 
We proceed by induction on ~$\al$. If $\al=1$ then $X$ is the freely reduced form of $W$ and 
$|X|_1 \le \ze |X| < 0.3$. 
Let $\al > 1$. Let $W \greq W_1 a$, $a \in \cA^{\pm1}$ and $W_1 = X_1$
in $G_{\al-1}$ where $X_1 \in \cR_{\al-1}$.  
By Corollary \ref{co:fragment-length-bound}, the inductive hypothesis and
Proposition \ref{pr:trigon-active-fragments}, 
equality $X = X_1 a$ holds already
in $G_{\al-1}$. By Corollary \ref{co:trigon-side-length}$_{\al-1}$
$$
    |X|_\al \le \ze |X|_{\al-1} \le \ze( 1.3 (0.3 + 0.3) + 4.8) < 0.3.
$$

\begin{corollary} \label{co:absolute-reduction}
Every element of $G$ can be represented by a word $X$ reduced in $G$ such that for some $\al\ge 1$,
$X$ contains no fragments $F$ of rank $\be \ge \al$ with $\muf(F) \ge 3\om$.
\end{corollary}

\section{A graded presentation for the Burnside group} \label{s:elementary-periods}


In this section we show that for sufficiently large odd $n$ the Burnside group $B(m,n)$ has a graded presentation 
which satisfies the iterated small cancellation condition formulated in Section \ref{s:condition}.

We fix an odd number $n > 2000$. We are going to construct a graded presentation of the form
\begin{equation} \label{eq:Burnside-presentation}
  \bpresn{\cA}{C^n= 1 \ (C \in \bigcup_{\al\ge 1} \cE_\al)}
\end{equation}
where all relators of all ranks $\al$ are $n$-th powers. 
We assume that values of the parameters ~$\la$ and $\Om$ are chosen as in 
Theorem \ref{th:main}, i.e.\
$$
  \la = \frac{80}{n}, \quad \Om = 0.25 n.
$$
We will use also the following extra parameters:
$$
  p_0 = 39, \quad p_1 = p_0 + 26 = 65.
$$

In what follows, we define the set $\cE_{\al+1}$ 
under the assumption that sets $\cE_\be$ are already defined for all $\be \le \al$.
We fix the value of rank $\al \ge 0$ and assume 
that the presentation \eqref{eq:Burnside-presentation}
satisfies small cancellation conditions (S0)--(S3) in \ref{ss:condition}, \ref{ss:no-inverse-conjugation}
and in normalized in the sense Definition \ref{df:normalized-presentation}
for all values of the rank up to ~$\al$.

We can therefore assume that all statements in 
Sections \ref{s:diagrams}--\ref{s:overlapping-periodicity}
are true for the current value of $\al$ and below.


According to Propositions \ref{pr:cyclic-reduction} and \ref{pr:root-existence}
each element of infinite order of $G_\al$ is conjugate to a power of a simple period over $G_\al$.
We will define $\cE_{\al+1}$ as a certain set of simple periods over $G_\al$.
This will automatically imply condition (S0) with $\al := \al+1$.

Since $n$ is odd, by Corollary \ref{co:no-involutions-no-inverse-conjugates} we obtain also that (S3) 
holds with $\al := \al+1$.

Before going to the chain of definitions in the next section, 
we formulate the following two conditions (P1) and (P2) on $\cE_{\al+1}$
(which can be viewed as ``periodic'' versions of (S1) and (S2) for the value of rank $\al := \al+1$).
\begin{itemize}
\item[(P1)] 
For each $A \in \cE_{\al+1}$, $[A]_\al \ge 0.25$.

\item[(P2)] 
Let $!L_1$ and $!L_2$ be periodic lines in $\Ga_\al$ with periods $A,B \in \cE_{\al+1}$ respectively.
Assume that a subpath $!P$ of $!L_1$ and a subpath of $!Q$ of $!L_2$ are close and $|!P| \ge p_1 |A|$.
Then $!L_1$ and $!L_2$ are parallel.
\end{itemize}
The main goal of the construction of $\cE_{\al+1}$ will be to satisfy (P1) and (P2).
Note that (P1) immediately implies (S1) for $\al := \al+1$ because of the assumption $n > 2000$.
Later we prove that (P2) implies (S2)$_{\al+1}$. (The difference between (P2) and (S2)$_{\al+1}$
is that in (P2) we measure periodic words by the number of periods while
in (S2)$_{\al+1}$ we use the length function $|\cdot|_{\al}$. An appropriate
bound will be given in Proposition \ref{pr:length-in-periods}.)

Our first step is to define a set of simple periods over $G_\al$ which 
potentially violate (P2)
(they will be excluded in the definition of $\cE_{\al+1}$).

\begin{definition} \label{df:start-suspended}
A simple period $A$ over $G_\al$ is {\em suspended of level 0} if there exist a simple period $B$ 
not conjugate in $G_\al$ to $A$ and 
words $P \in \per(A)$ and $Q \in \per(B)$ such that $P$ and ~$Q$ are close in $G_\al$
and $|Q| \ge p_1|B|$.
\end{definition}

At first sight, we could simply define $\cE_{\al+1}$ by excluding 
periods $A$ as in Definition \ref{df:start-suspended} from the set of all simple periods over $G_\al$.
However, in this case we cannot guarantee a necessary lower bound on $[A]_\al$ for $A \in \cE_{\al+1}$ in (P1).
Roughly speaking, we need to claim 
that a fragment of rank $\be \le \al$ can cover only a ``small'' part of a periodic
word with a period $A \in \cE_{\al+1}$; moreover, we need
an exponentially decreasing upper bound on the size of this part when $\be$ decreases
(compare with the definition of the function $|\cdot|_\al$ in \ref{ss:alpha-length}).
To achieve this, we enlarge the set of excluded simple periods over $G_{\al+1}$
by adding potentially ``bad'' examples of this sort.

\begin{definition} \label{df:next-suspended}
A simple period $A$ over $G_\al$ is {\em suspended of level $m \ge 1$} if there exist 
a suspended period $B$ of level $m-1$ not conjugate to ~$A$ in $G_\al$, and a reduced in $G_\al$
word of the form $XQY$ such that $Q \in \per(B)$, $|Q| \ge 4|B|$ and $XQY$ is close in $G_\al$ to a word $P \in \per(A)$.
\end{definition}

\begin{definition} \label{df:elementary-periods}
Let $\cP_\al$ denote the set of all simple periods over $G_\al$
and $\cS_\al$ denote the set of all suspended simple periods over $G_\al$ of all levels $m \ge 0$.
For $\cE_{\al+1}$ we take any set of representatives of equivalence classes in $\cP_\al \sm \cS_\al$
with respect to the equivalence
$$
  A \sim B \lequiv A \text{ is conjugate to $B^{\pm 1}$ in } G_\al.
$$
\end{definition}

The definition implies that any simple period over $G_\al$ in $\cP_\al \sm \cS_\al$ has finite order
in $G_{\al+1}$. 
Since $\cP_{\al+1} \seq \cP_\al$, 
it follows that any simple period over $G_{\al+1}$ and, in particular, any
word in $\cE_\be$ for $\be \ge \al+1$ belongs to $\cS_\al$.
As a consequence, we prove now that a fragment of rank $\al+1$ cannot
cover a large periodic word with a simple period $A$ over $G_{\al+1}$. (So here is the trick: 
the definition of the set of suspended periods over $G_\al$ of levels $m \ge 1$ serves condition (P1)
for the {\em future} rank $\al+1$.)

\begin{remark} \label{rm:G-normalized}
By construction, we obtain a normalized presentation  \eqref{eq:Burnside-presentation}
(see Definition \ref{df:normalized-presentation}).
\end{remark}

\begin{proposition} \label{pr:suspended-by-fragment}
Let $A$ be a simple period over $G_{\al+1}$. If an $A$-periodic word $P$ is 
a subword of a fragment of rank $\al+1$ then $|P| < 4|A|$.
\end{proposition}

\begin{proof}
As observed above, $A \in \cS_\al$.
Let $UPV$ be a fragment of rank $\al+1$ where $P \in \per(A)$. Then $UPV$ is close in $G_\al$
to a word $Q \in \per(B)$ where $B \in \cE_{\al+1}$.
Since $A$ is of infinite order in $G_{\al+1}$, it is not conjugate to $B$ in $G_\al$.
In this case, Definition \ref{df:next-suspended} says that if $|P| \ge 4|A|$ then $B \in \cS_\al$
which would contradict Definition \ref{df:elementary-periods}.
\end{proof}

Proposition \ref{pr:suspended-by-fragment} with $\al: = \al-1$ is an important but not sufficient
ingredient in the proof of (P1). We need also to ensure that if a subword of fragment of rank $\be< \al$ is
a subword of an $A$-periodic word with $A \in \cE_{\al+1}$ then its length compared to $|A|$
is ``exponentially decreasing when $\be$ decreases''. 
We prove a precise form of this statement in the next section by showing that coarsely periodic words 
have a certain property of hierarchical containment: a coarsely $A$-periodic word $S$ over $G_\al$ has $t$
disjoint occurrences of coarsely periodic words over $G_{\al-1}$ with sufficiently large number 
of periods where $t$ is approximately the number of periods $A$ in $S$.

\section{Hierarchical containment of coarsely periodic words} 
\label{s:hierarchical-containment}

{\em Starting from this point, all statements are formulated and proved under assumption that the group $G$
has a specific presentation \eqref{eq:Burnside-presentation} defined in Section \ref{s:elementary-periods}.}
The goal of this section is to prove the following property of suspended periods over $G_\al$
and to finalize the proof of the fact that the presentation \eqref{eq:Burnside-presentation} satisfies conditions (S0)--(S3). As in Section \ref{s:elementary-periods} 
we assume fixed the value of rank $\al \ge 0$ and assume that the normalized
presentation \eqref{eq:Burnside-presentation} satisfies conditions (S0)--(S3) 
for ranks less or equal ~$\al$; so we can use all statements in 
Sections \ref{s:diagrams}--\ref{s:elementary-periods} for any rank up to ~$\al$.

\begin{proposition} \label{pr:hierarchical-containment}
Let $A$ be a suspended period over $G_\al$. Then there exists a simple period ~$B$ over $G_\al$
such that:
\begin{enumerate} \label{pri:hierarchical-containment-periodic}
\item
A cyclic shift of $A$ contains
a coarsely $B$-periodic word $T$ over $G_\al$ with $\ell_B(T) \ge p_0$.

\item
Moreover, this subword $T$ has the following property.
Let $!S$ be a coarsely $A$-periodic segment in $\Ga_\al$ with $\ell_A(!S) \ge 4$.
Then there are an $A$-periodic base $!P$ for $!S$, $\ell_A(!S)-3$ translates $!T$, $s_{A,!P} !T$, $\dots$, $s_{A,!P}^{\ell(!S)-4} !T$ of a coarsely $B$-periodic segment $!T$ in $!P$ with 
$\lab(!T) \greq T$
and $\ell_A(!S)-3$  disjoint coarsely $B$-periodic segments ~$!V_i$ $(i=0,1,\dots,\ell(!S)-4)$ in $!S$ 
such that $!V_i \approx s_{A,!P}^i !T$
for all $i$.
\end{enumerate}
\end{proposition}

We start with showing how Proposition \ref{pr:hierarchical-containment}$_{\al-1}$ implies (P1)
in the case $\al\ge 1$.

\begin{lemma} \label{lm:containment-coarsely-periodic}
Let $A$ be a simple period over $G_{\al}$ and 
let $!S$ and $!V_i$  $(i=0,1,\dots,\ell_A(!S)-4)$ be as in 
Proposition \ref{pr:hierarchical-containment}$_{\al-1}$.
Then for any $i$, $!V_i \cup !V_{i+4}$ is not contained in a fragment of rank ~$\al$.
\end{lemma}

\begin{proof}
As in Proposition \ref{pr:hierarchical-containment}$_{\al-1}$, 
let $!P$ be an $A$-periodic base for $!S$ in $\Ga_{\al-1}$
containing $t-3$ translates $!T$, $s_{A,!P} !T$, $\dots$, $s_{A,!P}^{t-4} !T$ where $!T$ is a coarsely 
periodic segment with another period ~$B$ and $\ell_B(!T) \ge p_0$.
Assume that a fragment $!K$ of rank $\al$ in $\Ga_{\al-1}$
contains $!V_i$ and $!V_{i+4}$. Let $!L$ be the base axis for $!K$, 
so $!L$ is a $C$-periodic line with $C \in \cE_{\al}$.
Denoting $!V_i^*$ the stable part of $!V_i$, 
by Proposition \ref{pr:coarse-periodic-stability}$_{\al-1}$ we find $!W$ and $!W'$ in $!L$
such that $!W \approx !V_i^*$ and $!W' \approx !V_{i+4}^*$. 
Then $!W \cup !W'$ is close to 
$s_{A,!P}^i !T^* \cup s_{A,!P}^{i+4} !T^*$.
Since $A \in \cS_{\al-1}$, according to Definition \ref{df:next-suspended}$_{\al-1}$ this should imply 
$C \in \cS_{\al-1}$,
a contradiction.
\end{proof}

\begin{lemma} \label{lm:hierarchical-length-step}
Let $\al\ge 1$. Assume that a (linear or cyclic) word $X$ 
has $r$ disjoint occurrences of coarsely $A$-periodic words $U_i$ $(i=1,\dots,r)$ over $G_{\al-1}$ with $\ell_A(U_i) \ge p_0$.
Then $|X|_{\al-1} \ge 5 r$. 
\end{lemma}

\begin{proof}
The statement is immediate if $\al = 1$. Assume that $\al > 1$.

Consider a fragmentation $\cF$ of rank $\al-1$ of $X$ (definition \ref{ss:alpha-length}).
Let $S_1$, $\dots$, $S_k$ be the subwords of fragments of rank $\al-1$ in $\cF$.
By Proposition \ref{pr:hierarchical-containment}$_{\al-1}$
each $U_i$ contains $p_0-3=36$ disjoint coarsely 
$B$-periodic words $V_{i,j}$ $(j=1,\dots,36)$ 
over $G_{\al-2}$ with $\ell_B(V_{i,j}) \ge p_0$. 
We can assume that $U_i$ and $V_{i,j}$ are indexed in their natural order 
from the start to the end in ~$X$.
By Lemma \ref{lm:containment-coarsely-periodic}, each $S_i$
intersects at most 6 consequent subwords $V_{i,j}, V_{i,j+1}, \dots, V_{i,j+5}$.
Excluding $V_{i,j}$ with $1 \le j \le 6$, we obtain that each $S_i$ intersects at most 6 of all the remaining ~$V_{i,j}$.
By induction, we conclude that
$$
  |X|_{\al-1} \ge k + 5 \ze \max\set{0, \: 30r - 6k}
$$
With fixed $r$, 
the minimal value of the right-hand side is achieved when $30r - 6k = 0$. This gives the 
bound $|X|_{\al-1} \ge 5 r$.
\end{proof}

We prove the following stronger form of (P1):

\begin{proposition} \label{pr:period-is-large}
For any simple period $A$ over $G_\al$ we have $[A]_\al \ge 0.25$ and, consequently, $h_\al(A) \le 6$.
\end{proposition}

\begin{proof}
If $\al=0$ then $[A]_0 \ge 1$ by the definition of $[\cdot]_0$. Let $\al\ge1$.
Take any $r \ge 1$.
Consider a fragmentation $\cF$ of rank ~$\al$ of the cyclic word $(A^r)^\circ$. 
Assume that $\cF$ consists of words $S_i$, $i=1,2,\dots,N$
where the first $k$ are subwords of fragments of rank $\al$.
By Proposition \ref{pr:suspended-by-fragment}$_{\al-1}$ we have $|S_i| < 4|A|$ for $i=1,2,\dots,k$.
This implies that the cyclic word $(A^{r-4k})^\circ$ can be partitioned 
into subwords of words in some subset of the remaining $S_i$, $i=k+1,k+2,\dots,N$.
Therefore,
$$
  |(A^r)^\circ|_\al \ge k + \ze |(A^{r-4k})^\circ|_{\al-1}.
$$
Proposition \ref{pri:hierarchical-containment-periodic}$_{\al-1}$ 
says that $(A^{r-4k})^\circ$ has at least $r-4k$ disjoint occurrences
of a coarsely $B$-periodic word $K$ over $G_{\al-1}$ with $\ell_B(K) \ge p_0$.
Then by Lemma \ref{lm:hierarchical-length-step},
$$
  |(A^r)^\circ|_\al \ge k + 5 \ze (r - 4k) = 0.25 r.
$$
This holds for all $r\ge1$, so by Definition \ref{df:period-stability-parameter} we get 
$[A]_\al \ge 0.25$ and hence $h_\al(A) \le 6$.
\end{proof}

The following lemma is a key tool in the proof of Proposition \ref{pr:hierarchical-containment}.
Very roughly, it corresponds to the statement ``if a word $W$ is periodic with two simple periods $A$
and $B$ at the same time, and if $|W| \ge 2|A|$, $|W| \ge 2|B|$ then $B$ is a cyclic shift of $A$''.

\begin{lemma} \label{lm:coarsely-periodic-closeness-step}
Let $!L_0$ and $!L_1$ be periodic lines in $\Ga_\al$ with simple periods $A$ and $B$ over $G_\al$, respectively.
Let $!S$ be a coarsely $C$-periodic segment in $!L_0$ where $C$ is another simple period over $G_\al$, 
$\ell_C(!S) \ge 25$. Assume that 
there exist coarsely $C$-periodic segments $!T_0, !T_1, !T_2$ in $!L_1$ such that 
$!T_0 < !T_1< !T_2$ and $!T_i \approx s_{A,!L_0}^i !S$, $i=0,1,2$.

If $!T_0 \lesssim s_{B,!L_1}^{-1} !T_1$ or $s_{B,!L_1} !T_1 \lesssim !T_2$ then, if fact,
$!T_0 \approx s_{B,!L_1}^{-1} !T_1$, $s_{B,!L_1} !T_1 \approx !T_2$, words $A$ and $B$ 
represent conjugate elements of $G_\al$ and periodic lines $!L_0$ and $!L_1$ are parallel.
\end{lemma} 

\begin{proof}
Denote $!P_0 = !S\cup s_{A,!L_0}^2 !S$ and $!P_1 = !T_0 \cup !T_2$.
Let $!S^*$ and $!T_i^*$ be stable parts of $!S$ and ~$!T_i$.

The crucial argument is similar to one in the proof of 
Proposition \ref{pr:coarsely-periodic-overlapping-general}.
Denote $\cP$ the set of all coarsely $C$-periodic segments $!U$ in $\Ga_\al$ such
that $!U \approx g !S^*$ for some $g \in G_\al$
(i.e.\ $!U$ and ~$!S^*$ have the same labels of their periodic bases).
We introduce translations and jumps on the set
of coarsely $C$-periodic segments $!U \in \cP$ which occur in $!P_0$ or $!P_1$.
As in the proof of Proposition \ref{pr:coarsely-periodic-overlapping-general},
it will be convenient to consider two disjoint sets of those $!U \in \cP$
which occur in $!P_0$ and in $!P_1$.
(So formally we introduce the set $\cP_i$ $(i=0,1)$ of pairs 
$(!U, !P_i)$ where $!U$ occurs in $!P_i$;
thus $s_{A,!L_0}^i !S^*$  belongs to $\cP_0$ 
and $!T_i^*$ belongs to ~$\cP_1$ for $i=0,1,2$.
For a coarsely $C$-periodic segment $!U \in \cP$, saying `$!U$ occurs in $!P_i$'
we mean the corresponding element of ~$\cP_i$.)

Let $!U, !V \in \cP$ be coarsely $C$-periodic segments each occurring in some ~$!P_i$.
\begin{enumerate}
\item 
If $!U$ and $!V$ occur in different paths $!P_i$ and $!U \approx !V$ then $!U$ {\em jumps} to $!V$. 
\item
$!U$ {\em translates} to $!V$ in the following cases:
\begin{itemize}
\item[] 
$!U$ and $!V$ occur in $!P_0$ and $!U \approx s_{A,!L_0}^k !V$ for some $k \in \Z$; or
\item[] 
$!U$ and $!V$ occur in $!P_1$ and $!U \approx s_{B,!L_1}^k !V$ for some $k \in \Z$.
\end{itemize} 
\end{enumerate}
Let $\cM$ be a maximal set of pairwise non-(strictly compatible) segments which can be obtained
by these two operations from $!S^*$. Lemma \ref{lm:strict-compatibility-ordering} implies that $\cM$ is a finite set.
As in the proof of Proposition \ref{pr:coarsely-periodic-overlapping-general} we prove the following claim.

\begin{claim*}
The jump operation is always possible inside $\cM$; 
that is, for any $!U \in \cM$ in $!P_i$, $i \in \set{0,1}$, there exists $!V \in \cP$ 
in $!P_{1-i}$ such that $!V \approx !U$.
\end{claim*}

To prove the claim, we will apply Lemma \ref{lm:penetration-lemma} and do a necessary preparatory work.
Assume that $!U \in \cM$ belongs to $!P_0$ (the other case differs only in notation).
Let $!V_0 = !S^*$, $!V_1$, $\dots$, $!V_l = !U$ be a sequence of coarsely $C$-periodic segments $!V_i \in \cM$ such that $!V_{i+1}$ is obtained from $!V_i$
by one of the operations (i) or (ii).
We can assume that $!V_{2j} \to !V_{2j+1}$ are translations and 
$!V_{2j+1} \to !V_{2j+2}$ are jumps, so
$l = 2k-1$ for some $k$. Under this assumption, $!V_{2j} \to !V_{2j+1}$
is a translation inside $!P_0$ if $j$ is even and inside $!P_1$ if $j$ is odd.
We then define a sequence $!Y_0$, $!Y_1$, $\dots$, $!Y_k$ of paths in $\Ga_\al$
($!Y_j$ will be periodic segments with alternating periods $A$ and ~$B$)
and a sequence $!W_j \in \cP$ of coarsely $C$-periodic segments in $!Y_j$ for $j=0,1,\dots,k-1$
such that $!W_0 = !V_1$ and $!W_i \approx !W_0$ for all $i$.
For each $j$ we will have $!W_j = f_j !V_{2j+1}$ for some $f_j \in G_\al$.
The definition of $!Y_j$ and $f_j$ goes as follows.

We start with $!Y_0 = !P_0$ and $!W_0 = !V_1$, so $f_0 = 1$.
Assume that $j < k-1$ and $!Y_j$ and $f_j$ are already defined.
For even $j$, $!V_{2j}$ translates to $!V_{2j+1}$ inside $!P_0$, so there exists $f_{j+1} \in G_\al$
of the form $f_j s_{A,!P_0}^t$ such that $f_{j+1} !V_{2j+1} \approx f_j !V_{2j}$.
Thus, $f_i !P_0$ and $f_{j+1} !P_0$ have a common $A$-periodic extension and we take 
$!Y_{j+1} = f_i !P_0 \cup f_{j+1} !P_0$. Similarly, for odd $j$ $!V_{2j}$ translates to $!V_{2j+1}$
inside $!P_1$. We take $f_{j+1} \in G_\al$ of the form $f_j s_{B,!P_1}^t$ such that 
$f_{j+1} !V_{2j+1} \approx f_j !V_{2j}$ and take $!Y_{j+1} = f_i !P_0 \cup f_{j+1} !P_0$
inside a common $B$-periodic extension of $f_i !P_0$ and $f_{j+1} !P_0$.
Note that $k$ is odd because $!V_{2k+1} = !U$ is assumed to occur in $!P_0$.
We finally set $!Y_k = f_{k-1} !P_1$.

We now apply Lemma \ref{lm:penetration-lemma} where:
\begin{itemize}
\item 
$S_j$ is the set of all coarsely $C$-periodic segments $!V \in \cP$ in $!Y_j$.
\item
$S_j$ is pre-ordered by `$\lnapprox$'.
\item
Equivalence is strict compatibility.
\item
A segment $!V \in \bigcup_j S_j$ is defined to be stable if $!V$ is the stable part of some coarsely $C$-periodic
segment in $!Y_j$.
\item
For $a_j$, $b_j$, $a_j'$ and $b_j'$ we take appropriate translates of $!S^*$ and $!T_i^*$; namely,
$f_j !S^*$, $f_j s_{A,!L_0}^2 !S^*$, $f_j !T_0^*$ and $f_j !T_2^*$ if $j$ is even or
$f_j !T_0^*$, $f_j !T_2^*$, $f_j !S^*$ and $f_j s_{A,!L_0}^2 !S^*$ if $j$ is odd, respectively.
\item
$c_0$ is $!V_1$.
\end{itemize}
Note that by Proposition \ref{pr:period-is-large} we have $h_\al(C) \le 6$.
Hence the hypothesis $\ell_C(!S) \ge 25$ implies 
$\ell_C(!V) \ge 13 \ge 2h_\al(C)+1$ for any $!V \in \cP$.
Condition (ii) of Lemma \ref{lm:penetration-lemma} holds by 
Lemma ~\ref{lm:strict-compatibility-ordering}.
Conditions (iii) and (iv) of Lemma \ref{lm:penetration-lemma} hold by Lemma 
\ref{lm:coarse-periodic-stability-between}.
By the lemma, there exists a coarsely $C$-periodic segment $!V_k \in \cP$ in $f_{k-1} !P_1$
such that $!V_k \approx f_{k-1} !U$. This gives the required jump $!U \to f_{k-1}^{-1} !V_k$.
The claim is proved.

Let $r$ be the number of coarsely $C$-periodic segments $!V \in \cM$ such that 
and $!K^* \lnapprox !V \lessapprox s_{A,!L_0} !K^*$  and let $q$ be 
the number of coarsely $C$-periodic segments $!V \in \cM$ 
such that $!T_1^* \lnapprox !N \lessapprox s_{B,!L_1} !T_1^*$
(in other words, $r$ and $q$ are the numbers of coarsely $C$-periodic
segments $!V \in \cM$ in one period $A$ and in one period $B$, respectively).
Note that $\gcd(r,q) = 1$ because $\cM$ is generated by a single segment $!S^*$.

We assume first that either
$!T_0 \lessapprox s_{B,!L_1}^{-1} !T_1$ or $s_{B,!L_1} !T_1 \lessapprox !T_2$.
Since $\cM$ is closed under translations modulo equivalence `$\approx$', 
each of these relations implies $q \le r$ and hence implies the other one.
Let $!U_0, !U_1,\dots, !U_t$ be all coarsely $C$-periodic segments in $\cM$ belonging to ~$!P_0$
arranged in their order in $!P_0$ (so $!U_i$ form a set of representatives of coarsely $C$-periodic 
segments in $\cM$ modulo `$\approx$').
The group $G_\al$ acts on the set $\cP/{\approx}$. 
It follows from Corollary ~\ref{co:compatible-periodic-axes}
that the action is free. For equivalence classes $[!U_i]$ of $!U_i$ we have
$$
     s_{A,!L_0} [!U_i] = [!U_{i+r}], \ i= 0,1,\dots, t-r \quad 
     s_{B,!L_1} [!U_i] = [!U_{i+q}], \  i= 0,1,\dots, t-q.
$$
Note also that $t \ge 2r +1$. Applying Lemma \ref{lm:overlapping-by-action} we get $s_{A,!L_0} = d^q$
and $s_{B,!L_1} = d^r$ for some $d \in G_\al$. Since $A$ and $B$ are non-powers we get $q=r=1$
which immediately implies the conclusion of the proposition. 
%

For the proof, it remains to consider cases  $!T_0 \sim s_{B,!L_1}^{-1} !T_1$
and $s_{B,!L_1} !T_1 \sim !T_2$. We consider the case  $s_{B,!L_1} !T_1 \sim !T_2$
(the case $!T_0 \sim s_{B,!L_1}^{-1} !T_1$ is symmetric). By the already proved part, we can assume that 
$!T_2 \lnapprox s_{B,!L_1} !T_1$. We show that the assumption leads to a contradiction.

We have $!T_0 \lnapprox s_{B,!L_1}^{-1} !T_2 \lnapprox !T_1$, so there exists $!T_3 \in \cM$
such that $!T_3 \approx s_{B,!L_1}^{-1} !T_2$. 
$!T_3$ jumps to some $!S_3 \in \cM$ in $!L_0$ such that $!S_3 \sim !S$ and $!S_3 \lnapprox !S$.
Then $!S_3$ translates to $!S_4 \approx s_{A,!L_0} !S_3$ and we have $!S_4 \sim !S_2$ and 
$!S_4 \lnapprox !S_2$.    
Then $!S_4$ jumps to some $!T_4$ in $!L_1$ and we can continue the process infinitely 
(see Figure \ref{fig:coarsely-periodic-closeness-step}).
\begin{figure}[h]
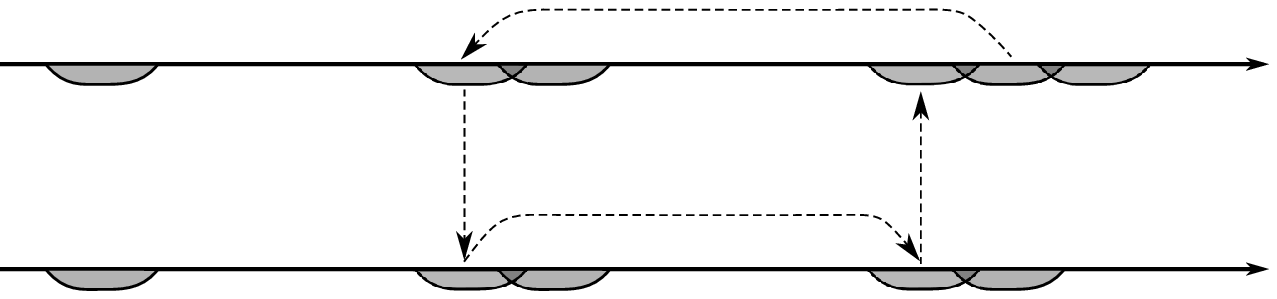
\caption{}  \label{fig:coarsely-periodic-closeness-step}
\end{figure}
\end{proof}

\begin{proof}[Proof of Proposition \ref{pr:hierarchical-containment}]
Let $A$ be a suspended period of level $m$ over $G_\al$,

Assume first that $m=0$. Then by Definition \ref{df:start-suspended}
and Proposition \ref{pr:coarsely-periodic-from-close} an $A$-periodic segment $!R$ in $G_\al$
contains a coarsely $B$-periodic segment $\hat{!T}$ with $\ell_B(\hat{!T}) \ge p_1 -  2h_\al(B) - 2 \ge 51$ 
where $B$ is not conjugate to ~$A$ in $G_\al$.  
By Lemma \ref{lm:coarsly-periodic-in-periodic} we have 
$\hat{!T} \not\sim s_{A,!R} \hat{!T}$ and $|\hat{!T}| < 2|A|$.
Let $!T$ be the stable part of $\hat{!T}$. Since $h_\al(B) \ge 2$ by Definition \ref{df:stable-part},
we have $|!T| < |A|$ by Corollary \ref{co:coarsely-periodic-cut}.
Note also that $\ell_B(!T) \ge \ell_B(\hat{!T}) - 2 h_\al(B) \ge p_0$.
Let $T \greq \lab(!T)$. We show that $T$ has the required property (ii) formulated in 
Proposition \ref{pr:hierarchical-containment}

Let $!S$ be a coarsely $A$-periodic segment in $\Ga_\al$ with $\ell_A(!S) \ge 4$ 
and let $!P$ be a periodic base for $!S$.
Denote $t = \ell(!S)$. 
By Remark \ref{rm:base-whole-periods} we can assume that $|!P| \ge t |A|$.
Up to placing ~$\hat{!T}$ in ~$\Ga_\al$ we can assume that 
$!P$ contains $t-2$ translates $\hat{!T}$, $s_{A,!P} \hat{!T}$, $\dots$, $s_{A,!P}^{t-3} \hat{!T}$
of $\hat{!T}$.
Using Lemma \ref{lmi:compatible-close} 
(which implies that strictly compatible coarsely periodic segments are close) and 
Proposition \ref{pr:coarse-periodic-stability} 
we find disjoint $!V_i$ $(i=0,\dots,t-3)$ in $!S$ such that $!V_i \approx s_{A,!P}^i !T$.
This proves the proposition in the case $m=0$.

Let $m \ge 1$. The proof consists of two parts. First we provide a construction of 
a coarsely $B$-periodic segment $T$ satisfying condition (i) of 
Proposition  \ref{pr:hierarchical-containment} and then we prove (ii).

{\em Construction of $T$}.
According to Definition \ref{df:next-suspended}, there exists a sequence $A_0$, $A_1$, $\dots$, $A_m=A$
of simple periods over $G_\al$ where $A_0$ is suspended of level 0, for each $i \le m-1$
$A_i$ is not conjugate to $A_{i+1}$ and there are reduced in $G_\al$ close 
words $X_i Q_i Y_i$ and $P_{i+1} \in \per(A_{i+1})$ where $Q_i \in \per(A_i)$ and $|Q_i| \ge 4|A_i|$.
For each $i$, we consider corresponding close paths $!X_i !Q_i !Y_i$ and $!P_{i+1}$ in $\Ga_\al$ 
and place then in such a way that $!Q_i$ and $!P_i$ have the common infinite $A_i$-periodic extension $!L_i$.
We denote also $!L_0$ the infinite $A_0$-periodic extension of $!Q_0$.

As we proved above, there is a coarsely $B$-periodic segment $\hat{!T}_0$ in $!Q_0$ with 
$\ell(\hat{!T}_0) \ge 51$ and the stable part $!T_0$ satisfying $\ell(!T_0) \ge p_0$  
and $|!T_0| < |A|$. 
Up to positioning $\hat{!T}_0$ in $!L_0$ we can assume that 
$!Q_0$ contains translates $s_{A_0,!L_0}^{-1} !T_0$ and $s_{A_0,!L_0} !T_0$ of $!T_0$.
In what follows, if $!Z$ is a coarsely $B$-periodic segment in $\Ga_\al$ then $!Z^*$ denotes
the stable part of $!Z$. By Lemma \ref{lm:coarsly-periodic-in-periodic}, 
$s_{A_0,!L_0}^t !T_0 \not\sim !T_0$ for any $t\ne 0$ and hence 
$s_{A_0,!L_0}^{-1} !T_0  \lnsim \hat{!T}_0  \lnsim s_{A_0,!L_0} !T_0$.
By Proposition \ref{pr:coarse-periodic-stability} there are $!T_1$, $!U_{1,1}$ and $!W_{1,1}$ in $!P_1$ such that $!T_1 \approx !T_0$, $!U_{1,1} \approx s_{A_0,!L_0}^{-1} !T_0^*$
and $!W_{1,1} \approx s_{A_0,!L_0} !T_0^*$. 
Application of Lemma \ref{lm:coarsely-periodic-closeness-step} with $!S:= !T_0^*$
(note that $\ell_B(!T_0^*) \ge p_0 - 12 \ge 27$)
gives $s_{A_1,!L_1}^{-1} !T_1 \lnsim !U_{1,1}$ and $!W_{1,1} \lnsim s_{A_1,!L_1} !T_1$. 
In particular, we have $|!T_1| \le |A_1|$. In  the case $m=1$ we take $T := \lab(!T_1)$.

Assume that $m \ge 2$. We continue a procedure of finding coarsely 
$B$-periodic segments ~$!T_i$ in $!P_i$.
Up to positioning $!Q_1$ in $!L_1$ we can assume that $!Q_1$ contains both 
$s_{A_1,!L_1}^{-1} !T_1$ and $s_{A_1,!L_1} !T_1$. 
Using Proposition \ref{pr:coarse-periodic-stability} 
we find $!U_{2,2}$, $!U_{2,1}$,  $!W_{2,1}$ and $!W_{2,2}$ in $!P_2$ such that 
$!U_{2,2} \approx s_{A_1,!L_1}^{-1} !T_1^*$, $!U_{2,1} \approx !U_{1,1}^*$,
$!W_{2,1} \approx !W_{1,1}^*$ and  $!W_{2,2} \approx s_{A_1,!L_1} !T_1^*$.
By Lemma \ref{lm:coarse-periodic-oredering-close}, 
$!U_{2,2} \lnsim  !U_{2,1} \lnsim !W_{2,1} \lnsim !W_{2,2}$.
We have $!U_{2,1} \approx s_{A_0,!L_0}^{-1} !T_0^{**}$, 
$!W_{2,1} \approx s_{A_0,!L_0} !T_0^{**}$ and 
using Proposition \ref{pr:coarse-periodic-stability} once more with 
$!X := s_{A_0,!L_0}^{-1} !T_0^{**} \cup s_{A_0,!L_0} !T_0^{**}$ and 
$!Y := !U_{2,1} \cup !W_{2,1}$ we find $!T_2$ in $!P_2$ such that $!T_2 \approx !T_0$.
Application of Lemma \ref{lm:coarsely-periodic-closeness-step} gives 
$s_{A_2,!L_2}^{-1} !T_2 \lnsim !U_{2,2}$ and $!W_{2,2} \lnsim s_{A_2,!L_2} !T_2$. 
In particular,  $|!T_2| \le |A_2|$.

Repeating in a similar manner, we find $!U_{m,m}$, $!U_{m,m-1}$, $!W_{m,m-1}$ and $!W_{m,m}$ in $!P_m$
such that $!U_{m,m} \approx s_{A_{m-1},!L_{m-1}}^{-1} !T_{m-1}^*$, $!U_{m,m-1} \approx !U_{m-1,m-1}^*$,
$!W_{m,m-1} \approx !W_{m-1,m-1}^*$, $!W_{m,m} \approx s_{A_{m-1},!L_{m-1}} !T_{m-1}^*$
and $!U_{m,m} \lnsim !U_{m,m-1} \lnsim !W_{m,m-1} \lnsim !W_{m,m}$.
Then we successively find $!U_{m,m-2}$, $!W_{m,m-2}$, $!U_{m,m-3}$, $!W_{m,m-3}$, $\dots$, 
$!U_{m,1}$, $!W_{m,1}$ such that $!U_{m,i} \approx !U_{i,i}^* \approx s_{A_{i-1},!L_{i-1}}^{-1} !T_{i-1}^{**}$ 
and $!W_{m,i} \approx !V_{i,i}^* \approx s_{A_{i-1},!L_{i-1}} !T_{i-1}^{**}$.
Finally, we find $!T_m$ in $!P_m$ such that $!T_m \approx !T_0$.
Application of Lemma \ref{lm:coarsely-periodic-closeness-step} gives 
$s_{A_m,!L_m}^{-1} !T_m \lnsim !U_{m,m}$ and $!W_{m,m} \lnsim s_{A_m,!L_m} !T_m$
which implies $|!T_m| \le |A_m|$. We take $T := \lab(!T_m)$. This completes the construction,
The whole procedure is schematically shown in Figure \ref{fig:hierarchical-containment-i}.   
\begin{figure}[h]
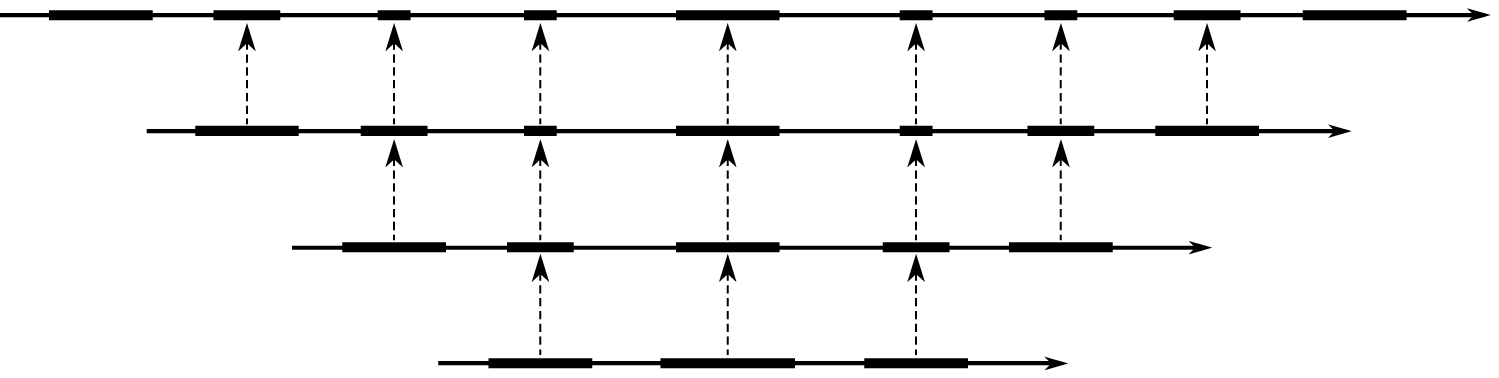
\caption{}  \label{fig:hierarchical-containment-i}
\end{figure}
Note that in $!P_m$ we have
$$
    s_{A_m,!L_m}^{-1} !T_m \lnsim !U_{m,m} \lnsim !U_{m,m-1} \lnsim\dots\lnsim !U_{m,1} \lnsim !T_m
    \lnsim !W_{m,1} \lnsim \dots \lnsim !W_{m,m} \lnsim s_{A_m,!L_m} !T_m.
$$
{\em Proof of (ii).}
Let $!S$ be a coarsely $A_m$-periodic segment in $\Ga_\al$ and let $!P$ be a periodic base for ~$!S$.
Denote $t = \ell_{A_m}(!S)$.
By Remark \ref{rm:base-whole-periods} we can assume that $|!P| \ge t |A|$, so $!P$ contains $t-1$
translates $!T$, $s_{A,!P} !T$, $\dots$, $s_{A,!P}^{t-2} !T$ of a coarsely $B$-periodic segment $!T$
which is a translate of $!T_m$ constructed above.
By Proposition \ref{pr:coarse-periodic-stability}, 
$!S$ contains coarsely $B$-periodic segments $!Z_0$, $!Z_1$, $\dots$, $!Z_{t-2}$ such that
$!Z_i \approx s_{A,!P}^i !T^*$. We claim, moreover, that for $1 \le i \le t-3$ there exist ~$!V_i$ in ~$!S$
such that $!V_i \approx s_{A,!P}^i !T$ and $!V_i$ are all disjoint. Since 
 $\ell_B(!V_i) = \ell_B(!T_m) \ge p_0$ this will finish the proof.

Fix an index $k$ in the interval $1 \le i \le t-3$. 
Up to positioning $!P$ and $!S$ in $\Ga_\al$ we can assume that $!P$ 
and $!P_m$ have the common $A_m$-periodic extension $!L_m$ 
and $s_{A,!P}^k !T = !T_m$.
By Lemma ~\ref{lm:coarsely-periodic-closeness-step}, 
$s_{A_m,!L_m}^{-1} !T  \lnsim !U_{m,m}$ and  $!W_{m,m}  \lnsim s_{A_m,!L_m} !T$.
Then using  Proposition \ref{pr:coarse-periodic-stability} as in the procedure above,
we successively find pairs $(!U_i,!W_i)$ for $i=m,m-1,\dots,1$ such that 
$!Z_{k-1} \lnsim !U_m \lnsim !U_{m-1} \lnsim\dots\lnsim !U_1 \lnsim !Z_k \lnsim !W_1 
\lnsim\dots\lnsim !W_m \lnsim !Z_{k+1}$ and
$!U_i \approx !U_{i,i}^*$, $!W_i \approx !W_{i,i}^*$ for $i=m,m-1,\dots,1$.
Then using Proposition \ref{pr:coarse-periodic-stability} again
with $!X := s_{A_0,!L_0}^{-1} !T_0^{**} \cup s_{A_0,!L_0} !T_0^{**}$,
$!Y := !U_1 \cup !W_1$ and $!S = \hat{!T}_0$ gives $!V_k$ with $!U_1 \lnsim !V_k \lnsim !W_1$
and $!V_k \approx !T_0 \approx s_{A,!P}^k !T$. The proof is finished.
\end{proof}

\begin{proposition} \label{pr:length-in-periods}
Let $A \in \cE_{\al+1}$ and $t\ge 1$ be an integer. 
Let $P$ be an $A$-periodic word with $|P| = t|A|$. Then 
$$ 
  \frac{t}{n+t} < \mu(P) < \frac{t}{n-t} + \om.
$$
Moreover, for $t \ge 200$ we have also
$$
  0.89 \frac{t}{n} < \mu(P) < 1.12 \frac{t}{n}.
$$
\end{proposition}

\begin{proof}
Denote $N = |(A^n)^\circ|_\al$.
Recall that $\mu(P) = |P|_\al / N$.
Up to cyclic shift of $A$, we assume that $P \greq A^t$.
For the lower bound on $\mu(P)$ in the first inequality, 
we observe that the cyclic word $(A^n)^\circ$ can be covered with $\ceil{\frac{n}{t}}$
copies of $P$. By \ref{ss:alpha-length-properties}, this implies 
$$
  N < \left(\frac{n}{t} + 1 \right) |P|_\al
$$
which is equivalent to $\frac{t}{n+t} < \mu(P)$.
Similarly, for the upper bound we observe that $\floor{\frac{n}{t}}$ disjoint copies of $P$ can be placed 
inside $(A^n)^\circ$. Then again by \ref{ss:alpha-length-properties},
$$
  N \ge \bbfloor{\frac{n}{t}} (|P|_\al -1) > \left(\frac{n}{t} - 1 \right) (|P|_\al -1)
$$
which implies by (S1) with $\al := \al+1$
$$
  \mu(P) < \frac{t}{n-t} + \frac{1}{N} \le \frac{t}{n-t} + \om.
$$

If $t \ge 200$ then we partition $A^t$ into $k$ subwords $A^{t_i}$ with $80 \le t_i \le 120$.
We have
$$
  \sum_i |A_{t_i}|_\al  - (k-1) \le |P|_\al \le \sum_i |A_{t_i}|_\al.
$$
and by the already proved bounds on $\mu(A^{t_i})$, for each $i$ we have
$$
  0.94 \frac{t_i}{n} < \mu(A^{t_i}) < 1.07 \frac{t_i}{n} + \frac{1}{N}.
$$
Then
$$
  \mu(P) \ge \sum_i \mu(A^{t_i}) - \frac{k-1}{N} > 0.94 \frac{t}{n}  - \frac{k}{N}.
$$
By Proposition \ref{pr:period-is-large}, $N \ge 0.25 n$. Hence
$$
  \frac{k}{N} \le \frac{t}{80} \left(\frac{n}{N} \right) \frac{1}{n} \le 0.05 \frac{t}{n}
$$
and we obtain the required bound $\mu(P) > 0.89 \frac{t}{n}$.
Similarly, for the upper bound on $\mu(P)$ we get
$$
  \mu(P) \le \sum_i \mu(A^{t_i}) < 1.07 \frac{t}{n} + \frac{k}{N} \le 1.12 \frac{t}{n}.
$$
\end{proof}

\begin{corollary}
(P2) implies (S2)$_{\al+1}$. 
\end{corollary}

\begin{proof}
By Proposition \ref{pr:length-in-periods}, 
if $P$ is a subword of $A^n$ with $A \in \cE_{\al+1}$ and $\mu(P) \ge \la$
then $|P| \ge t|A|$ where $t$ satisfies
$$
    \frac{t}{n-t} \ge \la - \om \ge \frac{1}{24} - \frac{1}{480}
$$
and hence $t > 76$.
Since $76 > p_1$, the required implication is straightforward.
\end{proof}

\begin{proposition} \label{pr:G-small-cancellation}
Presentation \eqref{eq:Burnside-presentation} satisfies (P2) and therefore satisfies the iterated
small cancellation condition (S0)--(S3) for all $\al\ge1$.
\end{proposition}

\begin{proof}
Indeed, assume that $!L_1$ and $!L_2$ are periodic lines in $\Ga_\al$ with periods $A,B \in \cE_{\al+1}$ respectively.
Let $!P$ and $!Q$ be close subpath of $!L_1$ and $!L_2$, respectively, such that $|!P| \ge p_1 |A|$.
If $A$ is conjugate to $B$ in $G_\al$ then $A=B$ according to Definition \ref{df:elementary-periods}
and the statement follows from Proposition \ref{pr:close-parallel-one-period}.
If $A$ is not conjugate to $B$ in $G_\al$ then $B$ is suspended of level 0 
as a simple period over $G_\al$ and hence cannot belong to $\cE_{\al+1}$.
\end{proof}

From this point, we may assume that all statements in 
Sections \ref{s:diagrams}--\ref{s:hierarchical-containment}
are true for all values of rank $\al$.

\begin{proposition}
Every element of $G$ is conjugate to a power of some $C \in \bigcup_{\al\ge 1} \cE_\al$.
\end{proposition}

\begin{proof} 
Let $g \in G$. If $g$ has finite order then by Proposition \ref{pr:cyclic-reduction}, 
$g$  is conjugate to a power of some $C \in \bigcup_{\al\ge 1} \cE_\al$.
We assume that $g$ has infinite order and come to a contradiction.

By Corollary \ref{co:absolute-reduction} we represent $g$ by a word $X$ reduced in $G$ such 
that for some $\al\ge 1$, 
$X$ contains no fragments $F$ of rank $\be \ge \al$ with $\muf(F) \ge 3\om$.
By our assumption, $X$ has infinite order in all $G_\be$ for $\be \ge \al$.
By Propositions \ref{pr:root-existence} and \ref{pr:cyclic-reduction}, $X$ is conjugate in $G_\al$
to a word of the form $A^t$ where $A$ is a simple period over $G_\al$.
Using Proposition \ref{pri:cyclic-monogon-cell} we conclude that $X$ is conjugate to $A^t$
already in $G_{\al-1}$. Then applying Proposition \ref{pr:fragment-cyclic-monogon-previous}
with $\be := \al, \al+1, \dots$ we see that no cyclic shift of $A$ contains a fragment $K$ of rank 
$\be \ge \al$ with $\muf(K) \ge 9\om$ and 
that $A$ is cyclically reduced in $G_\be$ for all $\be > \al$.
Moreover, by Propositions \ref{pri:fragment-in-periodic-small} and \ref{pr:dividing-fragment},
$A$ is strongly cyclically reduced in $G_\be$ for all $\be > \al$.

Assume that for some $\be \ge \al$, $A$ is conjugate in $G_\be$ 
to a power $B^r$ of a simple period over ~$G_\be$. By Proposition \ref{pr:no-active-fragments-cyclic},
$A$ and $B^r$ are conjugate already in ~$G_{\al}$. Since $A$ is a non-power in $G_\al$, we have
$r=1$ and then by Propositions \ref{pr:root-existence} and  \ref{pr:cyclic-reduction}, $A$ 
is a non-power in ~$G_\be$. We showed that $A$ is a simple period over $G_\be$
for any $\be\ge\al$.  But this is impossible because by Proposition \ref{pr:period-is-large}
we should have $|A|_\be \ge 0.25$ and hence $|A| \ge 0.25 \ze^{-\be}$ for any $\be\ge\al$.
\end{proof}

As an immediate consequence we get:

\begin{corollary} \label{co:G-Burnside}
$G$ satisfies the identity $x^n = 1$ and therefore is isomorphic to the free Burnside group $B(m,n)$.
\end{corollary}

\bibliography{bibliography}

\end{document}

%% file: diagram-previous.eps_tex
\begingroup%
  \makeatletter%
  \providecommand\color[2][]{%
    \errmessage{(Inkscape) Color is used for the text in Inkscape, but the package 'color.sty' is not loaded}%
    \renewcommand\color[2][]{}%
  }%
  \providecommand\transparent[1]{%
    \errmessage{(Inkscape) Transparency is used (non-zero) for the text in Inkscape, but the package 'transparent.sty' is not loaded}%
    \renewcommand\transparent[1]{}%
  }%
  \providecommand\rotatebox[2]{#2}%
  \newcommand*\fsize{\dimexpr\f@size pt\relax}%
  \newcommand*\lineheight[1]{\fontsize{\fsize}{#1\fsize}\selectfont}%
  \ifx\svgwidth\undefined%
    \setlength{\unitlength}{343.34646798bp}%
    \ifx\svgscale\undefined%
      \relax%
    \else%
      \setlength{\unitlength}{\unitlength * \real{\svgscale}}%
    \fi%
  \else%
    \setlength{\unitlength}{\svgwidth}%
  \fi%
  \global\let\svgwidth\undefined%
  \global\let\svgscale\undefined%
  \makeatother%
  \begin{picture}(1,0.35254884)%
    \lineheight{1}%
    \setlength\tabcolsep{0pt}%
    \put(0,0){\includegraphics[width=\unitlength]{diagram-previous.eps}}%
    \put(0.34780908,0.2992368){\color[rgb]{0,0,0}\makebox(0,0)[lt]{\lineheight{1.25}\smash{\begin{tabular}[t]{l}cells of rank $\al$\end{tabular}}}}%
    \put(0.15312879,0.00028526){\color[rgb]{0,0,0}\makebox(0,0)[lt]{\lineheight{1.25}\smash{\begin{tabular}[t]{l}$\De$\end{tabular}}}}%
    \put(0.76211401,0.00028526){\color[rgb]{0,0,0}\makebox(0,0)[lt]{\lineheight{1.25}\smash{\begin{tabular}[t]{l}$\De_{\al-1}$\end{tabular}}}}%
  \end{picture}%
\endgroup%

%% file: diagram-cell-cell.eps_tex
\begingroup%
  \makeatletter%
  \providecommand\color[2][]{%
    \errmessage{(Inkscape) Color is used for the text in Inkscape, but the package 'color.sty' is not loaded}%
    \renewcommand\color[2][]{}%
  }%
  \providecommand\transparent[1]{%
    \errmessage{(Inkscape) Transparency is used (non-zero) for the text in Inkscape, but the package 'transparent.sty' is not loaded}%
    \renewcommand\transparent[1]{}%
  }%
  \providecommand\rotatebox[2]{#2}%
  \newcommand*\fsize{\dimexpr\f@size pt\relax}%
  \newcommand*\lineheight[1]{\fontsize{\fsize}{#1\fsize}\selectfont}%
  \ifx\svgwidth\undefined%
    \setlength{\unitlength}{437.69793672bp}%
    \ifx\svgscale\undefined%
      \relax%
    \else%
      \setlength{\unitlength}{\unitlength * \real{\svgscale}}%
    \fi%
  \else%
    \setlength{\unitlength}{\svgwidth}%
  \fi%
  \global\let\svgwidth\undefined%
  \global\let\svgscale\undefined%
  \makeatother%
  \begin{picture}(1,0.27158131)%
    \lineheight{1}%
    \setlength\tabcolsep{0pt}%
    \put(0,0){\includegraphics[width=\unitlength]{diagram-cell-cell.eps}}%
    \put(0.10428771,0.15348707){\color[rgb]{0,0,0}\makebox(0,0)[lt]{\lineheight{1.25}\smash{\begin{tabular}[t]{l}$!a$\end{tabular}}}}%
    \put(0.16416709,0.12830233){\color[rgb]{0,0,0}\makebox(0,0)[lt]{\lineheight{1.25}\smash{\begin{tabular}[t]{l}$!b$\end{tabular}}}}%
    \put(0.13310501,0.13222277){\color[rgb]{0,0,0}\makebox(0,0)[lt]{\lineheight{1.25}\smash{\begin{tabular}[t]{l}$!p$\end{tabular}}}}%
    \put(0.07388599,0.16326778){\color[rgb]{0,0,0}\makebox(0,0)[lt]{\lineheight{1.25}\smash{\begin{tabular}[t]{l}$!C$\end{tabular}}}}%
    \put(0.19168823,0.12283902){\color[rgb]{0,0,0}\makebox(0,0)[lt]{\lineheight{1.25}\smash{\begin{tabular}[t]{l}$!D$\end{tabular}}}}%
    \put(0.07514976,0.21987168){\color[rgb]{0,0,0}\makebox(0,0)[lt]{\lineheight{1.25}\smash{\begin{tabular}[t]{l}$!Q$\end{tabular}}}}%
    \put(0.19403002,0.18242375){\color[rgb]{0,0,0}\makebox(0,0)[lt]{\lineheight{1.25}\smash{\begin{tabular}[t]{l}$!R$\end{tabular}}}}%
    \put(0.41107036,0.16236747){\color[rgb]{0,0,0}\makebox(0,0)[lt]{\lineheight{1.25}\smash{\begin{tabular}[t]{l}$\Th$\end{tabular}}}}%
    \put(0.80540415,0.13907056){\color[rgb]{0,0,0}\makebox(0,0)[lt]{\lineheight{1.25}\smash{\begin{tabular}[t]{l}$!D$\end{tabular}}}}%
    \put(0.91766789,0.13892474){\color[rgb]{0,0,0}\makebox(0,0)[lt]{\lineheight{1.25}\smash{\begin{tabular}[t]{l}$!C$\end{tabular}}}}%
    \put(0.87848027,0.09279113){\color[rgb]{0,0,0}\makebox(0,0)[lt]{\lineheight{1.25}\smash{\begin{tabular}[t]{l}$!v$\end{tabular}}}}%
    \put(0.86766443,0.11896064){\color[rgb]{0,0,0}\makebox(0,0)[lt]{\lineheight{1.25}\smash{\begin{tabular}[t]{l}$!S$\end{tabular}}}}%
    \put(0.87193098,0.1809195){\color[rgb]{0,0,0}\makebox(0,0)[lt]{\lineheight{1.25}\smash{\begin{tabular}[t]{l}$!w$\end{tabular}}}}%
    \put(0.14238405,0.22936283){\color[rgb]{0,0,0}\makebox(0,0)[lt]{\lineheight{1.25}\smash{\begin{tabular}[t]{l}$\De$\end{tabular}}}}%
    \put(0.48016206,0.23071909){\color[rgb]{0,0,0}\makebox(0,0)[lt]{\lineheight{1.25}\smash{\begin{tabular}[t]{l}$\De$\end{tabular}}}}%
    \put(0.77099279,0.22187038){\color[rgb]{0,0,0}\makebox(0,0)[lt]{\lineheight{1.25}\smash{\begin{tabular}[t]{l}$\De$\end{tabular}}}}%
    \put(0.13203516,0.0002488){\color[rgb]{0,0,0}\makebox(0,0)[lt]{\lineheight{1.25}\smash{\begin{tabular}[t]{l}a\end{tabular}}}}%
    \put(0.47890269,0.0002488){\color[rgb]{0,0,0}\makebox(0,0)[lt]{\lineheight{1.25}\smash{\begin{tabular}[t]{l}$!b$\end{tabular}}}}%
    \put(0.82040946,0.0002488){\color[rgb]{0,0,0}\makebox(0,0)[lt]{\lineheight{1.25}\smash{\begin{tabular}[t]{l}c\end{tabular}}}}%
    \put(0.97251727,0.14001225){\color[rgb]{0,0,0}\makebox(0,0)[lt]{\lineheight{1.25}\smash{\begin{tabular}[t]{l}$!T$\end{tabular}}}}%
  \end{picture}%
\endgroup%

%% file: bond-excluded-cases.eps_tex
\begingroup%
  \makeatletter%
  \providecommand\color[2][]{%
    \errmessage{(Inkscape) Color is used for the text in Inkscape, but the package 'color.sty' is not loaded}%
    \renewcommand\color[2][]{}%
  }%
  \providecommand\transparent[1]{%
    \errmessage{(Inkscape) Transparency is used (non-zero) for the text in Inkscape, but the package 'transparent.sty' is not loaded}%
    \renewcommand\transparent[1]{}%
  }%
  \providecommand\rotatebox[2]{#2}%
  \newcommand*\fsize{\dimexpr\f@size pt\relax}%
  \newcommand*\lineheight[1]{\fontsize{\fsize}{#1\fsize}\selectfont}%
  \ifx\svgwidth\undefined%
    \setlength{\unitlength}{368.37554458bp}%
    \ifx\svgscale\undefined%
      \relax%
    \else%
      \setlength{\unitlength}{\unitlength * \real{\svgscale}}%
    \fi%
  \else%
    \setlength{\unitlength}{\svgwidth}%
  \fi%
  \global\let\svgwidth\undefined%
  \global\let\svgscale\undefined%
  \makeatother%
  \begin{picture}(1,0.32652779)%
    \lineheight{1}%
    \setlength\tabcolsep{0pt}%
    \put(0,0){\includegraphics[width=\unitlength]{bond-excluded-cases.eps}}%
    \put(0.17557005,0.20851668){\color[rgb]{0,0,0}\makebox(0,0)[lt]{\lineheight{1.25}\smash{\begin{tabular}[t]{l}$!u$\end{tabular}}}}%
    \put(0.89549515,0.20941869){\color[rgb]{0,0,0}\makebox(0,0)[lt]{\lineheight{1.25}\smash{\begin{tabular}[t]{l}$!u$\end{tabular}}}}%
    \put(0.98103263,0.24715511){\color[rgb]{0,0,0}\makebox(0,0)[lt]{\lineheight{1.25}\smash{\begin{tabular}[t]{l}$!v$\end{tabular}}}}%
    \put(0.97116837,0.16189069){\color[rgb]{0,0,0}\makebox(0,0)[lt]{\lineheight{1.25}\smash{\begin{tabular}[t]{l}$!p$\end{tabular}}}}%
    \put(0.91854372,0.31009492){\color[rgb]{0,0,0}\makebox(0,0)[lt]{\lineheight{1.25}\smash{\begin{tabular}[t]{l}$!q$\end{tabular}}}}%
  \end{picture}%
\endgroup%

%% file: bond-to-bridge.eps_tex
\begingroup%
  \makeatletter%
  \providecommand\color[2][]{%
    \errmessage{(Inkscape) Color is used for the text in Inkscape, but the package 'color.sty' is not loaded}%
    \renewcommand\color[2][]{}%
  }%
  \providecommand\transparent[1]{%
    \errmessage{(Inkscape) Transparency is used (non-zero) for the text in Inkscape, but the package 'transparent.sty' is not loaded}%
    \renewcommand\transparent[1]{}%
  }%
  \providecommand\rotatebox[2]{#2}%
  \newcommand*\fsize{\dimexpr\f@size pt\relax}%
  \newcommand*\lineheight[1]{\fontsize{\fsize}{#1\fsize}\selectfont}%
  \ifx\svgwidth\undefined%
    \setlength{\unitlength}{338.70326789bp}%
    \ifx\svgscale\undefined%
      \relax%
    \else%
      \setlength{\unitlength}{\unitlength * \real{\svgscale}}%
    \fi%
  \else%
    \setlength{\unitlength}{\svgwidth}%
  \fi%
  \global\let\svgwidth\undefined%
  \global\let\svgscale\undefined%
  \makeatother%
  \begin{picture}(1,0.26348459)%
    \lineheight{1}%
    \setlength\tabcolsep{0pt}%
    \put(0,0){\includegraphics[width=\unitlength]{bond-to-bridge.eps}}%
    \put(0.19610527,0.12265462){\color[rgb]{0,0,0}\makebox(0,0)[lt]{\lineheight{1.25}\smash{\begin{tabular}[t]{l}$!u$\end{tabular}}}}%
    \put(0.68005053,0.12265462){\color[rgb]{0,0,0}\makebox(0,0)[lt]{\lineheight{1.25}\smash{\begin{tabular}[t]{l}$!u$\end{tabular}}}}%
    \put(0.75238667,0.12233026){\color[rgb]{0,0,0}\makebox(0,0)[lt]{\lineheight{1.25}\smash{\begin{tabular}[t]{l}$!v$\end{tabular}}}}%
    \put(0.81558747,0.12265462){\color[rgb]{0,0,0}\makebox(0,0)[lt]{\lineheight{1.25}\smash{\begin{tabular}[t]{l}$!u'$\end{tabular}}}}%
  \end{picture}%
\endgroup%

%% file: folded-cell.eps_tex
\begingroup%
  \makeatletter%
  \providecommand\color[2][]{%
    \errmessage{(Inkscape) Color is used for the text in Inkscape, but the package 'color.sty' is not loaded}%
    \renewcommand\color[2][]{}%
  }%
  \providecommand\transparent[1]{%
    \errmessage{(Inkscape) Transparency is used (non-zero) for the text in Inkscape, but the package 'transparent.sty' is not loaded}%
    \renewcommand\transparent[1]{}%
  }%
  \providecommand\rotatebox[2]{#2}%
  \newcommand*\fsize{\dimexpr\f@size pt\relax}%
  \newcommand*\lineheight[1]{\fontsize{\fsize}{#1\fsize}\selectfont}%
  \ifx\svgwidth\undefined%
    \setlength{\unitlength}{179.6109222bp}%
    \ifx\svgscale\undefined%
      \relax%
    \else%
      \setlength{\unitlength}{\unitlength * \real{\svgscale}}%
    \fi%
  \else%
    \setlength{\unitlength}{\svgwidth}%
  \fi%
  \global\let\svgwidth\undefined%
  \global\let\svgscale\undefined%
  \makeatother%
  \begin{picture}(1,0.91752234)%
    \lineheight{1}%
    \setlength\tabcolsep{0pt}%
    \put(0,0){\includegraphics[width=\unitlength]{folded-cell.eps}}%
    \put(0.12303574,0.4413979){\color[rgb]{0,0,0}\makebox(0,0)[lt]{\lineheight{1.25}\smash{\begin{tabular}[t]{l}$!Q$\end{tabular}}}}%
    \put(0.34174642,0.44218489){\color[rgb]{0,0,0}\makebox(0,0)[lt]{\lineheight{1.25}\smash{\begin{tabular}[t]{l}$!P$\end{tabular}}}}%
    \put(0.75060919,0.44970787){\color[rgb]{0,0,0}\makebox(0,0)[lt]{\lineheight{1.25}\smash{\begin{tabular}[t]{l}$!u$\end{tabular}}}}%
    \put(0.70242387,0.36248934){\color[rgb]{0,0,0}\makebox(0,0)[lt]{\lineheight{1.25}\smash{\begin{tabular}[t]{l}$!a$\end{tabular}}}}%
    \put(0.69746245,0.54873502){\color[rgb]{0,0,0}\makebox(0,0)[lt]{\lineheight{1.25}\smash{\begin{tabular}[t]{l}$!b$\end{tabular}}}}%
  \end{picture}%
\endgroup%

%% file: contiguity-map-exists-1.eps_tex
\begingroup%
  \makeatletter%
  \providecommand\color[2][]{%
    \errmessage{(Inkscape) Color is used for the text in Inkscape, but the package 'color.sty' is not loaded}%
    \renewcommand\color[2][]{}%
  }%
  \providecommand\transparent[1]{%
    \errmessage{(Inkscape) Transparency is used (non-zero) for the text in Inkscape, but the package 'transparent.sty' is not loaded}%
    \renewcommand\transparent[1]{}%
  }%
  \providecommand\rotatebox[2]{#2}%
  \newcommand*\fsize{\dimexpr\f@size pt\relax}%
  \newcommand*\lineheight[1]{\fontsize{\fsize}{#1\fsize}\selectfont}%
  \ifx\svgwidth\undefined%
    \setlength{\unitlength}{203.89650973bp}%
    \ifx\svgscale\undefined%
      \relax%
    \else%
      \setlength{\unitlength}{\unitlength * \real{\svgscale}}%
    \fi%
  \else%
    \setlength{\unitlength}{\svgwidth}%
  \fi%
  \global\let\svgwidth\undefined%
  \global\let\svgscale\undefined%
  \makeatother%
  \begin{picture}(1,0.40240413)%
    \lineheight{1}%
    \setlength\tabcolsep{0pt}%
    \put(0,0){\includegraphics[width=\unitlength]{contiguity-map-exists-1.eps}}%
    \put(0.79111345,0.1873164){\color[rgb]{0,0,0}\makebox(0,0)[lt]{\lineheight{1.25}\smash{\begin{tabular}[t]{l}$\Th_2$\end{tabular}}}}%
    \put(0.56846613,0.15155413){\color[rgb]{0,0,0}\makebox(0,0)[lt]{\lineheight{1.25}\smash{\begin{tabular}[t]{l}$\Th_1$\end{tabular}}}}%
    \put(0.42372894,0.1653825){\color[rgb]{0,0,0}\makebox(0,0)[lt]{\lineheight{1.25}\smash{\begin{tabular}[t]{l}$\Pi$\end{tabular}}}}%
    \put(0.66950416,0.163322){\color[rgb]{0,0,0}\makebox(0,0)[lt]{\lineheight{1.25}\smash{\begin{tabular}[t]{l}$!v$\end{tabular}}}}%
    \put(0.49537047,0.16793534){\color[rgb]{0,0,0}\makebox(0,0)[lt]{\lineheight{1.25}\smash{\begin{tabular}[t]{l}$!u$\end{tabular}}}}%
  \end{picture}%
\endgroup%

%% file: contiguity-map-exists-2.eps_tex
\begingroup%
  \makeatletter%
  \providecommand\color[2][]{%
    \errmessage{(Inkscape) Color is used for the text in Inkscape, but the package 'color.sty' is not loaded}%
    \renewcommand\color[2][]{}%
  }%
  \providecommand\transparent[1]{%
    \errmessage{(Inkscape) Transparency is used (non-zero) for the text in Inkscape, but the package 'transparent.sty' is not loaded}%
    \renewcommand\transparent[1]{}%
  }%
  \providecommand\rotatebox[2]{#2}%
  \newcommand*\fsize{\dimexpr\f@size pt\relax}%
  \newcommand*\lineheight[1]{\fontsize{\fsize}{#1\fsize}\selectfont}%
  \ifx\svgwidth\undefined%
    \setlength{\unitlength}{388.85183594bp}%
    \ifx\svgscale\undefined%
      \relax%
    \else%
      \setlength{\unitlength}{\unitlength * \real{\svgscale}}%
    \fi%
  \else%
    \setlength{\unitlength}{\svgwidth}%
  \fi%
  \global\let\svgwidth\undefined%
  \global\let\svgscale\undefined%
  \makeatother%
  \begin{picture}(1,0.18464581)%
    \lineheight{1}%
    \setlength\tabcolsep{0pt}%
    \put(0,0){\includegraphics[width=\unitlength]{contiguity-map-exists-2.eps}}%
    \put(0.06129038,0.0839192){\color[rgb]{0,0,0}\makebox(0,0)[lt]{\lineheight{1.25}\smash{\begin{tabular}[t]{l}$\De_{\al-1}$\end{tabular}}}}%
    \put(0.1915245,0.03615526){\color[rgb]{0,0,0}\makebox(0,0)[lt]{\lineheight{1.25}\smash{\begin{tabular}[t]{l}$!u$\end{tabular}}}}%
    \put(0.73378442,0.03129229){\color[rgb]{0,0,0}\makebox(0,0)[lt]{\lineheight{1.25}\smash{\begin{tabular}[t]{l}$!u'$\end{tabular}}}}%
    \put(0.83506977,0.0426755){\color[rgb]{0,0,0}\makebox(0,0)[lt]{\lineheight{1.25}\smash{\begin{tabular}[t]{l}$!u''$\end{tabular}}}}%
    \put(0.78690565,0.03418261){\color[rgb]{0,0,0}\makebox(0,0)[lt]{\lineheight{1.25}\smash{\begin{tabular}[t]{l}$\Pi$\end{tabular}}}}%
  \end{picture}%
\endgroup%

%% file: tiling-sides.eps_tex
\begingroup%
  \makeatletter%
  \providecommand\color[2][]{%
    \errmessage{(Inkscape) Color is used for the text in Inkscape, but the package 'color.sty' is not loaded}%
    \renewcommand\color[2][]{}%
  }%
  \providecommand\transparent[1]{%
    \errmessage{(Inkscape) Transparency is used (non-zero) for the text in Inkscape, but the package 'transparent.sty' is not loaded}%
    \renewcommand\transparent[1]{}%
  }%
  \providecommand\rotatebox[2]{#2}%
  \newcommand*\fsize{\dimexpr\f@size pt\relax}%
  \newcommand*\lineheight[1]{\fontsize{\fsize}{#1\fsize}\selectfont}%
  \ifx\svgwidth\undefined%
    \setlength{\unitlength}{238.34623744bp}%
    \ifx\svgscale\undefined%
      \relax%
    \else%
      \setlength{\unitlength}{\unitlength * \real{\svgscale}}%
    \fi%
  \else%
    \setlength{\unitlength}{\svgwidth}%
  \fi%
  \global\let\svgwidth\undefined%
  \global\let\svgscale\undefined%
  \makeatother%
  \begin{picture}(1,0.50533891)%
    \lineheight{1}%
    \setlength\tabcolsep{0pt}%
    \put(0,0){\includegraphics[width=\unitlength]{tiling-sides.eps}}%
    \put(0.14223811,0.3991156){\color[rgb]{0,0,0}\makebox(0,0)[lt]{\lineheight{1.25}\smash{\begin{tabular}[t]{l}{\scriptsize 1}\end{tabular}}}}%
    \put(0.13824987,0.05751831){\color[rgb]{0,0,0}\makebox(0,0)[lt]{\lineheight{1.25}\smash{\begin{tabular}[t]{l}{\scriptsize 1}\end{tabular}}}}%
    \put(0.44956936,0.3846276){\color[rgb]{0,0,0}\makebox(0,0)[lt]{\lineheight{1.25}\smash{\begin{tabular}[t]{l}{\scriptsize 1}\end{tabular}}}}%
    \put(0.33008122,0.2298444){\color[rgb]{0,0,0}\makebox(0,0)[lt]{\lineheight{1.25}\smash{\begin{tabular}[t]{l}{\scriptsize 1}\end{tabular}}}}%
    \put(0.46211627,0.07335151){\color[rgb]{0,0,0}\makebox(0,0)[lt]{\lineheight{1.25}\smash{\begin{tabular}[t]{l}{\scriptsize 1}\end{tabular}}}}%
    \put(0.71974454,0.06782059){\color[rgb]{0,0,0}\makebox(0,0)[lt]{\lineheight{1.25}\smash{\begin{tabular}[t]{l}{\scriptsize 1}\end{tabular}}}}%
    \put(0.68765768,0.40877472){\color[rgb]{0,0,0}\makebox(0,0)[lt]{\lineheight{1.25}\smash{\begin{tabular}[t]{l}{\scriptsize 1}\end{tabular}}}}%
    \put(0.83833747,0.25581221){\color[rgb]{0,0,0}\makebox(0,0)[lt]{\lineheight{1.25}\smash{\begin{tabular}[t]{l}{\scriptsize 1}\end{tabular}}}}%
    \put(0.58374205,0.23755561){\color[rgb]{0,0,0}\makebox(0,0)[lt]{\lineheight{1.25}\smash{\begin{tabular}[t]{l}{\scriptsize 1}\end{tabular}}}}%
    \put(0.32577334,0.33776115){\color[rgb]{0,0,0}\makebox(0,0)[lt]{\lineheight{1.25}\smash{\begin{tabular}[t]{l}{\small 0}\end{tabular}}}}%
    \put(0.58029821,0.36230758){\color[rgb]{0,0,0}\makebox(0,0)[lt]{\lineheight{1.25}\smash{\begin{tabular}[t]{l}{\small 0}\end{tabular}}}}%
    \put(0.81859146,0.39541835){\color[rgb]{0,0,0}\makebox(0,0)[lt]{\lineheight{1.25}\smash{\begin{tabular}[t]{l}{\small 0}\end{tabular}}}}%
    \put(0.32846308,0.10173229){\color[rgb]{0,0,0}\makebox(0,0)[lt]{\lineheight{1.25}\smash{\begin{tabular}[t]{l}{\small 0}\end{tabular}}}}%
    \put(0.58545697,0.09961519){\color[rgb]{0,0,0}\makebox(0,0)[lt]{\lineheight{1.25}\smash{\begin{tabular}[t]{l}{\small 0}\end{tabular}}}}%
    \put(0.86674515,0.08798163){\color[rgb]{0,0,0}\makebox(0,0)[lt]{\lineheight{1.25}\smash{\begin{tabular}[t]{l}{\small 0}\end{tabular}}}}%
    \put(0.11317218,0.21986346){\color[rgb]{0,0,0}\makebox(0,0)[lt]{\lineheight{1.25}\smash{\begin{tabular}[t]{l}2\end{tabular}}}}%
    \put(0.45507251,0.22355684){\color[rgb]{0,0,0}\makebox(0,0)[lt]{\lineheight{1.25}\smash{\begin{tabular}[t]{l}2\end{tabular}}}}%
    \put(0.71415646,0.23081401){\color[rgb]{0,0,0}\makebox(0,0)[lt]{\lineheight{1.25}\smash{\begin{tabular}[t]{l}2\end{tabular}}}}%
  \end{picture}%
\endgroup%

%% file: monogon-regularity.eps_tex
\begingroup%
  \makeatletter%
  \providecommand\color[2][]{%
    \errmessage{(Inkscape) Color is used for the text in Inkscape, but the package 'color.sty' is not loaded}%
    \renewcommand\color[2][]{}%
  }%
  \providecommand\transparent[1]{%
    \errmessage{(Inkscape) Transparency is used (non-zero) for the text in Inkscape, but the package 'transparent.sty' is not loaded}%
    \renewcommand\transparent[1]{}%
  }%
  \providecommand\rotatebox[2]{#2}%
  \newcommand*\fsize{\dimexpr\f@size pt\relax}%
  \newcommand*\lineheight[1]{\fontsize{\fsize}{#1\fsize}\selectfont}%
  \ifx\svgwidth\undefined%
    \setlength{\unitlength}{190.91521061bp}%
    \ifx\svgscale\undefined%
      \relax%
    \else%
      \setlength{\unitlength}{\unitlength * \real{\svgscale}}%
    \fi%
  \else%
    \setlength{\unitlength}{\svgwidth}%
  \fi%
  \global\let\svgwidth\undefined%
  \global\let\svgscale\undefined%
  \makeatother%
  \begin{picture}(1,0.48913258)%
    \lineheight{1}%
    \setlength\tabcolsep{0pt}%
    \put(0,0){\includegraphics[width=\unitlength]{monogon-regularity.eps}}%
    \put(0.56627459,0.22431105){\color[rgb]{0,0,0}\makebox(0,0)[lt]{\lineheight{1.25}\smash{\begin{tabular}[t]{l}$!D$\end{tabular}}}}%
    \put(0.19940662,0.22331756){\color[rgb]{0,0,0}\makebox(0,0)[lt]{\lineheight{1.25}\smash{\begin{tabular}[t]{l}$\De'$\end{tabular}}}}%
    \put(0.56694095,0.03686232){\color[rgb]{0,0,0}\makebox(0,0)[lt]{\lineheight{1.25}\smash{\begin{tabular}[t]{l}$\Pi_1$\end{tabular}}}}%
    \put(0.55196596,0.42212998){\color[rgb]{0,0,0}\makebox(0,0)[lt]{\lineheight{1.25}\smash{\begin{tabular}[t]{l}$\Pi_2$\end{tabular}}}}%
  \end{picture}%
\endgroup%

%% file: diagram-regularity-exception.eps_tex
\begingroup%
  \makeatletter%
  \providecommand\color[2][]{%
    \errmessage{(Inkscape) Color is used for the text in Inkscape, but the package 'color.sty' is not loaded}%
    \renewcommand\color[2][]{}%
  }%
  \providecommand\transparent[1]{%
    \errmessage{(Inkscape) Transparency is used (non-zero) for the text in Inkscape, but the package 'transparent.sty' is not loaded}%
    \renewcommand\transparent[1]{}%
  }%
  \providecommand\rotatebox[2]{#2}%
  \newcommand*\fsize{\dimexpr\f@size pt\relax}%
  \newcommand*\lineheight[1]{\fontsize{\fsize}{#1\fsize}\selectfont}%
  \ifx\svgwidth\undefined%
    \setlength{\unitlength}{419.13843126bp}%
    \ifx\svgscale\undefined%
      \relax%
    \else%
      \setlength{\unitlength}{\unitlength * \real{\svgscale}}%
    \fi%
  \else%
    \setlength{\unitlength}{\svgwidth}%
  \fi%
  \global\let\svgwidth\undefined%
  \global\let\svgscale\undefined%
  \makeatother%
  \begin{picture}(1,0.34149581)%
    \lineheight{1}%
    \setlength\tabcolsep{0pt}%
    \put(0,0){\includegraphics[width=\unitlength]{diagram-regularity-exception.eps}}%
    \put(0.33591306,0.05699532){\color[rgb]{0,0,0}\makebox(0,0)[lt]{\lineheight{1.25}\smash{\begin{tabular}[t]{l}$!P_1$\end{tabular}}}}%
    \put(0.40820161,0.07154064){\color[rgb]{0,0,0}\makebox(0,0)[lt]{\lineheight{1.25}\smash{\begin{tabular}[t]{l}$!T_1$\end{tabular}}}}%
    \put(0.43814398,0.20265708){\color[rgb]{0,0,0}\makebox(0,0)[lt]{\lineheight{1.25}\smash{\begin{tabular}[t]{l}$!v$\end{tabular}}}}%
    \put(0.31099877,0.1826146){\color[rgb]{0,0,0}\makebox(0,0)[lt]{\lineheight{1.25}\smash{\begin{tabular}[t]{l}$!D$\end{tabular}}}}%
    \put(0.35109675,0.19185952){\color[rgb]{0,0,0}\makebox(0,0)[lt]{\lineheight{1.25}\smash{\begin{tabular}[t]{l}$!u_1$\end{tabular}}}}%
    \put(0.28968784,0.30531633){\color[rgb]{0,0,0}\makebox(0,0)[lt]{\lineheight{1.25}\smash{\begin{tabular}[t]{l}$!P_2$\end{tabular}}}}%
    \put(0.38165453,0.32215158){\color[rgb]{0,0,0}\makebox(0,0)[lt]{\lineheight{1.25}\smash{\begin{tabular}[t]{l}$!T_2$\end{tabular}}}}%
    \put(0.22211151,0.18427775){\color[rgb]{0,0,0}\makebox(0,0)[lt]{\lineheight{1.25}\smash{\begin{tabular}[t]{l}$!u_2$\end{tabular}}}}%
    \put(0.26862042,0.03558349){\color[rgb]{0,0,0}\makebox(0,0)[lt]{\lineheight{1.25}\smash{\begin{tabular}[t]{l}$!S_1$\end{tabular}}}}%
    \put(0.18803951,0.29341811){\color[rgb]{0,0,0}\makebox(0,0)[lt]{\lineheight{1.25}\smash{\begin{tabular}[t]{l}$!S_2$\end{tabular}}}}%
    \put(0.38336247,0.24692236){\color[rgb]{0,0,0}\makebox(0,0)[lt]{\lineheight{1.25}\smash{\begin{tabular}[t]{l}$\De'$\end{tabular}}}}%
    \put(0.32557074,0.10650142){\color[rgb]{0,0,0}\makebox(0,0)[lt]{\lineheight{1.25}\smash{\begin{tabular}[t]{l}$\Pi_1$\end{tabular}}}}%
    \put(0.29227498,0.25582251){\color[rgb]{0,0,0}\makebox(0,0)[lt]{\lineheight{1.25}\smash{\begin{tabular}[t]{l}$\Pi_2$\end{tabular}}}}%
    \put(0.96461973,0.11724215){\color[rgb]{0,0,0}\makebox(0,0)[lt]{\lineheight{1.25}\smash{\begin{tabular}[t]{l}$!v_1$\end{tabular}}}}%
    \put(0.90100519,0.31935736){\color[rgb]{0,0,0}\makebox(0,0)[lt]{\lineheight{1.25}\smash{\begin{tabular}[t]{l}$!S_1$\end{tabular}}}}%
    \put(0.81298229,0.2997037){\color[rgb]{0,0,0}\makebox(0,0)[lt]{\lineheight{1.25}\smash{\begin{tabular}[t]{l}$!P$\end{tabular}}}}%
    \put(0.63317121,0.29611493){\color[rgb]{0,0,0}\makebox(0,0)[lt]{\lineheight{1.25}\smash{\begin{tabular}[t]{l}$!S_2$\end{tabular}}}}%
    \put(0.88204433,0.27014035){\color[rgb]{0,0,0}\makebox(0,0)[lt]{\lineheight{1.25}\smash{\begin{tabular}[t]{l}$!u_1$\end{tabular}}}}%
    \put(0.82875399,0.21370191){\color[rgb]{0,0,0}\makebox(0,0)[lt]{\lineheight{1.25}\smash{\begin{tabular}[t]{l}$!Q$\end{tabular}}}}%
    \put(0.81313286,0.1296393){\color[rgb]{0,0,0}\makebox(0,0)[lt]{\lineheight{1.25}\smash{\begin{tabular}[t]{l}$!R_1$\end{tabular}}}}%
    \put(0.74536113,0.23627701){\color[rgb]{0,0,0}\makebox(0,0)[lt]{\lineheight{1.25}\smash{\begin{tabular}[t]{l}$!u_2$\end{tabular}}}}%
    \put(0.81438514,0.25544446){\color[rgb]{0,0,0}\makebox(0,0)[lt]{\lineheight{1.25}\smash{\begin{tabular}[t]{l}$\Pi$\end{tabular}}}}%
    \put(0.95594005,0.28054514){\color[rgb]{0,0,0}\makebox(0,0)[lt]{\lineheight{1.25}\smash{\begin{tabular}[t]{l}$!v_2$\end{tabular}}}}%
    \put(0.90792796,0.22862011){\color[rgb]{0,0,0}\makebox(0,0)[lt]{\lineheight{1.25}\smash{\begin{tabular}[t]{l}$!R_2$\end{tabular}}}}%
    \put(0.2097722,0.00026212){\color[rgb]{0,0,0}\makebox(0,0)[lt]{\lineheight{1.25}\smash{\begin{tabular}[t]{l}a\end{tabular}}}}%
    \put(0.7229038,0.00026212){\color[rgb]{0,0,0}\makebox(0,0)[lt]{\lineheight{1.25}\smash{\begin{tabular}[t]{l}b\end{tabular}}}}%
  \end{picture}%
\endgroup%

%% file: cell-regularity.eps_tex
\begingroup%
  \makeatletter%
  \providecommand\color[2][]{%
    \errmessage{(Inkscape) Color is used for the text in Inkscape, but the package 'color.sty' is not loaded}%
    \renewcommand\color[2][]{}%
  }%
  \providecommand\transparent[1]{%
    \errmessage{(Inkscape) Transparency is used (non-zero) for the text in Inkscape, but the package 'transparent.sty' is not loaded}%
    \renewcommand\transparent[1]{}%
  }%
  \providecommand\rotatebox[2]{#2}%
  \newcommand*\fsize{\dimexpr\f@size pt\relax}%
  \newcommand*\lineheight[1]{\fontsize{\fsize}{#1\fsize}\selectfont}%
  \ifx\svgwidth\undefined%
    \setlength{\unitlength}{333.37092121bp}%
    \ifx\svgscale\undefined%
      \relax%
    \else%
      \setlength{\unitlength}{\unitlength * \real{\svgscale}}%
    \fi%
  \else%
    \setlength{\unitlength}{\svgwidth}%
  \fi%
  \global\let\svgwidth\undefined%
  \global\let\svgscale\undefined%
  \makeatother%
  \begin{picture}(1,0.46710054)%
    \lineheight{1}%
    \setlength\tabcolsep{0pt}%
    \put(0,0){\includegraphics[width=\unitlength]{cell-regularity.eps}}%
    \put(0.05850636,0.44277955){\color[rgb]{0,0,0}\makebox(0,0)[lt]{\lineheight{1.25}\smash{\begin{tabular}[t]{l}$!S$\end{tabular}}}}%
    \put(0.12385274,0.3670388){\color[rgb]{0,0,0}\makebox(0,0)[lt]{\lineheight{1.25}\smash{\begin{tabular}[t]{l}$\Pi_1$\end{tabular}}}}%
    \put(0.29109381,0.3670388){\color[rgb]{0,0,0}\makebox(0,0)[lt]{\lineheight{1.25}\smash{\begin{tabular}[t]{l}$\Pi_2$\end{tabular}}}}%
    \put(0.21405101,0.3670388){\color[rgb]{0,0,0}\makebox(0,0)[lt]{\lineheight{1.25}\smash{\begin{tabular}[t]{l}$\Th$\end{tabular}}}}%
    \put(0.68704854,0.22674538){\color[rgb]{0,0,0}\makebox(0,0)[lt]{\lineheight{1.25}\smash{\begin{tabular}[t]{l}$!R$\end{tabular}}}}%
    \put(0.8025779,0.2324759){\color[rgb]{0,0,0}\makebox(0,0)[lt]{\lineheight{1.25}\smash{\begin{tabular}[t]{l}$!u$\end{tabular}}}}%
    \put(0.74415853,0.22659371){\color[rgb]{0,0,0}\makebox(0,0)[lt]{\lineheight{1.25}\smash{\begin{tabular}[t]{l}$\Th'$\end{tabular}}}}%
    \put(0.8624538,0.23263487){\color[rgb]{0,0,0}\makebox(0,0)[lt]{\lineheight{1.25}\smash{\begin{tabular}[t]{l}$\Pi$\end{tabular}}}}%
    \put(0.20960027,0.22301923){\color[rgb]{0,0,0}\makebox(0,0)[lt]{\lineheight{1.25}\smash{\begin{tabular}[t]{l}$!D$\end{tabular}}}}%
    \put(0.85138329,0.15818788){\color[rgb]{0,0,0}\makebox(0,0)[lt]{\lineheight{1.25}\smash{\begin{tabular}[t]{l}$!D$\end{tabular}}}}%
    \put(0.21727697,0.00032955){\color[rgb]{0,0,0}\makebox(0,0)[lt]{\lineheight{1.25}\smash{\begin{tabular}[t]{l}a\end{tabular}}}}%
    \put(0.78208725,0.00032955){\color[rgb]{0,0,0}\makebox(0,0)[lt]{\lineheight{1.25}\smash{\begin{tabular}[t]{l}b\end{tabular}}}}%
  \end{picture}%
\endgroup%

%% file: single-layer-bigon.eps_tex
\begingroup%
  \makeatletter%
  \providecommand\color[2][]{%
    \errmessage{(Inkscape) Color is used for the text in Inkscape, but the package 'color.sty' is not loaded}%
    \renewcommand\color[2][]{}%
  }%
  \providecommand\transparent[1]{%
    \errmessage{(Inkscape) Transparency is used (non-zero) for the text in Inkscape, but the package 'transparent.sty' is not loaded}%
    \renewcommand\transparent[1]{}%
  }%
  \providecommand\rotatebox[2]{#2}%
  \newcommand*\fsize{\dimexpr\f@size pt\relax}%
  \newcommand*\lineheight[1]{\fontsize{\fsize}{#1\fsize}\selectfont}%
  \ifx\svgwidth\undefined%
    \setlength{\unitlength}{362.97702988bp}%
    \ifx\svgscale\undefined%
      \relax%
    \else%
      \setlength{\unitlength}{\unitlength * \real{\svgscale}}%
    \fi%
  \else%
    \setlength{\unitlength}{\svgwidth}%
  \fi%
  \global\let\svgwidth\undefined%
  \global\let\svgscale\undefined%
  \makeatother%
  \begin{picture}(1,0.30840666)%
    \lineheight{1}%
    \setlength\tabcolsep{0pt}%
    \put(0,0){\includegraphics[width=\unitlength]{single-layer-bigon.eps}}%
    \put(0.1981901,0.1517343){\color[rgb]{0,0,0}\makebox(0,0)[lt]{\lineheight{1.25}\smash{\begin{tabular}[t]{l}$!D$\end{tabular}}}}%
    \put(0.77975412,0.1517343){\color[rgb]{0,0,0}\makebox(0,0)[lt]{\lineheight{1.25}\smash{\begin{tabular}[t]{l}$!D$\end{tabular}}}}%
    \put(0.0562674,0.1517343){\color[rgb]{0,0,0}\makebox(0,0)[lt]{\lineheight{1.25}\smash{\begin{tabular}[t]{l}$\De_1$\end{tabular}}}}%
    \put(0.32630906,0.1517343){\color[rgb]{0,0,0}\makebox(0,0)[lt]{\lineheight{1.25}\smash{\begin{tabular}[t]{l}$\De_2$\end{tabular}}}}%
    \put(0.64133385,0.1517343){\color[rgb]{0,0,0}\makebox(0,0)[lt]{\lineheight{1.25}\smash{\begin{tabular}[t]{l}$\De_1$\end{tabular}}}}%
    \put(0.89974792,0.1517343){\color[rgb]{0,0,0}\makebox(0,0)[lt]{\lineheight{1.25}\smash{\begin{tabular}[t]{l}$\De_2$\end{tabular}}}}%
    \put(0.19775599,0.00030269){\color[rgb]{0,0,0}\makebox(0,0)[lt]{\lineheight{1.25}\smash{\begin{tabular}[t]{l}a\end{tabular}}}}%
    \put(0.7855456,0.00030269){\color[rgb]{0,0,0}\makebox(0,0)[lt]{\lineheight{1.25}\smash{\begin{tabular}[t]{l}b\end{tabular}}}}%
    \put(0.18725453,0.24489665){\color[rgb]{0,0,0}\makebox(0,0)[lt]{\lineheight{1.25}\smash{\begin{tabular}[t]{l}$\Pi_2$\end{tabular}}}}%
    \put(0.19645372,0.0650546){\color[rgb]{0,0,0}\makebox(0,0)[lt]{\lineheight{1.25}\smash{\begin{tabular}[t]{l}$\Pi_1$\end{tabular}}}}%
    \put(0.77964097,0.0600756){\color[rgb]{0,0,0}\makebox(0,0)[lt]{\lineheight{1.25}\smash{\begin{tabular}[t]{l}$\Pi_1$\end{tabular}}}}%
    \put(0.773487,0.24421957){\color[rgb]{0,0,0}\makebox(0,0)[lt]{\lineheight{1.25}\smash{\begin{tabular}[t]{l}$\Pi_2$\end{tabular}}}}%
  \end{picture}%
\endgroup%

%% file: annular-monogon.eps_tex
\begingroup%
  \makeatletter%
  \providecommand\color[2][]{%
    \errmessage{(Inkscape) Color is used for the text in Inkscape, but the package 'color.sty' is not loaded}%
    \renewcommand\color[2][]{}%
  }%
  \providecommand\transparent[1]{%
    \errmessage{(Inkscape) Transparency is used (non-zero) for the text in Inkscape, but the package 'transparent.sty' is not loaded}%
    \renewcommand\transparent[1]{}%
  }%
  \providecommand\rotatebox[2]{#2}%
  \newcommand*\fsize{\dimexpr\f@size pt\relax}%
  \newcommand*\lineheight[1]{\fontsize{\fsize}{#1\fsize}\selectfont}%
  \ifx\svgwidth\undefined%
    \setlength{\unitlength}{364.08112416bp}%
    \ifx\svgscale\undefined%
      \relax%
    \else%
      \setlength{\unitlength}{\unitlength * \real{\svgscale}}%
    \fi%
  \else%
    \setlength{\unitlength}{\svgwidth}%
  \fi%
  \global\let\svgwidth\undefined%
  \global\let\svgscale\undefined%
  \makeatother%
  \begin{picture}(1,0.28875381)%
    \lineheight{1}%
    \setlength\tabcolsep{0pt}%
    \put(0,0){\includegraphics[width=\unitlength]{annular-monogon.eps}}%
    \put(0.12112771,0.13416374){\color[rgb]{0,0,0}\makebox(0,0)[lt]{\lineheight{1.25}\smash{\begin{tabular}[t]{l}$!D$\end{tabular}}}}%
    \put(0.68036695,0.13119165){\color[rgb]{0,0,0}\makebox(0,0)[lt]{\lineheight{1.25}\smash{\begin{tabular}[t]{l}$!D$\end{tabular}}}}%
    \put(0.1085018,0.07641094){\color[rgb]{0,0,0}\makebox(0,0)[lt]{\lineheight{1.25}\smash{\begin{tabular}[t]{l}{\small $\Pi_1$}\end{tabular}}}}%
    \put(0.17638079,0.14114561){\color[rgb]{0,0,0}\makebox(0,0)[lt]{\lineheight{1.25}\smash{\begin{tabular}[t]{l}{\small $\Pi_2$}\end{tabular}}}}%
    \put(0.66984015,0.07369229){\color[rgb]{0,0,0}\makebox(0,0)[lt]{\lineheight{1.25}\smash{\begin{tabular}[t]{l}{\small $\Pi_1$}\end{tabular}}}}%
    \put(0.66831588,0.19251868){\color[rgb]{0,0,0}\makebox(0,0)[lt]{\lineheight{1.25}\smash{\begin{tabular}[t]{l}{\small $\Pi_2$}\end{tabular}}}}%
    \put(0.59322392,0.12775831){\color[rgb]{0,0,0}\makebox(0,0)[lt]{\lineheight{1.25}\smash{\begin{tabular}[t]{l}$\De_1$\end{tabular}}}}%
    \put(0.73618466,0.18854796){\color[rgb]{0,0,0}\makebox(0,0)[lt]{\lineheight{1.25}\smash{\begin{tabular}[t]{l}$\De_2$\end{tabular}}}}%
  \end{picture}%
\endgroup%

%% file: small-complexity-cell.eps_tex
\begingroup%
  \makeatletter%
  \providecommand\color[2][]{%
    \errmessage{(Inkscape) Color is used for the text in Inkscape, but the package 'color.sty' is not loaded}%
    \renewcommand\color[2][]{}%
  }%
  \providecommand\transparent[1]{%
    \errmessage{(Inkscape) Transparency is used (non-zero) for the text in Inkscape, but the package 'transparent.sty' is not loaded}%
    \renewcommand\transparent[1]{}%
  }%
  \providecommand\rotatebox[2]{#2}%
  \newcommand*\fsize{\dimexpr\f@size pt\relax}%
  \newcommand*\lineheight[1]{\fontsize{\fsize}{#1\fsize}\selectfont}%
  \ifx\svgwidth\undefined%
    \setlength{\unitlength}{340.13444011bp}%
    \ifx\svgscale\undefined%
      \relax%
    \else%
      \setlength{\unitlength}{\unitlength * \real{\svgscale}}%
    \fi%
  \else%
    \setlength{\unitlength}{\svgwidth}%
  \fi%
  \global\let\svgwidth\undefined%
  \global\let\svgscale\undefined%
  \makeatother%
  \begin{picture}(1,0.85329783)%
    \lineheight{1}%
    \setlength\tabcolsep{0pt}%
    \put(0,0){\includegraphics[width=\unitlength]{small-complexity-cell.eps}}%
  \end{picture}%
\endgroup%

%% file: bigon-Cayley.eps_tex
\begingroup%
  \makeatletter%
  \providecommand\color[2][]{%
    \errmessage{(Inkscape) Color is used for the text in Inkscape, but the package 'color.sty' is not loaded}%
    \renewcommand\color[2][]{}%
  }%
  \providecommand\transparent[1]{%
    \errmessage{(Inkscape) Transparency is used (non-zero) for the text in Inkscape, but the package 'transparent.sty' is not loaded}%
    \renewcommand\transparent[1]{}%
  }%
  \providecommand\rotatebox[2]{#2}%
  \newcommand*\fsize{\dimexpr\f@size pt\relax}%
  \newcommand*\lineheight[1]{\fontsize{\fsize}{#1\fsize}\selectfont}%
  \ifx\svgwidth\undefined%
    \setlength{\unitlength}{439.98431624bp}%
    \ifx\svgscale\undefined%
      \relax%
    \else%
      \setlength{\unitlength}{\unitlength * \real{\svgscale}}%
    \fi%
  \else%
    \setlength{\unitlength}{\svgwidth}%
  \fi%
  \global\let\svgwidth\undefined%
  \global\let\svgscale\undefined%
  \makeatother%
  \begin{picture}(1,0.28998681)%
    \lineheight{1}%
    \setlength\tabcolsep{0pt}%
    \put(0,0){\includegraphics[width=\unitlength]{bigon-Cayley.eps}}%
    \put(0.00547004,0.1196004){\color[rgb]{0,0,0}\makebox(0,0)[lt]{\lineheight{1.25}\smash{\begin{tabular}[t]{l}$!u$\end{tabular}}}}%
    \put(0.07030218,0.0547178){\color[rgb]{0,0,0}\makebox(0,0)[lt]{\lineheight{1.25}\smash{\begin{tabular}[t]{l}$!Y$\end{tabular}}}}%
    \put(0.19982039,0.00000134){\color[rgb]{0,0,0}\makebox(0,0)[lt]{\lineheight{1.25}\smash{\begin{tabular}[t]{l}$!M_1$\end{tabular}}}}%
    \put(0.4026154,0.00504932){\color[rgb]{0,0,0}\makebox(0,0)[lt]{\lineheight{1.25}\smash{\begin{tabular}[t]{l}$!M_2$\end{tabular}}}}%
    \put(0.76141993,0.00436462){\color[rgb]{0,0,0}\makebox(0,0)[lt]{\lineheight{1.25}\smash{\begin{tabular}[t]{l}$!M_3$\end{tabular}}}}%
    \put(0.97476618,0.13734378){\color[rgb]{0,0,0}\makebox(0,0)[lt]{\lineheight{1.25}\smash{\begin{tabular}[t]{l}$!v$\end{tabular}}}}%
    \put(0.12005228,0.14039519){\color[rgb]{0,0,0}\makebox(0,0)[lt]{\lineheight{1.25}\smash{\begin{tabular}[t]{l}$!R_1$\end{tabular}}}}%
    \put(0.30401949,0.1371328){\color[rgb]{0,0,0}\makebox(0,0)[lt]{\lineheight{1.25}\smash{\begin{tabular}[t]{l}$!R_2$\end{tabular}}}}%
    \put(0.67542909,0.14132325){\color[rgb]{0,0,0}\makebox(0,0)[lt]{\lineheight{1.25}\smash{\begin{tabular}[t]{l}$!R_3$\end{tabular}}}}%
    \put(0.06987338,0.19007316){\color[rgb]{0,0,0}\makebox(0,0)[lt]{\lineheight{1.25}\smash{\begin{tabular}[t]{l}$!X$\end{tabular}}}}%
    \put(0.18217435,0.26814772){\color[rgb]{0,0,0}\makebox(0,0)[lt]{\lineheight{1.25}\smash{\begin{tabular}[t]{l}$!K_1$\end{tabular}}}}%
    \put(0.38498327,0.27155909){\color[rgb]{0,0,0}\makebox(0,0)[lt]{\lineheight{1.25}\smash{\begin{tabular}[t]{l}$!K_2$\end{tabular}}}}%
    \put(0.7452501,0.27091688){\color[rgb]{0,0,0}\makebox(0,0)[lt]{\lineheight{1.25}\smash{\begin{tabular}[t]{l}$!K_3$\end{tabular}}}}%
    \put(0.20506081,0.0965651){\color[rgb]{0,0,0}\makebox(0,0)[lt]{\lineheight{1.25}\smash{\begin{tabular}[t]{l}$!S_1$\end{tabular}}}}%
    \put(0.39434537,0.09634888){\color[rgb]{0,0,0}\makebox(0,0)[lt]{\lineheight{1.25}\smash{\begin{tabular}[t]{l}$!S_2$\end{tabular}}}}%
    \put(0.76161094,0.09373796){\color[rgb]{0,0,0}\makebox(0,0)[lt]{\lineheight{1.25}\smash{\begin{tabular}[t]{l}$!S_3$\end{tabular}}}}%
    \put(0.20326171,0.18272053){\color[rgb]{0,0,0}\makebox(0,0)[lt]{\lineheight{1.25}\smash{\begin{tabular}[t]{l}$!Q_1$\end{tabular}}}}%
    \put(0.39732165,0.18417326){\color[rgb]{0,0,0}\makebox(0,0)[lt]{\lineheight{1.25}\smash{\begin{tabular}[t]{l}$!Q_2$\end{tabular}}}}%
    \put(0.75591458,0.18440339){\color[rgb]{0,0,0}\makebox(0,0)[lt]{\lineheight{1.25}\smash{\begin{tabular}[t]{l}$!Q_3$\end{tabular}}}}%
  \end{picture}%
\endgroup%

%% file: active-loop-induction-bigon.eps_tex
\begingroup%
  \makeatletter%
  \providecommand\color[2][]{%
    \errmessage{(Inkscape) Color is used for the text in Inkscape, but the package 'color.sty' is not loaded}%
    \renewcommand\color[2][]{}%
  }%
  \providecommand\transparent[1]{%
    \errmessage{(Inkscape) Transparency is used (non-zero) for the text in Inkscape, but the package 'transparent.sty' is not loaded}%
    \renewcommand\transparent[1]{}%
  }%
  \providecommand\rotatebox[2]{#2}%
  \newcommand*\fsize{\dimexpr\f@size pt\relax}%
  \newcommand*\lineheight[1]{\fontsize{\fsize}{#1\fsize}\selectfont}%
  \ifx\svgwidth\undefined%
    \setlength{\unitlength}{414.07160741bp}%
    \ifx\svgscale\undefined%
      \relax%
    \else%
      \setlength{\unitlength}{\unitlength * \real{\svgscale}}%
    \fi%
  \else%
    \setlength{\unitlength}{\svgwidth}%
  \fi%
  \global\let\svgwidth\undefined%
  \global\let\svgscale\undefined%
  \makeatother%
  \begin{picture}(1,0.5311802)%
    \lineheight{1}%
    \setlength\tabcolsep{0pt}%
    \put(0,0){\includegraphics[width=\unitlength]{active-loop-induction-bigon.eps}}%
    \put(0.36211961,0.50064022){\color[rgb]{0,0,0}\makebox(0,0)[lt]{\lineheight{1.25}\smash{\begin{tabular}[t]{l}$!X$\end{tabular}}}}%
    \put(0.36371182,0.26631536){\color[rgb]{0,0,0}\makebox(0,0)[lt]{\lineheight{1.25}\smash{\begin{tabular}[t]{l}$!Y$\end{tabular}}}}%
    \put(0.24709088,0.39022241){\color[rgb]{0,0,0}\makebox(0,0)[lt]{\lineheight{1.25}\smash{\begin{tabular}[t]{l}$!u$\end{tabular}}}}%
    \put(0.72800409,0.39064917){\color[rgb]{0,0,0}\makebox(0,0)[lt]{\lineheight{1.25}\smash{\begin{tabular}[t]{l}$!v$\end{tabular}}}}%
    \put(0.47862108,0.51186459){\color[rgb]{0,0,0}\makebox(0,0)[lt]{\lineheight{1.25}\smash{\begin{tabular}[t]{l}$!K$\end{tabular}}}}%
    \put(0.48137173,0.42465629){\color[rgb]{0,0,0}\makebox(0,0)[lt]{\lineheight{1.25}\smash{\begin{tabular}[t]{l}$!Q$\end{tabular}}}}%
    \put(0.48455761,0.3597908){\color[rgb]{0,0,0}\makebox(0,0)[lt]{\lineheight{1.25}\smash{\begin{tabular}[t]{l}$!S$\end{tabular}}}}%
    \put(0.54841465,0.46292729){\color[rgb]{0,0,0}\makebox(0,0)[lt]{\lineheight{1.25}\smash{\begin{tabular}[t]{l}$!w_1$\end{tabular}}}}%
    \put(0.39880781,0.4604896){\color[rgb]{0,0,0}\makebox(0,0)[lt]{\lineheight{1.25}\smash{\begin{tabular}[t]{l}$!w_2$\end{tabular}}}}%
    \put(0.40009912,0.32163069){\color[rgb]{0,0,0}\makebox(0,0)[lt]{\lineheight{1.25}\smash{\begin{tabular}[t]{l}$!w_3$\end{tabular}}}}%
    \put(0.48196855,0.26192341){\color[rgb]{0,0,0}\makebox(0,0)[lt]{\lineheight{1.25}\smash{\begin{tabular}[t]{l}$!M$\end{tabular}}}}%
    \put(0.55929791,0.32181025){\color[rgb]{0,0,0}\makebox(0,0)[lt]{\lineheight{1.25}\smash{\begin{tabular}[t]{l}$!w_4$\end{tabular}}}}%
    \put(0.13370433,0.00000001){\color[rgb]{0,0,0}\makebox(0,0)[lt]{\lineheight{1.25}\smash{\begin{tabular}[t]{l}$!P_1$\end{tabular}}}}%
    \put(0.39869504,0.3926976){\color[rgb]{0,0,0}\makebox(0,0)[lt]{\lineheight{1.25}\smash{\begin{tabular}[t]{l}$!R$\end{tabular}}}}%
    \put(0.84174754,0.00000001){\color[rgb]{0,0,0}\makebox(0,0)[lt]{\lineheight{1.25}\smash{\begin{tabular}[t]{l}$!P_2$\end{tabular}}}}%
  \end{picture}%
\endgroup%

%% file: active-loop-induction-trigon.eps_tex
\begingroup%
  \makeatletter%
  \providecommand\color[2][]{%
    \errmessage{(Inkscape) Color is used for the text in Inkscape, but the package 'color.sty' is not loaded}%
    \renewcommand\color[2][]{}%
  }%
  \providecommand\transparent[1]{%
    \errmessage{(Inkscape) Transparency is used (non-zero) for the text in Inkscape, but the package 'transparent.sty' is not loaded}%
    \renewcommand\transparent[1]{}%
  }%
  \providecommand\rotatebox[2]{#2}%
  \newcommand*\fsize{\dimexpr\f@size pt\relax}%
  \newcommand*\lineheight[1]{\fontsize{\fsize}{#1\fsize}\selectfont}%
  \ifx\svgwidth\undefined%
    \setlength{\unitlength}{409.35321641bp}%
    \ifx\svgscale\undefined%
      \relax%
    \else%
      \setlength{\unitlength}{\unitlength * \real{\svgscale}}%
    \fi%
  \else%
    \setlength{\unitlength}{\svgwidth}%
  \fi%
  \global\let\svgwidth\undefined%
  \global\let\svgscale\undefined%
  \makeatother%
  \begin{picture}(1,0.33048483)%
    \lineheight{1}%
    \setlength\tabcolsep{0pt}%
    \put(0,0){\includegraphics[width=\unitlength]{active-loop-induction-trigon.eps}}%
    \put(0.13074869,0){\color[rgb]{0,0,0}\makebox(0,0)[lt]{\lineheight{1.25}\smash{\begin{tabular}[t]{l}$!X_1$\end{tabular}}}}%
    \put(0.21102203,0.00391244){\color[rgb]{0,0,0}\makebox(0,0)[lt]{\lineheight{1.25}\smash{\begin{tabular}[t]{l}$!K_1$\end{tabular}}}}%
    \put(0.3127785,0.17189599){\color[rgb]{0,0,0}\makebox(0,0)[lt]{\lineheight{1.25}\smash{\begin{tabular}[t]{l}$!K_2$\end{tabular}}}}%
    \put(0.28545102,0.23710412){\color[rgb]{0,0,0}\makebox(0,0)[lt]{\lineheight{1.25}\smash{\begin{tabular}[t]{l}$!X_2$\end{tabular}}}}%
    \put(0.05207804,0.20593084){\color[rgb]{0,0,0}\makebox(0,0)[lt]{\lineheight{1.25}\smash{\begin{tabular}[t]{l}$!X_3$\end{tabular}}}}%
    \put(0.14697236,0.15325394){\color[rgb]{0,0,0}\makebox(0,0)[lt]{\lineheight{1.25}\smash{\begin{tabular}[t]{l}$!R$\end{tabular}}}}%
    \put(0.58962493,0.03702465){\color[rgb]{0,0,0}\makebox(0,0)[lt]{\lineheight{1.25}\smash{\begin{tabular}[t]{l}$!P_1$\end{tabular}}}}%
    \put(0.83836541,0.03968768){\color[rgb]{0,0,0}\makebox(0,0)[lt]{\lineheight{1.25}\smash{\begin{tabular}[t]{l}$!P_2$\end{tabular}}}}%
  \end{picture}%
\endgroup%

%% file: side-bound.eps_tex
\begingroup%
  \makeatletter%
  \providecommand\color[2][]{%
    \errmessage{(Inkscape) Color is used for the text in Inkscape, but the package 'color.sty' is not loaded}%
    \renewcommand\color[2][]{}%
  }%
  \providecommand\transparent[1]{%
    \errmessage{(Inkscape) Transparency is used (non-zero) for the text in Inkscape, but the package 'transparent.sty' is not loaded}%
    \renewcommand\transparent[1]{}%
  }%
  \providecommand\rotatebox[2]{#2}%
  \newcommand*\fsize{\dimexpr\f@size pt\relax}%
  \newcommand*\lineheight[1]{\fontsize{\fsize}{#1\fsize}\selectfont}%
  \ifx\svgwidth\undefined%
    \setlength{\unitlength}{413.84443467bp}%
    \ifx\svgscale\undefined%
      \relax%
    \else%
      \setlength{\unitlength}{\unitlength * \real{\svgscale}}%
    \fi%
  \else%
    \setlength{\unitlength}{\svgwidth}%
  \fi%
  \global\let\svgwidth\undefined%
  \global\let\svgscale\undefined%
  \makeatother%
  \begin{picture}(1,0.3051595)%
    \lineheight{1}%
    \setlength\tabcolsep{0pt}%
    \put(0,0){\includegraphics[width=\unitlength]{side-bound.eps}}%
    \put(0.05710041,0.28050293){\color[rgb]{0,0,0}\makebox(0,0)[lt]{\lineheight{1.25}\smash{\begin{tabular}[t]{l}$!X_1$\end{tabular}}}}%
    \put(0.17391733,0.27328481){\color[rgb]{0,0,0}\makebox(0,0)[lt]{\lineheight{1.25}\smash{\begin{tabular}[t]{l}$!z$\end{tabular}}}}%
    \put(0.277619,0.28358827){\color[rgb]{0,0,0}\makebox(0,0)[lt]{\lineheight{1.25}\smash{\begin{tabular}[t]{l}$!X_2$\end{tabular}}}}%
    \put(0.04110955,0.18575925){\color[rgb]{0,0,0}\makebox(0,0)[lt]{\lineheight{1.25}\smash{\begin{tabular}[t]{l}$!Y_{11}$\end{tabular}}}}%
    \put(0.10165901,0.14261879){\color[rgb]{0,0,0}\makebox(0,0)[lt]{\lineheight{1.25}\smash{\begin{tabular}[t]{l}$!w_1$\end{tabular}}}}%
    \put(0.11959258,0.06347163){\color[rgb]{0,0,0}\makebox(0,0)[lt]{\lineheight{1.25}\smash{\begin{tabular}[t]{l}$!Y_{12}$\end{tabular}}}}%
    \put(0.2266802,0.0647862){\color[rgb]{0,0,0}\makebox(0,0)[lt]{\lineheight{1.25}\smash{\begin{tabular}[t]{l}$!Y_{21}$\end{tabular}}}}%
    \put(0.24260725,0.14526322){\color[rgb]{0,0,0}\makebox(0,0)[lt]{\lineheight{1.25}\smash{\begin{tabular}[t]{l}$!w_2$\end{tabular}}}}%
    \put(0.29535998,0.18910583){\color[rgb]{0,0,0}\makebox(0,0)[lt]{\lineheight{1.25}\smash{\begin{tabular}[t]{l}$!Y_{22}$\end{tabular}}}}%
    \put(0.62495475,0.28188804){\color[rgb]{0,0,0}\makebox(0,0)[lt]{\lineheight{1.25}\smash{\begin{tabular}[t]{l}$!X_1$\end{tabular}}}}%
    \put(0.85775093,0.28556781){\color[rgb]{0,0,0}\makebox(0,0)[lt]{\lineheight{1.25}\smash{\begin{tabular}[t]{l}$!X_2$\end{tabular}}}}%
    \put(0.57173857,0.1611168){\color[rgb]{0,0,0}\makebox(0,0)[lt]{\lineheight{1.25}\smash{\begin{tabular}[t]{l}$!Y_1$\end{tabular}}}}%
    \put(0.70153783,0.13529843){\color[rgb]{0,0,0}\makebox(0,0)[lt]{\lineheight{1.25}\smash{\begin{tabular}[t]{l}$!Y_{21}$\end{tabular}}}}%
    \put(0.80400224,0.13780451){\color[rgb]{0,0,0}\makebox(0,0)[lt]{\lineheight{1.25}\smash{\begin{tabular}[t]{l}$!Y_{22}$\end{tabular}}}}%
    \put(0.92601037,0.15931192){\color[rgb]{0,0,0}\makebox(0,0)[lt]{\lineheight{1.25}\smash{\begin{tabular}[t]{l}$!Y_3$\end{tabular}}}}%
    \put(0.18517611,0.00026547){\color[rgb]{0,0,0}\makebox(0,0)[lt]{\lineheight{1.25}\smash{\begin{tabular}[t]{l}a\end{tabular}}}}%
    \put(0.7622481,0.00026547){\color[rgb]{0,0,0}\makebox(0,0)[lt]{\lineheight{1.25}\smash{\begin{tabular}[t]{l}b\end{tabular}}}}%
  \end{picture}%
\endgroup%

%% file: fellow-traveling.eps_tex
\begingroup%
  \makeatletter%
  \providecommand\color[2][]{%
    \errmessage{(Inkscape) Color is used for the text in Inkscape, but the package 'color.sty' is not loaded}%
    \renewcommand\color[2][]{}%
  }%
  \providecommand\transparent[1]{%
    \errmessage{(Inkscape) Transparency is used (non-zero) for the text in Inkscape, but the package 'transparent.sty' is not loaded}%
    \renewcommand\transparent[1]{}%
  }%
  \providecommand\rotatebox[2]{#2}%
  \newcommand*\fsize{\dimexpr\f@size pt\relax}%
  \newcommand*\lineheight[1]{\fontsize{\fsize}{#1\fsize}\selectfont}%
  \ifx\svgwidth\undefined%
    \setlength{\unitlength}{94.30812251bp}%
    \ifx\svgscale\undefined%
      \relax%
    \else%
      \setlength{\unitlength}{\unitlength * \real{\svgscale}}%
    \fi%
  \else%
    \setlength{\unitlength}{\svgwidth}%
  \fi%
  \global\let\svgwidth\undefined%
  \global\let\svgscale\undefined%
  \makeatother%
  \begin{picture}(1,1.36296728)%
    \lineheight{1}%
    \setlength\tabcolsep{0pt}%
    \put(0,0){\includegraphics[width=\unitlength]{fellow-traveling.eps}}%
    \put(0.3985283,1.27699473){\color[rgb]{0,0,0}\makebox(0,0)[lt]{\lineheight{1.25}\smash{\begin{tabular}[t]{l}$!X_i$\end{tabular}}}}%
    \put(0.00642849,1.07026264){\color[rgb]{0,0,0}\makebox(0,0)[lt]{\lineheight{1.25}\smash{\begin{tabular}[t]{l}$!w_{i1}$\end{tabular}}}}%
    \put(0.74111002,1.07097169){\color[rgb]{0,0,0}\makebox(0,0)[lt]{\lineheight{1.25}\smash{\begin{tabular}[t]{l}$!w_{i2}$\end{tabular}}}}%
    \put(0.0259582,0.24801233){\color[rgb]{0,0,0}\makebox(0,0)[lt]{\lineheight{1.25}\smash{\begin{tabular}[t]{l}$!w_{i3}$\end{tabular}}}}%
    \put(0.74925703,0.24970103){\color[rgb]{0,0,0}\makebox(0,0)[lt]{\lineheight{1.25}\smash{\begin{tabular}[t]{l}$!w_{i4}$\end{tabular}}}}%
    \put(0.40435529,0.86378085){\color[rgb]{0,0,0}\makebox(0,0)[lt]{\lineheight{1.25}\smash{\begin{tabular}[t]{l}$!R_i$\end{tabular}}}}%
    \put(0.41873392,0.41831586){\color[rgb]{0,0,0}\makebox(0,0)[lt]{\lineheight{1.25}\smash{\begin{tabular}[t]{l}$!S_i$\end{tabular}}}}%
    \put(0.42724475,0.00828102){\color[rgb]{0,0,0}\makebox(0,0)[lt]{\lineheight{1.25}\smash{\begin{tabular}[t]{l}$!Y_i$\end{tabular}}}}%
  \end{picture}%
\endgroup%

%% file: fellow-traveling-subcases.eps_tex
\begingroup%
  \makeatletter%
  \providecommand\color[2][]{%
    \errmessage{(Inkscape) Color is used for the text in Inkscape, but the package 'color.sty' is not loaded}%
    \renewcommand\color[2][]{}%
  }%
  \providecommand\transparent[1]{%
    \errmessage{(Inkscape) Transparency is used (non-zero) for the text in Inkscape, but the package 'transparent.sty' is not loaded}%
    \renewcommand\transparent[1]{}%
  }%
  \providecommand\rotatebox[2]{#2}%
  \newcommand*\fsize{\dimexpr\f@size pt\relax}%
  \newcommand*\lineheight[1]{\fontsize{\fsize}{#1\fsize}\selectfont}%
  \ifx\svgwidth\undefined%
    \setlength{\unitlength}{412.80000943bp}%
    \ifx\svgscale\undefined%
      \relax%
    \else%
      \setlength{\unitlength}{\unitlength * \real{\svgscale}}%
    \fi%
  \else%
    \setlength{\unitlength}{\svgwidth}%
  \fi%
  \global\let\svgwidth\undefined%
  \global\let\svgscale\undefined%
  \makeatother%
  \begin{picture}(1,0.32489245)%
    \lineheight{1}%
    \setlength\tabcolsep{0pt}%
    \put(0,0){\includegraphics[width=\unitlength]{fellow-traveling-subcases.eps}}%
    \put(0.14895934,0.30551734){\color[rgb]{0,0,0}\makebox(0,0)[lt]{\lineheight{1.25}\smash{\begin{tabular}[t]{l}$!X$\end{tabular}}}}%
    \put(0.15221936,0.04941799){\color[rgb]{0,0,0}\makebox(0,0)[lt]{\lineheight{1.25}\smash{\begin{tabular}[t]{l}$!Y$\end{tabular}}}}%
    \put(0.01485266,0.29084279){\color[rgb]{0,0,0}\makebox(0,0)[lt]{\lineheight{1.25}\smash{\begin{tabular}[t]{l}$!w_{1}$\end{tabular}}}}%
    \put(0.01821643,0.17867977){\color[rgb]{0,0,0}\makebox(0,0)[lt]{\lineheight{1.25}\smash{\begin{tabular}[t]{l}$!P$\end{tabular}}}}%
    \put(0.01852614,0.07625934){\color[rgb]{0,0,0}\makebox(0,0)[lt]{\lineheight{1.25}\smash{\begin{tabular}[t]{l}$!w_{2}$\end{tabular}}}}%
    \put(0.28765472,0.08227147){\color[rgb]{0,0,0}\makebox(0,0)[lt]{\lineheight{1.25}\smash{\begin{tabular}[t]{l}$!w_{3}$\end{tabular}}}}%
    \put(0.29639155,0.18081548){\color[rgb]{0,0,0}\makebox(0,0)[lt]{\lineheight{1.25}\smash{\begin{tabular}[t]{l}$!Q$\end{tabular}}}}%
    \put(0.28719345,0.28347363){\color[rgb]{0,0,0}\makebox(0,0)[lt]{\lineheight{1.25}\smash{\begin{tabular}[t]{l}$!w_{4}$\end{tabular}}}}%
    \put(0.7469133,0.25479697){\color[rgb]{0,0,0}\makebox(0,0)[lt]{\lineheight{1.25}\smash{\begin{tabular}[t]{l}$!X_1$\end{tabular}}}}%
    \put(0.81772815,0.23955014){\color[rgb]{0,0,0}\makebox(0,0)[lt]{\lineheight{1.25}\smash{\begin{tabular}[t]{l}$!X_2$\end{tabular}}}}%
    \put(0.89224656,0.2570691){\color[rgb]{0,0,0}\makebox(0,0)[lt]{\lineheight{1.25}\smash{\begin{tabular}[t]{l}$!X_3$\end{tabular}}}}%
    \put(0.75043601,0.09583828){\color[rgb]{0,0,0}\makebox(0,0)[lt]{\lineheight{1.25}\smash{\begin{tabular}[t]{l}$!Y_1$\end{tabular}}}}%
    \put(0.82213887,0.11455548){\color[rgb]{0,0,0}\makebox(0,0)[lt]{\lineheight{1.25}\smash{\begin{tabular}[t]{l}$!Y_2$\end{tabular}}}}%
    \put(0.90340533,0.10050771){\color[rgb]{0,0,0}\makebox(0,0)[lt]{\lineheight{1.25}\smash{\begin{tabular}[t]{l}$!Y_3$\end{tabular}}}}%
    \put(0.14557695,0.00625434){\color[rgb]{0,0,0}\makebox(0,0)[lt]{\lineheight{1.25}\smash{\begin{tabular}[t]{l}(a)\end{tabular}}}}%
    \put(0.48057984,0.00625434){\color[rgb]{0,0,0}\makebox(0,0)[lt]{\lineheight{1.25}\smash{\begin{tabular}[t]{l}(b)\end{tabular}}}}%
    \put(0.82001362,0.00625434){\color[rgb]{0,0,0}\makebox(0,0)[lt]{\lineheight{1.25}\smash{\begin{tabular}[t]{l}(c)\end{tabular}}}}%
  \end{picture}%
\endgroup%

%% file: monogon-graph-version.eps_tex
\begingroup%
  \makeatletter%
  \providecommand\color[2][]{%
    \errmessage{(Inkscape) Color is used for the text in Inkscape, but the package 'color.sty' is not loaded}%
    \renewcommand\color[2][]{}%
  }%
  \providecommand\transparent[1]{%
    \errmessage{(Inkscape) Transparency is used (non-zero) for the text in Inkscape, but the package 'transparent.sty' is not loaded}%
    \renewcommand\transparent[1]{}%
  }%
  \providecommand\rotatebox[2]{#2}%
  \newcommand*\fsize{\dimexpr\f@size pt\relax}%
  \newcommand*\lineheight[1]{\fontsize{\fsize}{#1\fsize}\selectfont}%
  \ifx\svgwidth\undefined%
    \setlength{\unitlength}{230.02424632bp}%
    \ifx\svgscale\undefined%
      \relax%
    \else%
      \setlength{\unitlength}{\unitlength * \real{\svgscale}}%
    \fi%
  \else%
    \setlength{\unitlength}{\svgwidth}%
  \fi%
  \global\let\svgwidth\undefined%
  \global\let\svgscale\undefined%
  \makeatother%
  \begin{picture}(1,0.39753114)%
    \lineheight{1}%
    \setlength\tabcolsep{0pt}%
    \put(0,0){\includegraphics[width=\unitlength]{monogon-graph-version.eps}}%
    \put(0.21878378,0.00625821){\color[rgb]{0,0,0}\makebox(0,0)[lt]{\lineheight{1.25}\smash{\begin{tabular}[t]{l}$!X$\end{tabular}}}}%
    \put(0.30178655,0.01879845){\color[rgb]{0,0,0}\makebox(0,0)[lt]{\lineheight{1.25}\smash{\begin{tabular}[t]{l}$!a_{1}$\end{tabular}}}}%
    \put(0.46037657,0.01906025){\color[rgb]{0,0,0}\makebox(0,0)[lt]{\lineheight{1.25}\smash{\begin{tabular}[t]{l}$!Y$\end{tabular}}}}%
    \put(0.64284754,0.02501465){\color[rgb]{0,0,0}\makebox(0,0)[lt]{\lineheight{1.25}\smash{\begin{tabular}[t]{l}$!a_{2}$\end{tabular}}}}%
    \put(0.26134173,0.11776671){\color[rgb]{0,0,0}\makebox(0,0)[lt]{\lineheight{1.25}\smash{\begin{tabular}[t]{l}$!u_{1}$\end{tabular}}}}%
    \put(0.66054642,0.11827535){\color[rgb]{0,0,0}\makebox(0,0)[lt]{\lineheight{1.25}\smash{\begin{tabular}[t]{l}$!u_{2}$\end{tabular}}}}%
    \put(0.45710537,0.16579639){\color[rgb]{0,0,0}\makebox(0,0)[lt]{\lineheight{1.25}\smash{\begin{tabular}[t]{l}$!R_1$\end{tabular}}}}%
    \put(0.45798154,0.33670698){\color[rgb]{0,0,0}\makebox(0,0)[lt]{\lineheight{1.25}\smash{\begin{tabular}[t]{l}$!R_2$\end{tabular}}}}%
    \put(0.33230133,0.20261916){\color[rgb]{0,0,0}\makebox(0,0)[lt]{\lineheight{1.25}\smash{\begin{tabular}[t]{l}$!b_{1}$\end{tabular}}}}%
    \put(0.58192545,0.20787735){\color[rgb]{0,0,0}\makebox(0,0)[lt]{\lineheight{1.25}\smash{\begin{tabular}[t]{l}$!b_{2}$\end{tabular}}}}%
  \end{picture}%
\endgroup%

%% file: stability-tetragon-previous.eps_tex
\begingroup%
  \makeatletter%
  \providecommand\color[2][]{%
    \errmessage{(Inkscape) Color is used for the text in Inkscape, but the package 'color.sty' is not loaded}%
    \renewcommand\color[2][]{}%
  }%
  \providecommand\transparent[1]{%
    \errmessage{(Inkscape) Transparency is used (non-zero) for the text in Inkscape, but the package 'transparent.sty' is not loaded}%
    \renewcommand\transparent[1]{}%
  }%
  \providecommand\rotatebox[2]{#2}%
  \newcommand*\fsize{\dimexpr\f@size pt\relax}%
  \newcommand*\lineheight[1]{\fontsize{\fsize}{#1\fsize}\selectfont}%
  \ifx\svgwidth\undefined%
    \setlength{\unitlength}{171.66787939bp}%
    \ifx\svgscale\undefined%
      \relax%
    \else%
      \setlength{\unitlength}{\unitlength * \real{\svgscale}}%
    \fi%
  \else%
    \setlength{\unitlength}{\svgwidth}%
  \fi%
  \global\let\svgwidth\undefined%
  \global\let\svgscale\undefined%
  \makeatother%
  \begin{picture}(1,0.62694819)%
    \lineheight{1}%
    \setlength\tabcolsep{0pt}%
    \put(0,0){\includegraphics[width=\unitlength]{stability-tetragon-previous.eps}}%
    \put(0.44801472,0.58035795){\color[rgb]{0,0,0}\makebox(0,0)[lt]{\lineheight{1.25}\smash{\begin{tabular}[t]{l}$!X$\end{tabular}}}}%
    \put(0.44646309,0.40183785){\color[rgb]{0,0,0}\makebox(0,0)[lt]{\lineheight{1.25}\smash{\begin{tabular}[t]{l}$!Y$\end{tabular}}}}%
    \put(0.77118553,0.42454977){\color[rgb]{0,0,0}\makebox(0,0)[lt]{\lineheight{1.25}\smash{\begin{tabular}[t]{l}$!S$\end{tabular}}}}%
    \put(0.36600894,0.26571059){\color[rgb]{0,0,0}\makebox(0,0)[lt]{\lineheight{1.25}\smash{\begin{tabular}[t]{l}$!Z$\end{tabular}}}}%
    \put(0.0263636,0.1019253){\color[rgb]{0,0,0}\makebox(0,0)[lt]{\lineheight{1.25}\smash{\begin{tabular}[t]{l}$!T_1$\end{tabular}}}}%
    \put(0.47586973,0.12894727){\color[rgb]{0,0,0}\makebox(0,0)[lt]{\lineheight{1.25}\smash{\begin{tabular}[t]{l}$!T_2$\end{tabular}}}}%
    \put(0.9075621,0.09864203){\color[rgb]{0,0,0}\makebox(0,0)[lt]{\lineheight{1.25}\smash{\begin{tabular}[t]{l}$!T_3$\end{tabular}}}}%
  \end{picture}%
\endgroup%

%% file: fragment-stability-bigon.eps_tex
\begingroup%
  \makeatletter%
  \providecommand\color[2][]{%
    \errmessage{(Inkscape) Color is used for the text in Inkscape, but the package 'color.sty' is not loaded}%
    \renewcommand\color[2][]{}%
  }%
  \providecommand\transparent[1]{%
    \errmessage{(Inkscape) Transparency is used (non-zero) for the text in Inkscape, but the package 'transparent.sty' is not loaded}%
    \renewcommand\transparent[1]{}%
  }%
  \providecommand\rotatebox[2]{#2}%
  \newcommand*\fsize{\dimexpr\f@size pt\relax}%
  \newcommand*\lineheight[1]{\fontsize{\fsize}{#1\fsize}\selectfont}%
  \ifx\svgwidth\undefined%
    \setlength{\unitlength}{332.55610547bp}%
    \ifx\svgscale\undefined%
      \relax%
    \else%
      \setlength{\unitlength}{\unitlength * \real{\svgscale}}%
    \fi%
  \else%
    \setlength{\unitlength}{\svgwidth}%
  \fi%
  \global\let\svgwidth\undefined%
  \global\let\svgscale\undefined%
  \makeatother%
  \begin{picture}(1,0.45696333)%
    \lineheight{1}%
    \setlength\tabcolsep{0pt}%
    \put(0,0){\includegraphics[width=\unitlength]{fragment-stability-bigon.eps}}%
    \put(0.44397348,0.39971459){\color[rgb]{0,0,0}\makebox(0,0)[lt]{\lineheight{1.25}\smash{\begin{tabular}[t]{l}$!S$\end{tabular}}}}%
    \put(0.90768082,0.43291311){\color[rgb]{0,0,0}\makebox(0,0)[lt]{\lineheight{1.25}\smash{\begin{tabular}[t]{l}$!L$\end{tabular}}}}%
    \put(0.22581388,0.35136121){\color[rgb]{0,0,0}\makebox(0,0)[lt]{\lineheight{1.25}\smash{\begin{tabular}[t]{l}$!K_i$\end{tabular}}}}%
    \put(0.64210691,0.33505319){\color[rgb]{0,0,0}\makebox(0,0)[lt]{\lineheight{1.25}\smash{\begin{tabular}[t]{l}$!F_i$\end{tabular}}}}%
    \put(0.76194856,0.36201259){\color[rgb]{0,0,0}\makebox(0,0)[lt]{\lineheight{1.25}\smash{\begin{tabular}[t]{l}$!K_{i+1}$\end{tabular}}}}%
    \put(0.96396593,0.28219362){\color[rgb]{0,0,0}\makebox(0,0)[lt]{\lineheight{1.25}\smash{\begin{tabular}[t]{l}$!X$\end{tabular}}}}%
    \put(0.44280964,0.25267951){\color[rgb]{0,0,0}\makebox(0,0)[lt]{\lineheight{1.25}\smash{\begin{tabular}[t]{l}$!K$\end{tabular}}}}%
    \put(0.13669969,0.29231369){\color[rgb]{0,0,0}\makebox(0,0)[lt]{\lineheight{1.25}\smash{\begin{tabular}[t]{l}$!w_{1}$\end{tabular}}}}%
    \put(0.33983733,0.20152533){\color[rgb]{0,0,0}\makebox(0,0)[lt]{\lineheight{1.25}\smash{\begin{tabular}[t]{l}$!S_1$\end{tabular}}}}%
    \put(0.35696852,0.05794093){\color[rgb]{0,0,0}\makebox(0,0)[lt]{\lineheight{1.25}\smash{\begin{tabular}[t]{l}$!w_{2}$\end{tabular}}}}%
    \put(0.49695869,0.0131465){\color[rgb]{0,0,0}\makebox(0,0)[lt]{\lineheight{1.25}\smash{\begin{tabular}[t]{l}$!H_i$\end{tabular}}}}%
    \put(0.6815006,0.03312912){\color[rgb]{0,0,0}\makebox(0,0)[lt]{\lineheight{1.25}\smash{\begin{tabular}[t]{l}$!w_{3}$\end{tabular}}}}%
    \put(0.63049689,0.19813042){\color[rgb]{0,0,0}\makebox(0,0)[lt]{\lineheight{1.25}\smash{\begin{tabular}[t]{l}$!S_2$\end{tabular}}}}%
    \put(0.87988134,0.27762327){\color[rgb]{0,0,0}\makebox(0,0)[lt]{\lineheight{1.25}\smash{\begin{tabular}[t]{l}$!w_{4}$\end{tabular}}}}%
    \put(0.96602333,0.04879762){\color[rgb]{0,0,0}\makebox(0,0)[lt]{\lineheight{1.25}\smash{\begin{tabular}[t]{l}$!Y$\end{tabular}}}}%
  \end{picture}%
\endgroup%

%% file: fragment-stability-trigon.eps_tex
\begingroup%
  \makeatletter%
  \providecommand\color[2][]{%
    \errmessage{(Inkscape) Color is used for the text in Inkscape, but the package 'color.sty' is not loaded}%
    \renewcommand\color[2][]{}%
  }%
  \providecommand\transparent[1]{%
    \errmessage{(Inkscape) Transparency is used (non-zero) for the text in Inkscape, but the package 'transparent.sty' is not loaded}%
    \renewcommand\transparent[1]{}%
  }%
  \providecommand\rotatebox[2]{#2}%
  \newcommand*\fsize{\dimexpr\f@size pt\relax}%
  \newcommand*\lineheight[1]{\fontsize{\fsize}{#1\fsize}\selectfont}%
  \ifx\svgwidth\undefined%
    \setlength{\unitlength}{355.63101736bp}%
    \ifx\svgscale\undefined%
      \relax%
    \else%
      \setlength{\unitlength}{\unitlength * \real{\svgscale}}%
    \fi%
  \else%
    \setlength{\unitlength}{\svgwidth}%
  \fi%
  \global\let\svgwidth\undefined%
  \global\let\svgscale\undefined%
  \makeatother%
  \begin{picture}(1,0.35812806)%
    \lineheight{1}%
    \setlength\tabcolsep{0pt}%
    \put(0,0){\includegraphics[width=\unitlength]{fragment-stability-trigon.eps}}%
    \put(0.08743269,0.33548248){\color[rgb]{0,0,0}\makebox(0,0)[lt]{\lineheight{1.25}\smash{\begin{tabular}[t]{l}$!K$\end{tabular}}}}%
    \put(0.19677838,0.30969084){\color[rgb]{0,0,0}\makebox(0,0)[lt]{\lineheight{1.25}\smash{\begin{tabular}[t]{l}$!X$\end{tabular}}}}%
    \put(0.07925084,0.16472617){\color[rgb]{0,0,0}\makebox(0,0)[lt]{\lineheight{1.25}\smash{\begin{tabular}[t]{l}$!Y_1$\end{tabular}}}}%
    \put(0.31119836,0.16738921){\color[rgb]{0,0,0}\makebox(0,0)[lt]{\lineheight{1.25}\smash{\begin{tabular}[t]{l}$!Y_2$\end{tabular}}}}%
    \put(0.1053488,0.11774535){\color[rgb]{0,0,0}\makebox(0,0)[lt]{\lineheight{1.25}\smash{\begin{tabular}[t]{l}$!N_1$\end{tabular}}}}%
    \put(0.27657782,0.11572414){\color[rgb]{0,0,0}\makebox(0,0)[lt]{\lineheight{1.25}\smash{\begin{tabular}[t]{l}$!N_2$\end{tabular}}}}%
    \put(0.71986086,0.31423781){\color[rgb]{0,0,0}\makebox(0,0)[lt]{\lineheight{1.25}\smash{\begin{tabular}[t]{l}$!N_1$\end{tabular}}}}%
    \put(0.64083775,0.16775491){\color[rgb]{0,0,0}\makebox(0,0)[lt]{\lineheight{1.25}\smash{\begin{tabular}[t]{l}$!N_2$\end{tabular}}}}%
    \put(0.65659085,0.33563833){\color[rgb]{0,0,0}\makebox(0,0)[lt]{\lineheight{1.25}\smash{\begin{tabular}[t]{l}$!K$\end{tabular}}}}%
    \put(0.20161703,0.00030894){\color[rgb]{0,0,0}\makebox(0,0)[lt]{\lineheight{1.25}\smash{\begin{tabular}[t]{l}a\end{tabular}}}}%
    \put(0.77341409,0.00030894){\color[rgb]{0,0,0}\makebox(0,0)[lt]{\lineheight{1.25}\smash{\begin{tabular}[t]{l}b\end{tabular}}}}%
  \end{picture}%
\endgroup%

%% file: fragment-stability-trigon-non-active.eps_tex
\begingroup%
  \makeatletter%
  \providecommand\color[2][]{%
    \errmessage{(Inkscape) Color is used for the text in Inkscape, but the package 'color.sty' is not loaded}%
    \renewcommand\color[2][]{}%
  }%
  \providecommand\transparent[1]{%
    \errmessage{(Inkscape) Transparency is used (non-zero) for the text in Inkscape, but the package 'transparent.sty' is not loaded}%
    \renewcommand\transparent[1]{}%
  }%
  \providecommand\rotatebox[2]{#2}%
  \newcommand*\fsize{\dimexpr\f@size pt\relax}%
  \newcommand*\lineheight[1]{\fontsize{\fsize}{#1\fsize}\selectfont}%
  \ifx\svgwidth\undefined%
    \setlength{\unitlength}{248.54707002bp}%
    \ifx\svgscale\undefined%
      \relax%
    \else%
      \setlength{\unitlength}{\unitlength * \real{\svgscale}}%
    \fi%
  \else%
    \setlength{\unitlength}{\svgwidth}%
  \fi%
  \global\let\svgwidth\undefined%
  \global\let\svgscale\undefined%
  \makeatother%
  \begin{picture}(1,0.59334167)%
    \lineheight{1}%
    \setlength\tabcolsep{0pt}%
    \put(0,0){\includegraphics[width=\unitlength]{fragment-stability-trigon-non-active.eps}}%
    \put(0.30773571,0.56116246){\color[rgb]{0,0,0}\makebox(0,0)[lt]{\lineheight{1.25}\smash{\begin{tabular}[t]{l}$!K$\end{tabular}}}}%
    \put(0.46669375,0.55461986){\color[rgb]{0,0,0}\makebox(0,0)[lt]{\lineheight{1.25}\smash{\begin{tabular}[t]{l}$!X$\end{tabular}}}}%
    \put(0.02161585,0.54042965){\color[rgb]{0,0,0}\makebox(0,0)[lt]{\lineheight{1.25}\smash{\begin{tabular}[t]{l}$!v_1$\end{tabular}}}}%
    \put(0.13625113,0.42829197){\color[rgb]{0,0,0}\makebox(0,0)[lt]{\lineheight{1.25}\smash{\begin{tabular}[t]{l}$!Q_1$\end{tabular}}}}%
    \put(-0.00095035,0.34394533){\color[rgb]{0,0,0}\makebox(0,0)[lt]{\lineheight{1.25}\smash{\begin{tabular}[t]{l}$!w_1$\end{tabular}}}}%
    \put(0.23740378,0.21096087){\color[rgb]{0,0,0}\makebox(0,0)[lt]{\lineheight{1.25}\smash{\begin{tabular}[t]{l}$!Y_1$\end{tabular}}}}%
    \put(0.40054254,0.007985){\color[rgb]{0,0,0}\makebox(0,0)[lt]{\lineheight{1.25}\smash{\begin{tabular}[t]{l}$!v_2$\end{tabular}}}}%
    \put(0.48676168,0.17586267){\color[rgb]{0,0,0}\makebox(0,0)[lt]{\lineheight{1.25}\smash{\begin{tabular}[t]{l}$!Q_2$\end{tabular}}}}%
    \put(0.59713806,0.01472984){\color[rgb]{0,0,0}\makebox(0,0)[lt]{\lineheight{1.25}\smash{\begin{tabular}[t]{l}$!w_2$\end{tabular}}}}%
    \put(0.70166835,0.20920879){\color[rgb]{0,0,0}\makebox(0,0)[lt]{\lineheight{1.25}\smash{\begin{tabular}[t]{l}$!Y_2$\end{tabular}}}}%
    \put(0.90013224,0.35043661){\color[rgb]{0,0,0}\makebox(0,0)[lt]{\lineheight{1.25}\smash{\begin{tabular}[t]{l}$!v_3$\end{tabular}}}}%
    \put(0.79869454,0.4313424){\color[rgb]{0,0,0}\makebox(0,0)[lt]{\lineheight{1.25}\smash{\begin{tabular}[t]{l}$!Q_3$\end{tabular}}}}%
    \put(0.8835264,0.53450105){\color[rgb]{0,0,0}\makebox(0,0)[lt]{\lineheight{1.25}\smash{\begin{tabular}[t]{l}$!w_3$\end{tabular}}}}%
    \put(0.46717932,0.39349487){\color[rgb]{0,0,0}\makebox(0,0)[lt]{\lineheight{1.25}\smash{\begin{tabular}[t]{l}$!Z$\end{tabular}}}}%
  \end{picture}%
\endgroup%

%% file: fragment-stability-cyclic1.eps_tex
\begingroup%
  \makeatletter%
  \providecommand\color[2][]{%
    \errmessage{(Inkscape) Color is used for the text in Inkscape, but the package 'color.sty' is not loaded}%
    \renewcommand\color[2][]{}%
  }%
  \providecommand\transparent[1]{%
    \errmessage{(Inkscape) Transparency is used (non-zero) for the text in Inkscape, but the package 'transparent.sty' is not loaded}%
    \renewcommand\transparent[1]{}%
  }%
  \providecommand\rotatebox[2]{#2}%
  \newcommand*\fsize{\dimexpr\f@size pt\relax}%
  \newcommand*\lineheight[1]{\fontsize{\fsize}{#1\fsize}\selectfont}%
  \ifx\svgwidth\undefined%
    \setlength{\unitlength}{468.57723592bp}%
    \ifx\svgscale\undefined%
      \relax%
    \else%
      \setlength{\unitlength}{\unitlength * \real{\svgscale}}%
    \fi%
  \else%
    \setlength{\unitlength}{\svgwidth}%
  \fi%
  \global\let\svgwidth\undefined%
  \global\let\svgscale\undefined%
  \makeatother%
  \begin{picture}(1,0.15896402)%
    \lineheight{1}%
    \setlength\tabcolsep{0pt}%
    \put(0,0){\includegraphics[width=\unitlength]{fragment-stability-cyclic1.eps}}%
    \put(0.39038103,0.14142631){\color[rgb]{0,0,0}\makebox(0,0)[lt]{\lineheight{1.25}\smash{\begin{tabular}[t]{l}$!N_{11}^{(i)}$\end{tabular}}}}%
    \put(0.46477883,0.14142631){\color[rgb]{0,0,0}\makebox(0,0)[lt]{\lineheight{1.25}\smash{\begin{tabular}[t]{l}$!N_{12}^{(i)}$\end{tabular}}}}%
    \put(0.54441231,0.14166077){\color[rgb]{0,0,0}\makebox(0,0)[lt]{\lineheight{1.25}\smash{\begin{tabular}[t]{l}$!N_{13}^{(i)}$\end{tabular}}}}%
    \put(0.39869173,0.02256503){\color[rgb]{0,0,0}\makebox(0,0)[lt]{\lineheight{1.25}\smash{\begin{tabular}[t]{l}$!N_{21}^{(i)}$\end{tabular}}}}%
    \put(0.45748999,0.0274287){\color[rgb]{0,0,0}\makebox(0,0)[lt]{\lineheight{1.25}\smash{\begin{tabular}[t]{l}$!N_{22}^{(i)}$\end{tabular}}}}%
    \put(0.52175725,0.01786295){\color[rgb]{0,0,0}\makebox(0,0)[lt]{\lineheight{1.25}\smash{\begin{tabular}[t]{l}$!N_{23}^{(i)}$\end{tabular}}}}%
    \put(0.36541773,0.01608434){\color[rgb]{0,0,0}\makebox(0,0)[lt]{\lineheight{1.25}\smash{\begin{tabular}[t]{l}$!Y_i$\end{tabular}}}}%
    \put(0.73623636,0.14142631){\color[rgb]{0,0,0}\makebox(0,0)[lt]{\lineheight{1.25}\smash{\begin{tabular}[t]{l}$!N_{11}^{(i+1)}$\end{tabular}}}}%
    \put(0.8106342,0.14142631){\color[rgb]{0,0,0}\makebox(0,0)[lt]{\lineheight{1.25}\smash{\begin{tabular}[t]{l}$!N_{12}^{(i+1)}$\end{tabular}}}}%
    \put(0.89026763,0.14166077){\color[rgb]{0,0,0}\makebox(0,0)[lt]{\lineheight{1.25}\smash{\begin{tabular}[t]{l}$!N_{13}^{(i+1)}$\end{tabular}}}}%
    \put(0.74454705,0.02256503){\color[rgb]{0,0,0}\makebox(0,0)[lt]{\lineheight{1.25}\smash{\begin{tabular}[t]{l}$!N_{21}^{(i+1)}$\end{tabular}}}}%
    \put(0.8082066,0.025314){\color[rgb]{0,0,0}\makebox(0,0)[lt]{\lineheight{1.25}\smash{\begin{tabular}[t]{l}$!N_{22}^{(i+1)}$\end{tabular}}}}%
    \put(0.87212124,0.01588221){\color[rgb]{0,0,0}\makebox(0,0)[lt]{\lineheight{1.25}\smash{\begin{tabular}[t]{l}$!N_{23}^{(i+1)}$\end{tabular}}}}%
    \put(0.70274733,0.01204886){\color[rgb]{0,0,0}\makebox(0,0)[lt]{\lineheight{1.25}\smash{\begin{tabular}[t]{l}$!Y_{i+1}$\end{tabular}}}}%
    \put(0.97854697,0.12201332){\color[rgb]{0,0,0}\makebox(0,0)[lt]{\lineheight{1.25}\smash{\begin{tabular}[t]{l}$\bar{!X}$\end{tabular}}}}%
    \put(0.63709209,0.07539763){\color[rgb]{0,0,0}\makebox(0,0)[lt]{\lineheight{1.25}\smash{\begin{tabular}[t]{l}$!u_i$\end{tabular}}}}%
    \put(0.27954157,0.07067384){\color[rgb]{0,0,0}\makebox(0,0)[lt]{\lineheight{1.25}\smash{\begin{tabular}[t]{l}$!u_{i-1}$\end{tabular}}}}%
  \end{picture}%
\endgroup%

%% file: bigon-stability.eps_tex
\begingroup%
  \makeatletter%
  \providecommand\color[2][]{%
    \errmessage{(Inkscape) Color is used for the text in Inkscape, but the package 'color.sty' is not loaded}%
    \renewcommand\color[2][]{}%
  }%
  \providecommand\transparent[1]{%
    \errmessage{(Inkscape) Transparency is used (non-zero) for the text in Inkscape, but the package 'transparent.sty' is not loaded}%
    \renewcommand\transparent[1]{}%
  }%
  \providecommand\rotatebox[2]{#2}%
  \newcommand*\fsize{\dimexpr\f@size pt\relax}%
  \newcommand*\lineheight[1]{\fontsize{\fsize}{#1\fsize}\selectfont}%
  \ifx\svgwidth\undefined%
    \setlength{\unitlength}{461.05900874bp}%
    \ifx\svgscale\undefined%
      \relax%
    \else%
      \setlength{\unitlength}{\unitlength * \real{\svgscale}}%
    \fi%
  \else%
    \setlength{\unitlength}{\svgwidth}%
  \fi%
  \global\let\svgwidth\undefined%
  \global\let\svgscale\undefined%
  \makeatother%
  \begin{picture}(1,0.25415538)%
    \lineheight{1}%
    \setlength\tabcolsep{0pt}%
    \put(0,0){\includegraphics[width=\unitlength]{bigon-stability.eps}}%
    \put(0.13313067,0.22929164){\color[rgb]{0,0,0}\makebox(0,0)[lt]{\lineheight{1.25}\smash{\begin{tabular}[t]{l}$!a_1$\end{tabular}}}}%
    \put(0.1708827,0.2305793){\color[rgb]{0,0,0}\makebox(0,0)[lt]{\lineheight{1.25}\smash{\begin{tabular}[t]{l}$!X_1$\end{tabular}}}}%
    \put(0.27578571,0.23680825){\color[rgb]{0,0,0}\makebox(0,0)[lt]{\lineheight{1.25}\smash{\begin{tabular}[t]{l}$!X$\end{tabular}}}}%
    \put(0.13174386,0.15584063){\color[rgb]{0,0,0}\makebox(0,0)[lt]{\lineheight{1.25}\smash{\begin{tabular}[t]{l}$!a_2$\end{tabular}}}}%
    \put(0.17306554,0.15050193){\color[rgb]{0,0,0}\makebox(0,0)[lt]{\lineheight{1.25}\smash{\begin{tabular}[t]{l}$!Y_1$\end{tabular}}}}%
    \put(0.23266604,0.15870223){\color[rgb]{0,0,0}\makebox(0,0)[lt]{\lineheight{1.25}\smash{\begin{tabular}[t]{l}$!S$\end{tabular}}}}%
    \put(0.08691687,0.20179115){\color[rgb]{0,0,0}\makebox(0,0)[lt]{\lineheight{1.25}\smash{\begin{tabular}[t]{l}$!u_1'$\end{tabular}}}}%
    \put(0.2153418,0.2022992){\color[rgb]{0,0,0}\makebox(0,0)[lt]{\lineheight{1.25}\smash{\begin{tabular}[t]{l}$!u_2'$\end{tabular}}}}%
    \put(0.0066074,0.17920895){\color[rgb]{0,0,0}\makebox(0,0)[lt]{\lineheight{1.25}\smash{\begin{tabular}[t]{l}$!v_{11}$\end{tabular}}}}%
    \put(0.00377584,0.04701979){\color[rgb]{0,0,0}\makebox(0,0)[lt]{\lineheight{1.25}\smash{\begin{tabular}[t]{l}$!v_{12}$\end{tabular}}}}%
    \put(0.30058354,0.05604728){\color[rgb]{0,0,0}\makebox(0,0)[lt]{\lineheight{1.25}\smash{\begin{tabular}[t]{l}$!v_{21}$\end{tabular}}}}%
    \put(0.2965418,0.17530263){\color[rgb]{0,0,0}\makebox(0,0)[lt]{\lineheight{1.25}\smash{\begin{tabular}[t]{l}$!v_{22}$\end{tabular}}}}%
    \put(0.15480512,0.04039857){\color[rgb]{0,0,0}\makebox(0,0)[lt]{\lineheight{1.25}\smash{\begin{tabular}[t]{l}$!T$\end{tabular}}}}%
    \put(0.08461118,0.10748679){\color[rgb]{0,0,0}\makebox(0,0)[lt]{\lineheight{1.25}\smash{\begin{tabular}[t]{l}$!R_1$\end{tabular}}}}%
    \put(0.21957989,0.10812882){\color[rgb]{0,0,0}\makebox(0,0)[lt]{\lineheight{1.25}\smash{\begin{tabular}[t]{l}$!R_2$\end{tabular}}}}%
    \put(0.04977966,0.1215466){\color[rgb]{0,0,0}\makebox(0,0)[lt]{\lineheight{1.25}\smash{\begin{tabular}[t]{l}$!b$\end{tabular}}}}%
    \put(0.03946054,0.01308557){\color[rgb]{0,0,0}\makebox(0,0)[lt]{\lineheight{1.25}\smash{\begin{tabular}[t]{l}$!a_3$\end{tabular}}}}%
    \put(0.46938757,0.04143854){\color[rgb]{0,0,0}\makebox(0,0)[lt]{\lineheight{1.25}\smash{\begin{tabular}[t]{l}$!a_3=!b$\end{tabular}}}}%
    \put(0.9535117,0.11804841){\color[rgb]{0,0,0}\makebox(0,0)[lt]{\lineheight{1.25}\smash{\begin{tabular}[t]{l}$!b$\end{tabular}}}}%
    \put(0.97459886,0.02013511){\color[rgb]{0,0,0}\makebox(0,0)[lt]{\lineheight{1.25}\smash{\begin{tabular}[t]{l}$!a_3$\end{tabular}}}}%
  \end{picture}%
\endgroup%

%% file: trigon-stability-claim1.eps_tex
\begingroup%
  \makeatletter%
  \providecommand\color[2][]{%
    \errmessage{(Inkscape) Color is used for the text in Inkscape, but the package 'color.sty' is not loaded}%
    \renewcommand\color[2][]{}%
  }%
  \providecommand\transparent[1]{%
    \errmessage{(Inkscape) Transparency is used (non-zero) for the text in Inkscape, but the package 'transparent.sty' is not loaded}%
    \renewcommand\transparent[1]{}%
  }%
  \providecommand\rotatebox[2]{#2}%
  \newcommand*\fsize{\dimexpr\f@size pt\relax}%
  \newcommand*\lineheight[1]{\fontsize{\fsize}{#1\fsize}\selectfont}%
  \ifx\svgwidth\undefined%
    \setlength{\unitlength}{266.80481575bp}%
    \ifx\svgscale\undefined%
      \relax%
    \else%
      \setlength{\unitlength}{\unitlength * \real{\svgscale}}%
    \fi%
  \else%
    \setlength{\unitlength}{\svgwidth}%
  \fi%
  \global\let\svgwidth\undefined%
  \global\let\svgscale\undefined%
  \makeatother%
  \begin{picture}(1,0.65128328)%
    \lineheight{1}%
    \setlength\tabcolsep{0pt}%
    \put(0,0){\includegraphics[width=\unitlength]{trigon-stability-claim1.eps}}%
    \put(0.41031159,0.60933636){\color[rgb]{0,0,0}\makebox(0,0)[lt]{\lineheight{1.25}\smash{\begin{tabular}[t]{l}$!X'$\end{tabular}}}}%
    \put(0.18208225,0.62114142){\color[rgb]{0,0,0}\makebox(0,0)[lt]{\lineheight{1.25}\smash{\begin{tabular}[t]{l}$!w_1$\end{tabular}}}}%
    \put(0.64669619,0.61964742){\color[rgb]{0,0,0}\makebox(0,0)[lt]{\lineheight{1.25}\smash{\begin{tabular}[t]{l}$!w_2$\end{tabular}}}}%
    \put(0.18905294,0.51386907){\color[rgb]{0,0,0}\makebox(0,0)[lt]{\lineheight{1.25}\smash{\begin{tabular}[t]{l}$!S$\end{tabular}}}}%
    \put(0.41043009,0.50150045){\color[rgb]{0,0,0}\makebox(0,0)[lt]{\lineheight{1.25}\smash{\begin{tabular}[t]{l}$!Y'$\end{tabular}}}}%
    \put(0.75552044,0.61697255){\color[rgb]{0,0,0}\makebox(0,0)[lt]{\lineheight{1.25}\smash{\begin{tabular}[t]{l}$!u_{22}$\end{tabular}}}}%
    \put(0.95676776,0.49601048){\color[rgb]{0,0,0}\makebox(0,0)[lt]{\lineheight{1.25}\smash{\begin{tabular}[t]{l}$!R$\end{tabular}}}}%
    \put(0.02878232,0.48394291){\color[rgb]{0,0,0}\makebox(0,0)[lt]{\lineheight{1.25}\smash{\begin{tabular}[t]{l}$!v_{11}$\end{tabular}}}}%
    \put(0.19091359,0.38895804){\color[rgb]{0,0,0}\makebox(0,0)[lt]{\lineheight{1.25}\smash{\begin{tabular}[t]{l}$!P_1$\end{tabular}}}}%
    \put(0.01167891,0.31255528){\color[rgb]{0,0,0}\makebox(0,0)[lt]{\lineheight{1.25}\smash{\begin{tabular}[t]{l}$!v_{12}$\end{tabular}}}}%
    \put(0.24105785,0.19308085){\color[rgb]{0,0,0}\makebox(0,0)[lt]{\lineheight{1.25}\smash{\begin{tabular}[t]{l}$!T_1$\end{tabular}}}}%
    \put(0.3820352,0.01250515){\color[rgb]{0,0,0}\makebox(0,0)[lt]{\lineheight{1.25}\smash{\begin{tabular}[t]{l}$!v_{21}$\end{tabular}}}}%
    \put(0.46506456,0.16221109){\color[rgb]{0,0,0}\makebox(0,0)[lt]{\lineheight{1.25}\smash{\begin{tabular}[t]{l}$!P_2$\end{tabular}}}}%
    \put(0.55287771,0.01578613){\color[rgb]{0,0,0}\makebox(0,0)[lt]{\lineheight{1.25}\smash{\begin{tabular}[t]{l}$!v_{22}$\end{tabular}}}}%
    \put(0.65175802,0.19329606){\color[rgb]{0,0,0}\makebox(0,0)[lt]{\lineheight{1.25}\smash{\begin{tabular}[t]{l}$!T_2$\end{tabular}}}}%
    \put(0.82893682,0.32031499){\color[rgb]{0,0,0}\makebox(0,0)[lt]{\lineheight{1.25}\smash{\begin{tabular}[t]{l}$!v_{31}$\end{tabular}}}}%
    \put(0.77792532,0.40455949){\color[rgb]{0,0,0}\makebox(0,0)[lt]{\lineheight{1.25}\smash{\begin{tabular}[t]{l}$!P_3$\end{tabular}}}}%
    \put(0.68856818,0.45545269){\color[rgb]{0,0,0}\makebox(0,0)[lt]{\lineheight{1.25}\smash{\begin{tabular}[t]{l}$!v_{32}$\end{tabular}}}}%
    \put(0.42946588,0.3240591){\color[rgb]{0,0,0}\makebox(0,0)[lt]{\lineheight{1.25}\smash{\begin{tabular}[t]{l}$!S_1$\end{tabular}}}}%
  \end{picture}%
\endgroup%

%% file: stability-cyclic-monogon-claim1.eps_tex
\begingroup%
  \makeatletter%
  \providecommand\color[2][]{%
    \errmessage{(Inkscape) Color is used for the text in Inkscape, but the package 'color.sty' is not loaded}%
    \renewcommand\color[2][]{}%
  }%
  \providecommand\transparent[1]{%
    \errmessage{(Inkscape) Transparency is used (non-zero) for the text in Inkscape, but the package 'transparent.sty' is not loaded}%
    \renewcommand\transparent[1]{}%
  }%
  \providecommand\rotatebox[2]{#2}%
  \newcommand*\fsize{\dimexpr\f@size pt\relax}%
  \newcommand*\lineheight[1]{\fontsize{\fsize}{#1\fsize}\selectfont}%
  \ifx\svgwidth\undefined%
    \setlength{\unitlength}{450.55077734bp}%
    \ifx\svgscale\undefined%
      \relax%
    \else%
      \setlength{\unitlength}{\unitlength * \real{\svgscale}}%
    \fi%
  \else%
    \setlength{\unitlength}{\svgwidth}%
  \fi%
  \global\let\svgwidth\undefined%
  \global\let\svgscale\undefined%
  \makeatother%
  \begin{picture}(1,0.26551184)%
    \lineheight{1}%
    \setlength\tabcolsep{0pt}%
    \put(0,0){\includegraphics[width=\unitlength]{stability-cyclic-monogon-claim1.eps}}%
    \put(0.34956689,0.24776013){\color[rgb]{0,0,0}\makebox(0,0)[lt]{\lineheight{1.25}\smash{\begin{tabular}[t]{l}$!X$\end{tabular}}}}%
    \put(0.22690165,0.22309937){\color[rgb]{0,0,0}\makebox(0,0)[lt]{\lineheight{1.25}\smash{\begin{tabular}[t]{l}$!u_1$\end{tabular}}}}%
    \put(0.3500358,0.17304322){\color[rgb]{0,0,0}\makebox(0,0)[lt]{\lineheight{1.25}\smash{\begin{tabular}[t]{l}$!Y$\end{tabular}}}}%
    \put(0.47420105,0.22213381){\color[rgb]{0,0,0}\makebox(0,0)[lt]{\lineheight{1.25}\smash{\begin{tabular}[t]{l}$!u_2$\end{tabular}}}}%
    \put(0.28508978,0.13290212){\color[rgb]{0,0,0}\makebox(0,0)[lt]{\lineheight{1.25}\smash{\begin{tabular}[t]{l}$!Z_i$\end{tabular}}}}%
    \put(0.68418291,0.13231464){\color[rgb]{0,0,0}\makebox(0,0)[lt]{\lineheight{1.25}\smash{\begin{tabular}[t]{l}$!Z_{i+1}$\end{tabular}}}}%
    \put(0.22461776,0.07106746){\color[rgb]{0,0,0}\makebox(0,0)[lt]{\lineheight{1.25}\smash{\begin{tabular}[t]{l}$!T_i$\end{tabular}}}}%
    \put(0.62183352,0.07268531){\color[rgb]{0,0,0}\makebox(0,0)[lt]{\lineheight{1.25}\smash{\begin{tabular}[t]{l}$!T_{i+1}$\end{tabular}}}}%
    \put(0.3750531,0.00101184){\color[rgb]{0,0,0}\makebox(0,0)[lt]{\lineheight{1.25}\smash{\begin{tabular}[t]{l}$!w_1^{(i)}$\end{tabular}}}}%
    \put(0.48154733,0.00101184){\color[rgb]{0,0,0}\makebox(0,0)[lt]{\lineheight{1.25}\smash{\begin{tabular}[t]{l}$!w_2^{(i)}$\end{tabular}}}}%
    \put(0.77404763,0.00101184){\color[rgb]{0,0,0}\makebox(0,0)[lt]{\lineheight{1.25}\smash{\begin{tabular}[t]{l}$!w_1^{(i+1)}$\end{tabular}}}}%
    \put(0.87750631,0.00101184){\color[rgb]{0,0,0}\makebox(0,0)[lt]{\lineheight{1.25}\smash{\begin{tabular}[t]{l}$!w_2^{(i+1)}$\end{tabular}}}}%
    \put(0.4092848,0.06736154){\color[rgb]{0,0,0}\makebox(0,0)[lt]{\lineheight{1.25}\smash{\begin{tabular}[t]{l}$!R_i$\end{tabular}}}}%
    \put(0.7933872,0.06683152){\color[rgb]{0,0,0}\makebox(0,0)[lt]{\lineheight{1.25}\smash{\begin{tabular}[t]{l}$!R_{i+1}$\end{tabular}}}}%
    \put(0.98088733,0.19185593){\color[rgb]{0,0,0}\makebox(0,0)[lt]{\lineheight{1.25}\smash{\begin{tabular}[t]{l}$\bar{!S}$\end{tabular}}}}%
  \end{picture}%
\endgroup%

%% file: closeness-order.eps_tex
\begingroup%
  \makeatletter%
  \providecommand\color[2][]{%
    \errmessage{(Inkscape) Color is used for the text in Inkscape, but the package 'color.sty' is not loaded}%
    \renewcommand\color[2][]{}%
  }%
  \providecommand\transparent[1]{%
    \errmessage{(Inkscape) Transparency is used (non-zero) for the text in Inkscape, but the package 'transparent.sty' is not loaded}%
    \renewcommand\transparent[1]{}%
  }%
  \providecommand\rotatebox[2]{#2}%
  \newcommand*\fsize{\dimexpr\f@size pt\relax}%
  \newcommand*\lineheight[1]{\fontsize{\fsize}{#1\fsize}\selectfont}%
  \ifx\svgwidth\undefined%
    \setlength{\unitlength}{194.6423595bp}%
    \ifx\svgscale\undefined%
      \relax%
    \else%
      \setlength{\unitlength}{\unitlength * \real{\svgscale}}%
    \fi%
  \else%
    \setlength{\unitlength}{\svgwidth}%
  \fi%
  \global\let\svgwidth\undefined%
  \global\let\svgscale\undefined%
  \makeatother%
  \begin{picture}(1,0.4193295)%
    \lineheight{1}%
    \setlength\tabcolsep{0pt}%
    \put(0,0){\includegraphics[width=\unitlength]{closeness-order.eps}}%
    \put(0.31818032,0.35362153){\color[rgb]{0,0,0}\makebox(0,0)[lt]{\lineheight{1.25}\smash{\begin{tabular}[t]{l}$!X_1$\end{tabular}}}}%
    \put(0.58665811,0.3517962){\color[rgb]{0,0,0}\makebox(0,0)[lt]{\lineheight{1.25}\smash{\begin{tabular}[t]{l}$!X_2$\end{tabular}}}}%
    \put(0.31703144,0.03049255){\color[rgb]{0,0,0}\makebox(0,0)[lt]{\lineheight{1.25}\smash{\begin{tabular}[t]{l}$!Y_1$\end{tabular}}}}%
    \put(0.58339299,0.034269){\color[rgb]{0,0,0}\makebox(0,0)[lt]{\lineheight{1.25}\smash{\begin{tabular}[t]{l}$!Y_2$\end{tabular}}}}%
    \put(0.29679006,0.18780249){\color[rgb]{0,0,0}\makebox(0,0)[lt]{\lineheight{1.25}\smash{\begin{tabular}[t]{l}$!u_1$\end{tabular}}}}%
    \put(0.61016895,0.19886117){\color[rgb]{0,0,0}\makebox(0,0)[lt]{\lineheight{1.25}\smash{\begin{tabular}[t]{l}$!u_2$\end{tabular}}}}%
    \put(0.442068,0.22197432){\color[rgb]{0,0,0}\makebox(0,0)[lt]{\lineheight{1.25}\smash{\begin{tabular}[t]{l}$!u_3$\end{tabular}}}}%
  \end{picture}%
\endgroup%

%% file: closeness-order-prev.eps_tex
\begingroup%
  \makeatletter%
  \providecommand\color[2][]{%
    \errmessage{(Inkscape) Color is used for the text in Inkscape, but the package 'color.sty' is not loaded}%
    \renewcommand\color[2][]{}%
  }%
  \providecommand\transparent[1]{%
    \errmessage{(Inkscape) Transparency is used (non-zero) for the text in Inkscape, but the package 'transparent.sty' is not loaded}%
    \renewcommand\transparent[1]{}%
  }%
  \providecommand\rotatebox[2]{#2}%
  \newcommand*\fsize{\dimexpr\f@size pt\relax}%
  \newcommand*\lineheight[1]{\fontsize{\fsize}{#1\fsize}\selectfont}%
  \ifx\svgwidth\undefined%
    \setlength{\unitlength}{221.59868315bp}%
    \ifx\svgscale\undefined%
      \relax%
    \else%
      \setlength{\unitlength}{\unitlength * \real{\svgscale}}%
    \fi%
  \else%
    \setlength{\unitlength}{\svgwidth}%
  \fi%
  \global\let\svgwidth\undefined%
  \global\let\svgscale\undefined%
  \makeatother%
  \begin{picture}(1,0.37446379)%
    \lineheight{1}%
    \setlength\tabcolsep{0pt}%
    \put(0,0){\includegraphics[width=\unitlength]{closeness-order-prev.eps}}%
    \put(0.33532506,0.11021354){\color[rgb]{0,0,0}\makebox(0,0)[lt]{\lineheight{1.25}\smash{\begin{tabular}[t]{l}$!u_1$\end{tabular}}}}%
    \put(0.67303929,0.11053238){\color[rgb]{0,0,0}\makebox(0,0)[lt]{\lineheight{1.25}\smash{\begin{tabular}[t]{l}$!u_2$\end{tabular}}}}%
    \put(0.46331785,0.20325716){\color[rgb]{0,0,0}\makebox(0,0)[lt]{\lineheight{1.25}\smash{\begin{tabular}[t]{l}$!X_1$\end{tabular}}}}%
    \put(0.78697152,0.20325716){\color[rgb]{0,0,0}\makebox(0,0)[lt]{\lineheight{1.25}\smash{\begin{tabular}[t]{l}$!X_2$\end{tabular}}}}%
    \put(0.16729393,0.00545037){\color[rgb]{0,0,0}\makebox(0,0)[lt]{\lineheight{1.25}\smash{\begin{tabular}[t]{l}$!Y_1$\end{tabular}}}}%
    \put(0.46393991,0.00545037){\color[rgb]{0,0,0}\makebox(0,0)[lt]{\lineheight{1.25}\smash{\begin{tabular}[t]{l}$!Y_2$\end{tabular}}}}%
    \put(0.95235971,0.24657787){\color[rgb]{0,0,0}\makebox(0,0)[lt]{\lineheight{1.25}\smash{\begin{tabular}[t]{l}$!v_1$\end{tabular}}}}%
    \put(0.46193607,0.33787553){\color[rgb]{0,0,0}\makebox(0,0)[lt]{\lineheight{1.25}\smash{\begin{tabular}[t]{l}$!Q$\end{tabular}}}}%
    \put(0.00584923,0.1796391){\color[rgb]{0,0,0}\makebox(0,0)[lt]{\lineheight{1.25}\smash{\begin{tabular}[t]{l}$!v_2$\end{tabular}}}}%
  \end{picture}%
\endgroup%

%% file: closeness-order-active.eps_tex
\begingroup%
  \makeatletter%
  \providecommand\color[2][]{%
    \errmessage{(Inkscape) Color is used for the text in Inkscape, but the package 'color.sty' is not loaded}%
    \renewcommand\color[2][]{}%
  }%
  \providecommand\transparent[1]{%
    \errmessage{(Inkscape) Transparency is used (non-zero) for the text in Inkscape, but the package 'transparent.sty' is not loaded}%
    \renewcommand\transparent[1]{}%
  }%
  \providecommand\rotatebox[2]{#2}%
  \newcommand*\fsize{\dimexpr\f@size pt\relax}%
  \newcommand*\lineheight[1]{\fontsize{\fsize}{#1\fsize}\selectfont}%
  \ifx\svgwidth\undefined%
    \setlength{\unitlength}{364.8657997bp}%
    \ifx\svgscale\undefined%
      \relax%
    \else%
      \setlength{\unitlength}{\unitlength * \real{\svgscale}}%
    \fi%
  \else%
    \setlength{\unitlength}{\svgwidth}%
  \fi%
  \global\let\svgwidth\undefined%
  \global\let\svgscale\undefined%
  \makeatother%
  \begin{picture}(1,0.34868039)%
    \lineheight{1}%
    \setlength\tabcolsep{0pt}%
    \put(0,0){\includegraphics[width=\unitlength]{closeness-order-active.eps}}%
    \put(0.22915641,0.22469303){\color[rgb]{0,0,0}\makebox(0,0)[lt]{\lineheight{1.25}\smash{\begin{tabular}[t]{l}$!R_1$\end{tabular}}}}%
    \put(0.3008932,0.20360326){\color[rgb]{0,0,0}\makebox(0,0)[lt]{\lineheight{1.25}\smash{\begin{tabular}[t]{l}$!S_1$\end{tabular}}}}%
    \put(0.40081958,0.23041417){\color[rgb]{0,0,0}\makebox(0,0)[lt]{\lineheight{1.25}\smash{\begin{tabular}[t]{l}$!R_2$\end{tabular}}}}%
    \put(0.63465577,0.1558148){\color[rgb]{0,0,0}\makebox(0,0)[lt]{\lineheight{1.25}\smash{\begin{tabular}[t]{l}$!S_2$\end{tabular}}}}%
    \put(0.80135475,0.21572348){\color[rgb]{0,0,0}\makebox(0,0)[lt]{\lineheight{1.25}\smash{\begin{tabular}[t]{l}$!R_3$\end{tabular}}}}%
    \put(0.91957288,0.17935563){\color[rgb]{0,0,0}\makebox(0,0)[lt]{\lineheight{1.25}\smash{\begin{tabular}[t]{l}$!X_2$\end{tabular}}}}%
    \put(0.08675011,0.00418555){\color[rgb]{0,0,0}\makebox(0,0)[lt]{\lineheight{1.25}\smash{\begin{tabular}[t]{l}$!Y_1$\end{tabular}}}}%
    \put(0.54216194,0.04745552){\color[rgb]{0,0,0}\makebox(0,0)[lt]{\lineheight{1.25}\smash{\begin{tabular}[t]{l}$!Y_2$\end{tabular}}}}%
    \put(0.25253857,0.11912489){\color[rgb]{0,0,0}\makebox(0,0)[lt]{\lineheight{1.25}\smash{\begin{tabular}[t]{l}$!u_1$\end{tabular}}}}%
    \put(0.46751486,0.31743438){\color[rgb]{0,0,0}\makebox(0,0)[lt]{\lineheight{1.25}\smash{\begin{tabular}[t]{l}$!u_3$\end{tabular}}}}%
    \put(0.80830343,0.11768805){\color[rgb]{0,0,0}\makebox(0,0)[lt]{\lineheight{1.25}\smash{\begin{tabular}[t]{l}$!u_2$\end{tabular}}}}%
  \end{picture}%
\endgroup%

%% file: reduction-step-graphs.eps_tex
\begingroup%
  \makeatletter%
  \providecommand\color[2][]{%
    \errmessage{(Inkscape) Color is used for the text in Inkscape, but the package 'color.sty' is not loaded}%
    \renewcommand\color[2][]{}%
  }%
  \providecommand\transparent[1]{%
    \errmessage{(Inkscape) Transparency is used (non-zero) for the text in Inkscape, but the package 'transparent.sty' is not loaded}%
    \renewcommand\transparent[1]{}%
  }%
  \providecommand\rotatebox[2]{#2}%
  \newcommand*\fsize{\dimexpr\f@size pt\relax}%
  \newcommand*\lineheight[1]{\fontsize{\fsize}{#1\fsize}\selectfont}%
  \ifx\svgwidth\undefined%
    \setlength{\unitlength}{451.82654554bp}%
    \ifx\svgscale\undefined%
      \relax%
    \else%
      \setlength{\unitlength}{\unitlength * \real{\svgscale}}%
    \fi%
  \else%
    \setlength{\unitlength}{\svgwidth}%
  \fi%
  \global\let\svgwidth\undefined%
  \global\let\svgscale\undefined%
  \makeatother%
  \begin{picture}(1,0.66815458)%
    \lineheight{1}%
    \setlength\tabcolsep{0pt}%
    \put(0,0){\includegraphics[width=\unitlength]{reduction-step-graphs.eps}}%
    \put(0.19011425,0.60890184){\color[rgb]{0,0,0}\makebox(0,0)[lt]{\lineheight{1.25}\smash{\begin{tabular}[t]{l}$!u_1$\end{tabular}}}}%
    \put(0.2882404,0.60865869){\color[rgb]{0,0,0}\makebox(0,0)[lt]{\lineheight{1.25}\smash{\begin{tabular}[t]{l}$!v_1$\end{tabular}}}}%
    \put(0.40019692,0.60890184){\color[rgb]{0,0,0}\makebox(0,0)[lt]{\lineheight{1.25}\smash{\begin{tabular}[t]{l}$!u_2$\end{tabular}}}}%
    \put(0.49461277,0.60865869){\color[rgb]{0,0,0}\makebox(0,0)[lt]{\lineheight{1.25}\smash{\begin{tabular}[t]{l}$!v_2$\end{tabular}}}}%
    \put(0.76660101,0.60890184){\color[rgb]{0,0,0}\makebox(0,0)[lt]{\lineheight{1.25}\smash{\begin{tabular}[t]{l}$!u_r$\end{tabular}}}}%
    \put(0.86058555,0.60865869){\color[rgb]{0,0,0}\makebox(0,0)[lt]{\lineheight{1.25}\smash{\begin{tabular}[t]{l}$!v_r$\end{tabular}}}}%
    \put(0.12353387,0.60637017){\color[rgb]{0,0,0}\makebox(0,0)[lt]{\lineheight{1.25}\smash{\begin{tabular}[t]{l}$!S_0$\end{tabular}}}}%
    \put(0.34164441,0.65020984){\color[rgb]{0,0,0}\makebox(0,0)[lt]{\lineheight{1.25}\smash{\begin{tabular}[t]{l}$!S_1$\end{tabular}}}}%
    \put(0.54685218,0.65020984){\color[rgb]{0,0,0}\makebox(0,0)[lt]{\lineheight{1.25}\smash{\begin{tabular}[t]{l}$!S_2$\end{tabular}}}}%
    \put(0.70087868,0.65020984){\color[rgb]{0,0,0}\makebox(0,0)[lt]{\lineheight{1.25}\smash{\begin{tabular}[t]{l}$!S_{r-1}^*$\end{tabular}}}}%
    \put(0.9361803,0.60452756){\color[rgb]{0,0,0}\makebox(0,0)[lt]{\lineheight{1.25}\smash{\begin{tabular}[t]{l}$!S_r$\end{tabular}}}}%
    \put(0.18998016,0.47617651){\color[rgb]{0,0,0}\makebox(0,0)[lt]{\lineheight{1.25}\smash{\begin{tabular}[t]{l}$!u_1^*$\end{tabular}}}}%
    \put(0.2882404,0.47593335){\color[rgb]{0,0,0}\makebox(0,0)[lt]{\lineheight{1.25}\smash{\begin{tabular}[t]{l}$!v_1^*$\end{tabular}}}}%
    \put(0.40019692,0.47617651){\color[rgb]{0,0,0}\makebox(0,0)[lt]{\lineheight{1.25}\smash{\begin{tabular}[t]{l}$!u_2^*$\end{tabular}}}}%
    \put(0.49426924,0.47593335){\color[rgb]{0,0,0}\makebox(0,0)[lt]{\lineheight{1.25}\smash{\begin{tabular}[t]{l}$!v_2^*$\end{tabular}}}}%
    \put(0.76660101,0.47617651){\color[rgb]{0,0,0}\makebox(0,0)[lt]{\lineheight{1.25}\smash{\begin{tabular}[t]{l}$!u_r^*$\end{tabular}}}}%
    \put(0.86058555,0.47593335){\color[rgb]{0,0,0}\makebox(0,0)[lt]{\lineheight{1.25}\smash{\begin{tabular}[t]{l}$!v_r^*$\end{tabular}}}}%
    \put(0.1509875,0.49967267){\color[rgb]{0,0,0}\makebox(0,0)[lt]{\lineheight{1.25}\smash{\begin{tabular}[t]{l}$!S_0^*$\end{tabular}}}}%
    \put(0.34164441,0.51759819){\color[rgb]{0,0,0}\makebox(0,0)[lt]{\lineheight{1.25}\smash{\begin{tabular}[t]{l}$!S_1^*$\end{tabular}}}}%
    \put(0.5063366,0.51759819){\color[rgb]{0,0,0}\makebox(0,0)[lt]{\lineheight{1.25}\smash{\begin{tabular}[t]{l}$!S_2^*$\end{tabular}}}}%
    \put(0.74141345,0.51759819){\color[rgb]{0,0,0}\makebox(0,0)[lt]{\lineheight{1.25}\smash{\begin{tabular}[t]{l}$!S_{r-1}^*$\end{tabular}}}}%
    \put(0.8977851,0.50473036){\color[rgb]{0,0,0}\makebox(0,0)[lt]{\lineheight{1.25}\smash{\begin{tabular}[t]{l}$!S_r^*$\end{tabular}}}}%
    \put(0.31169506,0.44067579){\color[rgb]{0,0,0}\makebox(0,0)[lt]{\lineheight{1.25}\smash{\begin{tabular}[t]{l}$!Q_1$\end{tabular}}}}%
    \put(0.40394048,0.42743186){\color[rgb]{0,0,0}\makebox(0,0)[lt]{\lineheight{1.25}\smash{\begin{tabular}[t]{l}$!Q_2$\end{tabular}}}}%
    \put(0.52646963,0.33018954){\color[rgb]{0,0,0}\makebox(0,0)[lt]{\lineheight{1.25}\smash{\begin{tabular}[t]{l}$!Y$\end{tabular}}}}%
    \put(0.09585499,0.44507321){\color[rgb]{0,0,0}\makebox(0,0)[lt]{\lineheight{1.25}\smash{\begin{tabular}[t]{l}$!T_0$\end{tabular}}}}%
    \put(0.15005217,0.4150549){\color[rgb]{0,0,0}\makebox(0,0)[lt]{\lineheight{1.25}\smash{\begin{tabular}[t]{l}$!U_0$\end{tabular}}}}%
    \put(0.24086303,0.33468108){\color[rgb]{0,0,0}\makebox(0,0)[lt]{\lineheight{1.25}\smash{\begin{tabular}[t]{l}$!R_1$\end{tabular}}}}%
    \put(0.34095401,0.38378429){\color[rgb]{0,0,0}\makebox(0,0)[lt]{\lineheight{1.25}\smash{\begin{tabular}[t]{l}$!U_1$\end{tabular}}}}%
    \put(0.44816568,0.33518965){\color[rgb]{0,0,0}\makebox(0,0)[lt]{\lineheight{1.25}\smash{\begin{tabular}[t]{l}$!R_2$\end{tabular}}}}%
    \put(0.53917168,0.4537481){\color[rgb]{0,0,0}\makebox(0,0)[lt]{\lineheight{1.25}\smash{\begin{tabular}[t]{l}$!U_2$\end{tabular}}}}%
    \put(0.69014488,0.45555138){\color[rgb]{0,0,0}\makebox(0,0)[lt]{\lineheight{1.25}\smash{\begin{tabular}[t]{l}$!U_{r-1}$\end{tabular}}}}%
    \put(0.82688263,0.33716024){\color[rgb]{0,0,0}\makebox(0,0)[lt]{\lineheight{1.25}\smash{\begin{tabular}[t]{l}$!R_r$\end{tabular}}}}%
    \put(0.91415307,0.41824716){\color[rgb]{0,0,0}\makebox(0,0)[lt]{\lineheight{1.25}\smash{\begin{tabular}[t]{l}$!U_r$\end{tabular}}}}%
    \put(0.96012615,0.4403863){\color[rgb]{0,0,0}\makebox(0,0)[lt]{\lineheight{1.25}\smash{\begin{tabular}[t]{l}$!T_r$\end{tabular}}}}%
    \put(-0.0011793,0.46454639){\color[rgb]{0,0,0}\makebox(0,0)[lt]{\lineheight{1.25}\smash{\begin{tabular}[t]{l}a\end{tabular}}}}%
    \put(-0.00170208,0.16912397){\color[rgb]{0,0,0}\makebox(0,0)[lt]{\lineheight{1.25}\smash{\begin{tabular}[t]{l}b\end{tabular}}}}%
  \end{picture}%
\endgroup%

%% file: reduction-active-upper-bound.eps_tex
\begingroup%
  \makeatletter%
  \providecommand\color[2][]{%
    \errmessage{(Inkscape) Color is used for the text in Inkscape, but the package 'color.sty' is not loaded}%
    \renewcommand\color[2][]{}%
  }%
  \providecommand\transparent[1]{%
    \errmessage{(Inkscape) Transparency is used (non-zero) for the text in Inkscape, but the package 'transparent.sty' is not loaded}%
    \renewcommand\transparent[1]{}%
  }%
  \providecommand\rotatebox[2]{#2}%
  \newcommand*\fsize{\dimexpr\f@size pt\relax}%
  \newcommand*\lineheight[1]{\fontsize{\fsize}{#1\fsize}\selectfont}%
  \ifx\svgwidth\undefined%
    \setlength{\unitlength}{256.0119015bp}%
    \ifx\svgscale\undefined%
      \relax%
    \else%
      \setlength{\unitlength}{\unitlength * \real{\svgscale}}%
    \fi%
  \else%
    \setlength{\unitlength}{\svgwidth}%
  \fi%
  \global\let\svgwidth\undefined%
  \global\let\svgscale\undefined%
  \makeatother%
  \begin{picture}(1,0.48139412)%
    \lineheight{1}%
    \setlength\tabcolsep{0pt}%
    \put(0,0){\includegraphics[width=\unitlength]{reduction-active-upper-bound.eps}}%
    \put(0.29115704,0.42536006){\color[rgb]{0,0,0}\makebox(0,0)[lt]{\lineheight{1.25}\smash{\begin{tabular}[t]{l}$!K_i$\end{tabular}}}}%
    \put(0.66954448,0.38460963){\color[rgb]{0,0,0}\makebox(0,0)[lt]{\lineheight{1.25}\smash{\begin{tabular}[t]{l}$\bar{!K}_{i+1}$\end{tabular}}}}%
    \put(0.58313634,0.43201631){\color[rgb]{0,0,0}\makebox(0,0)[lt]{\lineheight{1.25}\smash{\begin{tabular}[t]{l}$!S$\end{tabular}}}}%
    \put(0.97162253,0.4501532){\color[rgb]{0,0,0}\makebox(0,0)[lt]{\lineheight{1.25}\smash{\begin{tabular}[t]{l}$!X$\end{tabular}}}}%
    \put(0.2658181,0.30224687){\color[rgb]{0,0,0}\makebox(0,0)[lt]{\lineheight{1.25}\smash{\begin{tabular}[t]{l}$\bar{!P}$\end{tabular}}}}%
    \put(0.26169605,0.16391316){\color[rgb]{0,0,0}\makebox(0,0)[lt]{\lineheight{1.25}\smash{\begin{tabular}[t]{l}$\bar{!Q}$\end{tabular}}}}%
    \put(0.31409971,0.24195129){\color[rgb]{0,0,0}\makebox(0,0)[lt]{\lineheight{1.25}\smash{\begin{tabular}[t]{l}$!W$\end{tabular}}}}%
    \put(0.5929912,0.04712904){\color[rgb]{0,0,0}\makebox(0,0)[lt]{\lineheight{1.25}\smash{\begin{tabular}[t]{l}$!V$\end{tabular}}}}%
    \put(0.31086296,0.05489171){\color[rgb]{0,0,0}\makebox(0,0)[lt]{\lineheight{1.25}\smash{\begin{tabular}[t]{l}$!N$\end{tabular}}}}%
    \put(0.89863342,0.09945367){\color[rgb]{0,0,0}\makebox(0,0)[lt]{\lineheight{1.25}\smash{\begin{tabular}[t]{l}$!L$\end{tabular}}}}%
    \put(0.96872199,0.00865013){\color[rgb]{0,0,0}\makebox(0,0)[lt]{\lineheight{1.25}\smash{\begin{tabular}[t]{l}$!Y$\end{tabular}}}}%
    \put(0.66994624,0.3127707){\color[rgb]{0,0,0}\makebox(0,0)[lt]{\lineheight{1.25}\smash{\begin{tabular}[t]{l}$!P_{i+1}$\end{tabular}}}}%
    \put(0.65618432,0.09041108){\color[rgb]{0,0,0}\makebox(0,0)[lt]{\lineheight{1.25}\smash{\begin{tabular}[t]{l}$!M_{i+1}$\end{tabular}}}}%
  \end{picture}%
\endgroup%

%% file: cyclic-reduction.eps_tex
\begingroup%
  \makeatletter%
  \providecommand\color[2][]{%
    \errmessage{(Inkscape) Color is used for the text in Inkscape, but the package 'color.sty' is not loaded}%
    \renewcommand\color[2][]{}%
  }%
  \providecommand\transparent[1]{%
    \errmessage{(Inkscape) Transparency is used (non-zero) for the text in Inkscape, but the package 'transparent.sty' is not loaded}%
    \renewcommand\transparent[1]{}%
  }%
  \providecommand\rotatebox[2]{#2}%
  \newcommand*\fsize{\dimexpr\f@size pt\relax}%
  \newcommand*\lineheight[1]{\fontsize{\fsize}{#1\fsize}\selectfont}%
  \ifx\svgwidth\undefined%
    \setlength{\unitlength}{440.31653829bp}%
    \ifx\svgscale\undefined%
      \relax%
    \else%
      \setlength{\unitlength}{\unitlength * \real{\svgscale}}%
    \fi%
  \else%
    \setlength{\unitlength}{\svgwidth}%
  \fi%
  \global\let\svgwidth\undefined%
  \global\let\svgscale\undefined%
  \makeatother%
  \begin{picture}(1,0.34016811)%
    \lineheight{1}%
    \setlength\tabcolsep{0pt}%
    \put(0,0){\includegraphics[width=\unitlength]{cyclic-reduction.eps}}%
    \put(0.45188168,0.25235831){\color[rgb]{0,0,0}\makebox(0,0)[lt]{\lineheight{1.25}\smash{\begin{tabular}[t]{l}$!X_1^{(i+1)}$\end{tabular}}}}%
    \put(0.15115992,0.18479211){\color[rgb]{0,0,0}\makebox(0,0)[lt]{\lineheight{1.25}\smash{\begin{tabular}[t]{l}$!N_1$\end{tabular}}}}%
    \put(0.27549035,0.17936171){\color[rgb]{0,0,0}\makebox(0,0)[lt]{\lineheight{1.25}\smash{\begin{tabular}[t]{l}$!N_2$\end{tabular}}}}%
    \put(0.46334287,0.18280843){\color[rgb]{0,0,0}\makebox(0,0)[lt]{\lineheight{1.25}\smash{\begin{tabular}[t]{l}$s_{Y,\bar{!Y}} !N_1$\end{tabular}}}}%
    \put(0.36125639,0.03962452){\color[rgb]{0,0,0}\makebox(0,0)[lt]{\lineheight{1.25}\smash{\begin{tabular}[t]{l}$!X_2$\end{tabular}}}}%
    \put(0.08109629,0.03873913){\color[rgb]{0,0,0}\makebox(0,0)[lt]{\lineheight{1.25}\smash{\begin{tabular}[t]{l}$!X_1^{(i)}$\end{tabular}}}}%
    \put(0.63990633,0.03873913){\color[rgb]{0,0,0}\makebox(0,0)[lt]{\lineheight{1.25}\smash{\begin{tabular}[t]{l}$!X_1^{(i)}$\end{tabular}}}}%
    \put(0.911482,0.03873913){\color[rgb]{0,0,0}\makebox(0,0)[lt]{\lineheight{1.25}\smash{\begin{tabular}[t]{l}$!X_1^{(i+1)}$\end{tabular}}}}%
    \put(0.6806319,0.22134135){\color[rgb]{0,0,0}\makebox(0,0)[lt]{\lineheight{1.25}\smash{\begin{tabular}[t]{l}$!N_1$\end{tabular}}}}%
    \put(0.76421684,0.24452629){\color[rgb]{0,0,0}\makebox(0,0)[lt]{\lineheight{1.25}\smash{\begin{tabular}[t]{l}$!N_2$\end{tabular}}}}%
    \put(0.21670568,0.28061147){\color[rgb]{0,0,0}\makebox(0,0)[lt]{\lineheight{1.25}\smash{\begin{tabular}[t]{l}$!P_i$\end{tabular}}}}%
    \put(0.78179431,0.14170624){\color[rgb]{0,0,0}\makebox(0,0)[lt]{\lineheight{1.25}\smash{\begin{tabular}[t]{l}$!P_i$\end{tabular}}}}%
    \put(0.16705021,0.32175429){\color[rgb]{0,0,0}\makebox(0,0)[lt]{\lineheight{1.25}\smash{\begin{tabular}[t]{l}$!v_1^{(i)}$\end{tabular}}}}%
    \put(0.27049092,0.32175429){\color[rgb]{0,0,0}\makebox(0,0)[lt]{\lineheight{1.25}\smash{\begin{tabular}[t]{l}$!v_2^{(i)}$\end{tabular}}}}%
    \put(0.72866003,0.32175429){\color[rgb]{0,0,0}\makebox(0,0)[lt]{\lineheight{1.25}\smash{\begin{tabular}[t]{l}$!v_1^{(i)}$\end{tabular}}}}%
    \put(0.83147623,0.32175429){\color[rgb]{0,0,0}\makebox(0,0)[lt]{\lineheight{1.25}\smash{\begin{tabular}[t]{l}$!v_2^{(i)}$\end{tabular}}}}%
    \put(0.22670854,0.00249113){\color[rgb]{0,0,0}\makebox(0,0)[lt]{\lineheight{1.25}\smash{\begin{tabular}[t]{l}a\end{tabular}}}}%
    \put(0.79063531,0.00249113){\color[rgb]{0,0,0}\makebox(0,0)[lt]{\lineheight{1.25}\smash{\begin{tabular}[t]{l}b\end{tabular}}}}%
  \end{picture}%
\endgroup%

%% file: conjugate-relator-roots.eps_tex
\begingroup%
  \makeatletter%
  \providecommand\color[2][]{%
    \errmessage{(Inkscape) Color is used for the text in Inkscape, but the package 'color.sty' is not loaded}%
    \renewcommand\color[2][]{}%
  }%
  \providecommand\transparent[1]{%
    \errmessage{(Inkscape) Transparency is used (non-zero) for the text in Inkscape, but the package 'transparent.sty' is not loaded}%
    \renewcommand\transparent[1]{}%
  }%
  \providecommand\rotatebox[2]{#2}%
  \newcommand*\fsize{\dimexpr\f@size pt\relax}%
  \newcommand*\lineheight[1]{\fontsize{\fsize}{#1\fsize}\selectfont}%
  \ifx\svgwidth\undefined%
    \setlength{\unitlength}{425.57217165bp}%
    \ifx\svgscale\undefined%
      \relax%
    \else%
      \setlength{\unitlength}{\unitlength * \real{\svgscale}}%
    \fi%
  \else%
    \setlength{\unitlength}{\svgwidth}%
  \fi%
  \global\let\svgwidth\undefined%
  \global\let\svgscale\undefined%
  \makeatother%
  \begin{picture}(1,0.40302221)%
    \lineheight{1}%
    \setlength\tabcolsep{0pt}%
    \put(0,0){\includegraphics[width=\unitlength]{conjugate-relator-roots.eps}}%
    \put(0.28611275,0.37558503){\color[rgb]{0,0,0}\makebox(0,0)[lt]{\lineheight{1.25}\smash{\begin{tabular}[t]{l}$!X_1$\end{tabular}}}}%
    \put(0.10512914,0.32674679){\color[rgb]{0,0,0}\makebox(0,0)[lt]{\lineheight{1.25}\smash{\begin{tabular}[t]{l}$!X_2$\end{tabular}}}}%
    \put(0.28774857,0.25439097){\color[rgb]{0,0,0}\makebox(0,0)[lt]{\lineheight{1.25}\smash{\begin{tabular}[t]{l}$!Z_1$\end{tabular}}}}%
    \put(0.17919782,0.09866897){\color[rgb]{0,0,0}\makebox(0,0)[lt]{\lineheight{1.25}\smash{\begin{tabular}[t]{l}$!S_1$\end{tabular}}}}%
    \put(0.17787399,0.0051008){\color[rgb]{0,0,0}\makebox(0,0)[lt]{\lineheight{1.25}\smash{\begin{tabular}[t]{l}$!Q_1$\end{tabular}}}}%
    \put(0.09615337,0.07371872){\color[rgb]{0,0,0}\makebox(0,0)[lt]{\lineheight{1.25}\smash{\begin{tabular}[t]{l}$!u_1$\end{tabular}}}}%
    \put(0.25937327,0.06919136){\color[rgb]{0,0,0}\makebox(0,0)[lt]{\lineheight{1.25}\smash{\begin{tabular}[t]{l}$!v_1$\end{tabular}}}}%
    \put(0.17520283,0.15379203){\color[rgb]{0,0,0}\makebox(0,0)[lt]{\lineheight{1.25}\smash{\begin{tabular}[t]{l}$!S_2$\end{tabular}}}}%
    \put(0.22642702,0.19796584){\color[rgb]{0,0,0}\makebox(0,0)[lt]{\lineheight{1.25}\smash{\begin{tabular}[t]{l}$!u_2$\end{tabular}}}}%
    \put(0.17220229,0.23660251){\color[rgb]{0,0,0}\makebox(0,0)[lt]{\lineheight{1.25}\smash{\begin{tabular}[t]{l}$!Q_2$\end{tabular}}}}%
    \put(0.1171791,0.19812784){\color[rgb]{0,0,0}\makebox(0,0)[lt]{\lineheight{1.25}\smash{\begin{tabular}[t]{l}$!v_2$\end{tabular}}}}%
    \put(0.81171213,0.2070318){\color[rgb]{0,0,0}\makebox(0,0)[lt]{\lineheight{1.25}\smash{\begin{tabular}[t]{l}$!R$\end{tabular}}}}%
    \put(0.85665964,0.06234595){\color[rgb]{0,0,0}\makebox(0,0)[lt]{\lineheight{1.25}\smash{\begin{tabular}[t]{l}$\bar{!X}_1$\end{tabular}}}}%
    \put(0.85456026,0.34670049){\color[rgb]{0,0,0}\makebox(0,0)[lt]{\lineheight{1.25}\smash{\begin{tabular}[t]{l}$\bar{!X}_2$\end{tabular}}}}%
    \put(0.13962876,0.12526994){\color[rgb]{0,0,0}\makebox(0,0)[lt]{\lineheight{1.25}\smash{\begin{tabular}[t]{l}$!D$\end{tabular}}}}%
  \end{picture}%
\endgroup%

%% file: coarsely-periodic-overlapping.eps_tex
\begingroup%
  \makeatletter%
  \providecommand\color[2][]{%
    \errmessage{(Inkscape) Color is used for the text in Inkscape, but the package 'color.sty' is not loaded}%
    \renewcommand\color[2][]{}%
  }%
  \providecommand\transparent[1]{%
    \errmessage{(Inkscape) Transparency is used (non-zero) for the text in Inkscape, but the package 'transparent.sty' is not loaded}%
    \renewcommand\transparent[1]{}%
  }%
  \providecommand\rotatebox[2]{#2}%
  \newcommand*\fsize{\dimexpr\f@size pt\relax}%
  \newcommand*\lineheight[1]{\fontsize{\fsize}{#1\fsize}\selectfont}%
  \ifx\svgwidth\undefined%
    \setlength{\unitlength}{194.7891064bp}%
    \ifx\svgscale\undefined%
      \relax%
    \else%
      \setlength{\unitlength}{\unitlength * \real{\svgscale}}%
    \fi%
  \else%
    \setlength{\unitlength}{\svgwidth}%
  \fi%
  \global\let\svgwidth\undefined%
  \global\let\svgscale\undefined%
  \makeatother%
  \begin{picture}(1,0.64587455)%
    \lineheight{1}%
    \setlength\tabcolsep{0pt}%
    \put(0,0){\includegraphics[width=\unitlength]{coarsely-periodic-overlapping.eps}}%
    \put(0.44714223,0.00708583){\color[rgb]{0,0,0}\makebox(0,0)[lt]{\lineheight{1.25}\smash{\begin{tabular}[t]{l}$!W_0$\end{tabular}}}}%
    \put(0.39347993,0.18090588){\color[rgb]{0,0,0}\makebox(0,0)[lt]{\lineheight{1.25}\smash{\begin{tabular}[t]{l}$!W_1$\end{tabular}}}}%
    \put(0.38877759,0.35126736){\color[rgb]{0,0,0}\makebox(0,0)[lt]{\lineheight{1.25}\smash{\begin{tabular}[t]{l}$!W_2$\end{tabular}}}}%
    \put(0.44422658,0.60425051){\color[rgb]{0,0,0}\makebox(0,0)[lt]{\lineheight{1.25}\smash{\begin{tabular}[t]{l}$!W_3$\end{tabular}}}}%
  \end{picture}%
\endgroup%

%% file: length-compare-close.eps_tex
\begingroup%
  \makeatletter%
  \providecommand\color[2][]{%
    \errmessage{(Inkscape) Color is used for the text in Inkscape, but the package 'color.sty' is not loaded}%
    \renewcommand\color[2][]{}%
  }%
  \providecommand\transparent[1]{%
    \errmessage{(Inkscape) Transparency is used (non-zero) for the text in Inkscape, but the package 'transparent.sty' is not loaded}%
    \renewcommand\transparent[1]{}%
  }%
  \providecommand\rotatebox[2]{#2}%
  \newcommand*\fsize{\dimexpr\f@size pt\relax}%
  \newcommand*\lineheight[1]{\fontsize{\fsize}{#1\fsize}\selectfont}%
  \ifx\svgwidth\undefined%
    \setlength{\unitlength}{448.47585581bp}%
    \ifx\svgscale\undefined%
      \relax%
    \else%
      \setlength{\unitlength}{\unitlength * \real{\svgscale}}%
    \fi%
  \else%
    \setlength{\unitlength}{\svgwidth}%
  \fi%
  \global\let\svgwidth\undefined%
  \global\let\svgscale\undefined%
  \makeatother%
  \begin{picture}(1,0.17917325)%
    \lineheight{1}%
    \setlength\tabcolsep{0pt}%
    \put(0,0){\includegraphics[width=\unitlength]{length-compare-close.eps}}%
    \put(0.05526466,0.09582899){\color[rgb]{0,0,0}\makebox(0,0)[lt]{\lineheight{1.25}\smash{\begin{tabular}[t]{l}$!D_i$\end{tabular}}}}%
    \put(0.18812761,0.09703292){\color[rgb]{0,0,0}\makebox(0,0)[lt]{\lineheight{1.25}\smash{\begin{tabular}[t]{l}$!D_{i+1}$\end{tabular}}}}%
    \put(0.28083698,0.16133941){\color[rgb]{0,0,0}\makebox(0,0)[lt]{\lineheight{1.25}\smash{\begin{tabular}[t]{l}$\ti{!X}$\end{tabular}}}}%
    \put(0.27915657,0.03353521){\color[rgb]{0,0,0}\makebox(0,0)[lt]{\lineheight{1.25}\smash{\begin{tabular}[t]{l}$\ti{!Y}$\end{tabular}}}}%
    \put(0.12154193,0.00964158){\color[rgb]{0,0,0}\makebox(0,0)[lt]{\lineheight{1.25}\smash{\begin{tabular}[t]{l}(a)\end{tabular}}}}%
    \put(0.4581121,0.00964158){\color[rgb]{0,0,0}\makebox(0,0)[lt]{\lineheight{1.25}\smash{\begin{tabular}[t]{l}(b)\end{tabular}}}}%
    \put(0.81460311,0.00964158){\color[rgb]{0,0,0}\makebox(0,0)[lt]{\lineheight{1.25}\smash{\begin{tabular}[t]{l}(c)\end{tabular}}}}%
  \end{picture}%
\endgroup%

%% file: length-compare-close-ii.eps_tex
\begingroup%
  \makeatletter%
  \providecommand\color[2][]{%
    \errmessage{(Inkscape) Color is used for the text in Inkscape, but the package 'color.sty' is not loaded}%
    \renewcommand\color[2][]{}%
  }%
  \providecommand\transparent[1]{%
    \errmessage{(Inkscape) Transparency is used (non-zero) for the text in Inkscape, but the package 'transparent.sty' is not loaded}%
    \renewcommand\transparent[1]{}%
  }%
  \providecommand\rotatebox[2]{#2}%
  \newcommand*\fsize{\dimexpr\f@size pt\relax}%
  \newcommand*\lineheight[1]{\fontsize{\fsize}{#1\fsize}\selectfont}%
  \ifx\svgwidth\undefined%
    \setlength{\unitlength}{109.7715997bp}%
    \ifx\svgscale\undefined%
      \relax%
    \else%
      \setlength{\unitlength}{\unitlength * \real{\svgscale}}%
    \fi%
  \else%
    \setlength{\unitlength}{\svgwidth}%
  \fi%
  \global\let\svgwidth\undefined%
  \global\let\svgscale\undefined%
  \makeatother%
  \begin{picture}(1,0.62915527)%
    \lineheight{1}%
    \setlength\tabcolsep{0pt}%
    \put(0,0){\includegraphics[width=\unitlength]{length-compare-close-ii.eps}}%
    \put(0.91089791,0.55629446){\color[rgb]{0,0,0}\makebox(0,0)[lt]{\lineheight{1.25}\smash{\begin{tabular}[t]{l}$\ti{!X}$\end{tabular}}}}%
    \put(0.90734967,0.03082702){\color[rgb]{0,0,0}\makebox(0,0)[lt]{\lineheight{1.25}\smash{\begin{tabular}[t]{l}$\ti{!Y}$\end{tabular}}}}%
    \put(0.15555684,0.30577876){\color[rgb]{0,0,0}\makebox(0,0)[lt]{\lineheight{1.25}\smash{\begin{tabular}[t]{l}$\ti{!u}$\end{tabular}}}}%
    \put(0.57489615,0.28688224){\color[rgb]{0,0,0}\makebox(0,0)[lt]{\lineheight{1.25}\smash{\begin{tabular}[t]{l}$!D_1$\end{tabular}}}}%
  \end{picture}%
\endgroup%

%% file: coarsely-periodic-closeness-step.eps_tex
\begingroup%
  \makeatletter%
  \providecommand\color[2][]{%
    \errmessage{(Inkscape) Color is used for the text in Inkscape, but the package 'color.sty' is not loaded}%
    \renewcommand\color[2][]{}%
  }%
  \providecommand\transparent[1]{%
    \errmessage{(Inkscape) Transparency is used (non-zero) for the text in Inkscape, but the package 'transparent.sty' is not loaded}%
    \renewcommand\transparent[1]{}%
  }%
  \providecommand\rotatebox[2]{#2}%
  \newcommand*\fsize{\dimexpr\f@size pt\relax}%
  \newcommand*\lineheight[1]{\fontsize{\fsize}{#1\fsize}\selectfont}%
  \ifx\svgwidth\undefined%
    \setlength{\unitlength}{383.18513102bp}%
    \ifx\svgscale\undefined%
      \relax%
    \else%
      \setlength{\unitlength}{\unitlength * \real{\svgscale}}%
    \fi%
  \else%
    \setlength{\unitlength}{\svgwidth}%
  \fi%
  \global\let\svgwidth\undefined%
  \global\let\svgscale\undefined%
  \makeatother%
  \begin{picture}(1,0.25492698)%
    \lineheight{1}%
    \setlength\tabcolsep{0pt}%
    \put(0,0){\includegraphics[width=\unitlength]{coarsely-periodic-closeness-step.eps}}%
    \put(0.06231635,0.1679238){\color[rgb]{0,0,0}\makebox(0,0)[lt]{\lineheight{1.25}\smash{\begin{tabular}[t]{l}$!T_0$\end{tabular}}}}%
    \put(0.31066848,0.1679238){\color[rgb]{0,0,0}\makebox(0,0)[lt]{\lineheight{1.25}\smash{\begin{tabular}[t]{l}$!T_3$\end{tabular}}}}%
    \put(0.40598839,0.16763709){\color[rgb]{0,0,0}\makebox(0,0)[lt]{\lineheight{1.25}\smash{\begin{tabular}[t]{l}$!T_1$\end{tabular}}}}%
    \put(0.65081174,0.16763709){\color[rgb]{0,0,0}\makebox(0,0)[lt]{\lineheight{1.25}\smash{\begin{tabular}[t]{l}$!T_4$\end{tabular}}}}%
    \put(0.74663373,0.16763709){\color[rgb]{0,0,0}\makebox(0,0)[lt]{\lineheight{1.25}\smash{\begin{tabular}[t]{l}$!T_2$\end{tabular}}}}%
    \put(0.80185596,0.1679238){\color[rgb]{0,0,0}\makebox(0,0)[lt]{\lineheight{1.25}\smash{\begin{tabular}[t]{l}$s_{B,!L_0} !T_1$\end{tabular}}}}%
    \put(0.04013493,0.01206546){\color[rgb]{0,0,0}\makebox(0,0)[lt]{\lineheight{1.25}\smash{\begin{tabular}[t]{l}$s_{A,!L_0}^{-1} !S$\end{tabular}}}}%
    \put(0.33881391,0.01206546){\color[rgb]{0,0,0}\makebox(0,0)[lt]{\lineheight{1.25}\smash{\begin{tabular}[t]{l}$!S_3$\end{tabular}}}}%
    \put(0.40867013,0.01206546){\color[rgb]{0,0,0}\makebox(0,0)[lt]{\lineheight{1.25}\smash{\begin{tabular}[t]{l}$!S$\end{tabular}}}}%
    \put(0.68481882,0.01206546){\color[rgb]{0,0,0}\makebox(0,0)[lt]{\lineheight{1.25}\smash{\begin{tabular}[t]{l}$!S_4$\end{tabular}}}}%
    \put(0.74132024,0.01206546){\color[rgb]{0,0,0}\makebox(0,0)[lt]{\lineheight{1.25}\smash{\begin{tabular}[t]{l}$s_{A,!L_0} !S$\end{tabular}}}}%
    \put(0.96193663,0.20325705){\color[rgb]{0,0,0}\makebox(0,0)[lt]{\lineheight{1.25}\smash{\begin{tabular}[t]{l}$!L_1$\end{tabular}}}}%
    \put(0.96102683,0.04983508){\color[rgb]{0,0,0}\makebox(0,0)[lt]{\lineheight{1.25}\smash{\begin{tabular}[t]{l}$!L_0$\end{tabular}}}}%
  \end{picture}%
\endgroup%

%% file: hierarchical-containment-i.eps_tex
\begingroup%
  \makeatletter%
  \providecommand\color[2][]{%
    \errmessage{(Inkscape) Color is used for the text in Inkscape, but the package 'color.sty' is not loaded}%
    \renewcommand\color[2][]{}%
  }%
  \providecommand\transparent[1]{%
    \errmessage{(Inkscape) Transparency is used (non-zero) for the text in Inkscape, but the package 'transparent.sty' is not loaded}%
    \renewcommand\transparent[1]{}%
  }%
  \providecommand\rotatebox[2]{#2}%
  \newcommand*\fsize{\dimexpr\f@size pt\relax}%
  \newcommand*\lineheight[1]{\fontsize{\fsize}{#1\fsize}\selectfont}%
  \ifx\svgwidth\undefined%
    \setlength{\unitlength}{446.14361112bp}%
    \ifx\svgscale\undefined%
      \relax%
    \else%
      \setlength{\unitlength}{\unitlength * \real{\svgscale}}%
    \fi%
  \else%
    \setlength{\unitlength}{\svgwidth}%
  \fi%
  \global\let\svgwidth\undefined%
  \global\let\svgscale\undefined%
  \makeatother%
  \begin{picture}(1,0.2719025)%
    \lineheight{1}%
    \setlength\tabcolsep{0pt}%
    \put(0,0){\includegraphics[width=\unitlength]{hierarchical-containment-i.eps}}%
    \put(0.31154295,0.00604883){\color[rgb]{0,0,0}\makebox(0,0)[lt]{\lineheight{1.25}\smash{\begin{tabular}[t]{l}$s_{A_0,!L_0}^{-1} !T_0$\end{tabular}}}}%
    \put(0.46100592,0.00604883){\color[rgb]{0,0,0}\makebox(0,0)[lt]{\lineheight{1.25}\smash{\begin{tabular}[t]{l}$\hat{!T}_0$\end{tabular}}}}%
    \put(0.55486466,0.00604883){\color[rgb]{0,0,0}\makebox(0,0)[lt]{\lineheight{1.25}\smash{\begin{tabular}[t]{l}$s_{A_0,!L_0} !T_0$\end{tabular}}}}%
    \put(0.21907618,0.0858418){\color[rgb]{0,0,0}\makebox(0,0)[lt]{\lineheight{1.25}\smash{\begin{tabular}[t]{l}$s_{A_1,!L_1}^{-1} !T_1$\end{tabular}}}}%
    \put(0.35602098,0.0858418){\color[rgb]{0,0,0}\makebox(0,0)[lt]{\lineheight{1.25}\smash{\begin{tabular}[t]{l}$!U_{11}$\end{tabular}}}}%
    \put(0.47616969,0.08559554){\color[rgb]{0,0,0}\makebox(0,0)[lt]{\lineheight{1.25}\smash{\begin{tabular}[t]{l}$!T_1$\end{tabular}}}}%
    \put(0.59878875,0.08559554){\color[rgb]{0,0,0}\makebox(0,0)[lt]{\lineheight{1.25}\smash{\begin{tabular}[t]{l}$!V_{11}$\end{tabular}}}}%
    \put(0.64911333,0.0858418){\color[rgb]{0,0,0}\makebox(0,0)[lt]{\lineheight{1.25}\smash{\begin{tabular}[t]{l}$s_{A_1,!L_1} !T_1$\end{tabular}}}}%
    \put(0.12375891,0.15971494){\color[rgb]{0,0,0}\makebox(0,0)[lt]{\lineheight{1.25}\smash{\begin{tabular}[t]{l}$s_{A_2,!L_2}^{-1} !T_2$\end{tabular}}}}%
    \put(0.26237933,0.15971494){\color[rgb]{0,0,0}\makebox(0,0)[lt]{\lineheight{1.25}\smash{\begin{tabular}[t]{l}$!U_{22}$\end{tabular}}}}%
    \put(0.35669484,0.15971494){\color[rgb]{0,0,0}\makebox(0,0)[lt]{\lineheight{1.25}\smash{\begin{tabular}[t]{l}$!U_{21}$\end{tabular}}}}%
    \put(0.47485976,0.15946869){\color[rgb]{0,0,0}\makebox(0,0)[lt]{\lineheight{1.25}\smash{\begin{tabular}[t]{l}$!T_2$\end{tabular}}}}%
    \put(0.59809992,0.15946869){\color[rgb]{0,0,0}\makebox(0,0)[lt]{\lineheight{1.25}\smash{\begin{tabular}[t]{l}$!V_{21}$\end{tabular}}}}%
    \put(0.69288714,0.15946869){\color[rgb]{0,0,0}\makebox(0,0)[lt]{\lineheight{1.25}\smash{\begin{tabular}[t]{l}$!V_{22}$\end{tabular}}}}%
    \put(0.74491783,0.15971494){\color[rgb]{0,0,0}\makebox(0,0)[lt]{\lineheight{1.25}\smash{\begin{tabular}[t]{l}$s_{A_2,!L_2} !T_2$\end{tabular}}}}%
    \put(0.02708335,0.23273341){\color[rgb]{0,0,0}\makebox(0,0)[lt]{\lineheight{1.25}\smash{\begin{tabular}[t]{l}$s_{A_3,!L_3}^{-1} !T_3$\end{tabular}}}}%
    \put(0.16612341,0.23273341){\color[rgb]{0,0,0}\makebox(0,0)[lt]{\lineheight{1.25}\smash{\begin{tabular}[t]{l}$!U_{33}$\end{tabular}}}}%
    \put(0.26234903,0.23273341){\color[rgb]{0,0,0}\makebox(0,0)[lt]{\lineheight{1.25}\smash{\begin{tabular}[t]{l}$!U_{32}$\end{tabular}}}}%
    \put(0.3564452,0.23273341){\color[rgb]{0,0,0}\makebox(0,0)[lt]{\lineheight{1.25}\smash{\begin{tabular}[t]{l}$!U_{31}$\end{tabular}}}}%
    \put(0.47505756,0.23273341){\color[rgb]{0,0,0}\makebox(0,0)[lt]{\lineheight{1.25}\smash{\begin{tabular}[t]{l}$!T_3$\end{tabular}}}}%
    \put(0.5987751,0.23273341){\color[rgb]{0,0,0}\makebox(0,0)[lt]{\lineheight{1.25}\smash{\begin{tabular}[t]{l}$!V_{31}$\end{tabular}}}}%
    \put(0.69296976,0.23273341){\color[rgb]{0,0,0}\makebox(0,0)[lt]{\lineheight{1.25}\smash{\begin{tabular}[t]{l}$!V_{32}$\end{tabular}}}}%
    \put(0.78518211,0.23273341){\color[rgb]{0,0,0}\makebox(0,0)[lt]{\lineheight{1.25}\smash{\begin{tabular}[t]{l}$!V_{33}$\end{tabular}}}}%
    \put(0.83821894,0.23273341){\color[rgb]{0,0,0}\makebox(0,0)[lt]{\lineheight{1.25}\smash{\begin{tabular}[t]{l}$s_{A_3,!L_3} !T_3$\end{tabular}}}}%
    \put(0.69294711,0.02859734){\color[rgb]{0,0,0}\makebox(0,0)[lt]{\lineheight{1.25}\smash{\begin{tabular}[t]{l}$!L_0$\end{tabular}}}}%
    \put(0.78633157,0.1033433){\color[rgb]{0,0,0}\makebox(0,0)[lt]{\lineheight{1.25}\smash{\begin{tabular}[t]{l}$!L_1$\end{tabular}}}}%
    \put(0.87613589,0.17906851){\color[rgb]{0,0,0}\makebox(0,0)[lt]{\lineheight{1.25}\smash{\begin{tabular}[t]{l}$!L_2$\end{tabular}}}}%
    \put(0.96627265,0.25372918){\color[rgb]{0,0,0}\makebox(0,0)[lt]{\lineheight{1.25}\smash{\begin{tabular}[t]{l}$!L_3$\end{tabular}}}}%
  \end{picture}%
\endgroup%